\documentclass{memo-l}

\usepackage[utf8]{inputenc}
\usepackage{eucal}
\usepackage{mathtools}
\usepackage[mode=image]{standalone}
\usepackage{tikz}
\usetikzlibrary{matrix,arrows}
\usepackage{tikz-cd}
\usepackage[alphabetic]{amsrefs}
\usepackage[normalem]{ulem}
\usepackage{extarrows}
\usepackage{amssymb}
\usepackage{url}
\usepackage[shortlabels]{enumitem}
  \setitemize[1]{leftmargin=2em}
  \setenumerate[1]{leftmargin=*}
\usepackage[colorlinks=true,linktocpage=true,citecolor=blue]{hyperref}

\newtheorem{theorem}{Theorem}[section]
\newtheorem{lemma}[theorem]{Lemma}
\newtheorem{ThmB}{Theorem}
\newtheorem{corollary}[theorem]{Corollary}
\newtheorem{proposition}[theorem]{Proposition}

\theoremstyle{definition}
\newtheorem{definition}[theorem]{Definition}
\newtheorem{example}[theorem]{Example}

\newtheorem{examples}[theorem]{Example}
\newtheorem{notation}[theorem]{Notation}
\newtheorem{warning}[theorem]{Warning}

\newtheorem{variant}[theorem]{Variant}
\newtheorem{proc}[theorem]{Procedure}
\newtheorem{construction}[theorem]{Construction}
\newtheorem{problem}[theorem]{Problem}
\newtheorem{convention}[theorem]{Convention}

\theoremstyle{remark}
\newtheorem{remark}[theorem]{Remark}

\numberwithin{section}{chapter}
\numberwithin{equation}{chapter}

\newcommand{\IFF}{if and only if}
\DeclareMathOperator{\limIL}{lim}

\newcommand{\op}{\txt{op}}
\newcommand{\RR}{\mathbb{R}}

\newcommand{\CC}{\mathbb{C}}
\newcommand{\ZZ}{\mathbb{Z}}

\newcommand{\icat}{$\infty$-category}
\newcommand{\icats}{$\infty$-categories}
\newcommand{\icatl}{$\infty$-categorical}

\newcommand{\xto}[1]{\xrightarrow{#1}}

\newcommand{\from}{\leftarrow}
\newcommand{\xfrom}[1]{\xleftarrow{#1}}

\newcommand{\catname}[1]{\ensuremath{\text{\textup{#1}}}}
\newcommand{\txt}[1]{\ensuremath{\text{\textup{#1}}}}
\newcommand{\Set}{\catname{Set}}
\newcommand{\sSet}{\Set_{\Delta}}

\newcommand{\Cat}{\catname{Cat}}
\newcommand{\CatI}{\catname{Cat}_\infty}
\newcommand{\LCatI}{\widehat{\catname{Cat}}_\infty}

\newcommand{\Sp}{\catname{Sp}}

\newcommand{\Mod}{\catname{Mod}}

\newcommand{\Fun}{\txt{Fun}}

\newcommand{\Map}{\txt{Map}}
\newcommand{\Hom}{\txt{Hom}}

\newcommand{\Span}{\txt{Span}}
\newcommand{\SPAN}{\txt{SPAN}}
\newcommand{\Cospan}{\txt{Cospan}}
\newcommand{\COSPAN}{\txt{COSPAN}}
\newcommand{\oSPAN}{\overline{\txt{SPAN}}}
\newcommand{\oCOSPAN}{\overline{\txt{COSPAN}}}
\newcommand{\lsp}{\txt{sp}}

\newcommand{\Der}{\txt{Der}}
\newcommand{\ev}{\txt{ev}}
\newcommand{\Art}{\txt{Art}}
\newcommand{\fin}{\txt{fin}}
\newcommand{\dAff}{\txt{dAff}}
\newcommand{\et}{\txt{\'{e}t}}

\newcommand{\dSt}{\txt{dSt}}
\newcommand{\dStg}{\dSt^{\txt{geom}}}
\newcommand{\dStgop}{\dSt^{\txt{geom},\op}}
\newcommand{\dStArt}{\dSt^{\txt{Art}}}
\newcommand{\UCC}{\txt{UCC}}
\newcommand{\dStUCC}{\dSt^{\UCC}}
\newcommand{\dStUCCop}{\dSt^{\UCC,\op}}
\newcommand{\PSymp}{\txt{PreSymp}}

\newcommand{\POr}{\txt{PreOr}}

\newcommand{\POrc}{\POr^{\txt{cpt}}}

\newcommand{\Lag}{\txt{Lag}}
\newcommand{\LAG}{\txt{LAG}}

 \newcommand{\Or}{\txt{Or}}
\newcommand{\OR}{\txt{OR}}

\newcommand{\ORcSd}{\OR^{\txt{cpt},S,d}}
\newcommand{\OrcSd}{\Or^{\txt{cpt},S,d}}
\newcommand{\bbS}{\bbSigma}
\newcommand{\pbbS}{\partial\bbS}
\newcommand{\tbbS}{\widetilde{\bbS}}

\newcommand{\bbL}{\bbLambda}
\newcommand{\tbbL}{\widetilde{\bbL}}

\newcommand{\Ch}{\sfop{Ch}}
\newcommand{\Chk}{\Ch_{\k}}
\newcommand{\Chkd}{\Ch_{\k/\k[-d]}}
\newcommand{\CAlg}{\txt{CAlg}}
\newcommand{\CAlgk}{\CAlg_{\k}}
\newcommand{\fX}{\mathfrak{X}}
\newcommand{\Xd}{\fX^{(d)}}
\newcommand{\Xdr}{\fX^{(d,\mathrm{red})}}
\newcommand{\Xdpr}{\fX^{(d+1,\mathrm{red})}}
\newcommand{\Xdj}{\Xd_{\ind{j}}}

\newcommand{\Xdjl}{\Xd_{\ind{j},l}}
\newcommand{\Xdrjl}{\Xdr_{\ind{j},l}}
\newcommand{\tX}{\widetilde{\fX}}
\newcommand{\tXd}{\tX^{(d)}}
\newcommand{\tXdj}{\tXd_{\ind{j}}}
\newcommand{\tXdjl}{\tXd_{\ind{j},l}}
\newcommand{\tXdr}{\tX^{(d,\mathrm{red})}}
\newcommand{\tXdrj}{\tXdr_{\ind{j}}}
\newcommand{\tXdrjl}{\tXdr_{\ind{j},l}}
\newcommand{\fY}{\mathfrak{Y}}
\newcommand{\Yd}{\fY^{(d)}}
\newcommand{\Ydr}{\fY^{(d,\mathrm{red})}}
\newcommand{\Ydpr}{\fY^{(d+1,\mathrm{red})}}
\newcommand{\Ydj}{\Yd_{\ind{j}}}

\newcommand{\Ydrj}{\Ydr_{\ind{j}}}

\newcommand{\tY}{\widetilde{\fY}}
\newcommand{\tYd}{\tY^{(d)}}
\newcommand{\tYdj}{\tYd_{\ind{j}}}
\newcommand{\tYdjl}{\tYd_{\ind{j},l}}
\newcommand{\tYdr}{\tY^{(d, \mathrm{red})}}
\newcommand{\tYdrj}{\tYdr_{\ind{j}}}
\newcommand{\tYdrjl}{\tYdr_{\ind{j},l}}
\newcommand{\tYdpr}{\tY^{(d+1,\mathrm{red})}}
\newcommand{\fZ}{\mathfrak{Z}}
\newcommand{\Zd}{\fZ^{(d)}}
\newcommand{\tZ}{\widetilde{\fZ}}
\newcommand{\tZd}{\tZ^{(d)}}
\newcommand{\Pj}{P_{\ind{j}}}
\newcommand{\Pjl}{P_{\ind{j},l}}
\newcommand{\AKSZ}{\txt{AKSZ}}
\newcommand{\nd}{\txt{nd}}
\newcommand{\semi}{\txt{semi}}
\newcommand{\Dnd}{\Delta_{\txt{nd}}}
\newcommand{\Dndspi}{\Delta_{\txt{nd},\spi}}
\DeclareMathOperator{\Tw}{Tw}
\DeclareMathOperator{\Nrv}{N}
\DeclareMathOperator{\Ex}{Ex}
\newcommand{\Sh}{\txt{Sh}}
\newcommand{\Shet}{\txt{Sh}^{\et}}

\newcommand{\tSp}{\widetilde{\Sp}}
\newcommand{\tlsp}{\widetilde{\lsp}}
\newcommand{\id}{\txt{id}}
\DeclareMathOperator{\colimP}{colim}
\newcommand{\colim}{\mathop{\colimP}}
\DeclareMathOperator{\hocolimP}{hocolim}
\newcommand{\hocolim}{\mathop{\hocolimP}}

\newcommand{\blank}{\text{\textendash}}

\newcommand{\isoto}{\xrightarrow{\sim}}
\newcommand{\isofrom}{\xleftarrow{\sim}}

\newcommand{\bbDelta}{\boldsymbol{\Delta}}

\newcommand{\bbSigma}{\boldsymbol{\Sigma}}
\newcommand{\bbLambda}{\boldsymbol{\Lambda}}

\newcommand{\Int}{\mathrm{Int}}

\DeclareMathOperator{\Bord}{Bord}

\DeclareMathOperator{\BORD}{BORD}
\newcommand{\sfop}[1]{\textmd{\textup{\textsf{#1}}}}
\newcommand{\sBORD}{\sfop{BORD}}
\newcommand{\sBord}{\sfop{Bord}}
\newcommand{\xint}{\txt{int}}
\newcommand{\ext}{\txt{ext}}
\newcommand{\BORDor}{\BORD^{\txt{or}}}
\newcommand{\Bordor}{\Bord^{\txt{or}}}
\newcommand{\sBORDor}{\sBORD^{\txt{or}}}
\newcommand{\sBORDorsemi}{\sBORD^{\txt{or},\semi}}
\newcommand{\sBordor}{\sBord^{\txt{or}}}
\newcommand{\sBordorsemi}{\sBord^{\txt{or},\semi}}

\newcommand{\R}{{\mathbb{R}}}

\DeclareMathOperator{\sd}{sd}
\DeclareFontFamily{OT1}{pzc}{}
\DeclareFontShape{OT1}{pzc}{m}{it}{<-> s * [1.10] pzcmi7t}{}
\DeclareMathAlphabet{\mathpzc}{OT1}{pzc}{m}{it}
\DeclareMathOperator{\Mfd}{Mfd}

\renewcommand{\O}{\mathcal O}

\newcommand{\oul}[1]{{\overline{\underline{#1}}}}

\renewcommand{\k}{\mathbb{K}}

\newcommand{\cut}{\operatorname{cut}}
\newcommand{\scut}{\sfop{cut}}
\newcommand{\orcut}{\operatorname{orcut}}
\newcommand{\sorcut}{\sfop{orcut}}
\newcommand{\cutT}{\widetilde{\cut}}
\newcommand{\scutT}{\widetilde{\scut}}

\newcommand{\grid}{{\mathrm{grid}}}
\newcommand{\gridT}{{\widetilde{\mathrm{grid}}}}
\newcommand{\ind}[1]{\mathbf{#1}}
\newcommand{\spi}{{\txt{sp}}}
\newcommand{\pred}{(\txt{red})}
\newcommand{\igpd}{$\infty$-groupoid}
\newcommand{\igpds}{$\infty$-groupoids}
\newcommand{\sdDl}{\sd(\Delta^{l})}
\newcommand{\ponej}{(1,\ind{j})}
\newcommand{\bbSj}{\bbS^{\ind{j}}}
\newcommand{\bbSponej}{\bbS^{\ponej}}
\newcommand{\Modkd}{\Mod_{\k/\k[-d]}}
\newcommand{\sMod}{\sfop{Mod}}
\newcommand{\Modconn}{\Mod^{c}}
\newcommand{\CAlgconn}{\CAlg^{c}}
\newcommand{\CAlgkc}{\CAlgconn_{\k}}
\newcommand{\CAlgkcop}{\CAlg_{\k}^{c,\op}}

\newcommand{\eg}{e.g.\@}
\newcommand{\ie}{i.e.\@}

\newcommand{\cf}{cf.\@}

\newcommand{\xF}{\mathbb{F}}
\newcommand{\angled}[1]{\langle #1 \rangle}

\newcommand{\opctriangle}[6]{ %
\[ %
\begin{tikzpicture} %
\matrix (m) [matrix of math nodes,row sep=3em,column sep=1.2em,text height=1.5ex,text depth=0.25ex] %
{  #1 \pgfmatrixnextcell \pgfmatrixnextcell #2 \\ %
  \pgfmatrixnextcell #3 \pgfmatrixnextcell \\ %
}; %
\path[->,font=\footnotesize] %
(m-1-1) edge node[above] {$#4$} (m-1-3)%
(m-1-1) edge node[below left] {$#5$} (m-2-2)%
(m-1-3) edge node[below right] {$#6$} (m-2-2);%
\end{tikzpicture}%
\]%
}

\newcommand{\csquare}[8]{ %
\[ %
\begin{tikzpicture} %
\matrix (m) [matrix of math nodes,row sep=3em,column sep=2.5em,text height=1.5ex,text depth=0.25ex] %
{ #1 \pgfmatrixnextcell #2 \\ %
  #3 \pgfmatrixnextcell #4 \\ }; %
\path[->,font=\footnotesize] %
(m-1-1) edge node[auto] {$#5$} (m-1-2)%
(m-1-1) edge node[left] {$#6$} (m-2-1)%
(m-1-2) edge node[auto] {$#7$} (m-2-2)%
(m-2-1) edge node[below] {$#8$} (m-2-2);%
\end{tikzpicture}%
\]%
}

\newcommand{\nodispcsquare}[8]{ %
\begin{tikzpicture} %
\matrix (m) [matrix of math nodes,row sep=3em,column sep=2.5em,text height=1.5ex,text depth=0.25ex] %
{ #1 \pgfmatrixnextcell #2 \\ %
  #3 \pgfmatrixnextcell #4 \\ }; %
\path[->,font=\footnotesize] %
(m-1-1) edge node[auto] {$#5$} (m-1-2)%
(m-1-1) edge node[left] {$#6$} (m-2-1)%
(m-1-2) edge node[auto] {$#7$} (m-2-2)%
(m-2-1) edge node[below] {$#8$} (m-2-2);%
\end{tikzpicture}%
}

\newcommand{\nolabelcsquare}[4]{\csquare{#1}{#2}{#3}{#4}{}{}{}{}}

\DeclareMathOperator{\Spec}{Spec}
\newcommand{\QCoh}{\txt{QCoh}}
\newcommand{\QCohconn}{\QCoh^{c}}

\newcommand{\QC}{\mathcal{Q}}
\newcommand{\QCconn}{\mathcal{Q}^c}
\newcommand{\QCAlg}{\txt{QCAlg}}
\newcommand{\QCAlgconn}{\QCAlg^{c}}

\newcommand{\QCAconn}{\mathcal{QA}^{c}}
\newcommand{\QCAconnop}{\mathcal{QA}^{c,\op}}

\newcommand{\Top}{\txt{Top}}
\newcommand{\ThmA}{\cite[Theorem 4.1.3.1]{HTT}}
\newcommand{\Dop}{\simp^{\op}}
\newcommand{\Dnop}{\simp^{n,\op}}

\newcommand{\Lagn}{\Lag_{n}}
\newcommand{\LagnSs}{\Lag_{n}^{S,s}}

\newcommand{\LAGnSs}{\LAG_{n}^{S,s}}
\newcommand{\LAGnpSs}{\LAG_{n}^{+,S,s}}
\newcommand{\Orn}{\Or_{n}}
\newcommand{\OrnSd}{\Or_{n}^{S,d}}

\newcommand{\ORnSd}{\OR_{n}^{S,d}}
\newcommand{\ORnpSd}{\OR_{n}^{+,S,d}}
\newcommand{\por}{(\txt{or})}

\newcommand{\simp}{\bbDelta}

\newcommand{\Seg}{\txt{Seg}}

\newcommand{\Mon}{\txt{Mon}}

\newcommand{\Alg}{\catname{Alg}}

\newcommand{\PrL}{\catname{Pr}^{\mathrm{L}}}
\newcommand{\iopd}{$\infty$-operad}
\newcommand{\iopds}{$\infty$-operads}

\newcommand{\ffcsquare}[4]{
  \[
  \begin{tikzcd}
    #1 \arrow[hookrightarrow]{r}
    \arrow[hookrightarrow]{d} \pgfmatrixnextcell #2
    \arrow[hookrightarrow]{d}\\
    #3 \arrow[hookrightarrow]{r}
    \pgfmatrixnextcell #4
  \end{tikzcd}
  \]
}

\newcommand{\GMC}{\txt{GMCAlg}}
\newcommand{\GC}{\txt{GrCAlg}}
\newcommand{\GMM}{\txt{GMMod}}
\newcommand{\sGMM}{\sfop{GMMod}}

\newcommand{\DR}{\txt{DR}}
\DeclareMathOperator{\Sym}{Sym}
\DeclareMathOperator{\disc}{disc}
\DeclareMathOperator{\diag}{diag}
\DeclareMathOperator{\const}{const}
\newcommand{\Apcl}{\mathcal{A}^{p,\txt{cl}}}
\newcommand{\Ap}{\mathcal{A}^{p}}
\newcommand{\Atcl}{\mathcal{A}^{2,\txt{cl}}}

\newcommand{\hS}{\widehat{\mathcal{S}}}

\newcommand{\Dint}{\simp_{\txt{int}}}
\newcommand{\inj}{\txt{inj}}
\newcommand{\Dinj}{\simp_{\txt{inj}}}
\newcommand{\Del}{\simp_{\txt{el}}}
\newcommand{\Delop}{\Del^{\op}}
\newcommand{\Dintop}{\Dint^{\op}}
\newcommand{\Dinjop}{\Dinj^{\op}}
\newcommand{\Delnop}{\Del^{n,\op}}
\newcommand{\Dintnop}{\Dint^{n,\op}}
\newcommand{\DnelopI}{\simp^{n,\op}_{\txt{el}/\mathbf{i}}}
\newcommand{\DelnI}{\simp^{n}_{\txt{el}/\mathbf{i}}}
\newcommand{\ssSet}{\sSet^{\txt{semi}}}

\begin{document}

\frontmatter

\title[The AKSZ Construction in Derived Algebraic Geometry as an Extended
TFT]{The AKSZ 
 Construction in Derived Algebraic Geometry as an Extended Topological Field Theory}

\author{Damien Calaque}
\address{IMAG, Univ Montpellier, CNRS, Montpellier, France}
\urladdr{http://imag.umontpellier.fr/~calaque}
\email{damien.calaque@umontpellier.fr}
\thanks{Damien Calaque has received funding from the European Research Council (ERC) under the European Union’s
Horizon 2020 research and innovation programme (grant agreement No 768679).}

\author{Rune Haugseng}
\address{NTNU, Trondheim, Norway}
\urladdr{http://folk.ntnu.no/runegha}
\email{rune.haugseng@ntnu.no}
\thanks{Parts of this paper were written while Rune Haugseng was employed by the IBS Center for Geometry and Physics in 
a position funded by the grant IBS-R003-D1 of the Institute for Basic Science, Republic of Korea.}

\author{Claudia Scheimbauer}
\address{Technische Universit\"at M\"unchen, Munich, Germany}
\urladdr{http://www.scheimbauer.at}
\email{scheimbauer@tum.de}
\thanks{Claudia Scheimbauer received funding from the Bergen Research (Trond Mohn) Foundation, and was partially supported by
the Swiss National Science Foundation (grants P300P2 164652 and P2EZP2 159113) during her wonderful time at the 
Mathematical Institute at Oxford University and at the Max Planck Institute for Mathematics in Bonn.}

\date{}




\begin{abstract}
  We construct a family of oriented extended topological field
  theories using the AKSZ construction in derived algebraic geometry,
  which can be viewed as an algebraic and topological version of the
  classical AKSZ field theories that occur in physics.  These have as
  their targets higher categories of symplectic derived stacks, with
  higher morphisms given by iterated Lagrangian correspondences. We
  define these, as well as analogous higher categories of oriented
  derived stacks and iterated oriented cospans, and prove that all
  objects are fully dualizable. Then we set up a functorial version of
  the AKSZ construction, first implemented in this context by
  Pantev-To\"{e}n-Vaquié-Vezzosi, and show that it induces a family of
  symmetric monoidal functors from oriented stacks to symplectic
  stacks.  Finally, we construct forgetful functors from the
  unoriented bordism $(\infty,n)$-category to cospans of spaces, and
  from the oriented bordism $(\infty,n)$-category to cospans of spaces
  equipped with an orientation; the latter combines with the AKSZ functors
  by viewing spaces as constant stacks, giving the desired field theories.
\end{abstract}

\maketitle

\tableofcontents


\mainmatter

\chapter{Introduction}
In physics, classical field theories typically associate to a suitable
manifold $X$ a space $\mathcal{F}(X)$ of \emph{fields}
on $X$, where the precise notion of ``space'' required might be some
sort of stack, or even derived stack. In many examples the space of
fields is a mapping space $\Map(X, T)$ for some target $T$; such field
theories are called \emph{$\sigma$-models}. The space of fields often
has a symplectic structure (or more generally a Poisson structure),
which plays an important role in quantizing the classical theory, for
example in the context of geometric quantization.

In the case of a $\sigma$-model with target $T$, one can construct a
symplectic structure on the mapping space $\Map(X, T)$ from a
symplectic structure on $T$ and an orientation on $X$ using the so-called 
\emph{AKSZ construction}, originally introduced by
Alexandrov--Kontsevich--Schwarz--Zaboronsky~\cite{AKSZ}. 
This gives the class of \emph{AKSZ $\sigma$-models}, which includes many
interesting examples of classical field theories (including classical
Chern-Simons theory and the classical versions of the A-- and B--models,
the theories relevant to mirror symmetry).

Roughly speaking, the AKSZ symplectic structure on the mapping space is given via a 
transgression procedure as
\[
\int_{[X]}\ev^*\omega\,,
\]
where $\omega$ is the symplectic form on $T$, $[X]$ is a fundamental
class on $X$, and $\ev \colon X\times\Map(X, T)\to T$ is the
evaluation morphism. The original reference~\cite{AKSZ} uses the
language of $Q$-manifolds, and $\int_{[X]}$ is given by Berezin
integration. These tools can be viewed as pre-derived geometric
concepts, and are very much model-dependent (in the sense that they
are not invariant under weak homotopy equivalences, such as \eg{} quasi-isomorphisms).

The AKSZ construction was promoted to the more general context of
derived (algebraic) geometry by
Pantev--To\"en--Vaqui\'{e}--Vezzosi~\cite{PTVV}. The main advantage of
passing to this setting is that derived geometry allows one to make
sense of many constructions with only a small number of assumptions
(and is model-independent, by definition); it is also essential for
rigorously stating the functorial nature of the AKSZ
procedure. 

A crucial feature of this functoriality is that it ensures that the
symplectic structure on the space of fields of a gluing of two
bordisms along a common boundary component is determined by the spaces
of fields of the individual bordisms. As discussed in
\cite{CalaqueTFT}, this can be used to define an Atiyah-style
functorial field theory for every symplectic derived stack.  Our main
goal in this paper is to give a mathematical construction of classical
AKSZ $\sigma$-models using the more refined formalism of
\emph{extended} topological field theories (TFTs), thereby showing in
particular that the AKSZ construction of \cite{PTVV} is compatible
with iterated gluings of manifolds with corners.\footnote{See
  Example~\ref{intr:ex of AKSZ} for some examples of field theories
  that fit into our framework.}

\section{The 1-Categorical Story}\label{1-cat-sketch}
Before we describe our results in more detail, it is helpful to
first briefly consider the simplest, 1-categorical, version of it; see
also~\cite{CalaqueTFT} for a more detailed discussion of this.

A (non-extended) oriented $n$-dimensional TFT, according to the original definition of
Atiyah~\cite{AtiyahTQFT}, is a symmetric monoidal functor
\[ \mathcal{Z} \colon h\!\Bordor_{n-1,n} \longrightarrow \mathbf{C}, \]
where $h\!\Bordor_{n-1,n}$ is a symmetric monoidal category whose
\begin{itemize}
\item objects are $(n-1)$-dimensional closed oriented manifolds,
\item morphisms are (diffeomorphism classes of) $n$-dimensional oriented cobordisms,
\item composition of morphisms is gluing of cobordisms,
\item tensor product is disjoint union of manifolds.
\end{itemize}
To describe topological quantum field theories, the target
$\mathbf{C}$ is typically taken to be the category of complex vector
spaces (or variations thereof). Here, however, we want to model
(semi-)classical topological field theories, which requires a
different target: A classical $n$-dimensional field theory should
assign to a closed $(n-1)$-manifold $X$ the space, or more generally derived
stack, $\mathcal{F}(X)$, of fields on $X$, together with its
symplectic structure. The objects of our target category should thus
be symplectic derived stacks. More precisely, we use the theory of symplectic
structures on derived Artin stacks introduced by
Pantev--To\"en--Vaqui\'{e}--Vezzosi~\cite{PTVV}.\footnote{We expect
  that most of our constructions will apply equally in the setting of
  derived $C^{\infty}$-stacks, but appropriate notions of differential
  and symplectic forms have not yet been studied in this setting.}  

To see what the morphisms should be, let us consider what a classical
field theory $\mathcal{F}$ should assign to a cobordism $X$ from $M$
to $N$: We have the space $\mathcal{F}(X)$ of fields on $X$, but we
can also \emph{restrict} these fields to the boundaries $M$ and
$N$. Thus, we have a span (or correspondence), which is a zig-zag
\[
  \begin{tikzcd}[row sep=tiny]
    {} & \mathcal{F}(X) \arrow{dl} \arrow{dr} \\
    \mathcal{F}(M)& & \mathcal{F}(N).
  \end{tikzcd}
  \]
 Moreover,
if the cobordism $X$ is a composite $X_{1} \cup_{K} X_{2}$ then the
fields on $X$ should be the same as pairs of fields on $X_{1}$ and
$X_{2}$ that agree on the boundary $K$, \ie{} we should have a
pullback square
\nolabelcsquare{\mathcal{F}(X)}{\mathcal{F}(X_1)}{\mathcal{F}(X_2)}{\mathcal{F}(K).}
Ignoring the symplectic structures, a morphism in the target should
thus be a span (or correspondence) of derived stacks, with composition
of spans given by taking pullbacks. The required compatibility with
the symplectic structures on the objects is then that this should be a
\emph{Lagrangian correspondence}.

In symplectic geometry, a Lagrangian correspondence between
symplectic manifolds $(M,\omega_{M})$ and $(N,\omega_{N})$ is a
Lagrangian submanifold (i.e. a submanifold of maximal dimension on
which the symplectic form vanishes) of the product $M \times N$, where
this is equipped with the symplectic form $\pi_{M}^{*}\omega_{M} -
\pi_{N}^{*}\omega_{N}$. The idea that these give a good notion of
\emph{morphisms} between symplectic manifolds was introduced by
Weinstein~\cite{WeinsteinSymplGeom,WeinsteinSymplCat}. However, due to
the potential failure of intersections to be transverse, it is not
always possible to compose Lagrangian correspondences by taking
pullbacks. Fortunately, this issue disappears when we work in the
setting of \emph{derived} geometry, as derived intersections are
always well-behaved.

On the other hand, in the derived setting we can consider
\emph{shifted} symplectic structures, essentially because 2-forms on a
derived stack form a chain complex instead of a vector space. A key example is the classifying stack $BG$ for a semisimple algebraic group $G$, whose Killing form can be reinterpreted as a 2-shifted symplectic structure.

The
target of our family of field theories is therefore a symmetric
monoidal category $h\Lag_{1}^{s}$ whose
\begin{itemize}
\item objects are
$s$-symplectic derived stacks, \ie{} derived stacks equipped with an
$s$-shifted symplectic structure,
\item morphisms are (equivalence classes of) Lagrangian correspondences, 
\item composition of morphisms is given by pullback (which always makes sense in derived geometry),
\item tensor product is the cartesian product of derived stacks.
\end{itemize}

In this context, Pantev, To\"en, Vaqui\'e, and Vezzosi introduced an
algebro-geometric version of the AKSZ construction by reinterpreting
pullback and pushforward constructions on differential forms. If
$\Sigma$ is a \emph{$d$-oriented} derived stack and $X$ is an
$s$-symplectic derived stack, then this produces an $(s-d)$-shifted
symplectic structure on the mapping stack $\Map(\Sigma,X)$. The key
example of a $d$-oriented stack is the constant stack (or Betti stack)
$M_{B}$ determined by the singular chains on a closed oriented
$d$-manifold $M$.

Oriented $(d+1)$-dimensional cobordisms also have an interesting
structure on their Betti stacks: the inclusion of the boundaries into
the cobordism is a cospan, and Poincar\'e-Lefschetz duality can be
reinterpreted as a non-degeneracy condition. Based on this we can
formulate a general 
notion of a $d$-oriented cospan of derived stacks
(see~\cite{CalaqueTFT}), with the main example given by an oriented
$(d+1)$-dimensional cobordism.

We can now state a key formal property of the AKSZ construction: 
If
\[
  \begin{tikzcd}[row sep=tiny]
    \Sigma \arrow{dr} &  & \Sigma' \arrow{dl}\\
     & T
  \end{tikzcd}
  \]
is a $d$-oriented cospan and $X$ is an $s$-symplectic derived
stack, then the AKSZ construction makes the span
\[
  \begin{tikzcd}[row sep=tiny]
    {} & \Map(T,X) \arrow{dl} \arrow{dr} \\
    \Map(\Sigma,X)& & \Map(\Sigma',X)
  \end{tikzcd}
  \]
a Lagrangian correspondence. Moreover, this construction is compatible
with composition of cospans and spans, and so determines a symmetric
monoidal functor
\begin{equation}\label{intro:eqn aksz} h\txt{Or}^{d}_{1} \xto{\txt{aksz}_{X}} h\Lag^{s-d}_{1},	
\end{equation}
where $h\txt{Or}^{d}_{1}$ is a category whose
\begin{itemize}
\item objects are $d$-oriented derived stacks,
\item morphisms are (equivalence classes of) $d$-oriented cospans, with composition given by pushouts,
\item symmetric monoidal structure is the coproduct of derived stacks.
\end{itemize}
As we already mentioned, we can obtain $d$-oriented stacks from closed
oriented $d$-manifolds, and $d$-oriented cospans from oriented
$(d+1)$-dimensional cobordisms, by taking constant
stacks. This gives a symmetric monoidal functor
\begin{equation}\label{intro:eqn Betti functor} h\!\Bordor_{n-1,n} \longrightarrow h\Or^{n-1}_{1}.	
\end{equation}
Composing this with the previous functor we see that for any
$s$-symplectic derived stack $X$ the AKSZ construction determines an
oriented $n$-dimensional TFT valued in symplectic stacks as the composite
\[ h\!\Bordor_{n-1,n} \longrightarrow h\Or^{n-1}_{1} \xto{\txt{aksz}_{X}}
  h\Lag^{s-n+1}_{1}.\]
In this paper we will extend this construction to obtain fully
extended TFTs.

\section{Extended TFTs}
The definition of \emph{extended} topological field theories is a
higher-categorical analogue of Atiyah's axiomatization: An extended
field theory is a symmetric monoidal functor from a symmetric monoidal
higher category of iterated cobordisms to some target symmetric
monoidal higher category.

Heuristically, the $n$-category of iterated bordisms $\Bord^{}_{0,n}$
can be described as follows:
\begin{itemize}
\item objects are $0$-dimensional manifolds, 
\item $1$-morphisms are $1$-dimensional cobordisms between $0$-dimensional manifolds, 
\item $2$-morphisms are $2$-dimensional cobordisms (with $0$-dimensional corners) between $1$-dimensional cobordisms, 
\item $\dots$
\item $n$-morphisms are diffeomorphism classes of $n$-dimensional cobordisms (with corners of all codimensions), 
\item various compositions are given by gluings of cobordisms, 
\item the tensor product is the disjoint union. 
\end{itemize}
The idea behind this extension is that it is a fully 
local version of Atiyah's functorial field theories: indeed, the value of an extended TFT on an $n$-dimensional manifold 
can be recovered from a single elementary building block by gluing iterated cobordisms. 

This locality is captured by the \emph{Cobordism Hypothesis}, the key statement for the classification of extended TFTs.
This was first conjectured by Baez and
Dolan~\cite{BaezDolanTQFT} and though a complete proof has not yet
appeared, Lurie~\cite{LurieCob} has written a detailed sketch of a
proof.
The simplest case of the Cobordism Hypothesis classifies
\emph{framed} TFTs, where all the manifolds involved are equipped with
a framing: Such a framed $n$-dimensional extended TFT is completely
determined by its value at the point, which is the mathematical incarnation 
of the locality principle that we mentioned above. 
The objects that correspond to framed extended TFTs are precisely the \emph{fully
 dualizable} objects in the target, where being fully dualizable is
an inductively defined finiteness condition. 

In fact, Lurie's approach~\cite{LurieCob} makes use of
$(\infty,n)$-categories rather than mere $n$-categories, in order to
take into account the action of the diffeomorphism group on top
dimensional cobordisms: this roughly amounts to replacing the set of
diffeomorphisms classes of $n$-dimensional cobordisms by the
classifying \emph{space} or {\em $\infty$-groupoid} thereof.  For
instance, there are invertible $(n+1)$-morphisms, which are
diffeomorphisms of $n$-dimensional cobordisms. We refer
to~\cite{LurieCob,CS} for a precise definition of the symmetric
monoidal $(\infty,n)$-category of $n$-iterated cobordisms; we will
also review the definition we use in more detail in \S\ref{sec bordism
  background}.

\section{Overview of Results}
As already explained, the main goal of this paper is to give a
mathematical construction of classical AKSZ $\sigma$-models as
oriented extended topological field theories (TFTs), by developing a
higher-categorical version of the
construction of 1-categorical field theories we sketched in
\S\ref{1-cat-sketch}. In other words, for every $s$-symplectic
derived stack $X$ we want to construct a symmetric monoidal functor of
$(\infty,n)$-categories
\[
\Bord^{\txt{or}}_{0,n} \longrightarrow \Lag_{n}^{s}
\]
where $\Bord_{0,n}^{\txt{or}}$ is the oriented bordism
$(\infty,n)$-category, which has $i$-dimensional oriented cobordisms
with corners as its $i$-morphisms, and $\Lag_{n}^{s}$ is a suitable
symmetric monoidal $(\infty,n)$-category of symplectic derived stacks
and Lagrangian correspondences.  Roughly speaking, this functor will
be given by the AKSZ construction applied functorially to maps from all pieces of the boundary of a cobordism
with corners to a fixed target $X$.

To carry out this programme, we must first define a symmetric monoidal
$(\infty,n)$-category $\Lag_{n}^{s}$ that will be the target for our
field theories. This should  have $s$-symplectic derived stacks as
its objects and Lagrangian correspondences as its morphisms, while for
the $i$-morphisms we need a notion of \emph{$i$-fold Lagrangian
  correspondences} (equivalent to that considered in
\cite{CalaqueTFT}). We introduce this in \S\ref{sec:lagor}, where we
use the construction of higher categories of iterated spans from
\cite{spans} to prove:
\begin{ThmB}[see \S\ref{subsec:itlagcat} and \S\ref{subsec:itlagsymmon}]
  There is a symmetric monoidal $(\infty,n)$-category $\Lag_{n}^{s}$
  whose objects are $s$-symplectic derived stacks and whose
  $i$-morphisms for $i \leq n$ are $i$-fold Lagrangian
  correspondences.
\end{ThmB}
The homotopy bicategory of $\Lag_{2}^{s}$ is the bicategory
previously constructed by Amorim and Ben-Bassat~\cite{AmorimBenBassatLag}.
  
Recall that the $1$-categorical construction from \S\ref{1-cat-sketch}
has two steps: The first step \eqref{intro:eqn aksz} maps the category
of (oriented) manifolds and cobordisms into that of oriented stacks
and cospans, while the second step \eqref{intro:eqn Betti functor} is
the formal application of a functorial AKSZ construction.

Following the same path, in \S\ref{sec:lagor} we also introduce a
notion of $i$-fold oriented cospans, which in a sense is ``dual'' to
that of an $i$-fold Lagrangian correspondence. Then we construct
symmetric monoidal $(\infty,n)$-categories $\Or_{n}^{d}$ of
$d$-oriented derived stacks where these are the $i$-morphisms, in a
similar way to the construction of $\Lag_{n}^{s}$.

In \S\ref{subsec:dual}, we prove:
\begin{ThmB}
  All objects of $\Lag_{n}^{s}$ and $\Or_{n}^{d}$ are fully
  dualizable, and so correspond to $n$-dimensional framed extended
  TFTs under the Cobordism Hypothesis.
\end{ThmB}

Our main goal, however, is to construct \emph{oriented} TFTs. There is
also a version of the Cobordism Hypothesis for oriented TFTs, but
applying this requires understanding an $O(n)$-action on the space of
fully dualizable objects. However, this $O(n)$-action is defined via
the Cobordism Hypothesis by translating a canonical $O(n)$-action on
the space of fully extended framed $n$-dimensional TFTs, and it is
unfortunately still unclear how to obtain the action on fully
dualizable objects without passing through this big machinery.

Moreover, our goal is to extend the AKSZ construction of field
theories, and so we want
to show that our TFTs are explicitly given by applying the AKSZ
construction. Therefore, in \S\ref{sec:aksz} we define a completely
functorial version of the AKSZ construction, the higher categorical
analog of \eqref{intro:eqn aksz}, and prove:
\begin{ThmB}[see \S\ref{subsec:aksznondeg}]
  For every $s$-symplectic derived stack $X$ the AKSZ construction gives a
  symmetric monoidal functor
  \[ \txt{aksz}_{X} \colon \Or_{n}^{d} \longrightarrow
    \Lag_{n}^{s-d}.\]
\end{ThmB}

We are now left with constructing the analog of \eqref{intro:eqn Betti functor}, a symmetric monoidal functor from
$\Bordor_{0,n}$ to $\Or_{n}^{0}$, which we can then compose with
$\txt{aksz}_{X}$ to obtain the desired TFT.

Before doing so, we
consider the simpler case of \emph{unoriented} cobordisms (where the
manifolds involved are not equipped with \emph{any} tangential structure) in
\S\ref{sec:cobcospan}, where we show:
\begin{ThmB}[see \S\ref{subsec:TFTspan}]
  There is a symmetric monoidal functor
  \[ \cut_{0,n} \colon \Bord_{0,n} \longrightarrow
    \Cospan_{n}(\mathcal{S})\] from the unoriented cobordism
  $(\infty,n)$-category $\Bord_{0,n}$ to the $(\infty,n)$-category of
  iterated cospans of spaces.
\end{ThmB}
This functor assigns to an $i$-dimensional cobordism with corners $M$
the homotopy type of the manifold $M$, together with the inclusions of
the homotopy types of all the parts of the boundary of $M$, which
together form an $i$-fold cospan.

In \S\ref{sec:orcob} we then upgrade this to take orientations into
account:
\begin{ThmB}[see \S\ref{sec:cobnondeg}]
  There is a symmetric monoidal functor
  \[ \orcut_{0,n} \colon \Bordor_{0,n} \longrightarrow
    \Or_{n}^{0}.\]  
\end{ThmB}
If we ignore the orientations, this is given by applying $\cut_{0,n}$
and then taking constant stacks to pass from spaces to derived stacks.

Putting these results together, for any $s$-symplectic derived
stack $X$ we get an $n$-dimensional oriented extended TFT as the composite
\[ \Bordor_{0,n}\xto{\orcut_{0,n}} \Or_{n}^{0} \xto{\txt{aksz}_{X}}
  \Lag_{n}^{s}.\]

\begin{remark}[Related works]
  In~\cite{CMNBV}, Cattaneo--Mnev--Reshetikhin develop an extension of
  the BV-BFV formalism on manifolds with corners, which can be used to
  describe similar field theories to those we consider here.  Another
  way of formalizing the concept of locality in quantum field theory
  uses the notion of a factorization algebra, and the work of
  Costello--Gwilliam~\cite{CostelloGwilliam2} can be seen as a
  systematic way of dealing with locality for BV-type field theories,
  which includes classical AKSZ theories.  Neither of these frameworks
  have so far been used to construct fully extended functorial field
  theories, but on the other hand they also encompass
  \emph{quantizations} of classical field theories, which we do not
  consider at all.
\end{remark}

For the convenience of the reader, we include two appendices: In
\S\ref{sec:cat} we recall some background material on higher
categories, including the construction of $(\infty,n)$-categories of
iterated spans from \cite{spans}, and in \S\ref{sec:dag} we do the
same for derived algebraic geometry. Here we also prove some new but
rather technical
results on these topics that are needed in the main body of the paper,
including a strict model of higher categories of spans
in a model category in \S\ref{subsec:spanmodel} and
a description of the left adjoint to the twisted arrow
\icat{} functor in \S\ref{subsec:twarr}.

\begin{examples}[Examples of AKSZ Theories]\label{intr:ex of AKSZ}
As further motivation for our construction, we mention some examples
from physics that fit into our formalism:
\begin{itemize}
\item If $G$ is a reductive algebraic group, then the classifying
  stack $BG$ has a 2-symplectic structure (given by the Killing form
  of the corresponding Lie algebra $\mathfrak{g}$ when $G$ is
  semisimple); the corresponding 3-dimensional field theory is
  \emph{classical Chern--Simons theory}.
\item If $X$ is a smooth algebraic variety equipped with an ordinary 
  algebraic symplectic structure, then it is
  $0$-shifted symplectic; the corresponding 3-dimensional field theory
  is \emph{classical Rozansky--Witten theory}~\cite{RW}.
\item There are also examples that would require some adaptation to
  fit into our framework, for instance when $X$ is the quotient of a
  shifted symplectic Lie algebroid in the sense of~\cite{PymSafronov}.
  In principle, if the shifted symplectic Lie algebroid is the
  one associated with a Poisson manifold (or smooth Poisson algebraic
  variety), then the corresponding 2-dimensional field theory will be
  the \emph{Poisson sigma model} (originally introduced by
  Ikeda~\cite{Ikeda} and Schaller--Strobl~\cite{SchStr}, and made into
  an AKSZ theory by Cattaneo--Felder~\cite{CatFel}).  
\end{itemize}
\end{examples}

\section{Acknowledgments}
We are grateful to Owen Gwilliam, Tyler Lawson, Irakli Patchkoria, Christopher
Schommer-Pries, Vivek Shende, and Mahmoud Zeinalian for helpful
discussions and comments related to this paper. RH also thanks the
denizens of the homotopy theory chat room back in 2015, in particular
Marc Hoyois and Adeel Khan, for algebro-geometric assistance.

This paper has taken a while to finish, and various subsets of the
authors have had the pleasure of working together during visits to the
Center for Symmetry and Deformation in Copenhagen, the Mathematical
Sciences Research Institute in Berkeley, the Max Planck Institute for
Mathematics in Bonn, the Newton Institute in Cambridge, and the
Perimeter Institute for Theoretical Physics in Waterloo. 


\section{Notation}
This paper is written in the language of \icats{}. 
We reuse much of the notation from \cite{spans}, as some of our
results build on the constructions of that paper. For the reader's
convenience, we recall some of this notation here, as well as other
conventions we use throughout the paper:
\begin{itemize}
\item By a \emph{commutative diagram} of some shape $\mathcal{I}$ we
  mean a homotopy-coherent diagram, i.e. a functor of \icats{} with
  source $\mathcal{I}$. Thus when we say that, for instance, a certain
  square \emph{commutes}, we mean that there is an implicitly chosen
  homotopy that fills in the square.
\item $\simp$ is the simplex category, with objects the
  non-empty finite ordered sets $[n] := \{0, 1, \ldots, n\}$
  and morphisms order-preserving functions between them.
\item If $f \colon \mathcal{C} \to \mathcal{D}$ is left adjoint to a
  functor $g \colon \mathcal{D} \to \mathcal{C}$, we will refer to the
  adjunction as $f \dashv g$.
\item $\mathcal{S}$ is the \icat{} of spaces (\ie{} $\infty$-groupoids
  or homotopy types); $\mathcal{S}_{\fin} \subseteq \mathcal{S}$ is the full
  subcategory of finite cell complexes, \ie{} the smallest full
  subcategory containing the point and closed under finite
  colimits.
\item $\sSet$ is the category of simplicial sets.
\item We say a functor $F \colon \mathcal{C} \to \mathcal{D}$ of
  \icats{} is \emph{coinitial} if the opposite functor $F^{\op} \colon
  \mathcal{C}^{\op} \to \mathcal{D}^{\op}$ is cofinal, i.e. if
  composing with $F$ always takes $\mathcal{D}$-indexed limits to
  $\mathcal{C}$-indexed limits.
\item If $\mathcal{C}$ is an \icat{}, we denote the cone points of the
  \icats{} $\mathcal{C}^{\triangleright}$ and $\mathcal{C}^{\triangleleft}$,
   obtained by freely adjoining a final and an initial object to $\mathcal{C}$,
   by $\infty$ and $-\infty$, respectively.
 \item We use three nested Grothendieck universes, and refer to the sets
   contained in these as \emph{small}, \emph{large}, and \emph{very
     large}. If $\mathcal{C}$ is the (large) \icat{} of small objects
   of some kind, we refer to the (very large) \icat{} of large objects
   of this type as $\widehat{\mathcal{C}}$. In particular,
   $\widehat{\mathcal{S}}$ is the (very large) \icat{} of large
   spaces, and $\LCatI$ is the \icat{} of large \icats{}.
 \item If $\mathcal{C}$ is an \icat{} and $x$ and $y$ are objects of
   $\mathcal{C}$, we write $\mathcal{C}_{/x,y}$ for the ``double
   slice'' \icat{} of diagrams $x \from c \to y$; this can be defined
   as the pullback of \icats{}
   \[
     \begin{tikzcd}
       \mathcal{C}_{/x,y} \arrow{r} \arrow{d} & \Fun(\Lambda^{2}_{0},
       \mathcal{C}) \arrow{d}{\txt{ev}_{1,2}} \\
       \{(x,y)\} \arrow{r} & \mathcal{C} \times \mathcal{C}.
     \end{tikzcd}
   \]
   (If the product $x \times y$ exists in $\mathcal{C}$, then its
   universal property says precisely that composition
   with the projections to $x$ and $y$ induces an equivalence
   $\mathcal{C}_{/x\times y} \isoto \mathcal{C}_{/x,y}$.)
 \item If $\mathcal{C}$ is an \icat{}, we'll denote the space obtained
   by inverting all the morphisms in $\mathcal{C}$ by
   $\|\mathcal{C}\|$ (in other words, this is the value at
   $\mathcal{C}$ of the left adjoint to the inclusion $\mathcal{S}
   \hookrightarrow \CatI$).
 \item We write $\xF_{*}$ for a skeleton of the category of finite
   pointed sets, with objects $\angled{n} := (\{0,1,\ldots,n\},0)$. A
   symmetric monoidal \icat{} can then be defined as a cocartesian
   fibration over $\xF_{*}$ whose corresponding functor $F \colon
   \xF_{*} \to \CatI$ is
   a commutative monoid in the sense that the maps $F(\angled{n}) \to
   \prod_{i=1}^{n} F(\angled{1})$, induced by the maps $\angled{n} \to
   \angled{1}$ that send all but one element of the source to the base
   point, are all equivalences.
 \item If $\mathcal{C}$ is a closed symmetric monoidal \icat{}, we
   write $X^{\vee}$ for the (linear) dual of an object $X \in
   \mathcal{C}$, meaning the internal Hom from $X$ to the monoidal unit.
\end{itemize}

\chapter{Symplectic and Oriented Derived Stacks}\label{sec:lagor}

Our goal in this chapter is to construct symmetric monoidal
$(\infty,n)$-categories of symplectic and oriented derived stacks and
prove that all their objects are fully dualizable.

We begin by reviewing
the basic objects we will work with: (pre)symplectic derived stacks in
\S\ref{subsec:symp} and (pre)oriented derived stacks in
\S\ref{subsec:or}. Then we discuss the \icats{} of presymplectic and
preoriented stacks in \S\ref{subsec:psympcat}.  Next, we review
Lagrangian correspondences in \S\ref{sec:lag} and oriented cospans in
\S\ref{subsec:orcospan} before we introduce their higher-dimensional
analogues in \S\ref{subsec:defitlag} and
\S\ref{subsec:higherorcospan}, respectively. In \S\ref{subsec:nondeg}
we provide some alternative characterizations of these objects, before
we construct the symplectic and oriented $(\infty,n)$-categories in
\S\ref{subsec:itlagcat}. We provide symmetric monoidal
structures on them in \S\ref{subsec:itlagsymmon} and prove that all  objects therein
are fully dualizable in \S\ref{subsec:dual}. Finally, we discuss the
special case of  oriented (higher) cospans of spaces in
\S\ref{subsec:orspc}.

\section{Symplectic Derived Stacks}\label{subsec:symp}
In this section we review the definition of symplectic
structures on derived stacks, introduced in \cite{PTVV}.

\begin{notation}\label{not:symp}
  We will use the following notation and terminology, which is
  discussed in more detail in appendix~\ref{sec:dag}:
  \begin{itemize}
  \item $\k$ is a fixed base field of characteristic zero
  \item $\dSt$ is the \icat{} of derived stacks over $\k$
    (Definition~\ref{defn:dSt}) and for $S$ a derived stack, $\dSt_{S}
    := \dSt_{/S}$ is the \icat{} of \emph{derived $S$-stacks} (\ie{}
    derived stacks \emph{over} $S$). We will usually leave the
    adjective ``derived'' implicit and just talk about $S$-stacks.
  \item $\dStArt_{S}$ is the full subcategory of $\dSt_{S}$ on the
    Artin $S$-stacks (Definition~\ref{defn:Artst}).
  \item For $X \in \dSt_{S}$, $s$ an integer, and $p$ a non-negative
    integer, $\Apcl_{S}(X,s)$ is the space of \emph{closed $s$-shifted
      relative $p$-forms} and $\Ap_{S}(X,s)$ is the space of
    \emph{$s$-shifted relative $p$-forms}
    (Definition~\ref{defn:pforms}). As functors in $X$, these are both
    represented by objects of $\dSt_{S}$
    (Proposition~\ref{propn:Apsheaf}), which we denote by
    $\Apcl_{S}(s)$ and $\Ap_{S}(s)$, respectively.
  \item For $X \in \dSt$, $\QCoh(X)$ is the (symmetric monoidal
    stable) \icat{} of quasicoherent sheaves on $X$ (Definition~\ref{defn:QCoh}).
  \item For $X \in \dSt_{S}$, $\mathbb{L}_{X/S}$
    is the relative cotangent complex of $X$ over $S$ (Definition~\ref{defn:LXS}), when this
    exists and is unique. 
  \item If $X$ is an Artin $S$-stack, then
    $\mathbb{L}_{X/S}$ exists (uniquely) and is a dualizable object of
    $\QCoh(X)$ (Theorem~\ref{thm:geomhascotgt}). Its dual
    $\mathbb{T}_{X/S}:= \mathbb{L}_{X/S}^{\vee}$ is the \emph{relative
      tangent complex} of $X$ over $S$.
  \end{itemize}
\end{notation}

Classically, a symplectic manifold is a manifold equipped with a
closed non-degenerate 2-form. Since we have a good theory of
differential forms on relative Artin stacks, we can mimic this in the
setting of derived stacks:
\begin{definition}
  Let $X$ be an Artin $S$-stack. By
  Proposition~\ref{propn:formscotgt} we can identify $\Ap_{S}(X,s)$
  with the space of shifted sections
  $\mathcal{O}_{X}[-s] \to \Lambda^{p}\mathbb{L}_{X/S}$ in
  $\QCoh(X)$. In particular, any $s$-shifted 2-form $\omega$
  corresponds to a map
  $\mathcal{O}_{X}[-s] \to \Lambda^{2}\mathbb{L}_{X/S}$, which induces
  a map $\widetilde{\omega} \colon \mathbb{T}_{X/S} \to \mathbb{L}_{X/S}[s]$ since
  $\mathbb{L}_{X/S}$ is dualizable. We say $\omega$ is
  \emph{non-degenerate} if this morphism in $\QCoh(X)$ is an
  equivalence. We say that a \emph{closed} relative 2-form is
  non-degenerate if its underlying relative 2-form is non-degenerate;
  we also refer to a non-degenerate closed relative 2-form as an
  \emph{$s$-symplectic form} (or \emph{$s$-shifted symplectic form}).
\end{definition}

\begin{definition}
  An \emph{$s$-presymplectic $S$-stack} $(X,\omega)$ is an
  Artin $S$-stack $X$ together with
  $\omega \in \Atcl_{S}(X,s)$. If $\omega$ is $s$-symplectic, we say
  that $(X,\omega)$ is an \emph{$s$-symplectic $S$-stack}.
\end{definition}

\begin{examples}\ 
  \begin{enumerate}[(i)]
\item For the terminal object $S \in \dSt_{S}$ we have $\Atcl_{S}(S,s)
  \simeq *$ and $\mathbb{L}_{S/S}
  \simeq 0$, so $S$ has a unique $s$-presymplectic structure $(S,0)$,
  which is symplectic.
\item For $k$ a noetherian $\k$-algebra, and $G$ an affine
  algebraic $k$-group with Lie $k$-algebra $\mathfrak{g}$, every
  element $c\in\mathrm{Sym}^2(\mathfrak{g}^\vee)^G$ determines a
  $2$-presymplectic structure on the classifying $\Spec(k)$-stack
  $BG$, which is symplectic if and only if the pairing $c$ is
  non-degenerate in the usual sense (meaning that the induced map
  $\tilde{c}:\mathfrak{g}\to\mathfrak{g}^\vee$ is an isomorphism). It
  is very likely that a similar result is still true over an arbitrary
  noetherian base scheme $S$, and it would be interesting to see if it
  holds over a more general base stack.
  \item The derived stack $\mathbf{Perf}_{k}$ ($k$ as above) of perfect $k$-modules has a 
  $2$-symplectic structure, which is essentially given by the weight two component of the Chern 
  character (see~\cite[\S2.3]{PTVV}). 
  \item If $X$ is a derived Artin $\k$-stack, then its shifted cotangent stack $\mathbf{T}^*[n]X$, 
  that is the stack over $X$ classifying sections of $\mathbb{L}_X[n]$, carries a tautological degree $n$ 
  one-form $\lambda_X\in\mathcal{A}^1_{\k}(n)$. By composing with the de Rham differential
  $\mathcal{A}^1_{\k}(n)\to \Atcl_{\k}(X,n)$ we get an $n$-presymplectic structure on 
  $\mathbf{T}^*[n]X$, that turns out to be $n$-symplectic (see \cite{CalCot}). The proof in \cite{CalCot}
  extends vertabim over a more general base $S$. 
  \end{enumerate}
\end{examples}

\section{Oriented Derived Stacks}\label{subsec:or}
In this section we review the definition of oriented derived stacks,
also introduced in \cite{PTVV}.

\begin{notation}
  In addition to Notation~\ref{not:symp}, we will here use the
  following terminology from appendix~\ref{sec:dag}:
  \begin{itemize}
  \item A morphism $f \colon X \to Y$ in $\dSt$ is
    \emph{cocontinuous} if
    $f_{*} \colon \QCoh(X) \to \QCoh(Y)$ preserves colimits, and
    \emph{universally cocontinuous} if every base change
    of $f$ is cocontinuous
    (Definition~\ref{defn:pfcc}). Universally cocontinuous
    morphisms satisfy base change for $\QCoh$
    (Theorem~\ref{thm:bcpf}). We denote the full
subcategories of
    universally cocontinuous morphisms in $\dSt_{S}$ and
    $\Fun(\Delta^{1}, \dSt)$ by $\dSt_{S}^{\txt{UCC}}$ and
    $\Fun(\Delta^{1}, \dSt)^{\txt{UCC}}$, respectively.
  \item If $X$ is an $S$-stack, corresponding to a
    morphism $f \colon X \to S$ in $\dSt$, we will often write $\Gamma_{S}
    := f_{*}$ to avoid introducing notation for the map $f$.
  \end{itemize}
\end{notation}

\begin{definition}
  A \emph{$d$-preorientation} of an $S$-stack $X$ is a
  morphism \[[X] \colon \Gamma_{S}\mathcal{O}_{X} \to \mathcal{O}_{S}[-d]\]
  in $\QCoh(S)$.
  A \emph{$d$-preoriented} stack $(X,[X])$ over $X$ is a universally
  cocontinuous\footnote{Although the assumption of universal
  cocontinuity is not strictly speaking required to make this definition, it
  will be convenient to include it as part of the notion of
  $d$-preoriented stack as we will repeatedly need to make use of base
  change below.} morphism $X \to S$ together with a $d$-preorientation
  $[X]$ of $X$ over $S$. 
\end{definition}

\begin{definition}\label{defn:weakor}
  Let $(X,[X])$ be a $d$-preoriented $S$-stack. Given any
  $\mathcal{E}$ in $\QCoh(X)$ we have an evaluation map
  $\mathcal{E} \otimes \mathcal{E}^{\vee} \to \mathcal{O}_{X}$, where
  $\mathcal{E}^{\vee}$ denotes the (linear) dual (\ie{} the internal
  Hom from $\mathcal{E}$ to $\mathcal{O}_{X}$).  Since
  $\Gamma_{S}$ is a lax monoidal functor, this induces a morphism
  \[ \Gamma_{S}(\mathcal{E}) \otimes \Gamma_{S}(\mathcal{E}^{\vee})
    \to \Gamma_{S}(\mathcal{E} \otimes \mathcal{E}^{\vee}) \to
    \Gamma_{S}\mathcal{O}_{X} \xto{[X]} \mathcal{O}_{S}[-d].\]
  Adjoint to this we have a morphism
  \[ \Gamma_{S}(\mathcal{E}^{\vee}) \to
    \Gamma_{S}(\mathcal{E})^{\vee}[-d].\]
  We say that $(X,[X])$ is \emph{weakly $d$-oriented} if this morphism
  is an equivalence whenever $\mathcal{E}$ is a dualizable object of
  $\QCoh(X)$.
\end{definition}

\begin{definition}
  Given a pullback square
  \csquare{\sigma^{*}X}{X}{S'}{S}{\xi}{}{}{\sigma}
  in $\dSt$, from a $d$-preorientation $[X]$ of $X$ over $S$ we get a
  $d$-preorientation  of $\sigma^{*}X$ over $S'$ using base change, as the composite
  \[ \Gamma_{S'}\mathcal{O}_{\sigma^{*}X} \simeq
    \Gamma_{S'}\xi^{*}\mathcal{O}_{X} \simeq
    \sigma^{*}\Gamma_{S}\mathcal{O}_{X} \xto{\sigma^{*}[X]}
    \sigma^{*}\mathcal{O}_{S}[-d] \simeq \mathcal{O}_{S'}[-d];\]
  we will denote this by $\sigma^{*}[X]$ for simplicity. We then say
  that $(X,[X])$ is \emph{$d$-oriented} if for every
  morphism $\sigma \colon S' \to S$ the base change
  $(\sigma^{*}X,\sigma^{*}[X])$ is weakly $d$-oriented over $S'$.
\end{definition}

\begin{proposition}\label{propn:doraff}
  A $d$-preoriented $S$-stack $(X, [X])$ is $d$-oriented \IFF{} for
  every map $p \colon \Spec A \to S$, the base change
  $(p^{*}X, p^{*}[X])$ is weakly $d$-oriented over $\Spec A$.
\end{proposition}
\begin{proof}
  It suffices to show that if $(p^{*}X, p^{*}[X])$ is weakly
  $d$-oriented over $\Spec A$ for every map $p \colon \Spec A \to S$,
  then $(X,[X])$ is weakly oriented. Thus given $\mathcal{E} \in
  \QCoh(X)$ dualizable, we want to show that under this assumption the
  induced map $\Gamma_{S}(\mathcal{E}^{\vee}) \to
  (\Gamma_{S}\mathcal{E})^{\vee}[-d]$ is an equivalence. By
  Corollary~\ref{cor:qcoheqonpts} it suffices to check that $p^{*}\Gamma_{S}(\mathcal{E}^{\vee}) \to
  p^{*}(\Gamma_{S}\mathcal{E})^{\vee}[-d]$ is an equivalence for every
  point $p \colon \Spec A \to S$. Applying base change to the pullback
  square
  \csquare{p^{*}X}{X}{\Spec A}{S}{q}{}{}{p}
  we can identify this with the corresponding map
  $\Gamma_{A}(q^{*}\mathcal{E})^{\vee} \to
  (\Gamma_{A}q^{*}\mathcal{E})^{\vee}[-d]$; this is an equivalence
  since by assumption $p^{*}X$ is weakly oriented over $\Spec A$.
\end{proof}

\begin{examples}\ 
  \begin{enumerate}[(i)]
  \item The initial $S$-stack $\emptyset$ has a unique
    $d$-preorientation for every $d$ (since
    $\mathcal{O}_{\emptyset} \simeq 0$), which is a $d$-orientation
    (since $\QCoh(\emptyset) \simeq \{0\}$).
  \item Let $M$ be a compact $d$-manifold. Then the underlying space of 
  	$M$ belongs to $\mathcal{S}_{\fin}$, and thus the Betti $\k$-stack $M_B$ 
    of $M$ (that is, the constant stack associated to this space, see Definition~\ref{defn:betti}) belongs to 
    $\dSt_{\Spec_{\k}}^{\txt{UCC}}$. If $M$ is closed and oriented, then 
    integrating along the fundamental class provides a $d$-preorientation 
    $\int_M:\Gamma\mathcal O_{M_B}\simeq C^*(M,\k)\to\k[-d]$.
    One can show that this is a $d$-orientation --- this is roughly
    Poincar\'e duality for local systems on $M$, which are precisely
    quasicoherent sheaves on $M_B$ (see \S\ref{subsec:orspc}). 
  \item More generally, the Betti $\k$-stack $M_B$ associated
    with a Poincar\'e duality space $(M,[M])$ is $d$-oriented. Indeed,
    the class $[M]$ provides a $d$-orientation
    $\Gamma\mathcal{O}_{M_B}\simeq
    C^*(M,\k)\to\k[-d]$. In fact, over any base $S$,
    one can consider the constant $S$-stack associated with the space
    $M$, and $[M]$ will again provide a $d$-orientation.  It would be
    interesting to work out a generalization of this, for a (possibly
    non-trivial) $S$-family of Poincar\'e duality spaces.
  \item Similarly, it is shown in \cite{PTVV} that if $X=Y_{DR}$ is the de Rham stack of a smooth proper 
    Deligne-Mumford $k$-stack $Y$ ($k$ a noetherian $\k$-algebra), with connected geometric fibers, of $k$-dimension $d$, 
    then every choice of a fundamental de Rham class for $Y$ leads to a $2d$-orientation on $Y_{\txt{DR}}$. With the 
    same assumption on $Y$, the same conclusion holds for the Dolbeault stack $Y_{\txt{Dol}}$ if one picks a fundamental 
    class in Hodge cohomology instead. In both cases, one can safely replace $\Spec(k)$ with a (not necessarily affine) noetherian 
    $\k$-scheme $S$. It would be interesting to see if one can consider a more general base. 
  \item Finally, if $k$ and $Y$ are as above, it is also shown in
    \cite{PTVV} that every trivialization of the canonical sheaf
    $\wedge^d\Omega^1_{Y/\Spec(k)}$ leads to a $d$-orientation on $Y$
    itself. (The same observations and questions concerning
    generalizations over a more general base also apply here.)
  \end{enumerate}
\end{examples}

\section{The $\infty$-Categories of Presymplectic and Preoriented
  Stacks}\label{subsec:psympcat} 
Before we proceed any further, it is convenient to define the \icats{}
of presymplectic and preoriented stacks, and prove some of their basic
formal properties. First we recall some further notation:
\begin{notation}\ 
  \begin{itemize}
  \item As in Definition~\ref{def:Qfib}, $\mathcal{Q} \to \dSt^{\op}$
    denotes the cocartesian fibration for the functor $\QCoh(\blank)$
    (using pullback functors); it is also the cartesian fibration for
    the right adjoint pushforward functors. We write
    $\mathcal{O} \colon \dSt^{\op} \to \mathcal{Q}$ for the
    (cocartesian) section taking $X \in \dSt^{\op}$ to its structure
    sheaf $\mathcal{O}_{X}$.
  \item As in Remark~\ref{rmk:formfunctor}, we write
    $\mathrm{ClDF}^{2,s} \to \dSt^{\Delta^{1}}$ for the right
    fibration corresponding to the functor
    $\Atcl_{(\blank)}(\blank,s) \colon \dSt^{\Delta^{1},\op} \to
    \mathcal{S}$, which takes $X \to S$ to the space $\Atcl_{S}(X,s)$
    of $s$-shifted relative 2-forms on $X$ over $S$. The composite
    $\mathrm{ClDF}^{2,s} \to \dSt^{\Delta^{1}} \xto{\ev_{1}} \dSt$
    is then a cartesian fibration. The corresponding functor encodes
    the compatibility
    of closed differential 2-forms with base change: it takes
    $S$ to $\dSt_{S/\Atcl_{S}(s)}$ and a morphism $\sigma
    \colon S' \to S$ to a functor $\sigma^{*} \colon
    \dSt_{S/\Atcl_{S}(s)} \to \dSt_{S'/\Atcl_{S'}(s)}$, given by
    pullback along $\sigma$ on underlying stacks together with
    composition with a morphism $\sigma^{*}\Atcl_{S}(s) \to
    \Atcl_{S'}(s)$.
  \end{itemize}
\end{notation}

\begin{definition}
  An $s$-presymplectic stack over $S$ is equivalently a morphism
  $X \to \Atcl_{S}(s)$ in $\dSt_{S}$, so we can define the \icat{}
  $\PSymp_{S,s}$ of $s$-presymplectic stacks over $S$ as the pullback
  \[ \PSymp_{S,s} := \dStArt_{S} \times_{\dSt_{S}}
    \dSt_{S/\Atcl_{S}(s)}.\]
  Equivalently, we can identify this with the full subcategory of
  $\dSt_{S/\Atcl_{S}(s)}$ spanned by the morphisms $X \to \Atcl_{S}(s)$
  such that the composite $X \to \Atcl_{S}(s) \to S$ is an Artin
  $S$-stack. 
\end{definition}

\begin{remark}\label{rmk:PSympRfib}
  Note that the forgetful functor $\PSymp_{S,s} \to \dStArt_{S}$
  is a right fibration, since it is a base change of the right
  fibration $\dSt_{S/\Atcl_{S}(s)} \to \dSt_{S}$.
\end{remark}

\begin{definition}
  We define 
   \[ \PSymp_{s} := \dSt^{\Delta^{1},\Art} \times_{\dSt^{\Delta^{1}}}
     \mathrm{ClDF}^{2,s},\]
   where $\dSt^{\Delta^{1},\Art}$ is the full subcategory of
   $\dSt^{\Delta^{1}}$ spanned by the relative Artin stacks. Since
   these are closed under base change, the projection $\PSymp_{s} \to
   \dSt$ is again a cartesian fibration, which encodes compatible
   functors
   \[\sigma^{*} \colon \PSymp_{S,s} \to \PSymp_{S',s} \]
   for $\sigma \colon S' \to S$, given by pullback along $\sigma$ on
   underlying stacks.
 \end{definition}

\begin{remark}
  For
  $(X,\omega)$ in $\PSymp_{S,s}$, with $\omega$ corresponding to $\tilde{\omega}
  \colon \mathbb{T}_{X/S} \to \mathbb{L}_{X/S}[s]$, in terms of the pullback square
  \csquare{X'}{X}{S'}{S,}{\xi}{}{}{\sigma}
  the closed 2-form
  of the presymplectic
  stack $\sigma^{*}(X, \omega)$ corresponds to the map
  \[\mathbb{T}_{X'/S'} \simeq \xi^{*}\mathbb{T}_{X/S} \xto{\xi^{*}\tilde{\omega}}
    \xi^{*}\mathbb{L}_{X/S}[s] \simeq \mathbb{L}_{X'/S'}[s].\]
  
\end{remark}

\begin{lemma}\label{lem:symplocal}
  For an $s$-presymplectic $S$-stack $(X,\omega)$, the following are equivalent:
  \begin{enumerate}[(i)]
  \item $X$ is symplectic.
  \item For every morphism $\sigma \colon S' \to S$, the
    $s$-presymplectic stack $(\sigma^{*}X, \sigma^{*}\omega)$ over
    $S'$ is symplectic.
  \item For every point $p \colon \Spec A \to S$, the $s$-presymplectic stack $(p^{*}X, p^{*}\omega)$ over $\Spec A$ is symplectic.
  \end{enumerate}
\end{lemma}
\begin{proof}
  (i) holds when the morphism
  $\mathbb{T}_{X/S} \xto{\tilde{\omega}} \mathbb{L}_{X/S}[s]$ induced
  by $\omega$ is an equivalence. Given $\sigma \colon S' \to S$, by
  applying Lemma~\ref{lem:relcotgtpb} we can identify
  $\sigma^{*}\tilde{\omega}$ with
  $\widetilde{\sigma^{*}\omega} \colon \mathbb{T}_{\sigma^{*}X/S'} \to
  \mathbb{L}_{\sigma^{*}X/S'}[s]$. Hence this is an equivalence if
  $\tilde{\omega}$ is one, which proves that (i) implies (ii). Since (iii)
  is a special case of (ii), it remains to prove that (iii) implies
  (i). But by Lemma~\ref{cor:qcoheqonpts} the morphism
  $\tilde{\omega}$ is an equivalence \IFF{} the morphism
  $p^{*}\tilde{\omega}$ is an equivalence for every
  $p \colon \Spec A \to S$. Since
  $p^{*}\tilde{\omega} \simeq \widetilde{p^{*}\omega}$, this completes
  the proof.
\end{proof}

\begin{lemma}\label{lem:PSymppullback}
  The \icat{} $\PSymp_{S,s}$ has pullbacks (and more generally all
  weakly contractible finite limits), and these are computed in $\dSt_{S}$.
\end{lemma}
\begin{proof}
  The \icat{} $\dSt_{S/\Atcl_{S}(s)}$ has all limits, and the
  forgetful functor $\dSt_{S/\Atcl_{S}(s)} \to \dSt_{S}$ preserves
  weakly contractible limits (as is true for any overcategory). On the other hand, the full subcategory $\dStArt_{S}$ is
  closed under finite limits in $\dSt_{S}$ by
  Proposition~\ref{propn:geomfinlim}. In the pullback square
  \nolabelcsquare{\PSymp_{S,s}}{\dSt_{S/\Atcl_{S}(s)}}{\dStArt_{S}}{\dSt_{S}}
  we thus know that all the vertices except the top left have weakly
  contractible finite limits, and that both the bottom and right-hand
  morphisms preserve these. It follows by \cite{HTT}*{Lemma 5.4.5.5} that
  $\PSymp_{S,s}$ also has weakly contractible finite limits, and that
  these are computed in $\dSt_{S}$ (since both functors to $\dSt_{S}$
  detect equivalences).
\end{proof}

Next, we want to define an \icat{} $\POr_{S,d}$ of $d$-preoriented
$S$-stacks:
\begin{definition}
  From Definition~\ref{defn:GammaSO} we have a functor
  $\Gamma_{S}\mathcal{O} \colon \dSt_{S}^{\op} \to \QCoh(S)$, taking
  an $S$-stack $f \colon X \to S$ to the pushforward
  $f_{*}\mathcal{O}_{X}$ of its structure sheaf.
  Now we define $\POr_{S,d}$ by the pullback
  \csquare{\POr_{S,d}^{\op}}{\QCoh(S)_{/\mathcal{O}_{S}[-d]}}{\dStUCCop_{S}}{\QCoh(S),}{}{}{}{\Gamma_{S}\mathcal{O}}
  where the right vertical map is the forgetful functor.
\end{definition}

\begin{remark}\label{rmk:POrLfib}
  Note that the forgetful functor $\POr_{S,d} \to \dStUCC_{S}$ is a
  left fibration, since it is a base change of the left fibration
  $(\QCoh(S)_{/\mathcal{O}_{S}[-d]})^{\op} \to \QCoh(S)^{\op}$. In
  particular, if $(T,[T])$ is $d$-preoriented and $f \colon T \to T'$
  is any morphism in $\dSt_{S}$, then $T'$ has a canonical
  preorientation $f_{*}[T]$, given by the composite
  \[ \Gamma_{S}\mathcal{O}_{T'} \to \Gamma_{S}\mathcal{O}_{T}
    \xto{[T]} \mathcal{O}_{S}[-d].\]
\end{remark}

\begin{lemma}\label{lem:POreqces}
  The forgetful functor $\POr_{S,d} \to \dSt_{S}$ detects equivalences.
\end{lemma}
\begin{proof}
  A morphism in the pullback $\POr_{S,d}$ is an equivalence \IFF{} its
  images in $\dSt_{S}$ and $\QCoh(S)_{/\mathcal{O}_{S}[-d]}^{\op}$ are
  equivalences. Since the forgetful functor
  $\QCoh(S)_{/\mathcal{O}_{S}[-d]}^{\op} \to \QCoh(S)^{\op}$ detects
  equivalences, this is true \IFF{} the image in $\dSt_{S}$ is an
  equivalence.
\end{proof}

\begin{proposition}\label{propn:POrpushout}
  The \icat{} $\POr_{S,d}$ has pushouts (and more generally all finite
  weakly contractible colimits) and these are computed in $\dSt_{S}$.
\end{proposition}
\begin{proof}
  Consider the pullback diagram defining $\POr_{S,d}$:
  \csquare{\POr_{S,d}^{\op}}{\QCoh(S)_{/\mathcal{O}_{S}[-d]}}{\dSt_{S}^{\txt{UCC},\op}}{\QCoh(S).}{}{}{}{\Gamma_{S}\mathcal{O}}
  According to Proposition~\ref{propn:UCCcolim}, the \icat{}
  $\dSt_{S}^{\txt{UCC}}$ has finite colimits and these are computed
  in $\dSt_{S}$. Moreover,
  $\QCoh(S)$ has all limits, and the forgetful functor
  $\QCoh(S)_{/\mathcal{O}_{S}[-d]} \to \QCoh(S)$ preserves weakly
  contractible limits. It remains to show that $\Gamma_{S}\mathcal{O}$
  takes colimits in $\dSt_{S}$ to limits in $\QCoh(S)$. Given a colimit
  $X \simeq \colim_{i \in \mathcal{I}} X_{i}$ in $\dSt_{S}$, by
  Corollary~\ref{cor:qcohrellim} we have an equivalence
  $\mathcal{O}_{X} \simeq \lim_{i \in \mathcal{I}^{\op}}
  f_{i,*}\mathcal{O}_{X_{i}}$ in $\QCoh(X)$, where $f_{i}$ is the map
  $X_{i} \to X$ in the colimit diagram. Applying $\Gamma_{S}$, we get
  the required equivalence
  $\Gamma_{S}\mathcal{O}_{X} \simeq \lim_{i \in \mathcal{I}^{\op}}
  \Gamma_{S}\mathcal{O}_{X_{i}}$.
\end{proof}

Now we want to show that the \icats{} $\POr_{S,d}$ are functorial in
$S$ with respect to pullback.
\begin{construction}
  Let $\mathcal{Q}_{/\mathcal{O}[-d]}$ denote the fibrewise slice of
  the cocartesian fibration $\mathcal{Q} \to \dSt^{\op}$ over the
  section $\mathcal{O}[-d]$, \ie{}
  the pullback
  \[
    \begin{tikzcd}
      \mathcal{Q}_{/\mathcal{O}[-d]} \arrow{r} \arrow{d} &
      \mathcal{Q}^{\Delta^{1}} \times_{(\dSt^{\op})^{\Delta^{1}}} \dSt
      \arrow{d}{\txt{ev}_{1}} \\
      \dSt^{\op} \arrow{r}{\mathcal{O}[-d]} & \mathcal{Q}.
    \end{tikzcd}
    \]
  Since $\mathcal{O}$ is a cocartesian section,
  this is a pullback square of cocartesian fibrations over $\dSt^{\op}$
  along functors that preserve cocartesian morphisms. Hence, it gives a cocartesian fibration
  $\mathcal{Q}_{/\mathcal{O}[-d]}\to \dSt^{\op}$ such that the forgetful
  functor $\mathcal{Q}_{/\mathcal{O}[-d]} \to \mathcal{Q}$ preserves
  cocartesian morphisms. We can then form the pullback
  \csquare{\POr_{d}^{\op}}{\mathcal{Q}_{/\mathcal{O}[-d]}}{\Fun(\Delta^{1},
    \dSt)^{\txt{UCC},\op}}{\mathcal{Q},}{}{}{}{\Gamma\mathcal{O}}
  where $\Gamma \mathcal{O}$ is the functor defined in
  Construction~\ref{constr:GammaO} (which takes $f \colon X \to S$ to
  $f_{*}\mathcal{O}_{X} \in \QCoh(S)$). Since the bottom and
  right-hand functors preserve cocartesian morphisms, the result is a
  cocartesian fibration $\POr_{d}^{\op} \to \dSt^{\op}$. Passing to
  opposite categories, we get a cartesian fibration
  $\POr_{d} \to \dSt$, which corresponds to the desired functor
  $\dSt^{\op} \to \CatI$ taking $S$ to $\POr_{S,d}$.
\end{construction}

\begin{lemma}\label{lem:POrpbpo}
  For any map $\sigma \colon S' \to S$, the pullback functor
  $\POr_{S,d} \to \POr_{S',d}$ preserves pushouts.
\end{lemma}
\begin{proof}
  Since $\dSt$ is an $\infty$-topos,
  $\sigma^{*}\colon \dSt_{S} \to \dSt_{S'}$ preserves colimits. By
  Proposition~\ref{propn:POrpushout} pushouts in $\POr_{S',d}$ are
  computed in $\dSt_{S'}$; since the forgetful functor
  $\POr_{S',d} \to \dSt_{S'}$ also detects equivalences by
  Lemma~\ref{lem:POreqces}, it follows that $\sigma^{*}$ preserves
  pushouts.
\end{proof}

\section{Lagrangian Correspondences}
\label{sec:lag}
Here we review the definitions of isotropic and Lagrangian morphisms
between (pre)symplectic derived stacks from \cite{PTVV}, as well as those
of isotropic and Lagrangian correspondences from \cite{CalaqueTFT}. 

Recall that if $(M,\omega)$ is a symplectic manifold, then an
\emph{isotropic} morphism is a smooth morphism $f \colon N \to M$ such
that $f^{*}\omega = 0$ (though usually this is only considered for
submanifolds). This notion also has a natural analogue for derived
stacks, but in this setting ``being $0$'' is not a property, but
rather extra structure:
\begin{definition}
  Suppose $(X,\omega)$ is an $s$-presymplectic $S$-stack. An
  \emph{isotropic morphism} is a morphism of Artin $S$-stacks $f\colon L \to X$
  together with an equivalence $f^{*}\omega \simeq 0$ of closed relative
  2-forms in $L$, \ie{} a path from $0$ to $f^{*}\omega$ in the space
  $\Atcl_{S}(L,s)$. 
\end{definition}

\begin{remark}
  Equivalently, an isotropic structure on $f\colon L \to X$ is an equivalence
  $f^{*}\omega \simeq 0$ of morphisms $L \to \Atcl_{S}(s)$.
  Since $f^{*}\omega$ is the composite morphism $L \xto{f} X
  \xto{\omega} \Atcl_{S}(s)$, and $0$ is the composite $L \to S
  \xto{0} \Atcl_{S}(s)$, we can also describe this as a commutative
  square
  \csquare{L}{X}{S}{\Atcl_{S}(s).}{f}{}{\omega}{0}
  Thus, an isotropic morphism is the same thing as a \emph{span}
  \[
    \begin{tikzcd}
      {} & L \arrow{dl} \arrow{dr}{f} \\
      S & & X
    \end{tikzcd}
  \]
  in the \icat{} $\PSymp_{S,s}$. This motivates the following definition:
\end{remark}
\begin{definition}
  Let $X$ and $Y$ be $s$-presymplectic $S$-stacks. An
  \emph{isotropic correspondence} from $X$ to $Y$ is a span
    \[
    \begin{tikzcd}
      {} & L \arrow{dl}[swap]{f} \arrow{dr}{g} \\
      X & & Y
    \end{tikzcd}
  \]
  in the \icat{} $\PSymp_{S,s}$.
\end{definition}
\begin{remark}
  Unwinding the definition, we see that, if $\omega_{X}$ and
  $\omega_{Y}$ denote the presymplectic forms on $X$ and $Y$,
  respectively, then an isotropic correspondence is given by a span of
  derived stacks as above, together with an equivalence of closed relative 2-forms
  $f^{*}\omega_{X} \simeq g^{*}\omega_{Y}$. This corresponds to an
  equivalence $f^{*}\omega_{X} - g^{*}\omega_{Y} \simeq 0$, which
  allows us to identify isotropic correspondences from $X$ to $Y$ with
  isotropic morphisms to $X \times \overline{Y}$, where $\overline{Y}$ denotes
  $Y$ equipped with the negative presymplectic form $-\omega_{Y}$.
\end{remark}

In symplectic geometry, a key role is played by \emph{Lagrangian}
submanifolds, which are the isotropic submanifolds of maximal
dimension. In our setting, we can similarly single out a class of
\emph{Lagrangian} morphisms among the isotropic ones:
\begin{definition}
  Suppose $(X,\omega)$ is an $s$-presymplectic stack over $S$, and $\tilde{\omega}
  \colon \mathbb{T}_{X/S} \to \mathbb{L}_{X/S}[-s]$ the corresponding
  morphism in $\QCoh(X)$. Any morphism $f \colon L \to X$ gives
  morphisms $\mathbb{T}_{L/S} \to f^{*}\mathbb{T}_{X/S}$ and (dually)
  $f^{*}\mathbb{L}_{X/S} \to \mathbb{L}_{L/S}$. We thus have a
  composite morphism
  \[ \mathbb{T}_{L/S} \to f^{*}\mathbb{T}_{X/S} \to
    f^{*}\mathbb{L}_{X/S}[-s] \to \mathbb{L}_{L/S}[-s].\]
  An isotropic structure on $f$ gives an equivalence between this and
  the zero morphism, which we can interpret as a commutative square
  \nolabelcsquare{\mathbb{T}_{L/S}}{f^{*}\mathbb{T}_{X/S}}{0}{\mathbb{L}_{L/S}[-s].}
  We say the isotropic morphism is \emph{Lagrangian} if this square is
  cartesian, or in other words if we have an exact triangle (or (co)fibre sequence)
  \[ \mathbb{T}_{L/S} \to f^{*}\mathbb{T}_{X/S} \to
    \mathbb{L}_{L/S}[-s].\]
\end{definition}
\begin{remark}
  To see how this definition relates to the classical notion of
  Lagrangian submanifolds, observe that in the setting of symplectic manifolds an isotropic submanifold $L
  \to M$ also gives a commutative diagram of vector bundles on $L$,
  \nolabelcsquare{T_L}{f^* T_{M}}{0}{T_{L}^*,}
  where the top horizontal morphism is a subbundle.
  Since $\dim T_{L} = \dim T_{L}^{*} = \dim L$ and $\dim f^{*}T_{M} =
  \dim M$, we see that this is a pullback square \IFF{} we have $\dim
  M = 2 \dim L$, \ie{} \IFF{} $L$ is a Lagrangian submanifold.
\end{remark}

This definition has a natural generalization to general isotropic
correspondences:
\begin{definition}\label{defn:Lagcorr}
  An isotropic correspondence
    \[
    \begin{tikzcd}
      {} & L \arrow{dl}[swap]{f} \arrow{dr}{g} \\
      X & & Y
    \end{tikzcd}
  \]
  over $S$ induces a commutative square
  \csquare{\mathbb{T}_{L/S}}{g^{*}\mathbb{T}_{Y/S}}{f^{*}\mathbb{T}_{X/S}}{\mathbb{L}_{L/S}[-s]}{g_*}{f_*}{}{}
  of quasicoherent sheaves on $L$. We say the isotropic correspondence
  is \emph{Lagrangian} if this square is cartesian.
\end{definition}

\begin{lemma}\label{lem:laglocal}
  The following are equivalent for a span in $\PSymp_{S,s}$:
  \begin{enumerate}[(i)]
  \item The given span is a Lagrangian correspondence.
  \item For every morphism $\sigma \colon S' \to S$, the
    pullback of the span to $S'$ is Lagrangian.
  \item For every point $p \colon \Spec A \to S$, the pullback of the
    span to $\Spec A$ is Lagrangian.
  \end{enumerate}
\end{lemma}
\begin{proof}
  As Lemma~\ref{lem:symplocal}, using Lemma~\ref{lem:relcotgtpb} and
  Corollary~\ref{cor:colimeqonpoints} to see that we can check the
  limit condition pointwise.
\end{proof}

\begin{examples}\ 
  \begin{enumerate}[(i)]
  \item It follows from~\cite{CalCot} that for every morphism of Artin 
  stacks $f:X\to Y$, there is Lagrangian structure on the correspondence 
  $\mathbf{T}^*[s]X\leftarrow f^*\mathbf{T}^*[s]Y\rightarrow \mathbf{T}^*[s]Y$. 
  Indeed, the pullbacks of the tautological $n$-shifted $1$-forms $\lambda_X$ and 
  $\lambda_Y$ coincide on $f^*\mathbf{T}^*[s]Y$, so that we get an isotropic 
  correspondence, which happens to be nondegenerate. The statement and proofs in \cite{CalCot}
  are given over $\Spec \k$, but remain exactly the same over an
  arbitrary base. 
  \item Let $X$ be a smooth symplectic algebraic variety equipped with two commuting 
  Hamiltonian actions of affine algebraic groups $G$ and $H$, with respective moment 
  maps $\mu:X\to \mathfrak{g}^\vee$ and $\nu:X\to\mathfrak{h}^\vee$. Then 
  there is a Lagrangian structure on  
   \[
    \begin{tikzcd}
      {} & {[X/G\times H]} \arrow{dl}[swap]{\mu} \arrow{dr}{-\nu} \\
      {[\mathfrak{g}^\vee/G]} & & {[\mathfrak{h}^\vee/H]}
    \end{tikzcd}
  \]
  To see this, one uses the following facts: first of all $[\mathfrak{g}^\vee/G]\simeq\mathbf{T}^*[1](BG)$ is $1$-symplectic, 
  then a moment map $\mu$ as above leads to a Lagrangian morphism $[X/G]\to [\mathfrak{g}^\vee/G]$ (see~\cite{CalaqueTFT}), 
  and finally the fibrewise sign reversal on shifted cotangents gives an equivalence of $n$-symplectic stacks
  $\mathbb{T}^*[s]Z\simeq \overline{\mathbb{T}^*[s]Z}$ (as it sends $\lambda_Z$ to $-\lambda_Z$) for every Artin stack $Z$. 
\item Over a general base $S$, which we equip with the trivial
  $s$-symplectic structure, for every Artin $S$-stack $X$ the space of
  Lagrangian (isotropic) structures on $S\leftarrow X\rightarrow S$ is
  equivalent to the space of $(s-1)$-(pre)symplectic
  structures on $X$. This was first observed in~\cite{CalaqueTFT} for
  $S=\Spec \k$, and the generalization to an arbitrary base is
  straightforward.  Indeed, the statement for isotropic versus
  presymplectic structures is easy, and then one checks that the
  non-degeneracy condition (which is a property) is the same on both
  sides of the equivalence. See \S\ref{subsec:itlagsymmon} and
  \S\ref{subsec:dual} below for more general statements.
  \end{enumerate}
\end{examples}

\begin{remark}
  Since $\QCoh(L)$ is a stable \icat{}, the square in
  Definition~\ref{defn:Lagcorr} is cartesian \IFF{} the corresponding
  square \csquare{\mathbb{T}_{L/S}}{f^*\mathbb{T}_{X/S} \oplus
    g^*\mathbb{T}_{X/S}}{0}{\mathbb{L}_{L/S}[-s]}{(f_*,-g_*)}{}{}{} is
  cartesian. It follows that the isotropic correspondence is
  Lagrangian precisely when the corresponding isotropic morphism
  $L \to X \times \overline{Y}$ is Lagrangian. Thus, we could
  equivalently have defined Lagrangian correspondences from $X$ to $Y$
  to be Lagrangian morphisms with target $X \times \overline{Y}$. The
  latter is close to the usual definition of Lagrangian correspondences
  between symplectic manifolds, and is also the formulation of the
  definition introduced in \cite{CalaqueTFT}. On the other hand, our
  definition has the advantage that it suggests a construction of an
  \icat{} of symplectic stacks with Lagrangian correspondences as
  morphisms, using the following observation:
\end{remark}

\begin{remark}\label{rmk:SpanPSymp}
  The \icat{} $\PSymp_{S,s}$ has pullbacks by
  Lemma~\ref{lem:PSymppullback}. It follows (using for example the
  simplest case of the construction of \cite{spans}*{\S 5}) that there
  is an \icat{} $\Span(\PSymp_{S,s})$ whose objects are
  $s$-presymplectic stacks over $S$ and whose morphisms are isotropic
  correspondences, with composition given by taking pullbacks. Since
  $s$-symplectic stacks are $s$-presymplectic stacks satisfying an
  additional property, and similarly for Lagrangian correspondences,
  we can then define an \icat{} of symplectic stacks and Lagrangian
  correspondences as a \emph{subcategory} of $\Span(\PSymp_{S,s})$ ---
  the only thing we need to check is that the composite of two
  Lagrangian correspondences is again Lagrangian. This is \cite{CalaqueTFT}*{Theorem 4.4} (see also
  \cite{spans}*{Proposition 14.12} for a version of the proof that is
  closer to our generalization below).

The construction of \cite{spans}*{\S 5} (which we review in
\S\ref{subsec:spans}) also gives more generally an
$(\infty,n)$-category $\Span_{n}(\PSymp_{S,s})$ of $s$-presymplectic
stacks, with $i$-morphisms given by ``$i$-fold isotropic
correspondences''. Our strategy for defining an $(\infty,n)$-category
of $s$-symplectic stacks is to define $i$-fold Lagrangian
correspondences as $i$-fold isotropic correspondences satisfying a
non-degeneracy condition, and then check that these are closed under
composition in $\Span_{n}(\PSymp_{S,s})$, so that there is a
sub-$(\infty,n)$-category where the $i$-morphisms are Lagrangian for
all $i$.
\end{remark}

\section{Oriented Cospans}\label{subsec:orcospan}
In this section we review the definitions of oriented morphisms and
cospans between derived stacks, introduced in \cite{CalaqueTFT}.

First we consider preoriented morphisms and cospans, which are
analogous to isotropic morphisms and correspondences:
\begin{definition}
  Let $(X, [X])$ be a $d$-preoriented $S$-stack. A
  \emph{$d$-preorientation} of a morphism $\phi \colon X \to Y$ in
  $\dSt_{S}^{\UCC}$ is an equivalence $\phi_{*}[X] \simeq 0$ in the space of
  maps $\Gamma_{S}\mathcal{O}_{Y} \to \mathcal{O}_{S}[-d]$; a
  $d$-preoriented morphism is a morphism equipped with a
  $d$-preorientation.
\end{definition}

\begin{remark}
  Equivalently, a $d$-preorientation on $\phi \colon X \to Y$ is a commutative square
  \nolabelcsquare{\Gamma_{S}\mathcal{O}_{Y}}{\Gamma_{S}\mathcal{O}_{X},}{0}{\mathcal{O}_{S}[-d]}
  which we can identify with a cospan
  \[
    \begin{tikzcd}
      \emptyset \arrow{dr} & & X \arrow{dl} \\
       & Y
    \end{tikzcd}
  \]
  in $\POr_{S,d}$. More generally, we can make the following definition:
\end{remark}
\begin{definition}
  Let $X$ and $Y$ be $d$-preoriented $S$-stacks. A \emph{$d$-preoriented cospan} from $X$ to $Y$ is a cospan
    \[
    \begin{tikzcd}
      X \arrow{dr} & & Y \arrow{dl} \\
       & Z
    \end{tikzcd}
  \]
  in $\POr_{S,d}$. This amounts to specifying a commutative square
  \csquare{\Gamma_{S}\mathcal{O}_{Z}}{\Gamma_{S}\mathcal{O}_{X}}{\Gamma_{S}\mathcal{O}_{Y}}{\mathcal{O}_{S}[-d]}{}{}{{[X]}}{{[Y]}}
  in $\QCoh(S)$.
\end{definition}

Now we consider a non-degeneracy condition on preoriented cospans, in a sense analogous to that for Lagrangian correspondences:
\begin{definition}
  Suppose     \[
    \begin{tikzcd}
      X \arrow{dr}[swap]{f} & & Y \arrow{dl}{g} \\
       & Z
    \end{tikzcd}
  \]
  is a $d$-preoriented cospan over $S$, and $\mathcal{E}$ is a dualizable object of $\QCoh(Z)$. Then from the evaluation pairing $\mathcal{E} \otimes \mathcal{E}^{\vee}
\to \mathcal{O}_{Z}$ we get a diagram
\[
\begin{tikzcd}
  \Gamma_{S}(f^{*}\mathcal{E}) \otimes
  \Gamma_{S}(f^{*}\mathcal{E}^{\vee}) \arrow{d} &  \Gamma_{S}(\mathcal{E}) \otimes
  \Gamma_{S}(\mathcal{E}^{\vee}) \arrow{d} \arrow{l} \arrow{r}&  \Gamma_{S}(g^{*}\mathcal{E}) \otimes
  \Gamma_{S}(g^{*}\mathcal{E}^{\vee})\arrow{d} \\
  \Gamma_{S}(f^{*}\mathcal{E}\otimes f^{*}\mathcal{E}^{\vee})\arrow{d} &
  \Gamma_{S}(\mathcal{E} \otimes \mathcal{E}^{\vee})\arrow{d} \arrow{l} \arrow{r}&
  \Gamma_{S}(g^{*}\mathcal{E} \otimes g^{*}\mathcal{E}^{\vee})\arrow{d} \\
  \Gamma_{S}(\mathcal{O}_{X})\arrow{dr} & \Gamma_{S}(\mathcal{O}_{Z})
  \arrow{l} \arrow{d} \arrow{r}&
  \Gamma_{S}(\mathcal{O}_{Y}) \arrow{dl}\\
  {} & \mathcal{O}_{S}[-d]. &
\end{tikzcd}
\]
In particular, we get a commutative square
\nolabelcsquare{\Gamma_{S}(\mathcal{E}) \otimes
  \Gamma_{S}(\mathcal{E}^{\vee})}{\Gamma_{S}(f^{*}\mathcal{E}) \otimes
  \Gamma_{S}(f^{*}\mathcal{E}^{\vee})}{\Gamma_{S}(g^{*}\mathcal{E})
  \otimes \Gamma_{S}(g^{*}\mathcal{E}^{\vee})}{\mathcal{O}_{S}[-d].}
Taking duals, we get from this a diagram
\[ 
\begin{tikzcd}
{}    & \Gamma_{S}(\mathcal{E}^{\vee}) \arrow{dl} \arrow{dr}  \\
 \Gamma_{S}(f^{*}\mathcal{E}^{\vee}) \arrow{d} & & \Gamma_{S}(g^{*}\mathcal{E}^{\vee}) \arrow{d}\\
 \Gamma_{S}(f^{*}\mathcal{E})^{\vee}[-d]  \arrow{dr}& &
 \Gamma_{S}(g^{*}\mathcal{E})^{\vee}[-d]  \arrow{dl}\\ 
{}    & \Gamma_{S}(\mathcal{E})^{\vee}[-d].
\end{tikzcd}\]
If $X$ and $Y$ are $d$-oriented, then the two vertical morphisms are
equivalences, and this reduces to a commutative square
\nolabelcsquare{\Gamma_{S}(\mathcal{E}^{\vee})}{\Gamma_{S}(f^{*}\mathcal{E}^{\vee})}{\Gamma_{S}(g^{*}\mathcal{E}^{\vee})}{\Gamma_{S}(\mathcal{E})^{\vee}[-d].}
We say that the cospan is \emph{weakly $d$-oriented} over $S$ if this square is cartesian for every dualizable $\mathcal{E} \in \QCoh(Z)$, and that the cospan is \emph{$d$-oriented} if for every morphism $\sigma \colon S' \to S$ the base-change to a cospan over $S'$ is weakly $d$-oriented.
\end{definition}

\begin{lemma}
  A cospan
  \[
    \begin{tikzcd}
     X \arrow{dr}[swap]{f} & & Y \arrow{dl}{g} \\
       & Z
    \end{tikzcd}
  \]
  in $\POr_{S,d}$ is $d$-oriented \IFF{} for every point $p \colon \Spec A \to S$, the pulled-back cospan
    \[
    \begin{tikzcd}
      p^{*}X \arrow{dr} & & p^{*}Y \arrow{dl}\\
       & p^{*}Z
    \end{tikzcd}
  \]
  is weakly $d$-oriented over $\Spec A$.
\end{lemma}
\begin{proof}
  This follows by a slightly more involved version of the argument in
  the proof of Proposition~\ref{propn:doraff}. (We will give a careful
  proof of a more elaborate version of this statement below in
  Lemma~\ref{lem:ncospanorpts}.)
\end{proof}

\begin{remark}
  A $d$-preoriented morphism $f \colon X \to Y$ over $S$ is weakly
  $d$-oriented (when viewed as a $d$-preoriented cospan from
  $\emptyset$ to $X$) if for every dualizable $\mathcal{E} \in
  \QCoh(X)$, the induced diagram
  \nolabelcsquare{\Gamma_{S}(\mathcal{E}^{\vee})}{\Gamma_{S}(f^*
    \mathcal{E}^\vee)}{0}{\Gamma_{S}(\mathcal{E})^{\vee}[-d]}
  is cartesian, \ie{} if we have a canonical exact triangle (or (co)fibre sequence)
  \[ \Gamma_{S}(\mathcal{E}^{\vee}) \to
    \Gamma_{S}(f^{*}\mathcal{E}^{\vee}) \to
    \Gamma_{S}(\mathcal{E})^{\vee}[-d],\]
  in $\QCoh(S)$.
\end{remark}

\begin{remark}
  Just as in the case of Lagrangian correspondences, we can
  equivalently define a (weakly) $d$-oriented cospan from $X$ to $Y$
  to be a (weakly) $d$-oriented morphism from the coproduct $X \amalg \overline{Y}$,
  where $\overline{Y}$ denotes $Y$ with the negative orientation $-[Y]$;
  this is the definition used in \cite{CalaqueTFT}.
\end{remark}

\begin{examples}\
  \begin{enumerate}[(i)]
  \item Let us consider the initial $S$-stack $\emptyset$ together with its unique $d$-orientation. Then for every stack $X$, the space of $d$-orientations 
  on the cospan $\emptyset\rightarrow X\leftarrow \emptyset$ is equivalent to the space of $(d+1)$-orientations on $X$ (see~\cite{CalaqueTFT}). 
  As in the case of Lagrangian self-correspondences of the terminal $S$-stack, one first observes that the statement is very easy for preorientations, and 
  then that the non-degeneracy condition (a property) coincides on both sides of the equivalence. We again refer to \S\ref{subsec:itlagsymmon} and 
  \S\ref{subsec:dual} below for more general statements. 
\item Let $N$ be a compact oriented cobordism between oriented closed
  $d$-manifolds $M^+$ and $M^-$. As we have already seen, the Betti
  stacks of $M^+$ and $M^-$ are $d$-oriented stacks. This lifts to an
  orientation on the cospan $M^+_B\rightarrow N_B\leftarrow M^-_B$.
  The homotopy for the preorientations is given by $\int_N$; indeed,
  Stokes's theorem tells us that it is a homotopy for
  $\int_{M^+}-\int_{M^-}$.
\item Let $X$ be a geometrically connected smooth and proper algebraic
  $k$-variety of dimension $d+1$, with $k$ a noetherian $\k$-algebra,
  and let $s:\omega_{X/k}\to\mathcal O_X$ be a section of the
  anticanonical sheaf $\omega_{X/k}^{-1}$. Then the closed embedding
  $Z(s)\hookrightarrow X$ of the derived zero locus of $s$ carries a
  $d$-orientation (see~\cite[\S5.3]{ToenEMS}). In the case where
  $Z(s)$ is a smooth Cartier divisor, then the embedding is
  Calabi--Yau (\ie{} its canonical sheaf is trivializable), and one
  recovers the example from~\cite[\S3.2.1]{CalaqueTFT}. (It would be
  interesting to explore the situation where $Z(s)$ is a normal
  crossing divisor made of two smooth irreducible components that
  intersect transversely; we may obtain an oriented cospan between
  the two components.)
  \end{enumerate}
\end{examples}

\begin{remark}
  The \icat{} $\POr_{S,d}$ has pushouts by Proposition~\ref{propn:POrpushout}. We can therefore define an \icat{} $\Cospan(\POr_{S,d}) :=
\Span(\POr_{S,d}^{\op})$ of $d$-preoriented stacks and cospans, with
composition given by taking pushouts in $\POr_{S,d}$. Since
$d$-oriented stacks and cospans are $d$-preoriented ones satisfying an
additional property, we can define an \icat{} of these as a
subcategory of $\Cospan(\POr_{S,d})$ --- we need only check that
$d$-oriented cospans are closed under composition, which is proved as
\cite{CalaqueTFT}*{Theorem 4.3}. Just as in the symplectic case, we
will similarly consider an $(\infty,n)$-category
$\Cospan_{n}(\POr_{S,d})$ and obtain an $(\infty,n)$-category of
$d$-oriented stacks as a sub-$(\infty,n)$-category where the $i$-morphisms are
``$d$-oriented $i$-fold cospans'' by checking that these are closed
under composition.
\end{remark}

\begin{remark}
  The \icats{}
  $\Cospan(\POr_{S,d})$ are contravariantly functorial in $S$ via
  pullback by Lemma~\ref{lem:POrpbpo}. Moreover, the subcategories of $d$-oriented stacks and
  cospans are preserved by pullbacks, so these are also functorial in $S$.
\end{remark}

\section{Higher Lagrangian Correspondences}\label{subsec:defitlag}
In this section we will introduce the definition of higher
Lagrangian correspondences, which we will view as a non-degeneracy
condition on iterated spans in $\PSymp_{S,s}$. This requires two
preliminary steps, which are analogues of the following observations
that were implicit in our definition of Lagrangian correspondences
above:
\begin{enumerate}[(1)]
\item From a span $X \xfrom{f} L \xto{g} Y$ of $s$-presymplectic
  stacks we obtain a commutative square
  \nolabelcsquare{\mathcal{O}_{L}}{\Lambda^{2}
    f^{*}\mathbb{L}_{X/S}[s]}{\Lambda^{2} g^*
    \mathbb{L}_{Y/S}[s]}{\Lambda^{2} \mathbb{L}_{L/S}[s]}
  in $\QCoh(L)$.
\item From such a diagram, we can use the dualizability of the
  cotangent complexes to ``move one factor in the tensor products to
  the other side'', giving a diagram
  \[ 
    \begin{tikzcd}
      {}    & \mathbb{T}_{L/S} \arrow{dl} \arrow{dr}  \\
      f^{*}\mathbb{T}_{X/S}\arrow{d} & & g^{*}\mathbb{T}_{Y/S}
      \arrow{d} \\
      f^{*}\mathbb{L}_{X/S}[s] \arrow{dr}& & g^{*}\mathbb{L}_{Y/S}[s] \arrow{dl}\\ 
      {}    & \mathbb{L}_{L/S}[s]
    \end{tikzcd}\]
  in $\QCoh(L)$.
\end{enumerate}
We can now say that the span is \emph{Lagrangian} if the diagram in
(2) is a limit diagram. 
\begin{remark}\label{rmk:Lagsquare}
  To recover the definition we gave above, we also need to observe
  that the diagram in (2) is a limit diagram precisely when the square
  \nolabelcsquare{\mathbb{T}_{L/S}}{f^{*}\mathbb{T}_{X/S}}{g^{*}\mathbb{T}_{Y/S}}{\mathbb{L}_{L/S}[s]}
  is cartesian.
\end{remark}

\begin{remark}
  There are other useful ways to simplify and reformulate the notion
  of a Lagrangian correspondence. For example, if $X$ and $Y$ are symplectic,
  then the diagram in (2) is a limit precisely when the square
  \nolabelcsquare{\mathbb{T}_{L/S}}{f^{*}\mathbb{L}_{X/S}[s]}{g^{*}\mathbb{L}_{Y/S}[s]}{\mathbb{L}_{L/S}[s]}
  is cartesian --- or equivalently cocartesian, since $\QCoh(L)$ is
  stable, which is equivalent to the diagram in (2) being a colimit
  diagram. We will prove versions of these statements for iterated
  correspondences below in \S\ref{subsec:nondeg}.
\end{remark}

\begin{notation}
  We use the following notation and terminology from
  appendix \ref{sec:cat}:
  \begin{itemize}
  \item We have partially ordered sets $\Sp^{n}$ and $\Sp^{n,\circ}$
    (Definition~\ref{defn:uplespan}), and $\lsp^{n}$ and
    $\lsp^{n,\circ}$ (Definition~\ref{defn:foldspan}).
  \item An \emph{$n$-uple span} in an \icat{} $\mathcal{C}$ is a
    functor $\Sp^{n} \to \mathcal{C}$, and an \emph{$n$-fold span} is
    a functor $\lsp^{n} \to \mathcal{C}$. By
    Lemma~\ref{lem:foldspanloc} we may identify $n$-fold spans with
    the \emph{reduced} $n$-uple spans (Definition~\ref{defn:redspan}).
  \item We have a functor assigning to an \icat{} $\mathcal{C}$ its
    \emph{twisted arrow \icat{}} $\Tw^{r}(\mathcal{C})$
    (Definition~\ref{defn:Tw}), and this functor has a left adjoint
    $\Tw^{r}_{!}$, described in Proposition~\ref{propn:Tw!desc}.
  \end{itemize}
\end{notation}

The following is the general version of step (1) above, which we prove
using some results from \cite{paradj}:
\begin{proposition}\label{propn:twoformTwDiag}
  Suppose given a diagram $p \colon K^{\triangleleft} \to \PSymp_{S,s}$
  that sends $-\infty$ to the Artin $S$-stack $X$. Then there is
  an induced diagram $\tilde{p} \colon (K^{\triangleleft})^{\op} \to
  \Tw^{r}(\QCoh(X))$ valued in the twisted arrow
  \icat{} of quasicoherent sheaves on $X$; if $X_{k}$ is the value of
  $p$ at $k \in K$ and $f_{k} \colon X \to X_{k}$ is the value of $p$
  at the unique map $-\infty \to k$, then this diagram sends $k$ to
  the map $f_{k}^{*}\mathbb{T}_{X_{k}/S} \to
  f_{k}^{*}\mathbb{L}_{X_{k}/S}[n]$ induced by the 2-form on $X_{k}$.
\end{proposition}
\begin{proof}
  By Lemma~\ref{lem:twoformftr} there is a functor $\PSymp_{S,s}^{\op}
  \to \mathcal{Q}^{\Delta^{1}}$ that takes $(X,\omega)$ to the
  corresponding map $\mathcal{O}_{X} \to
  \otimes^{2}\mathbb{L}_{X/S}[s]$. Composing with $p$ we get a functor
  \[ (K^{\triangleleft})^{\op} \to \mathcal{Q}^{\Delta^{1}}.\]
  We can now take the cocartesian pushforward in $\mathcal{Q}$ to the
  fibre over $X$ to get a functor
  \[ (K^{\triangleleft})^{\op} \to \QCoh(X)^{\Delta^{1}},\]
  which takes $k \in K$ to \[\mathcal{O}_{X} \simeq
  f_{k}^{*}\mathcal{O}_{X_{k}} \to
  f_{k}^{*}(\otimes^{2}\mathbb{L}_{X/S}[s]) \simeq \otimes^{2}
  f^{*}\mathbb{L}_{X/S}[s].\]
We know from Theorem~\ref{thm:geomhascotgt}(ii) that the cotangent complexes of derived Artin stacks are dualizable, so we can now apply
  \cite{paradj}*{Corollary 4.2.9} as in \cite{paradj}*{Example 4.2.10} to get
  the required functor 
  \[ (K^{\triangleleft})^{\op} \to \Tw^{r}(\QCoh(X)),\]
  which takes $k$ to the adjoint morphism $f_{k}^{*}\mathbb{T}_{X/S}
  \to f^{*}\mathbb{L}_{X/S}[s]$.
\end{proof}

Since $\Tw^{r}$ has a left adjoint $\Tw^{r}_{!}$ by
Corollary~\ref{cor:Tw!}, we obtain the following generalization of
step (2) as an immediate consequence:
\begin{corollary}\label{cor:twoformTw!Diag}
  Suppose given a diagram $p \colon K^{\triangleleft} \to \PSymp_{S,s}$
  that sends $-\infty$ to the derived Artin stack $X$. Then there is
  an induced diagram \[\widehat{p} \colon \Tw^{r}_{!}(K^{\triangleleft})^{\op} \to \QCoh(X)\] valued in
  the \icat{} of quasicoherent sheaves on $X$. \qed
\end{corollary}

\begin{remark}\label{rmk:pbTwdiagnatural}
  It is clear from the construction (and Lemma~\ref{lem:relcotgtpb})
  that for any map $\sigma \colon S' \to S$  we have a natural equivalence
  $\widehat{\sigma^{*}p} \simeq \sigma^{*}\widehat{p}$.
\end{remark}

\begin{notation}
  We will use the following abbreviations: $\tSp{}^{n} :=
  \Tw^{r}_{!}\Sp^{n}$, $\tSp{}^{n,\circ}:= \Tw^{r}_{!}\Sp^{n,\circ}$,
  $\tlsp{}^{n} := \Tw^{r}_{!}\lsp^{n}$, $\tlsp{}^{n,\circ} :=
  \Tw^{r}_{!}\lsp^{n,\circ}$. 
\end{notation}

In the particular cases we are interested in here, Corollary~\ref{cor:twoformTw!Diag} specializes to:
\begin{corollary}\label{cor:spantSpDiag}\ 
  \begin{enumerate}[(i)]
  \item An $n$-uple span in $\PSymp_{S,s}$ with apex $X$ induces a diagram of shape $\tSp{}^{n}$ in $\QCoh(X)$.
  \item An $n$-fold span in $\PSymp_{S,s}$ with apex $X$ induces a diagram of shape $\tlsp{}^{n}$ in $\QCoh(X)$. \qed
  \end{enumerate}
\end{corollary}

\begin{remark}
  By Lemma~\ref{lem:Tw!initial} we have natural equivalences
  \[ \tSp{}^{n} \simeq (\tSp{}^{n,\circ})^{\triangleleft\triangleright}, \qquad \tlsp{}^{n} \simeq
  (\tlsp{}^{n,\circ})^{\triangleleft\triangleright}.\] In particular, the
  categories $\tSp{}^{n}$ and $\tlsp{}^{n}$ both have initial objects,
  which lets us make the following definitions:
\end{remark}
\begin{definition}
  Let $\mathcal{C}$ be an \icat{}. We say a diagram $\tSp{}^{n} \to
  \mathcal{C}$ or $\tlsp{}^{n} \to \mathcal{C}$ is \emph{non-degenerate} if it is a limit diagram.
\end{definition}

\begin{definition}
  An \emph{$n$-uple Lagrangian correspondence} is a $n$-uple span
  $\Sp{}^{n} \to \PSymp_{S,s}$ (with apex $X$) such that the induced
  diagram $\tSp{}^{n} \to \QCoh(X)$ is non-degenerate. Similarly, an
  \emph{$n$-fold Lagrangian correspondence} is an $n$-fold span 
  $\lsp{}^{n} \to \PSymp_{S,s}$ (with apex $X$) such that the induced
  diagram $\tlsp{}^{n} \to \QCoh(X)$ is non-degenerate.
\end{definition}

\begin{remark}
  In  \cite{CalaqueTFT}, an inductive definition of $n$-fold Lagrangian correspondences is given. 
  We will see below in Remark~\ref{rmk:Calaquedefinitioneq} that the two definitions agree.
\end{remark}

\begin{remark}
  Proposition~\ref{propn:Tw!desc} gives an explicit description of the
  categories $\tSp{}^{n}$ and $\tlsp{}^{n}$, which we now consider in the
  smallest cases $n = 1,2$. $\Sp^{1} \simeq \lsp^{1}$ obviously satisfies the criterion of
  Corollary~\ref{cor:Tw!poset}, and so $\tSp{}^{1} \simeq \tlsp{}^{1}$ is
  the partially ordered set
  
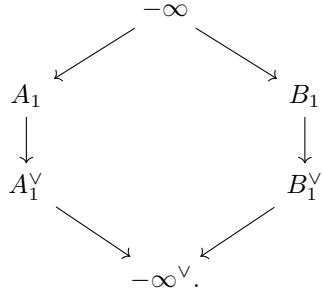
\[
    \begin{tikzcd}
      {} & -\infty \arrow{dl} \arrow{dr} \\
      A_{1} \arrow{d} & & B_{1} \arrow{d} \\
      A_{1}^{\vee} \arrow{dr} & & B_{1}^{\vee} \arrow{dl} \\
       & -\infty^{\vee}.
    \end{tikzcd}
  \]
  In general, the partially ordered set $\Sp^{n}$ satisfies the
  criterion of Corollary~\ref{cor:Tw!poset} for all $n$, and so
  $\tSp{}^{n}$ is again a partially ordered set, with objects $I$ and
  $I^{\vee}$ for all $I \in \Sp^{n}$, and ordering defined by:
  \begin{itemize}
  \item $I \leq J$ in $\tSp{}^{n}$ \IFF{} $I \leq J$ in $\Sp^{n}$,
  \item $I^{\vee} \leq J^{\vee}$ in $\tSp{}^{n}$ \IFF{} $J \leq I$ in $\Sp^{n}$,
  \item $I \leq J^{\vee}$ in $\tSp{}^{n}$ \IFF{} there exists $K \in \Sp^{n}$ such that
    $I \leq K$ and $J \leq K$,
  \item $I^{\vee}$ is never $\leq J$.
  \end{itemize}
  For example, $\tSp{}^{2}$ can be depicted as
  
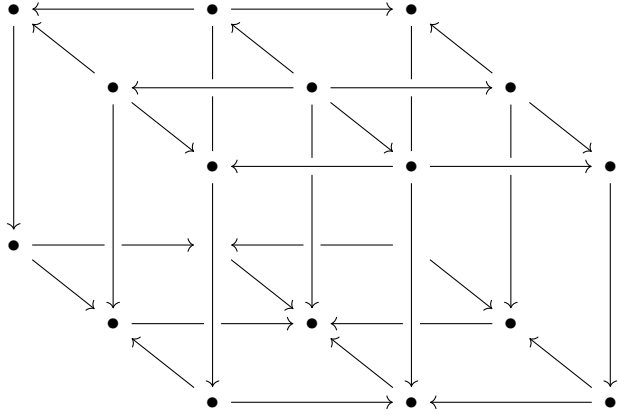
\[
    \begin{tikzcd}
      \bullet \arrow{ddd}& & \bullet\arrow[dash]{dd} \arrow{ll} \arrow{rr} & & \bullet \arrow[dash]{dd}\\
      & \bullet \arrow{ddd}\arrow{ul} \arrow{dr} & & \bullet
       \arrow[crossing over]{ll} \arrow[crossing over]{rr} \arrow{ul} \arrow{dr}& &
      \bullet \arrow{ul} \arrow{dr} \arrow{ddd} \\
      & & \bullet & & \bullet \arrow[crossing over]{rr} & & \bullet\\ 
      \bullet \arrow{rr} \arrow{dr} & & \bullet\arrow{dr} & & \bullet
      \arrow{ll} \arrow{dr}\\
      & \bullet\arrow{rr} \arrow[from=uuu,crossing over] & &
      \bullet \arrow[from=uuu,crossing over] & & \bullet\arrow{ll} \\
      & & \bullet \arrow{rr} \arrow{ul} \arrow[from=uuu,crossing
      over]& & \bullet \arrow{ul} \arrow[crossing over,from=uuu] & &
      \bullet\arrow{ll}\arrow{ul}\arrow[crossing over,from=uuu]
      \arrow[from=3-5,to=3-3,crossing over]
    \end{tikzcd}
  \]
  On the other hand $\tlsp{}^{n}$ is \emph{not} a partially ordered set
  for $n \geq 2$, though it is still an ordinary category: According
  to Proposition~\ref{propn:Tw!desc} we can view $\tlsp{}^{n}$ as
  containing a copy of $\lsp^{n}$  and a copy of $\lsp^{n,\op}$; if we
  write $X, X^{\vee} \in \tlsp{}^{n}$ for the objects corresponding to
  $X$ under these two embeddings, we have that
  $\Hom_{\tlsp{}^{n}}(A_{k}, B_{k}^{\vee})$ is a 2-element set for $k
  \neq 1$. We can depict $\tlsp{}^{2}$ as
  
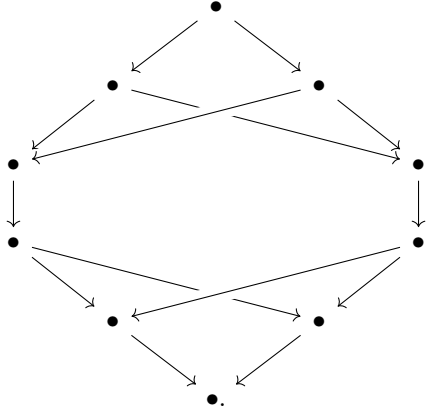
\[
    \begin{tikzcd}
      {} & & \bullet \arrow{dl} \arrow{dr} \\
      & \bullet  \arrow{dl} \arrow{drrr} & & \bullet \arrow{dr}
      \arrow[crossing over]{dlll} \\
      \bullet \arrow{d} & & & & \bullet\arrow{d} \\
      \bullet \arrow{dr} \arrow{drrr} & & & & \bullet \arrow{dl}
      \arrow[crossing over]{dlll} \\
      & \bullet \arrow{dr} & & \bullet \arrow{dl} \\
      {} & & \bullet.      
    \end{tikzcd}
    \]
\end{remark}

We also have the following generalization of
Remark~\ref{rmk:Lagsquare}:
\begin{proposition}\label{propn:nondeghalf}\ 
  \begin{enumerate}[(i)]
  \item The following are equivalent for a diagram $\Phi \colon
    \tSp{}^{n} \to \mathcal{C}$:
    \begin{enumerate}[(1)]
    \item $\Phi$ is non-degenerate.
    \item The composite $\Sp^{n,\triangleright} \to \tSp{}^{n}
      \xto{\Phi} \mathcal{C}$ is a limit diagram.
    \end{enumerate}
  \item The following are equivalent for a diagram $\Phi \colon
    \tlsp{}^{n} \to \mathcal{C}$:
    \begin{enumerate}[(1)]
    \item $\Phi$ is non-degenerate.
    \item The composite $\lsp^{n,\triangleright} \to \tlsp{}^{n}
      \xto{\Phi} \mathcal{C}$ is a limit diagram.
    \item The composite
      $\tSp{}^{n} \xto{\Tw^{r}_{!}(j_{n})} \tlsp{}^{n} \xto{\Phi}
      \mathcal{C}$ is non-degenerate, where $j_{n}$ is the functor of
      Definition~\ref{defn:jn}.
    \end{enumerate}
  \end{enumerate}
\end{proposition}

For the proof we need the following trivial observation:
\begin{lemma}\label{lem:coinittrright}
  A functor $\mathcal{C} \to \mathcal{D}$ is coinitial \IFF{} the
  induced functor $\mathcal{C}^{\triangleright} \to
  \mathcal{D}^{\triangleright}$ is coinitial.
\end{lemma}
\begin{proof}
  Immediate from \ThmA{}, as $(\mathcal{C}^{\triangleright})_{/d}
  \simeq \mathcal{C}_{/d}$ for $d \in \mathcal{D}$ and
  $(\mathcal{C}^{\triangleright})_{/\infty} \simeq
  \mathcal{C}^{\triangleright}$, which is weakly contractible since it
  has a terminal object.
\end{proof}

\begin{proof}[Proof of Proposition~\ref{propn:nondeghalf}]
  Proving (i) is equivalent to showing that the functor $\Sp^{n,\circ,\triangleright}
  \to \tSp{}^{n,\circ,\triangleright}$ is coinitial, which follows from
  Proposition~\ref{propn:CtoTw!Ccoinit}(i) and
  Lemma~\ref{lem:coinittrright}. The same argument for
  $\lsp^{n,\circ}$ shows that (ii)(1) and (ii)(2) are equivalent. To
  show that (ii)(1) and (ii)(3) are equivalent, we want to prove that
  $\tSp{}^{n,\circ,\triangleright} \to \tlsp{}^{n,\circ,\triangleright}$
  is coinitial, which follows by combining
  Lemma~\ref{lem:foldspanloc}(iii),
  Proposition~\ref{propn:CtoTw!Ccoinit}(ii), and Lemma~\ref{lem:coinittrright}.
\end{proof}

In Lemma~\ref{lem:foldspanloc} we show that we can identify $n$-fold
spans in an \icat{} $\mathcal{C}$ with the \emph{reduced} $n$-uple
spans in $\mathcal{C}$ (Definition~\ref{defn:redspan}), via the
functor $j_{n}$. In the case of higher spans in $\PSymp_{S,s}$, this
identifies the $n$-fold Lagrangian correspondences with the reduced
$n$-uple Lagrangian correspondences:
\begin{corollary}
  $\Phi \colon \lsp^{n} \to \PSymp_{S,s}$ is an $n$-fold Lagrangian
  correspondence \IFF{} $\Phi \circ j_{n} \colon \Sp^{n} \to
  \PSymp_{S,s}$ is a (reduced) $n$-uple Lagrangian correspondence.
\end{corollary}
\begin{proof}
  The construction of Corollary~\ref{cor:twoformTw!Diag} is natural,
  so we have a commutative triangle
  \opctriangle{\tSp{}^{n}}{\tlsp{}^{n}}{\QCoh(X)}{\Tw^{r}_{!}j_{n}}{\widehat{\Phi
      j_n}}{\widehat{\Phi}}
  The result then follows from Proposition~\ref{propn:nondeghalf}(ii).
\end{proof}

The naturality observed in Remark~\ref{rmk:pbTwdiagnatural}, together
with the same argument as in Lemma~\ref{lem:laglocal}, implies:
\begin{lemma}
  The following are equivalent for an $n$-fold ($n$-uple) span \[\Phi
  \colon \Sp^{n}\to \PSymp_{S,s} \quad (\lsp^{n} \to \PSymp_{S,s})\colon\]
  \begin{enumerate}[(i)]
  \item $\Phi$ is Lagrangian.
  \item For every morphism $\sigma \colon S' \to S$, the pulled-back
    diagram $\sigma^{*}\Phi$ is Lagrangian.
  \item For every point $p \colon \Spec A \to S$, the pulled-back
    diagram $p^{*}\Phi$ is Lagrangian.
  \end{enumerate}
\end{lemma}
\begin{proof}
  This follows from Lemma~\ref{lem:relcotgtpb} and
  Corollary~\ref{cor:colimeqonpoints} as in the proof of
  Lemmas~\ref{lem:symplocal} and \ref{lem:laglocal}, since the diagram
  shapes $\tSp{}^{n}$ and $\tlsp{}^{n}$ are finite.
\end{proof}

\section{Higher Oriented Cospans}\label{subsec:higherorcospan}
Having defined higher Lagrangian correspondences, we now want to
consider the analogous higher version of oriented cospans. This will
be a non-degeneracy condition on a higher cospan in $\POr_{S,d}$,
again defined as certain diagrams of shape $\tSp{}^{n}$ and $\tlsp{}^{n}$
being non-degenerate. Our first goal is to construct these diagrams,
for which we need the following result:
\begin{proposition}\label{propn:cospanqcohdiag}
  Suppose given a diagram $p \colon \mathcal{I}^{\triangleright} \to
  \POr_{S,d}$ with $X := p(\infty)$, $X_{i}:= p(i)$ and $f_{i} \colon
  X_{i} \to X$ the map $p(i \to \infty)$. For any $\mathcal{E} \in \QCoh(X)$
  the functor $p$ induces a diagram
  \[ p_{\mathcal{E}} \colon (\mathcal{I}^{\triangleright})^{\op} \to \QCoh(S)_{/\mathcal{O}_{S}[-d]}\]
  which takes $i \in \mathcal{I}$ to
  \[ \Gamma_{S}f_{i}^{*}\mathcal{E} \otimes
    \Gamma_{S}f_{i}^{*}(\mathcal{E}^{\vee}) \to
    \Gamma_{S}f_{i}^{*}(\mathcal{E} \otimes \mathcal{E}^{\vee}) \xto{\Gamma_{S}f_{i}^{*}\txt{ev}}
    \Gamma_{S}f_{i}^{*}\mathcal{O}_{X}[-d] \simeq
    \Gamma_{S}\mathcal{O}_{X_{i}} \xto{[X_{i}]} \mathcal{O}_{S}[-d].\]
  Moreover, for any morphism $\sigma \colon S' \to S$ we have a
  natural equivalence
  \[ \sigma^{*}p_{\mathcal{E}} \simeq (\sigma^{*}p)_{\sigma^{*}\mathcal{E}},\]
  where $\sigma^{*}\mathcal{E}$ denotes the pullback of $\mathcal{E}$ to
  $\sigma^{*}X$.
\end{proposition}
\begin{proof}
  We can view the evaluation map
  $\mathcal{E} \otimes \mathcal{E}^{\vee} \to \mathcal{O}_{X}$ as a
  morphism $(\mathcal{E},\mathcal{E}^{\vee}) \xto{\epsilon}
  \mathcal{O}_{X}$ in the symmetric monoidal \icat{}
  $\QCoh(X)^{\otimes}$ (over the unique active map $\angled{2} \to \angled{1}$ in
  $\xF_{*}$). We have a cocartesian fibration
  $\mathcal{Q}_{S}^{\otimes}\to \dSt_{S}^{\op} \times \xF_{*}$
  from Definition~\ref{defn:Qotimes}, which encodes the symmetric
  monoidal structures on quasicoherent sheaves and the symmetric
  monoidal functoriality of the pullback functors. Here we can
  take cocartesian pushforwards of $\epsilon$ along $p^{\op}$ to
  obtain a commutative square
  \[
    \begin{tikzcd}
      (\mathcal{I}^{\triangleright})^{\op} \times \Delta^{1} \arrow{r}
      \arrow{d} & \mathcal{Q}^{\otimes}_{S}  \arrow{d} \\
      \POr_{S,d}^{\op} \times \Delta^{1} \arrow{r} & \dSt_{S}^{\op} \times \xF_{*},
    \end{tikzcd}
  \]
  where the top horizontal map takes $i \in \mathcal{I}$ to
  $(f_{i}^{*}\mathcal{E}, f_{i}^{*}\mathcal{E}^{\vee})
  \xto{f_{i}^{*}\epsilon} f_{i}^{*}\mathcal{O}_{X}\simeq \mathcal{O}_{X_{i}}$.
  Now we can compose this with the functor $\Gamma_{S}^{\otimes}
  \colon \mathcal{Q}_{S}^{\otimes}\to \QCoh(S)^{\otimes}$ from
  Construction~\ref{constr:GammaSotimes}, which encodes the compatible
  lax monoidal structures on the pushforward functors to $S$, to get a
  commutative square
  \[
    \begin{tikzcd}
      (\mathcal{I}^{\triangleright})^{\op} \times \Delta^{1} \arrow{r}
      \arrow{d} & \QCoh(S)^{\otimes} \arrow{d} \\
      \POr_{S,d}^{\op} \times \Delta^{1} \arrow{r} & \xF_{*}.
    \end{tikzcd}
  \]
  Finally, since $\QCoh(S)^{\otimes}\to \xF_{*}$ is a
  cocartesian fibration, we can push forward along $\angled{2} \to
  \angled{1}$ to the fibre over $\angled{1}$, obtaining a functor
  \[ (\mathcal{I}^{\triangleright})^{\op} \times \Delta^{1} \to
    \QCoh(S)\]
  which takes $i \in \mathcal{I}$ to the composite morphism
  \[ \Gamma_{S}f_{i}^{*}\mathcal{E} \otimes
    \Gamma_{S}f_{i}^{*}(\mathcal{E}^{\vee}) \to
    \Gamma_{S}(f_{i}^{*}\mathcal{E} \otimes
    f_{i}^{*}\mathcal{E}^{\vee}) \xto{\Gamma_{S}f_{i}^{*}\epsilon}
    \Gamma_{S}f_{i}^{*}\mathcal{O}_{X} \simeq \Gamma_{S}\mathcal{O}_{X_{i}}.\]
  Evaluated at $1$, this is the composite
  $(\mathcal{I}^{\triangleright})^{\op} \to \POr_{S,d}^{\op} \to \dSt^{\op}_{S}
  \xto{\Gamma_{S}\mathcal{O}} \QCoh(S)$. Using the forgetful functor
  $\POr_{S,d}^{\op} \to \QCoh(S)_{/\mathcal{O}_{S}[-d]}$ we obtain
  another functor $(\mathcal{I}^{\triangleright})^{\op} \times
  \Delta^{1} \to \QCoh(S)$ whose value at $0$ agrees with that of the
  first functor at $1$. We can therefore combine them to a functor $
  (\mathcal{I}^{\triangleright})^{\op} \times \Delta^{2} \to \QCoh(S)$;
  restricting to $\Delta^{\{0,2\}}$ we get the desired diagram
  $p_{\mathcal{E}}$.

  To see that this is natural in $S$ with respect to pullback, we
  carry out a slightly more elaborate version of the same
  construction, using instead the functors $\Gamma$ and
  $\Gamma^{\otimes}$ of Constructions~\ref{constr:Gamma} and
  \ref{constr:Gammaotimes}. (For example, these imply that for
  $\sigma \colon S' \to S$ we have a commutative diagram
  \[
    \begin{tikzcd}
      {} & \mathcal{Q}_{S}^{\txt{UCC},\otimes}
      \arrow{r}{\Gamma_{S}^{\otimes}} \arrow{dd}{\sigma^{*}} &
      \QCoh(S) \arrow{dd}{\sigma^{*}} \\
      (\mathcal{I}^{\triangleright})^{\op} \times \Delta^{1} \arrow{dr} \arrow{ur}\\
       & \mathcal{Q}_{S'}^{\txt{UCC},\otimes}
       \arrow{r}{\Gamma_{S'}^{\otimes}} & \QCoh(S'),
    \end{tikzcd}
    \]
    which encodes the compatibility of the first part of the
    construction with pullback along $\sigma$.)
\end{proof}

We can now apply \cite{paradj}*{Corollary 4.2.9} and the left adjoint of $\Tw^{r}$ to the parametrized tensor-hom
adjunction in $\QCoh(S)$ to get:
\begin{corollary}
  Consider a diagram $p \colon \mathcal{I}^{\triangleright} \to
  \POr_{S,d}$ with $X := p(\infty)$, $X_{i}:= p(i)$ and $f_{i} \colon
  X_{i} \to X$ the map $p(i \to \infty)$. For any $\mathcal{E} \in \QCoh(X)$
  the functor $p$ induces
  \begin{enumerate}[(i)]
  \item   a diagram
    \[ \tilde{p}_{\mathcal{E}} \colon (\mathcal{I}^{\triangleright})^{\op} \to \Tw^{r}\QCoh(S)\]
    which takes $i \in \mathcal{I}$ to the morphism
    \[ \Gamma_{S}f_{i}^{*}\mathcal{E}^{\vee} \to
      (\Gamma_{S}f_{i}^{*}\mathcal{E})^{\vee}[-d]\]
    adjoint to $p_{\mathcal{E}}(i)$,
  \item a diagram
    \[ \widehat{p}_{\mathcal{E}} \colon
      \Tw^{r}_{!}((\mathcal{I}^{\triangleright})^{\op}) \to \QCoh(S)\]
    adjoint to $\tilde{p}_{\mathcal{E}}$.
  \end{enumerate}
  Moreover, for any morphism $\sigma \colon S' \to S$ we have
  natural equivalences
  \[ \sigma^{*}\tilde{p}_{\mathcal{E}} \simeq
    (\widetilde{\sigma^{*}p})_{\sigma^{*}\mathcal{E}},\qquad \sigma^{*}\widehat{p}_{\mathcal{E}} \simeq (\widehat{\sigma^{*}p})_{\sigma^{*}\mathcal{E}},  \]
  where $\sigma^{*}\mathcal{E}$ denotes the pullback of $\mathcal{E}$ to
  $\sigma^{*}X$. \qed
\end{corollary}

In particular, for higher cospans in $\POr_{S,d}$ we have:
\begin{corollary}\label{cor:cospanEdiag}
  An $n$-uple ($n$-fold) cospan $\Phi$ in $\POr_{S,d}$ with bottom vertex $X$
  induces for every $\mathcal{E} \in \QCoh(X)$ a diagram
  $\widehat{\Phi}_{\mathcal{E}}\colon \tSp{}^{n} \to \QCoh(S)$
  ($\tlsp{}^{n} \to \QCoh(S)$). \qed
\end{corollary}

\begin{definition}
  An $n$-uple ($n$-fold) cospan $\Phi$ in $\POr_{S,d}$ with bottom vertex $X$ is \emph{weakly
    oriented} if for every dualizable $\mathcal{E} \in \QCoh(X)$, the
  corresponding diagram $\widehat{\Phi}_{\mathcal{E}}$ is non-degenerate.
\end{definition}

\begin{definition}
  An $n$-uple ($n$-fold) cospan $\Phi$ in $\POr_{S,d}$ is
  \emph{oriented} if $\sigma^{*}\Phi$ is weakly oriented for every
  morphism $\sigma \colon S' \to S$ in $\dSt$.
\end{definition}

\begin{lemma}\label{lem:ncospanorpts}
  An $n$-uple ($n$-fold) cospan $\Phi$ in $\POr_{S,d}$ is oriented
  \IFF{} $p^{*}\Phi$ is weakly oriented for every point $p \colon \Spec A \to S$.
\end{lemma}
\begin{proof}
  It suffices to prove that if $p^{*}\Phi$ is weakly oriented for
  every point $p \colon \Spec A \to S$, then $\Phi$ is weakly
  oriented. Let $X$ be the bottom vertex of $\Phi$, and consider $\mathcal{E}
  \in \QCoh(X)$. We want to prove that if $\mathcal{E}$ is dualizable,
  then $\widehat{\Phi}_{\mathcal{E}}$ is a limit diagram in $\QCoh(S)$.
  Since \[\QCoh(S) \simeq \lim_{p \colon \Spec A \to S} \Mod_{A},\] by
  Corollary~\ref{cor:colimeqonpoints}
  the diagram
  $\widehat{\Phi}_{\mathcal{E}}$ is a limit diagram \IFF{}
  $p^{*}\widehat{\Phi}_{\mathcal{E}}$ is a limit diagram for all points
  $p$. But we have a natural equivalence
  $p^{*}\widehat{\Phi}_{\mathcal{E}} \simeq
  \widehat{p^{*}\Phi}_{p^{*}\mathcal{E}}$, so this is indeed a limit
  diagram as $p^{*}\Phi$ is by assumption weakly oriented, and
  $p^{*}\mathcal{E}$ is dualizable since $p^{*}$ is a symmetric
  monoidal functor.
\end{proof}

Using Proposition~\ref{propn:nondeghalf} we can identify 
$n$-fold oriented cospans with  (reduced) $n$-uple oriented cospans:
\begin{corollary}
  A functor $\Phi \colon \lsp^{n,\op} \to \POr_{S,d}$ is an $n$-fold
  oriented cospan \IFF{} the composite
  $\Phi \circ j_{n} \colon \Sp^{n,\op} \to \POr_{S,d}$ is an $n$-uple
  oriented cospan. \qed
\end{corollary}

\section{Characterizations of Non-Degeneracy}\label{subsec:nondeg}
In this section we will give some alternative
characterizations of non-degenerate diagrams of shape $\tSp{}^{n}$ and $\tlsp{}^{n}$, and
thus alternative descriptions of higher Lagrangian
correspondences and higher oriented cospans, which will be useful
later on. 

Recall that an isotropic correspondence $X \xfrom{f} L \xto{g} Y$ is
Lagrangian when the induced commutative square
\nolabelcsquare{\mathbb{T}_{L/S}}{f^{*}\mathbb{T}_{X/S}}{g^{*}\mathbb{T}_{Y/S}}{\mathbb{L}_{L/S}[s]}
is cartesian. If we in addition assume that $X$ and $Y$ are
\emph{symplectic}\footnote{Both conditions make sense for
  presymplectic $X$ and $Y$, but they are not equivalent in general.}
then this is equivalent to the square
\nolabelcsquare{\mathbb{T}_{L/S}}{f^{*}\mathbb{L}_{X/S}[s]}{g^{*}\mathbb{L}_{Y/S}[s]}{\mathbb{L}_{L/S}[s]}
being cartesian. We now extend this characterization to higher
spans\footnote{Here we consider only $n$-uple spans, as this is the
  case we make use of below, but the obvious variation for $n$-fold
  spans also holds, with essentially the same proof.}:
\begin{definition}
  We say a diagram $\tSp{}^{n} \to \mathcal{C}$ has \emph{non-degenerate
    boundary} if for every proper subset $S \subsetneq
  \{0,\ldots,n\}$ the restricted diagram $\tSp{}^{|S|} \hookrightarrow
  \tSp{}^{n} \to \mathcal{C}$ is non-degenerate.
\end{definition}

\begin{proposition}\label{propn:itlagRKE}\ 
  \begin{enumerate}[(i)]
  \item   A diagram $\Phi \colon \tSp{}^{n} \to \mathcal{C}$ has non-degenerate
    boundary \IFF{} $\Phi|_{\tSp{}^{n,\circ,\triangleright}}$ is a right
    Kan extension of $\Phi|_{\Sp^{n,\op}}$.
  \item Suppose $\Phi \colon \tSp{}^{n} \to \mathcal{C}$ is a diagram with
    non-degenerate boundary. Then the following are equivalent:
    \begin{enumerate}[(1)]
    \item $\Phi$ is non-degenerate.
    \item $\Phi$ is a right Kan extension of its
      restriction along $\Sp^{n,\op} \hookrightarrow \tSp{}^{n}$.
    \item $\Phi(-\infty)$ is the limit of
      $\Phi$ restricted to $\Sp^{n,\op}$.
    \end{enumerate}
      \end{enumerate}
\end{proposition}
\begin{proof}
  We will prove both statements by induction on $n$. For $n = 0$ they
  are both vacuous, so assume we have proved them for all $i < n$.

  Suppose $X = (I_{1},\ldots,I_{n}) \in \Sp^{n,\circ}$.  Let
  $m := |\{i : I_{i} = -\infty\}|$; then $\Sp^{n}_{X/} \simeq \Sp^{m}$
  and so we have an inclusion
  \[ \Sp^{m,\op} \simeq (\Sp^{n}_{X/})^{\op} \hookrightarrow
    (\Sp^{n,\op})_{X/} := \Sp^{n,\op} \times_{\tSp{}^{n,\circ}}
  \tSp{}^{n,\circ}_{X/}.\]
  We claim that this inclusion is 
  coinitial. To see this we must show that for every $Y^{\vee} \in
  (\Sp^{n,\op})_{X/}$ the category $(\Sp^{m,\op})_{/Y^{\vee}} \simeq
  (\Sp^{n,\op})_{/X,/Y}$ is weakly contractible; but since $Y^{\vee}
  \geq X$ we know by Example~\ref{ex:epsTwnprod} that this
  category has a terminal object, and so is weakly
  contractible.

  Thus $\Phi(X)$ is the limit of $\Phi$ restricted to
  $(\Sp^{n,\op})_{X/}$ \IFF{} it is the limit of the restriction of
  this diagram to $\Sp^{m,\op}$. The former condition for all $X$ is
  precisely the requirement for
  $\Phi|_{\tSp{}^{n,\circ,\triangleright}}$ to be a right Kan extension
  of its restriction to $\Sp^{n,\op}$, while by the inductive
  hypothesis the latter condition for all $X$ is equivalent to $\Phi$
  having non-degenerate boundary. This proves (i) for $n$.

  To prove (ii) for $n$, we may therefore assume that
  $\Phi|_{\tSp{}^{n,\circ,\triangleright}}$ is the right Kan extension
  of its restriction to $\Sp^{n,\op}$. Clearly $\Phi$ is
  non-degenerate \IFF{} it is the right Kan extension of its
  restriction to $\tSp{}^{n,\circ,\triangleright}$, so we see that (1)
  and (2) are equivalent, since right Kan extensions are
  transitive. Moreover, (2) and (3) are equivalent since in (3) we are
  evaluating at the initial object.
\end{proof}

The following consequence will be useful in \S\ref{subsec:itlagsymmon}:
\begin{corollary}\label{cor:Lagoneaxis}
  Suppose $\Phi \colon \tSp{}^{n} \to \mathcal{C}$ ($n>1$) is a diagram with
  non-degenerate boundary. Then the following are equivalent:
  \begin{enumerate}[(1)]
  \item $\Phi$ is non-degenerate.
  \item Let $\xi$ denote the inclusion $\Sp^{n,\op}
    \hookrightarrow \Sp^{1,\op,\triangleleft} \times
    \Sp^{n-1,\op}$, and consider the right Kan extension
    $\Psi := \xi_{*}\Phi|_{\Sp^{n,\op}}$ of the restriction of $\Phi$ along
    $\xi$. Then $\Phi(-\infty)$ is the limit of 
    $\Psi|_{(-\infty) \times \Sp^{n-1,\op}}$.
  \end{enumerate}
\end{corollary}
\begin{proof}
  By Proposition~\ref{propn:itlagRKE} we know that (1) is equivalent to
  $\Phi$ being the right Kan extension of its restriction to
  $\Sp^{n,\op}$. Consider the commutative square of fully faithful
  inclusions
  \ffcsquare{\Sp^{n,\op}}{\Sp^{n,\op,\triangleleft}}{\Sp^{1,\op,\triangleleft}
    \times \Sp^{n-1,\op}}{\Sp^{1,\op,\triangleleft} \times
    \Sp^{n-1,\op,\triangleleft}.}  Right Kan extensions along
  fully faithful maps do not change the value of the functor on the
  full subcategory, so the limit of $\Phi|_{\Sp^{n,\op}}$ is
  equivalent to the limit of $\Psi$. But the inclusion
  $\{-\infty\} \times \Sp^{n-1,\op} \hookrightarrow
  \Sp^{1,\op,\triangleleft} \times \Sp^{n-1,\op}$ is a product of
  coinitial maps, and hence is coinitial by \cite[Corollary
  4.1.1.13]{HTT}, so this limit is equivalent to the limit of $\Phi$
  restricted to $\{-\infty\} \times \Sp^{n-1,\op}$, as required.
\end{proof}

We also record an inductive description of non-degeneracy for $n$-fold
spans:
\begin{proposition}\label{propn:foldLagcrit}
  Suppose given a diagram $F \colon
  \lsp^{n,\triangleright} \to \mathcal{C}$. Define
  $M_{i}$ inductively by starting with $M_{0} := F(\infty)$ and then
  setting $M_{i}$ to be the pullback
  \nolabelcsquare{M_{i}}{F(A_{i})}{F(B_{i})}{M_{i-1}.}
  Then $F$ is a limit diagram \IFF{} the induced map $F(-\infty) \to
  M_{n}$ is an equivalence, \ie{} the square
  \nolabelcsquare{F(-\infty)}{F(A_{n})}{F(B_{n})}{M_{n-1}}
  is cartesian.  
\end{proposition}
\begin{proof}
  Let $p^{n} \colon \lsp^{n,\circ,\triangleright} \to \lsp^{n-1,\circ,\triangleright}$
  be the functor that sends $A_{1}$,
  $B_{1}$ and $\infty$ to $\infty$ and $A_{i}$ and $B_{i}$ to
  $A_{i-1}$ and $B_{i-1}$ for $i > 1$. Then the limit of a diagram $F$
  of shape $\lsp^{n,\circ,\triangleright}$ is given by the sequence of
  right Kan extensions $p^{0}_{*}\cdots p^{n-1}_{*}p^{n}_{*}F$ (where
  we take $\lsp^{0,\circ}$ to be empty). The value of $p^{n}_{*}F$ at $X \in
  \lsp^{n-1,\circ,\triangleright}$ is the limit of $F$ restricted to the
  partially ordered set of $Y$ such that $p^{n}(Y) \geq X$. If $X$ is
  $A_{i}$ or $B_{i}$ then this has an initial element (namely
  $A_{i+1}$ or $B_{i+1}$, respectively), so $p^{n}_{*}F(A_{i}) \simeq
  F(A_{i+1})$ and $p^{n}_{*}(B_{i}) \simeq F(B_{i+1})$. On the other
  hand, if $X = \infty$ then this set consists of $A_{1}$, $B_{1}$ and
  $\infty$, so $p^{n}_{*}F(\infty) \simeq F(A_{1}) \times_{F(\infty)}
  F(B_{1})$. Applying this observation inductively to the sequence of
  right Kan extensions $p^{i}_{*} \cdots p_{*}^{n}F$ then gives the result.
\end{proof}

\begin{corollary}\label{cor:Spnsphere} 
  There are coinitial functors $\lsp^{n,\circ}\to S^{n-1}$, $\Sp^{n,\circ}\to S^{n-1}$.
\end{corollary}
\begin{proof}
  By Lemma~\ref{lem:foldspanloc}(iii), the functor
  $j_{n}^{\circ}\colon \Sp^{n,\circ} \to \lsp^{n,\circ}$ is coinitial,
  so it suffices to prove the first statement. Now
  it follows from Proposition~\ref{propn:foldLagcrit}, applied to the functor
  $\lsp^{n,\circ} \to \mathcal{S}^{\op}$ constant at $*$, that
  $\colim_{\lsp^{n,\circ,\op}} * \simeq S^{n-1}$. This gives the
  required coinitial functor $\lsp^{n,\circ}\to
  \|\lsp^{n,\circ}\| \simeq S^{n-1}$.
\end{proof}

Our final characterization of non-degenerate diagrams extends the
observation that an isotropic span $X \xfrom{f} L \xto{g} Y$ is
Lagrangian precisely when a certain square
\nolabelcsquare{\mathbb{T}_{L/S}}{f^* \mathbb{T}_{X/S} \oplus g^*
  \mathbb{T}_{Y/S}}{0}{\mathbb{L}_{L/S}[s]} is cartesian:
\begin{proposition}\label{propn:Spklimfibseq}
  Suppose given a diagram
  $\Phi \colon \Sp^{n,\triangleright} \to \mathcal{C}$ where
  $\mathcal{C}$ is a stable \icat{}. Then there is a
  canonical commutative square \nolabelcsquare{\Phi(-\infty)}{\lim
    \Phi|_{\Sp^{n,\circ}}}{0}{\Phi(\infty)[1-n]} such that $\Phi$ is a limit
  diagram \IFF{} this square is cartesian.
\end{proposition}

We make use of the following observation:
\begin{lemma}\label{lem:overlimpb}
  Suppose $\mathcal{I}$ is a small \icat{} and $\mathcal{C}$ is an
  \icat{} with pullbacks and $\mathcal{I}$-indexed limits. Given a
  functor $F \colon \mathcal{I} \to \mathcal{C}_{/A}$ then there is a
  natural pullback square \nolabelcsquare{u(\lim_{\mathcal{I}}
    F)}{\lim_{\mathcal{I}} uF}{A}{\lim_\mathcal{I} A} where $u$ is the
  forgetful functor $\mathcal{C}_{/A}\to \mathcal{C}$.
\end{lemma}
\begin{proof}
  Let $X$ be the object of $\mathcal{C}$ defined by the pullback
  \nolabelcsquare{X}{\lim_{\mathcal{I}}
    uF}{A}{\lim_\mathcal{I} A.}
  We will show that $X$ satisfies the universal property of
  $\lim_{\mathcal{I}} F$. For any object $\pi \colon B \to A$ in
  $\mathcal{C}_{/A}$ we have a commutative diagram
  \[
  \begin{tikzcd}
    \Map_{/A}(B, X) \arrow{r} \arrow{d} & \Map(B, X) \arrow{r}
    \arrow{d} & \Map(B, \lim_{\mathcal{I}} uF) \arrow{d} \\
    \{\pi\} \arrow{r} & \Map(B, A) \arrow{r} & \Map(B, \lim_{\mathcal{I}}A),
  \end{tikzcd}
  \]
  where both squares are cartesian. Thus the composite square is also
  cartesian, so we have a cartesian square
  \nolabelcsquare{\Map_{/A}(B, X)}{\lim_{\mathcal{I}}\Map(X,
    uF)}{\lim_{\mathcal{I}} \{\pi\}}{\lim_{\mathcal{I}}\Map(X,A)}
  Since limits commute, this means $\Map_{/A}(B, X) \simeq \lim
  \Map_{/A}(B, F)$, and so $X \simeq \lim_{\mathcal{I}}F$ as required.
\end{proof}

\begin{proof}[Proof of Proposition~\ref{propn:Spklimfibseq}]
  A functor $F \colon \Sp^{n,\circ,\triangleright} \to \mathcal{C}$
  corresponds to a functor $F' \colon \Sp^{n,\circ} \to \mathcal{C}_{/A}$
  where $A := F(\infty)$, and the limit of $F$ in $\mathcal{C}$ can be
  identified with the image in $\mathcal{C}$ of the limit of $F'$ in
  $\mathcal{C}_{/A}$. By Lemma~\ref{lem:overlimpb} we therefore have a
  pullback square \nolabelcsquare{\lim F}{\lim
    F|_{\Sp^{n,\circ}}}{A}{\lim_{\Sp^{n,\circ}} A.}  Here the limit
  $\lim_{\Sp^{n,\circ}} A$ of the constant diagram indexed by $\Sp^{n,\circ}$ can
  be identified with the cotensoring $A^{X}$ where $X$ is the space
  obtained by inverting the morphisms in $\Sp^{n,\circ}$; by
  Corollary~\ref{cor:Spnsphere} this can be identified with the sphere
  $S^{n-1}$. Since $\mathcal{C}$ is stable, it is cotensored over
  \emph{pointed} finite spaces, and with respect to this we can identify
  $A^{S^{n-1}}$ with the cotensoring with the pointed space
  $S^{n-1}_{+}$; moreover, $\mathcal{C}$ is cotensored over finite spectra,
  so this in turn is identified with the cotensoring with
  $\Sigma^{\infty}_{+} S^{n-1} \simeq \Sigma^{n-1}\mathbb{S} \vee
  \mathbb{S}$ (where $\mathbb{S}$ is the sphere spectrum). Thus
  $\lim_{\Sp^{n,\circ}} A \simeq A[1-n] \oplus A$, from which we get a cartesian
  square \nolabelcsquare{A}{\lim_{\Sp^{n,\circ}} A}{0}{A[1-n].}
  Combining this with the first
  square gives a cartesian square
  \nolabelcsquare{\lim F}{\lim F|_{\Sp^{n,\circ}}}{0}{A[1-n],}
  as required.
\end{proof}

Combining Propositions~\ref{propn:nondeghalf} and
\ref{propn:Spklimfibseq}, we get:
\begin{corollary}
  Suppose $\mathcal{C}$ is a stable \icat{}. Then for $\Phi
  \colon \tSp^{n} \to \mathcal{C}$ there is a commutative
  square
  \nolabelcsquare{\Phi(-\infty)}{\lim
    \Phi|_{\Sp^{n,\circ}}}{0}{\Phi(\infty)[1-n]}
  such that $\Phi$ is non-degenerate \IFF{} the square is cartesian.\qed
\end{corollary}

\section{Higher Categories of Symplectic and Oriented Derived Stacks}\label{subsec:itlagcat}

Our goal in this section is to construct $(\infty,n)$-categories
$\LagnSs$ and $\OrnSd$ whose $i$-morphisms are $i$-fold Lagrangian
correspondences in $\PSymp_{S,s}$ and $i$-fold oriented cospans in
$\POr_{S,d}$, respectively. (We review the definition of
$(\infty,n)$-categories that we use in \S\ref{subsec:segsp}.) To do
this we will show that these determine sub-$(\infty,n)$-categories of
the $(\infty,n)$-categories $\Span_{n}(\PSymp_{S,s})$, where the $i$-morphisms are arbitrary $i$-fold spans in $\PSymp_{S,s}$, and $\Cospan_{n}(\POr_{S,d}) = \Span_{n}(\POr_{S,d}^{\op})$, respectively. These  $(\infty,n)$-categories arise as special
cases of the construction of general higher categories of iterated
spans in \cite{spans}, which we review in \S\ref{subsec:spans}.

\begin{notation}\ 
  \begin{itemize}
  \item For $\mathcal{C}$ an \icat{} with pullbacks,
    $\Span_{n}(\mathcal{C})$ is defined using the partially ordered
    sets $\bbS^{\mathbf{k}}$ and $\bbL^{\mathbf{k}}$ for $\mathbf{k}
    \in \simp^{n}$
    (Definition~\ref{defn:bbS}).
  \item $\oSPAN_{n}^{+}(\mathcal{C}) \to \Dnop$ is the cocartesian
    fibration for $\Fun(\bbS^{(\blank)}, \mathcal{C})$.
  \item $\SPAN_{n}^{+}(\mathcal{C})$ is the full subcategory of
    \emph{cartesian} functors $\bbS^{\mathbf{k}} \to \mathcal{C}$,
    meaning those that are right Kan
    extensions of their restriction to $\bbL^{\mathbf{k}}$; the
    restricted projection to $\Dnop$ is the cocartesian fibration for
    an $n$-uple category object in $\CatI$.
  \item $\SPAN_{n}(\mathcal{C}) \to \Dnop$ is the underlying left
    fibration of $\SPAN^{+}_{n}(\mathcal{C})$. This corresponds to an $n$-uple Segal space, and
    $\Span_{n}(\mathcal{C})$ is the underlying $n$-fold Segal space of this.
  \end{itemize}
\end{notation}

Just as the $(\infty,n)$-category $\Span_{n}(\mathcal{C})$ is
conveniently defined as a subobject of an $n$-uple category object
$\SPAN^{+}_{n}(\mathcal{C})$ in $\CatI$, we will show that iterated
Lagrangian correspondences and iterated oriented cospans determine
$n$-uple category objects $\LAGnpSs$ and $\ORnpSd$ inside
$\SPAN_{n}^{+}(\PSymp_{S,s})$ and $\SPAN_{n}^{+}(\POr_{S,d})$; the
$(\infty,n)$-categories $\LagnSs$ and $\OrnSd$ are
then simply the underlying $(\infty,n)$-categories of these $n$-uple
category objects. We thus wish to define $\LAGnpSs$ and $\ORnpSd$ as 
full subcategories of $\SPAN^{+}_{n}(\PSymp_{S,s})$ and $\COSPAN^{+}_{n}(\POr_{S,d})$, and so we need
to specify the $\bbS^{\mathbf{k}}$-shaped diagrams
that lie in these subcategories:
\begin{definition}
  A map $\bbL^{\mathbf{k}} \to \PSymp_{S,s}$ is \emph{Lagrangian} if for
  every $\mathbf{i} = (i_{1},\ldots,i_{n})$ with each $i_{t} = 0$ or
  $1$, and every inert map $\mathbf{i} \to \mathbf{k}$, the composite
  $\Sp^{m} \simeq \bbS^{\mathbf{i}} \to \bbL^{\mathbf{k}} \to \PSymp_{S,s}$ (where
  $m = \sum_{t} i_{t}$) is an $m$-fold Lagrangian correspondence. We define
  \emph{oriented} maps $\bbL^{\mathbf{k},\op} \to \POr_{S,d}$ similarly.
\end{definition}

\begin{definition}
  A map $F \colon \bbS^{\mathbf{k}} \to \PSymp_{S,s}$ is \emph{Lagrangian}
  if $F$ is cartesian and the restriction $F|_{\bbL^{\mathbf{k}}}$ is
  Lagrangian in the sense of the previous definition. We define
  \emph{oriented} maps $\bbS^{\mathbf{k}} \to \POr_{S,d}$ similarly.
\end{definition}

\begin{definition}
  Let $\LAGnpSs$ be the full subcategory of the \icat{}
  $\SPAN^{+}_{n}(\PSymp_{S,s})$ spanned by the Lagrangian maps
  $\bbS^{\mathbf{k}} \to \PSymp_{S,s}$ for all $\mathbf{k}$. We define
  $\ORnpSd$ similarly as the full subcategory of $\COSPAN^{+}_{n}(\POr_{S,d})$ spanned by the oriented maps.
\end{definition}

\begin{theorem}\label{thm:LAGcoCart}
  The restricted projections from $\LAGnpSs$ and $\ORnpSd$ to
  $\Dnop$ are cocartesian fibrations.
\end{theorem}

Before we turn to the proof of Theorem~\ref{thm:LAGcoCart}, we note
some immediate consequences of it:
\begin{corollary}\label{cor:LAGSegCond}
  The functors associated to the cocartesian fibrations
  \[\LAGnpSs, \ORnpSd \to \simp^{n,\op}\]  are
  $n$-uple category objects of $\CatI$.
\end{corollary}
\begin{proof}
  It remains to check the Segal condition. But by definition of
  $\LAGnpSs$ we have for every $\mathbf{i} =
  ([i_{1}],\ldots,[i_{k}])$ a pullback square
  \nolabelcsquare{(\LAGnpSs)_{\mathbf{i}}}{\SPAN^{+}_{n}(\PSymp_{S,s})_{\mathbf{i}}}{\lim_{\mathbf{e}
      \in \DnelopI}
    (\LAGnpSs)_{\mathbf{e}}}{\lim_{\mathbf{e} \in
      \DnelopI}
    \SPAN^{+}_{k}(\PSymp_{S,s})_{\mathbf{e}}.} 
    (We use the notation from Remark \ref{app:cat:notation Del}.) Here we know the right vertical map
  is an equivalence by \cite[Proposition 3.10]{spans}, hence so is the
  left vertical map. The proof for $\ORnpSd$ is the same.
\end{proof}

\begin{definition}
  Let $\LAGnSs \to \Dnop$ be the underlying left fibration
  of the cocartesian fibration $\LAGnpSs$, and define $\ORnSd$
  similarly; then Corollary~\ref{cor:LAGSegCond} implies that the
  associated functors are $n$-uple Segal spaces. We define $\LagnSs$
  and $\OrnSd$ to be their underlying $n$-fold Segal spaces
  $U_{\Seg}\LAGnSs$ and $U_{\Seg}\ORnSd$ (cf. \cite[Proposition 2.13]{spans}).
\end{definition}

\begin{proposition}\label{propn:Lagcompl}
  The $n$-fold Segal spaces $\LagnSs$ and $\OrnSd$ are complete.
\end{proposition}
\begin{proof}
  It is clear from \cite[Lemma 6.2]{spans} that equivalences between
  $i$-fold Lagrangian correspondences are again Lagrangian
  correspondences. Thus the completeness of the $n$-fold Segal space
  $\LagnSs$ follows from that of $\Span_{n}(\PSymp_{S,s})$,
  proved in \cite[Corollary 6.5]{spans}. The proof for $\OrnSd$ is the same.
\end{proof}

The rest of this section is devoted to the proof of
Theorem~\ref{thm:LAGcoCart}. To prove that $\LAGnpSs \to \Dnop$ is a
cocartesian fibration, it suffices to check that if we have a cocartesian
morphism $X \to X'$ in $\SPAN^{+}_{n}(\PSymp_{S,s})$ where $X$ is
Lagrangian, then $X'$ is also Lagrangian. It's enough to show this for
elementary degeneracies and face maps (in one of the $n$ variables)
since we can factor every morphism in $\Dnop$ as a composite of such
maps. For degeneracies, this will follow from the following observation:
\begin{proposition}\label{propn:deg}
  For $j = 1,\ldots,n$, let $\pi_{j} \colon \Sp^{n} \to \Sp^{n-1}$ be
  the map induced by projecting off the $j$th factor of
  $\Sp^{1}$. Then the induced functor $\widetilde{\pi}_{j} \colon
  \tSp{}^{n} \to \tSp{}^{n-1}$ has the property that for any functor $F
  \colon \tSp{}^{n-1} \to \mathcal{C}$ that is a limit diagram, the
  composite $F \circ \widetilde{\pi}_{j}$ also is a limit diagram.
\end{proposition}
Since $\widetilde{\pi}_{j}$ does not restrict to a functor from
$\tSp{}^{n,\circ,\triangleright}$ to $\tSp{}^{n-1,\circ,\triangleright}$
we cannot prove this by checking that a functor is coinitial. Instead,
we will use the following trivial observation:
\begin{lemma}\label{lem:Kanextsquare}
  Suppose given a commutative square of \icats{}
  \[
  \begin{tikzcd}
    \mathcal{A}_{00} \arrow[hookrightarrow]{r}
    \arrow[hookrightarrow]{d} & \mathcal{A}_{01}
    \arrow[hookrightarrow]{d}\\
    \mathcal{A}_{10} \arrow[hookrightarrow]{r}
    & \mathcal{A}_{11}
  \end{tikzcd}
  \]
  where all the morphisms are fully faithful inclusions, with the
  property that every object of $\mathcal{A}_{11}$ lies in the
  essential image of $\mathcal{A}_{01}$ or $\mathcal{A}_{10}$. Let $F
  \colon \mathcal{A}_{11} \to \mathcal{B}$ be a functor such that:
  \begin{enumerate}[(1)]
  \item $F|_{\mathcal{A}_{01}}$ is a (pointwise) right Kan extension of $F|_{\mathcal{A}_{00}}$,
  \item $F|_{\mathcal{A}_{10}}$ is a (pointwise) right Kan extension of $F|_{\mathcal{A}_{00}}$.
  \end{enumerate}
  Then:
  \begin{enumerate}[(i)]
  \item $F$ is a right Kan extension of $F|_{\mathcal{A}_{00}}$.
  \item $F$ is a right Kan extension of $F|_{\mathcal{A}_{01}}$.
  \item $F$ is a right Kan extension of $F|_{\mathcal{A}_{10}}$.    
  \end{enumerate}
\end{lemma}
\begin{proof}
  To prove (i), observe that if $x \in \mathcal{A}_{11}$ lies in
  $\mathcal{A}_{01}$ then 
  \[\mathcal{A}_{00} \times_{\mathcal{A}_{11}} \mathcal{A}_{11,x/}
    \simeq \mathcal{A}_{00} \times_{\mathcal{A}_{01}}
    \mathcal{A}_{01,x/},\]
  since the functors are fully faithful, and similarly for $x \in
  \mathcal{A}_{10}$. Now (ii) and (iii) follow since right Kan
  extensions are transitive.
\end{proof}

\begin{proof}[Proof of Proposition~\ref{propn:deg}]
  By Proposition~\ref{propn:nondeghalf}(i)  it suffices to prove that the
  corresponding map $\pi_{j}^{\triangleright} \colon \Sp^{n,\triangleright} \to
  \Sp^{n-1,\triangleright}$ has the same property; let $\Phi :=
  F|_{\Sp^{n-1,\triangleright}}$ and $\Phi' := \Phi \circ \pi_{j}^{\triangleright}$. Consider the
  commutative diagram
  \[
  \begin{tikzcd}
    (\Sp^{n-1,\circ} \times \Sp^{1})^{\triangleright}
    \arrow[hookrightarrow]{r} \arrow[hookrightarrow]{d}
    \arrow[bend left=15]{rr}&
    \Sp^{n,\circ,\triangleright} \arrow[hookrightarrow,crossing over]{d} &
    \Sp^{n-1,\circ,\triangleright} \arrow[hookrightarrow]{d}\\
    (\Sp^{n-1,\circ} \times \Sp^{1})^{\triangleleft\triangleright} \arrow[hookrightarrow]{r} &
    \Sp^{n,\triangleright} \arrow{r}[below]{\pi_{j}^{\triangleright}} &  \Sp^{n-1,\triangleright}.
  \end{tikzcd}
  \]
  We wish to show that $\Phi'$ is the right Kan extension of its
  restriction to $\Sp^{n,\circ,\triangleright}$. Applying
  Lemma~\ref{lem:Kanextsquare} to the left-hand square in the
  diagram, we see that to prove this it suffices to show the following:
  \begin{enumerate}[(a)]
  \item The restriction of $\Phi'$ to $(\Sp^{n-1,\circ} \times
    \Sp^{1})^{\triangleleft\triangleright}$ is the right Kan extension
    of its restriction to $(\Sp^{n-1,\circ} \times \Sp^{1})^{\triangleright}$.
  \item The restriction of $\Phi'$ to  $\Sp^{n,\circ,\triangleright}$ is the right Kan extension
    of its restriction to $(\Sp^{n-1,\circ} \times
    \Sp^{1})^{\triangleright}$.
  \end{enumerate}
  For (a), we need to check the value of $\Phi'$ at the initial object in $(\Sp^{n-1,\circ} \times
    \Sp^{1})^{\triangleleft\triangleright}$ is the limit of $\Phi'$ restricted to $(\Sp^{n-1,\circ} \times \Sp^1)^\triangleright$. Since we know that it is the limit of $\Phi$ on $(\Sp^{n-1,\circ})^{\triangleright}$,    it suffices by Lemma~\ref{lem:coinittrright} to check that the map $\Sp^{n-1,\circ} \times \Sp^{1} \to
  \Sp^{n-1,\circ}$ is coinitial. Since
  products of coinitial functors are coinitial by
  \cite{HTT}*{Corollary 4.1.1.3}, this is implied by $\Sp^{1}$ being
  weakly contractible.

  For (b), observe that the objects of $\Sp^{n,\circ}$ that do not lie
  in $\Sp^{n-1,\circ} \times \Sp^{1}$ are precisely the two objects
  $X_{0} := (-\infty, (0,0))$ and $X_{1}:= (-\infty, (1,1))$. We thus
  need to show that the value of $\Phi'$ at $X_{i}$ is the limit of
  its restriction to
  \[ (\Sp^{n-1,\circ} \times \Sp^{1})_{X_{i}/} \simeq
    \Sp^{n-1,\circ}_{-\infty/} \times \Sp^{1}_{(i,i)/} \simeq
    \Sp^{n-1,\circ},\]
  which is immediate from $\Phi$ being a limit diagram.
\end{proof}

Next we consider inner face maps. For this we first describe the
structure of the diagram of quasicoherent sheaves induced by a
Lagrangian map $\bbS^{2,1,\ldots,1} \to \PSymp_{S,s}$. This requires a bit
of notation:
\begin{definition}
  Let $\pbbS^{i}$ be the full subcategory of $\bbS^{i}$ spanned by the
  objects of the form $(0,j)$ and $(j,i)$, for $j = 0,\ldots,i$. For
  $\mathbf{i} = (i_{1},\ldots,i_{n})$, we then write
  $\pbbS^{\mathbf{i}}$ for the product $\pbbS^{i_{1}} \times \cdots
  \times \pbbS^{i_{n}}$.
\end{definition}

\begin{notation}
  We write $\tbbS^{\mathbf{i}} := \Tw_{!}\bbS^{\mathbf{i}}$ and $\tbbL^{\mathbf{i}} := \Tw_{!}\bbL^{\mathbf{i}}$.
\end{notation}

\begin{proposition}\label{propn:tbbS2RKE}
  Suppose given a diagram
 \[\Phi \colon \tbbS^{2,1,\ldots,1} \cong
  \Tw_{!}(\bbS^{2}\times \Sp^{n-1}) \to \mathcal{C}\] such
  that:
  \begin{enumerate}[(1)]
  \item For every $Y^{\vee} = ((1,1), X)^{\vee} \in
    (\bbS^{2,1,\ldots,1})^{\op} \subseteq \tbbS^{2,1,\ldots,1}$ the
    commutative square \nolabelcsquare{\Phi(((1,1),
      X)^{\vee})}{\Phi(((0,1),
      X)^{\vee})}{\Phi(((1,2),
      X)^{\vee})}{\Phi(((0,2), X)^{\vee})} is a pullback
    square.
  \item The restriction of $\Phi$ along any inclusion $\tSp{}^{n} \to
    \tbbS^{2,1,\ldots,1}$ induced by an inclusion $\Sp^{n}
    \hookrightarrow \bbL^{2,1,\ldots,1}$ is non-degenerate.
  \item For every $X \in \bbS^{2,1,\ldots,1} \subseteq
    \tbbS^{2,1,\ldots,1}$, the object $\Phi(X)$ is the
    limit of the restriction of $\Phi$ to~$\bbL^{2,1,\ldots,1}_{X/}$,
  \end{enumerate}
  Then $\Phi$ is a right Kan extension of its restriction along
  $(\pbbS^{2,1,\ldots,1})^{\op} \hookrightarrow \tbbS^{2,1,\ldots,1}$.
\end{proposition}
\begin{proof}
  Let $\mathcal{C}_{3} := \tbbS^{2,1,\ldots,1}$, and define
  $\mathcal{C}_{2} \subseteq \mathcal{C}_{3}$ to be the full
  subcategory spanned by the objects in $(\bbS^{2,1,\ldots,1})^{\op}$
  together with the objects in $\bbL^{2,1,\ldots,1} \subseteq
  \bbS^{2,1,\ldots,1}$. Then set $\mathcal{C}_{1} :=
  (\bbS^{2,1,\ldots,1})^{\op}$ and $\mathcal{C}_{0} :=
  (\pbbS^{2,1,\ldots,1})^{\op}$. We will then prove that the
  restriction of $\Phi$ to $\mathcal{C}_{i}$ is a right
  Kan extension of its restriction to $\mathcal{C}_{i-1}$ for $i =
  1,2,3$.

  First consider the inclusion $\mathcal{C}_{0} \hookrightarrow
  \mathcal{C}_{1}$. The objects $X^{\vee}$ of $\mathcal{C}_{1}$ that
  are not in $\mathcal{C}_{0}$ are those of the form $((1,1),
  Y)^{\vee}$ for $Y \in \Sp^{n-1}$. We must thus show that
  $\Phi(((1,1), Y)^{\vee})$ is the limit of
  $\Phi$ restricted to $\mathcal{C}_{0,X^{\vee}/} \simeq
  (\pbbS^{2})^{\op}_{(1,1)/} \times \Sp^{n-1,\op}_{Y/}$. The
  inclusion $(\pbbS^{2})^{\op}_{(1,1)/} \times \{Y\} \hookrightarrow
  (\pbbS^{2})^{\op}_{(1,1)/} \times \Sp^{n-1,\op}_{Y/}$ is a
  product of coinitial maps, and so is coinitial by \cite[Corollary
  4.1.1.3]{HTT}. Thus it suffices to know that the commutative square
  \nolabelcsquare{\Phi(((1,1),
    Y)^{\vee})}{\Phi(((0,1),
    Y)^{\vee})}{\Phi(((1,2),
    Y)^{\vee})}{\Phi(((0,2), Y)^{\vee})} is a pullback
  square, which is condition (1).

  Next we consider the inclusion $\mathcal{C}_{1} \hookrightarrow
  \mathcal{C}_{2}$. For $X \in \bbL^{2,1,\ldots,1}$, we must show that
  $\Phi(X)$ is the limit of $(\mathcal{C}_{1})_{X/}$. If
  $X = ((\alpha_{1}, \beta_{1}), \ldots, (\alpha_{k}, \beta_{k}))$,
  set $i := \sum_{t} (\beta_{t}-\alpha_{t})$. By assumption (2) the
  restriction of $\Phi$ to the copy of $\tSp{}^{i}$
  determined by $X$ is non-degenerate, hence by
  Proposition~\ref{propn:itlagRKE} we know that $\Phi(X)$
  is the limit of the restriction of $\Phi$ to the
  corresponding copy of $\Sp^{i,\op} \simeq
  (\bbS^{2,1,\ldots,1})^{\op}_{/X}$. It therefore suffices to show
  that the inclusion $\Sp^{i, \op} \hookrightarrow
  (\mathcal{C}_{1})_{X/}$ is coinitial. Calling on \ThmA{} yet again,
  we must thus prove that for every $Y \in (\mathcal{C}_{1})_{X/}$,
  the category $(\Sp^{i, \op})_{/Y} \simeq
  (\bbS^{2,1,\ldots,1})^{\op}_{/X,/Y}$ is weakly contractible. By
  Example~\ref{ex:epsTwnprod} we know this category has a terminal
  object, and hence is indeed weakly contractible.

  Finally we consider the inclusion $\mathcal{C}_{2} \hookrightarrow
  \mathcal{C}_{3}$. By assumption (3) we know that for $X \in
  \mathcal{C}_{3}$ not contained in $\mathcal{C}_{2}$, the object $\Phi(X)$ is the
  limit of the restriction of $\Phi$ to
  $\bbL^{2,1,\ldots,1}_{X/}$. Thus it suffices to show that the
  inclusion $\bbL^{2,1,\ldots,1}_{X/} \hookrightarrow
  (\mathcal{C}_{2})_{X/}$ is coinitial. We now apply \ThmA{} once more to see
  that it is enough to prove that for each $Y^{\vee} \in
  (\bbS^{2,1,\ldots,1})^{\op}_{X/}$ the category
  $(\bbL^{2,1,\ldots,1}_{X/})_{/Y^{\vee}}$ is weakly contractible. If $X
  = (X_{1},\ldots, X_{n})$ and $Y = (Y_{1},\ldots,Y_{n})$, then this
  category can be identified with the product $\prod_{i}
  (\bbL^{a_i}_{X_{i}/})_{/Y_{i}^{\vee}}$, where $a_{i} = 2$ for $i = 1$
  and 1 otherwise. A simple case-by-case analysis now finishes the
  proof.
\end{proof}

\begin{remark}
  By a more elaborate version of this argument, from which we will
  spare the reader, it is possible to show for any $\mathbf{i}$ that
  if a functor $\Phi \colon \tbbS^{\mathbf{i}} \to \mathcal{C}$
  satisfies analogues of assumptions (1)--(3), then it is a right
  Kan extension of its restriction to $(\pbbS^{\mathbf{i}})^{\op}$.
\end{remark}

\begin{corollary}\label{cor:tbbS2nd}
  Suppose a diagram $\Phi \colon \tbbS^{2,1,\ldots,1} \to \mathcal{C}$
  satisfies the three hypothesis in
  Proposition~\ref{propn:tbbS2RKE}. Then the restriction $\Psi \colon \tSp{}^{n} \to
  \mathcal{C}$ along the inclusion $\tSp{}^{n} \hookrightarrow
  \tbbS^{2,1,\ldots,1}$ induced by the functor $(d^{1},\id,\ldots,\id)
  \colon \Sp^{n} \to \bbS^{2,1,\ldots,1}$ is non-degenerate.
\end{corollary}
\begin{proof}
  We induct on $n$. For $n = 0$ there is nothing to prove, so assume
  the result is known for all $i < n$. Then $\Psi$ has non-degenerate
  boundary, so by Proposition~\ref{propn:itlagRKE} it suffices to
  show that $\Psi((0,1),\ldots,(0,1))$ is the limit of
  $\Psi$ restricted to $\Sp^{n,\op}$. By
  Proposition~\ref{propn:tbbS2RKE} we know that $\Phi$ is a
  right Kan extension of its restriction to $(\partial
  \bbS^{2,1,\ldots,1})^{\op}$, so in particular
  $\Psi((0,1),\ldots,(0,1)) =
  \Phi((0,2),\ldots,(0,1))$ is the limit of
  $\Phi$ restricted to $(\pbbS^{2,1,\ldots,1})^{\op}$. To
  complete the proof it this suffices to show that the inclusion
  $\Sp^{n} \hookrightarrow \pbbS^{2,1,\ldots,1}$ is cofinal. Appealing
  yet again to \ThmA{}, it suffices to show that for each $X \in
  \pbbS^{2,1,\ldots,1}$, the category $\Sp^{n}_{X/}$ is weakly
  contractible. If $X = ((\alpha_{1},\beta_{1}), \ldots,
  (\alpha_{n},\beta_{n}))$, then  $\Sp^{n}_{X/} \simeq
  \prod_{i} (\Sp^{1})_{(\alpha_{i},\beta_{i})/}$, and so it suffices
  to show that $(\Sp^{1})_{(\alpha,\beta)/}$ is weakly contractible
  when $(\alpha,\beta)$ is $(0,1)$ or $(1,2)$. But this is clear,
  since in these two cases this category consists of the single object
  $(0,0)$ or $(1,1)$, respectively.
\end{proof}

\begin{corollary}\label{cor:d1}\ 
  \begin{enumerate}[(i)]
  \item Suppose given a Lagrangian diagram \[F \colon \bbS^{2}
    \times \Sp^{n-1} = \bbS^{2,1,\ldots,1} \to \PSymp_{S,s}.\] Then the
    composite
    \[ F' \colon \Sp^{n} \xto{(d^{1},\id,\ldots,\id)}
    \bbS^{2,1,\ldots,1} \longrightarrow \PSymp_{S,s} \] is an $n$-uple Lagrangian
    correspondence.
  \item Suppose given an oriented diagram \[G \colon \bbS^{2} \times
    \Sp^{n-1} = \bbS^{2,1,\ldots,1} \to \POr_{S,d}^{\op}.\] Then the composite
    \[ G' \colon \Sp^{n} \xto{(d^{1},\id,\ldots,\id)}
    \bbS^{2,1,\ldots,1} \longrightarrow \POr_{S,d}^{\op} \] is an $n$-uple oriented
    cospan.
  \end{enumerate}
\end{corollary}
\begin{proof}
  Applying Corollary~\ref{cor:tbbS2nd} it suffices to check in case
  (i) that the three assumptions in Proposition~\ref{propn:tbbS2RKE}
  hold for the induced diagram $\widetilde{F} \colon
  \tbbS^{2,1,\ldots,1} \to \QCoh(S)$ ($S :=
  F((0,2),(0,1),\ldots,(0,1))$), and in case (ii) that they hold for
  the induced diagram $\widetilde{G}_{\mathcal{E}} \colon
  \tbbS^{2,1,\ldots,1} \to \QCoh(S')$ for any morphism $S' \to S$ and
  any dualizable $\mathcal{E} \in \QCoh(S' \times_{S} X)$.

  Condition (3) follows from Proposition~\ref{propn:tgtlim} in case
  (i) and in case (ii) from Corollary~\ref{cor:qcohrellim} and the fact
  that $\Gamma_{S'}$ is a right adjoint by Remark~\ref{rmk:GammaSradj}.

  Condition (2) says precisely that $F$ is Lagrangian in case (i) and
  oriented in case (ii).

  In both cases the commutative square in (1) is the dual of the
  commutative square \nolabelcsquare{\Phi((0,2), X)}{\Phi((0,1),
    X)}{\Phi((1,2), X)}{\Phi((1,1), X),} (with $\Phi = \widetilde{F}$
  or $\widetilde{G}_{\mathcal{E}}$, respectively). This is a pullback
  square by condition (3), hence the dual is a pushout square. But
  then as $\QCoh(S)$ and $\QCoh(S')$ are stable \icats{} it is also a
  pullback square.
\end{proof}

\begin{proof}[Proof of Theorem~\ref{thm:LAGcoCart}]
  Since we know the projections \[\SPAN^{+}_{n}(\PSymp_{S,s}), \COSPAN^{+}_{n}(\POr_{S,d}) \longrightarrow
  \Dnop\] are cocartesian fibrations by \cite[Proposition
  3.8]{spans}, it suffices to show that for every morphism $f \colon I
  \to J$ in $\Dnop$ and every $\Phi \in \LAGnpSs$ or $\ORnpSd$ over $I$,
  the cocartesian pushforward $f_{!}\Phi$ 
  is also in $\LAGnpSs$ or $\ORnpSd$. In other words, for $\Phi \colon \bbS^{I} \to
  \PSymp_{S,s}$ Lagrangian, we must show that the composite with $f^{*}
  \colon \bbS^{J} \to \bbS^{I}$ is also Lagrangian, and similarly for
  $\Phi \colon \bbS^{I} \to \POr_{S,d}$ oriented. By the definition of a
  Lagrangian or oriented functor, it suffices to consider this in the case where
  $\bbS^{J} \cong \Sp^{n}$ for some $n$. Moreover, the result is
  obvious if $f$ is inert, so using the active-inert factorization
  system it suffices to consider active $f$. Decomposing $f$ as a
  composite of elementary (inner) face maps and degeneracies, we
  conclude that it suffices to consider $f$ of the form
  $(\id_{[i_{1}]},\ldots, g, \ldots, \id_{[i_{k}]})$, where each
  $i_{t}$ is either $0$ or $1$, and $g$ is either the degeneracy
  $s^{0} \colon [1] \to [0]$ or the inner face map $d^{1} \colon [1]
  \to [2]$. In the first case the result follows from
  Proposition~\ref{propn:deg} and in the second case from
  Corollary~\ref{cor:d1}.
\end{proof}

\section{Symmetric Monoidal Structures}\label{subsec:itlagsymmon}
Our goal in this section is to show that the
$(\infty,n)$-categories $\Lag_{n}^{S,s}$ and $\Or_{n}^{S,d}$ have natural
symmetric monoidal structures. These will be restricted from symmetric
monoidal structures on $\Span_{n}(\PSymp_{S,s})$ and
$\Cospan_{n}(\POr_{S,d})$, which we will obtain by ``delooping'' as in
Remark~\ref{rmk:spanmondeloop}: We will show that there are natural
equivalences
\[ \Span_{n}(\PSymp_{S,s})(S,S) \simeq
  \Span_{n-1}(\PSymp_{S,s-1}),\]
\[  \Cospan_{n}(\POr_{S,d})(\emptyset,\emptyset) \simeq
  \Cospan_{n-1}(\POr_{S,d+1}),\] which implies that we have the
desired symmetric monoidal structures. By \cite{spans}*{Proposition
  8.3} we have equivalences of \icats{}
\[ \Span_{n}(\PSymp_{S,s})(S,S) \simeq \Span_{n-1}((\PSymp_{S,s})_{/S,S}),\]
\[ \Cospan_{n}(\POr_{S,d})(\emptyset,\emptyset)\simeq
\Cospan_{n-1}((\POr_{S,d})_{\emptyset,\emptyset/}),\] so this amounts
to identifying these slices.

We first consider the presymplectic case, where the following is the
key observation:
\begin{proposition}\label{propn:Atclloop}
  There is a pullback square
  \nolabelcsquare{\Atcl_{S}(s-1)}{S}{S}{\Atcl_S(s),}
  natural in $S$.
\end{proposition}
\begin{proof}
  For every $S$-stack $X$, we have a natural cartesian square
  \nolabelcsquare{\Atcl_{S}(X,s-1)}{*}{*}{\Atcl_{S}(X,s),}
  since by the definition of $\Atcl_{S}$ this is
  \[
    \begin{tikzcd}[column sep=-1em]
      \Map_{\txt{GMMod}_\k}(\k(2)[-2-s+1],
      \DR(X/S)) \arrow{r} \arrow{d} & * \arrow{d} \\
      * \arrow{r} & \Map_{\txt{GMMod}_\k}(\k(2)[-2-s], \DR(X/S)),
    \end{tikzcd}
  \]
  which is cartesian since we have a pushout square of graded mixed
  modules
  \nolabelcsquare{\k(2)[-2-s]}{0}{0}{\k(2)[-2-s+1].}
  These diagrams are functorial in $X$, and identifying them with
  $S$-stacks we get the desired pullback square; they are also
  natural in $S$ (since $\DR(X/S)$ is functorial in both $X$ and $S$).
\end{proof}

\begin{remark}
  Proposition~\ref{propn:Atclloop} shows that $\Atcl_{S}(s)$ is an
  infinite loop object in the $\infty$-topos $\dSt_{S}$, and hence a
  (grouplike) commutative monoid. It follows from the proof that this structure is
  equivalent to the commutative monoid structure arising from
  $\txt{GMMod}_\k$ being a stable \icat{} (which is the natural
  \emph{addition} of differential forms), since both arise from the
  shift functor in $\txt{GMMod}_\k$.
\end{remark}

\begin{corollary}\label{cor:deloopPSymp}
  There is an equivalence \[(\PSymp_{S,s})_{/S,S} \isoto
  \PSymp_{S,s-1},\] natural in $S$, which fits in a commutative square
  \[
    \begin{tikzcd}
      (\PSymp_{S,s})_{/S,S} \arrow{r}{\sim} \arrow{d} & \PSymp_{S,s-1}
      \arrow{d} \\
      (\dSt_{S})_{/S,S} \arrow{r}{\sim} & \dSt_{S},
    \end{tikzcd}
  \]
  where the vertical maps are the forgetful functors.\qed
\end{corollary}

In the preoriented case, we similarly have:
\begin{proposition}\label{propn:deloopPOr}
  There is an equivalence $(\POr_{S,d})_{\emptyset,\emptyset/} \simeq
  \POr_{S,d+1}$, natural in $S$, which fits in a commutative square
  \[
    \begin{tikzcd}
      (\POr_{S,d})_{\emptyset,\emptyset/} \arrow{r}{\sim} \arrow{d} & \POr_{S,d+1}
      \arrow{d} \\
      (\dSt_{S})_{\emptyset,\emptyset/} \arrow{r}{\sim} & \dSt_{S},
    \end{tikzcd}
  \]
  where the vertical maps are the forgetful functors.
\end{proposition}
\begin{proof}
  From the definition of $\POr_{S,d}$ we have a pullback square
  \[
    \begin{tikzcd}
      (\POr_{S,d})_{\emptyset,\emptyset/}^{\op} \arrow{r} \arrow{d} &
      \QCoh(S)_{/0 \to \mathcal{O}_{S}[-d] \from 0} \arrow{d} \\
      (\dSt_{S})_{\emptyset,\emptyset/}^{\op} \arrow{r} & \QCoh(S)_{/0,0}.
    \end{tikzcd}
  \]
  Since we have a pullback square
  \nolabelcsquare{\mathcal{O}_{S}[-d-1]}{0}{0}{\mathcal{O}_{S}[-d],}
  we may identify this with the pullback square defining $\POr_{S,d+1}$.
\end{proof}

Applying \cite{spans}*{Proposition 8.3}, we get:
\begin{corollary}\label{cor:SpanSSeq}
  For all $n,s,d,S$, there are natural equivalences
  \[\Span_{n}(\PSymp_{S,s})(S,S) \simeq \Span_{n-1}(\PSymp_{S,s-1}),\]
  \[\Cospan_{n}(\POr_{S,d})(\emptyset,\emptyset) \simeq
    \Cospan_{n-1}(\POr_{S,d+1}).\]
  \qed
\end{corollary}

Now Remark~\ref{rmk:spanmondeloop} gives the following:
\begin{corollary}\label{cor:SpanPSympSM}
  There are natural symmetric monoidal structures on the $(\infty,n)$-categories
  $\Span_{n}(\PSymp_{S,s})$ and $\Cospan_{n}(\POr_{S,d})$. The former
  is given on objects by assigning to $(X, \omega
  \in \Atcl_{S}(X,s))$ and $(X', \omega' \in \Atcl_{S}(X',s)$ the
  presymplectic $S$-stack $(X
  \times_{S} X', \pi_{X}^{*}\omega + \pi_{X'}^{*}\omega')$, where
  $\pi_{X}$ and $\pi_{X'}$ are the projections from $X \times_{S}X'$
  to $X$ and $X'$. The latter is given by assigning to $(T,[T])$ and
  $(T',[T'])$ the preoriented $S$-stack $(T \amalg T', i_{T,*}[T] +
  i_{T',*}[T'])$, where $i_{T}$ and $i_{T'}$ are the inclusions of $T$
  and $T'$ in the coproduct $T \amalg T'$. \qed
\end{corollary}

\begin{proposition}\label{propn:deloop}
  The equivalences of Corollary~\ref{cor:SpanSSeq} restrict to
  equivalences
  \[ \LagnSs(S,S) \simeq \Lag_{n-1}^{S,s-1}, \qquad
    \OrnSd(\emptyset,\emptyset) \simeq \Or_{n-1}^{S,d+1}.\]
\end{proposition}
\begin{proof}
  We will prove the Lagrangian case; the oriented case is similar.
  Inducting on $n$, it suffices to show that an iterated span $F \colon
  \Sp^{n} \to \PSymp_{S,s}$ from $S$ to $S$ with Lagrangian boundary is
  Lagrangian \IFF{} the restriction of $F$ to $\{(0,1)\} \times
  \Sp^{n-1}$ is Lagrangian when considered as a functor to
  $\PSymp_{S,s-1}$. Let $\widetilde{F}$ be the associated diagram $\tSp{}^{n}
  \to \QCoh(X)$ ($X = F((0,1),\ldots,(0,1))$). By
  Corollary~\ref{cor:Lagoneaxis} we know that $F$ is Lagrangian \IFF{}
  $\mathbb{T}_{X/S}$ is the limit of $\Xi|_{(-\infty) \times
    \Sp^{n-1,\op}}$, where \[\Xi \colon
  (\Sp^{1, \op})^{\triangleleft} \times \Sp^{n-1,\op} \longrightarrow
  \QCoh(X)\] is the right Kan extension of $\widetilde{F}$. Since $F$
  lies in $\Span_{n}(\PSymp_{S,s})$, the restrictions of $F$ to $(0,0)
  \times \Sp^{n-1}$ and $(1,1) \times \Sp^{n-1}$ are constant at the
  terminal object
  $S$ hence the restriction of $\widetilde{F}$ to $(0,0) \times
  \Sp^{n-1,\op}$ and $(1,1) \times \Sp^{n-1,\op}$ is constant at
  $0$. Thus $F$ is Lagrangian \IFF{} $\mathbb{T}_{X}$ is the limit of
  $\widetilde{F}|_{(0,1) \times
    \Sp_{-1}^{\op}}[-1]$. Appealing to Proposition~\ref{propn:itlagRKE}
  and the inductive hypothesis we see that this is equivalent to $F$
  being Lagrangian when considered as an $(n-1)$-fold span in
  $\PSymp_{S,s-1}$, which completes
  the proof.
\end{proof}

\begin{corollary}\label{cor:Lagsymmon}
  The symmetric monoidal structures of Corollary~\ref{cor:SpanPSympSM}
  restrict to the sub-$(\infty,n)$-categories $\LagnSs$ and $\OrnSd$. \qed
\end{corollary}

\section{Full Dualizability}\label{subsec:dual}
Our goal in this section is to prove that all $s$-symplectic
$S$-stacks are fully dualizable as objects of $\LagnSs$, and similarly
that 
all $d$-oriented $S$-stacks are fully dualizable in $\OrnSd$. As a first step,
which will allow us to reach this conclusion by induction, we will
show that 1- and 2-morphisms have adjoints:
\begin{proposition}\label{propn:adj12}\ 
  \begin{enumerate}[(i)]
  \item The $(\infty,n)$-category $\LagnSs$ has adjoints for 1-morphisms.
  \item Suppose $X$ and $X'$ are $s$-symplectic $S$-stacks. Then the
    $(\infty,n-1)$-category $\LagnSs(X, X')$ has adjoints for
    1-morphisms.
  \item The $(\infty,n)$-category $\OrnSd$ has adjoints for 1-morphisms.    
  \item Suppose $T$ and $T'$ are $d$-oriented $S$-stacks.  Then the
    $(\infty,n-1)$-category $\OrnSd(T, T')$ has adjoints for
    1-morphisms.
  \end{enumerate}
\end{proposition}
\begin{proof}
  We will prove the Lagrangian case; the proof in the oriented case is
  similar. We know from \cite{spans}*{Theorem 10.3} that
  $\Span_{n}(\PSymp_{S,s})$ has adjoints for $i$-morphisms for all $i =
  1,\ldots,n-1$. To show that $\LagnSs$ has adjoints for $i$-morphisms
  it therefore suffices to check that given an $i$-morphism, its
  adjoint as well as the unit and counit $(i+1)$-morphisms are all
  Lagrangian.
  
  Suppose $X$ and $X'$ are $s$-symplectic $S$-stacks, and $X
  \xfrom{f} Y \xto{g} X'$ is a Lagrangian correspondence, i.e. the
  induced square
  \nolabelcsquare{\mathbb{T}_{Y/S}}{\mathbb{T}_{X/S}}{\mathbb{T}_{X'/S}}{\mathbb{L}_{Y/S}[s]}
  is cartesian. (Here, and in the rest of the proof, we have
  simplified our notation by not indicating that the tangent and
  cotangent complexes are pulled back to $Y$.) The (left and right) adjoint in
  $\Span_{n}(\PSymp_{S,s})$ is given by reading the span the other way,
  which corresponds to reorienting the square of quasicoherent
  sheaves, so this is obviously again Lagrangian. We will prove that
  the unit map is Lagrangian; the proof for the counit is similar. The
  unit is given by the 2-fold span 
  $X \from Y \to Y \times_{X'} Y$ over
  $X$ and $X$. To prove that this is Lagrangian, we use the criterion
  of Proposition~\ref{propn:foldLagcrit}: let $M$ be defined by the pullback
  \nolabelcsquare{M}{\mathbb{T}_{X/S}}{\mathbb{T}_{X/S}}{\mathbb{L}_{Y/S}[s],}
  then we must show that the commutative square
  \nolabelcsquare{\mathbb{T}_{Y/S}}{\mathbb{T}_{Y\times_{X'}Y/S}}{\mathbb{T}_{X/S}}{M}
  is cartesian. Now consider the diagram
\[ \begin{tikzcd}
   \mathbb{T}_{Y \times_{X'} Y/S} \arrow{r}\arrow{d}& \mathbb{T}_{Y/S} \arrow{d}\\
   \mathbb{T}_{Y/S} \arrow{r}\arrow{d}& \mathbb{T}_{X'/S} \arrow{d}\\
   \mathbb{T}_{X/S} \arrow{r}& \mathbb{L}_{Y/S}[s]. \\
 \end{tikzcd}
\]
Here the top square is cartesian by Proposition~\ref{propn:tgtlim} and
the bottom square is cartesian since $Y$ is by assumption a Lagrangian
correspondence from $X$ to $X'$; thus the composite square is also
cartesian. Next look at the diagram
\[ \begin{tikzcd}
   \mathbb{T}_{Y \times_{X'} Y/S} \arrow{r}\arrow{d}& \mathbb{T}_{Y/S} \arrow{d}\\
   M \arrow{r}\arrow{d}& \mathbb{T}_{X/S} \arrow{d}\\
   \mathbb{T}_{X/S} \arrow{r}& \mathbb{L}_{Y/S}[s]. \\
 \end{tikzcd}
\]
The composite square is the same as before, and the bottom square is
cartesian by definition; it follows that the top square is also
cartesian. Finally we consider the diagram
 \[ \begin{tikzcd}
    \mathbb{T}_{Y/S} \arrow{r}\arrow{d}& \mathbb{T}_{Y \times_{X'} Y/S}
    \arrow{r}\arrow{d} & \mathbb{T}_{Y/S} \arrow{d}\\
    \mathbb{T}_{X/S} \arrow{r} & M \arrow{r} & \mathbb{T}_{X/S}\\
  \end{tikzcd}
 \]
Here we know the right-hand square is cartesian, and the composite
square is cartesian since the two horizontal composites are clearly
identity morphisms. It follows that the left-hand square is also
cartesian, which completes the proof.

The proof for 2-morphisms is essentially the same: we use
Proposition~\ref{propn:foldLagcrit} and the same arguments about cartesian
squares; we omit the details.
\end{proof}

\begin{remark}
  The same type of argument could be used to show that $\LagnSs$ and
  $\OrnSd$ have adjoints for $i$-morphisms for all $i$. We have chosen
  to instead give an inductive proof based on identifying the higher
  categories of maps in $\LagnSs$, as this will also allow us to
  relate the $(\infty,n)$-categories $\LagnSs$ to the definition of
  such higher categories sketched in \cite{CalaqueTFT}.
\end{remark}

\begin{definition}
  Suppose $X$ is an $s$-symplectic $S$-stack. We define
  $\Lagn^{X/S,s}$ to be the $(\infty,n)$-category
  $\Lag_{n+1}^{S,s}(S, X)$ of maps from the terminal object $S$ to $X$ in
  $\Lag_{n+1}^{S,s}$. The objects of $\Lagn^{X/S,s}$ are thus
  \emph{Lagrangian morphisms} with target $X$. Similarly, for $T$ a
  $d$-oriented $S$-stack, we define $\Orn^{T/S,d}$ to be
  $\Or_{n+1}^{S,d}(\emptyset, T)$.
\end{definition}

\begin{corollary}\label{cor:mapsareLagkS}
  Suppose $X$ is an $s$-symplectic $S$-stack, and $Y$ and
  $Z$ are objects of $\Lagn^{X/S,s}$. Then there is an equivalence
  \[ \Lagn^{X/S,s}(Y, Z) \simeq \Lag_{n-1}^{Z \times_{X}
    Y^{\txt{rev}}/S,s-1},\]
  where $Y^{\txt{rev}}$ means $Y$ viewed as a Lagrangian correspondence from
  $X$ to $S$, and $Z \times_{X} Y^{\txt{rev}}$ denotes the
  $(s-1)$-symplectic derived stack corresponding to the composite of
  $Z$ and $Y^{\txt{rev}}$ according to
  Proposition~\ref{propn:deloop}. Similarly, if $T$ is a $d$-oriented $S$-stack and $U$ and $V$ are objects of $\Orn^{T/S,d}$, then there
  is an equivalence
  \[ \Orn^{T/S,d}(U, V) \simeq \Or_{n-1}^{V \amalg_{T} U^{\txt{rev}}/S,d+1}.\]
\end{corollary}
\begin{proof}
  We will prove the Lagrangian case; the oriented case is proved
  similarly. If $\mathcal{C}$ is an $(\infty,n)$-category, and $f
  \colon A \to B$ is a 1-morphism in $\mathcal{C}$ that has a right
  adjoint $g \colon B \to A$, then the adjunction identities imply
  that for any $C \in \mathcal{C}$, composition with $f$ and $g$
  induces an adjunction
  \[ f_{*} : \mathcal{C}(C, A) \rightleftarrows \mathcal{C}(C,B) :
  g_{*},\]
  and so in particular a natural equivalence of $(\infty,n-2)$-categories
  \[ \mathcal{C}(C,B)(f_{*}\alpha, \beta) \simeq
  \mathcal{C}(C,A)(\alpha,g_{*}\beta)\]
  for any $\alpha \in \mathcal{C}(C,A)$ and $\beta \in
  \mathcal{C}(C,B)$. We apply this with $\mathcal{C} =
  \Lag_{n+1}^{S,s}$, $C = A = S$, $B = X$, $f = Y$ (viewed as a
  1-morphism from $S$ to $X$ in $\Lag_{n+1}^{S,s}$), $g =
  Y^{\txt{rev}}$, $\alpha = \id_{S}$ and $\beta = Z$ to get an
  equivalence
  \[ \Lagn^{X/S,s}(Y, Z) \simeq \Lag_{n+1}^{S,s}(S, X)(Y, Z) \simeq
  \Lag_{n+1}^{S,s}(S,S)(S, Z \times_{X} Y^{\txt{rev}}).\]
  Now we use Proposition~\ref{propn:deloop} to get an equivalence
  \[ \Lag_{n+1}^{S,s}(S,S)(S, Z \times_{X} Y^{\txt{rev}})
  \simeq \Lagn^{S,s-1}(S, Z \times_{X} Y^{\txt{rev}}) \simeq \Lag_{n-1}^{Z \times_{X} Y^{\txt{rev}}/S,s-1}.\qedhere\]  
\end{proof}

\begin{remark}\label{rmk:Calaquedefinitioneq}
  In \cite{CalaqueTFT} an inductive definition of
  iterated Lagrangian correspondences is given as follows: If $X$ is an
  $s$-symplectic $S$-stack and $Y \to X$ and $Z \to X$ are Lagrangian
  morphisms, then one defines an $n$-fold Lagrangian
  correspondence from $Y$ to $Z$ over $X$ to be an $(n-1)$-fold
  Lagrangian correspondence over $Z \times_{X}Y^{\txt{rev}}$. Taking a
  0-fold Lagrangian correspondence over $X$ to be a Lagrangian
  morphism to $X$, this defines $n$-fold Lagrangian correspondences
  over $S$ inductively in terms of Lagrangian morphisms; iterated
  Lagrangian correspondences are then defined to be iterated
  Lagrangian correspondences over $S$. It follows from
  Corollary~\ref{cor:mapsareLagkS} that this definition is equivalent
  to ours.
\end{remark}

\begin{corollary}\label{cor:LagnSadj}
  For any $s$-symplectic $S$-stack $X$ and any $d$-oriented
  $S$-stack $T$, the
  $(\infty,n)$-categories $\Lagn^{X/S,s}$ and $\Orn^{T/S,d}$ have adjoints.
\end{corollary}
\begin{proof}
  By Proposition~\ref{propn:adj12} we know that $\Lagn^{X/S,s}$ has
  adjoints for $1$-morphisms for all $n$ and~$X$. To prove that
  $\Lagn^{X/S,s}$ has adjoints for $i$-morphisms for $i > 1$ we must show
  that for all $Y, Z \in \Lagn^{X/S,s}$ the $(\infty,n-1)$-category
  $\Lagn^{X/S,s}(Y,Z)$ has adjoints for $(i-1)$-morphisms. But by
  Corollary~\ref{cor:mapsareLagkS} this $(\infty,n-1)$-category is
  equivalent to $\Lag_{n-1}^{Z \times_{X} Y^{\txt{rev}}/S,s-1}$, so
  this follows by induction on $i$. The proof for $\Orn^{T/S,d}$ is the
  same.
\end{proof}

\begin{corollary}
  For any $s$ and $d$, the symmetric monoidal $(\infty,n)$-categories
  $\LagnSs$ and $\OrnSd$ have duals. In particular, every $s$-symplectic
  derived Artin stack is fully dualizable as an object of $\LagnSs$ and
  every $d$-oriented derived stack is fully dualizable as an object of
  $\OrnSd$.
\end{corollary}
\begin{proof}
  Taking $X = S$ in Corollary~\ref{cor:LagnSadj} we know that $\LagnSs$
  has adjoints for all $n$ and $s$. Furthermore, the monoidal structure on
  $\LagnSs$ is defined by viewing $\LagnSs$ as the endomorphisms of $S$
  in $\Lag_{n+1}^{S,s+1}$. Hence, the objects of $\LagnSs$ have
  duals as they have adjoints when viewed as 1-morphisms in
  $\Lag_{n+1}^{S,s+1}$. The proof for $\OrnSd$ is the same.
\end{proof}

\begin{remark}
  Let $X$ be an $s$-symplectic $S$-stack. We may view $X$ as
  a Lagrangian correspondence $\xi$ from $S$ to $S$, and the reverse
  Lagrangian correspondence $\xi^{\txt{rev}}$ corresponds to the
  $s$-symplectic $S$-stack $\overline{X}$, meaning $X$ equipped
  with the negative of its symplectic form. Thus $\overline{X}$ is the
  dual of $X$ in $\LagnSs$.
\end{remark}

Invoking the cobordism hypothesis, we get:
\begin{corollary}\
  \begin{enumerate}[(i)]
  \item   Every $s$-symplectic $S$-stack $X$ gives rise to a framed
  extended $n$-dimensional topological field theory
  \[ \mathcal{Z}_{S} \colon \txt{Bord}_{0,n}^{\txt{fr}} \longrightarrow
  \LagnSs \]
  for every $n$.
\item Every $d$-oriented $S$-stack $T$ gives rise to a framed
  extended $n$-dimensional topological field theory
  \[ \mathcal{Z}_{T} \colon \txt{Bord}_{0,n}^{\txt{fr}} \longrightarrow
  \OrnSd \]
  for every $n$.
  \end{enumerate}
\end{corollary}

\section{Oriented Cospans of Spaces}\label{subsec:orspc}

We have already mentioned that  constant stacks on closed
oriented $d$-manifolds give $d$-oriented derived stacks, while oriented
$(d+1)$-dimensional cobordisms give $d$-oriented cospans. These and
higher-dimensional versions thereof are key to our construction of extended
TFTs via the AKSZ construction: we will later construct a symmetric
monoidal ``forgetful functor'' from the oriented cobordism
$(\infty,n)$-category to $\Or_{n}^{S,d}$ for any derived stack
$S$. For this it is convenient to isolate an intermediate
$(\infty,n)$-category of iterated oriented cospans of \emph{spaces},
given as a subobject of iterated cospans in \emph{preoriented spaces},
which we introduce in this section.

\begin{definition}
  For a space $X \in \mathcal{S}$ we write $C^{*}(X) \in \Mod_{\k}$
  for the $\k$-cochains on $X$, \ie{} $C^{*}(X) \coloneqq \lim_{X} \k$.  A
  \emph{$d$-preorientation} on a space $X \in \mathcal{S}$ is a
  morphism $[X]\colon C^{*}(X) \to \k[-d]$ in $\Mod_{\k}$. A
  \emph{$d$-preoriented space} is a finite cell complex $X$ equipped
  with a $d$-preorientation.
\end{definition}

\begin{remark}
  For a finite cell complex $X$ the cochains $C^{*}(X)$ are dualizable
  with dual the $\k$-chains $C_{*}(X) \coloneqq \colim_{X} \k$. A
  $d$-preorientation of $X$ is then equivalent to a map
  $\k[d] \to C_{*}(X)$, \ie{} a homology class in degree $d$.
\end{remark}

\begin{notation}
  For $A$ a commutative $\k$-algebra, a \emph{local system} of
  $A$-modules on $X$ is a functor $X \to \Mod_{A}$. For
  $\mathcal{E} \in \Fun(X, \Mod_{A})$ such a local system,
  we write
  \[ C_{*}(X;\mathcal{E}) \coloneqq \colim_{x \in X} \mathcal{E}_{x},\quad
    C^{*}(X;\mathcal{E}) \coloneqq \lim_{x \in X} \mathcal{E}_{x};\]
  these are the (co)chains for (co)homology with local coefficients on $X$.
  If $\mathcal{E}$ is the constant local system with value $A$, we
  also write $C_{*}(X; A)$ and $C^{*}(X; A)$.
\end{notation}

\begin{remark}
  An object $\mathcal{E} \in \Fun(X, \Mod_{A})$ is dualizable \IFF{}
  the value $\mathcal{E}_{x}$ at every point $x \in X$ is dualizable
  in $\Mod_{A}$. Note that if $X$ is a finite cell complex, then for
  $\mathcal{E}$ dualizable the $A$-module
  $C_{*}(X; \mathcal{E})$ is dualizable with dual
  $C^{*}(X;\mathcal{E}^{\vee})$, since dualizable objects are closed under
  finite (co)limits.
\end{remark}

\begin{definition}
  We say a $d$-preoriented space $(X, [X])$ is \emph{$d$-oriented} if
  for every commutative $\k$-algebra $A$ and every dualizable object $\mathcal{E}
  \in \Fun(X, \Mod_{A})$, the induced map
  \[ C^{*}(X; \mathcal{E}) \longrightarrow C_{*}(X; \mathcal{E})[-d] \]
  is an equivalence in $\Mod_{A}$, where this map is adjoint to the
  composite
  \[ C^{*}(X; \mathcal{E}^{\vee}) \otimes_{A} C^{*}(X; \mathcal{E}) \longrightarrow
    C^{*}(X; \mathcal{E}^{\vee}\otimes \mathcal{E}) \longrightarrow C^{*}(X; A)
    \simeq C^{*}(X) \otimes_{\k} A \xto{[X] \otimes_{\k} \id} A.\]
\end{definition}

\begin{example}
  If $X$ is the homotopy type of a closed oriented $d$-manifold and
  $[X]$ is the dual fundamental class of $X$, then Poincar\'e duality
  implies that $(X,[X])$ is $d$-oriented. More precisely, this follows
  easily from the approach to Poincar\'e duality through Verdier
  duality, as described in the textbooks \cite{DimcaSheaves}*{\S 3.3},
  \cite{IversenSheaves}*{\S 5.3}.
\end{example}

\begin{definition}
  We define $\POr_{d}^{\mathcal{S}}$ by the pullback
  \[
    \begin{tikzcd}
      \POr_{d}^{\mathcal{S}} \arrow{d} \arrow{r} &
      (\Mod_{\k/\k[-d]})^{\op} \arrow{d} \\
      \mathcal{S}_{\fin} \arrow{r}{C^{*}} & \Mod_{\k}^{\op}.
    \end{tikzcd}
  \]
\end{definition}

\begin{lemma}
  $\POr_{d}^{\mathcal{S}}$ has pushouts (and more generally weakly
  contractible finite colimits) and these are detected by the forgetful
  functor to $\mathcal{S}_{\fin}$.
\end{lemma}
\begin{proof}
  Same as Proposition~\ref{propn:POrpushout}.
\end{proof}

\begin{definition}
  For $X \in \mathcal{S}$, let $X_{B} \coloneqq \colim_{X} \Spec \k \in \dSt$ be the
  corresponding  constant stack, or \emph{Betti stack}. This determines a functor
  $(\blank)_{B} \colon \mathcal{S} \to \dSt$, which is the unique
  colimit-preserving functor taking $*$ to $\Spec \k$.
\end{definition}

\begin{proposition}
  For $X \in \mathcal{S}$ and $S \in \dSt$ there is a natural equivalence
  \[\QCoh(X_{B} \times S) \simeq \Fun(X, \QCoh(S)),\] whereby for a morphism
  $f \colon X \to Y$ in $\mathcal{S}$ the functor $(f_{B} \times \id_{S})^{*}$
  corresponds to the functor
  \[f^{*} \colon \Fun(Y, \QCoh(S)) \longrightarrow \Fun(X, \QCoh(S))\]
  given by composition with $f$,
  and $(f_{B} \times \id_{S})_{*}$ corresponds to the functor
  \[f_{*} \colon \Fun(X, \QCoh(S)) \longrightarrow \Fun(Y, \QCoh(S))\]
  given by right Kan extension along $f$.
\end{proposition}
\begin{proof}
  This follows from descent: $X_{B} \times S \simeq \colim_{X} S$
  since $\dSt$ is cartesian closed, hence
  \[ \QCoh(X_{B} \times S) \simeq \lim_{X} \QCoh(S) \simeq \Fun(X,
    \QCoh(S)).\]
  Moreover, the descent equivalence is natural in $X$, so for $f \colon X \to
  Y$ in $\mathcal{S}$ the functor \[(f_{B}\times S)^{*} \colon
  \QCoh(Y_{B}\times S) \longrightarrow \QCoh(X_{B} \times S)\] is identified with
  composition with $f$, as required; this implies the corresponding right adjoints are also
  identified.
\end{proof}

\begin{remark}\label{rmk:projpflim}
  In particular, applying this to the morphism
  $\pi_{S}= q_{B} \times \id_{S}\colon X_{B} \times S \to S$ where $q$
  is the unique map $X \to *$, we see that under this equivalence
  $\mathcal{O}_{X_{B} \times S} \simeq \pi_{S}^{*}\mathcal{O}_{S}$
  corresponds to the constant functor with value $\mathcal{O}_{S}$,
  while $\pi_{S,*} \colon \QCoh(X_{B} \times S) \to \QCoh(S)$
  corresponds to the functor
  \[ \limIL_{X} \colon \Fun(X, \QCoh(S)) \longrightarrow \QCoh(S).\]
\end{remark}

\begin{lemma}\label{lem:BettiOcpt}
  For any $S \in \dSt$ and
  $X \in \mathcal{S}_{\fin}$, the projection $\pi_{S}\colon X_{B} \times
  S \to S$ is an $\mathcal{O}$-compact morphism in the sense of
  Definition~\ref{defn:Ocompact}.
\end{lemma}
\begin{proof}
  Since any pullback of $\pi_{S}$ is a morphism of the same type, it
  suffices to check that $\pi_{S,*}$ preserves colimits and dualizable
  objects. By Remark~\ref{rmk:projpflim} we can identify this with the
  functor $\limIL_{X} \colon \Fun(X, \QCoh(S)) \to \QCoh(S)$, and
  $X$-indexed limits are finite since by assumption $X$ lies in
  $\mathcal{S}_{\fin}$.  Since $\QCoh(S)$ is stable, colimits there
  commute with finite limits. Moreover, since the tensor product in
  $\QCoh(S)$ preserves colimits it also preserves finite limits, and
  hence dualizable objects are closed under finite limits.
\end{proof}

\begin{corollary}\label{cor:BettiGamma}
  There is a commutative triangle
  \[
    \begin{tikzcd}
      \mathcal{S} \arrow{dr}[swap]{C^{*}} \arrow{rr}{(\blank)_{B}} & & \dSt
      \arrow{dl}{\Gamma(\mathcal{O}_{(\blank)})} \\
       & \Mod_{\k}.
     \end{tikzcd}
    \]
\end{corollary}
\begin{proof}
  We have natural equivalences
  $\Gamma\mathcal{O}_{X_{B}} \simeq \lim_{X} \Gamma \Spec \k \simeq
  C^{*}(X)$, as required.
\end{proof}

\begin{corollary}
  The Betti stack functor induces a functor
  \[ \POr_{d}^{\mathcal{S}} \xto{(\blank)_{B}} \POr_{\Spec \k,d},\]
  and more generally a functor
  \[ \POr_{d}^{\mathcal{S}} \xto{(\blank)_{B}\times S} \POr_{S,d}\]
  for any derived stack $S$, compatible with base change. These
  functors preserve pushouts.
\end{corollary}
\begin{proof}
  To see the functor exists, combine the commutative diagram from
  Corollary~\ref{cor:BettiGamma} with the pullback squares defining
  the two \icats{}. The functors preserve pushouts since these are
  computed in $\mathcal{S}$ and $\dSt_{S}$, respectively, and
  $(\blank)_{B}$ is by definition the unique colimit-preserving
  functor $\mathcal{S} \to \dSt$ that takes $*$  to $\Spec \k$.
\end{proof}

\begin{corollary}\label{cor:CospanPOrSftr}
  For any derived stack $S$, the functor $(\blank)_{B} \times S$
  induces a symmetric monoidal functor of $(\infty,n)$-categories
  \[ \Cospan_{n}(\POr_{d}^{\mathcal{S}}) \to
    \Cospan_{n}(\POr_{S,d}),\]
  natural in $S$. \qed
\end{corollary}

\begin{definition}\label{def:orcospS}
  By the same argument as for Proposition~\ref{propn:cospanqcohdiag},
  given an $n$-uple cospan
  $p \colon \Sp^{n,\op} \to \POr_{d}^{\mathcal{S}}$ taking the
  terminal object to $X$ and an object
  $\mathcal{E} \in \Fun(X, \Mod_{A})$ there is an induced diagram
  \[ p_{A,\mathcal{E}} \colon \Sp^{n} \to \Mod_{A/A[-d]} \]
  that takes $I$ to
  \[
    \begin{split}
    C^{*}(p(I); \mathcal{E}) \otimes_{A} C^{*}(p(I);
    \mathcal{E}^{\vee}) \to C^{*}(p(I); \mathcal{E}
    \otimes_{A} \mathcal{E}^{\vee})  & \to C^{*}(p(I); A)\\ &  \simeq
    C^{*}(p(I);\k) \otimes_{\k} A \to A[-d].      
    \end{split}
\]
  As in Corollary~\ref{cor:cospanEdiag},  this induces a diagram
  \[ \widehat{p}_{A,\mathcal{E}} \colon \tSp{}^{n} \to \Mod_{A}.\]
  We say that $p$ is \emph{oriented} if this diagram is non-degenerate
  whenever $\mathcal{E}$ is dualizable, for all $A \in \CAlg(\k)$.
\end{definition}

\begin{example}\label{ex:cobordnondeg}
  Suppose $M$ is a $(d+1)$-dimensional oriented cobordism from $M_{0}$
  to $M_{1}$. The relative fundamental class of $M$ makes the cospan
  \[ M_{0} \hookrightarrow M \hookleftarrow M_{1} \]
  of underlying homotopy types a $d$-preoriented cospan, and this is
  oriented: the condition is that the commutative square
  \[
    \begin{tikzcd}
    C^{*}(M; \mathcal{E}^{\vee}) \arrow{r} \arrow{d} & C^{*}(M_{0};
    \mathcal{E}^{\vee}|_{M_{0}}) \arrow{d} \\
    C^{*}(M_{1};\mathcal{E}^{\vee}|_{M_{1}}) \arrow{r} &
    C_{*}(M;\mathcal{E})[-d]      
    \end{tikzcd}
  \]
  should be cartesian. If $\mathcal{E}$ is the constant functor with
  value $\k$ this is equivalent to the usual statement of
  Poincar\'e--Lefschetz duality (\ie{} Poincar\'e duality for manifolds
  with boundary) with coefficients in $\k$, while the general version
  follows easily from the approach to Poincar\'e--Lefschetz duality
  through Verdier duality, as discussed in \cite{IversenSheaves}*{\S
    VI.3}.
\end{example}

\begin{definition}
  By the same proofs as for $\COSPAN_{n}(\POr_{d,S})$, the oriented
  iterated cospans are closed under composition in
  $\COSPAN_{n}(\POr_{d}^{\mathcal{S}})$ and so determine a
  sub-$n$-uple Segal space $\OR_{n}^{\mathcal{S},d}$ with underlying
  $n$-fold Segal space $\Or_{n}^{\mathcal{S},d}$, and this is moreover
  symmetric monoidal (via disjoint union). 
\end{definition}

Since the notions of non-degeneracy for preoriented spaces and
preoriented $S$-stacks correspond under the Betti stack functor, the
functor from Corollary~\ref{cor:CospanPOrSftr} restricts to give:
\begin{corollary}\label{cor:BettiOrn}
  For any derived stack $S$, the functor $(\blank)_{B} \times S$
  induces a symmetric monoidal functor of $(\infty,n)$-categories
  \[ \Or_{n}^{\mathcal{S},d}  \to \Or_{n}^{S,d} \]
  natural in $S$. \qed
\end{corollary}

\chapter{The AKSZ Construction}\label{sec:aksz}
In this chapter we implement the AKSZ construction for derived
symplectic stacks as a family of symmetric monoidal functors of
$(\infty,n)$-categories. We first describe a pushforward construction
on differential forms in \S\ref{subsec:DFpush} and then use this to
define the AKSZ construction on the level of differential forms in
\S\ref{subsec:AKSZdiff}; this material largely follows \cite{PTVV},
though we take care to set things up in a fully functorial manner. We
then check that the AKSZ construction induces functors on
$(\infty,n)$-categories of iterated spans in
\S\ref{subsec:akszonspans}, and then show that these restrict to the
oriented and Lagrangian sub-$(\infty,n)$-categories in
\S\ref{subsec:aksznondeg}.

\section{Pushforward of Differential Forms}\label{subsec:DFpush}
In this section we will recall that there is a natural pushforward
construction for differential forms, using the discussion of
pushforward and the de Rham complex in \S\ref{subsec:deRham}.

\begin{proposition}\label{propn:pushforwardexists}
  Suppose $f \colon X \to S$ is a geometric $S$-stack and $\phi \colon
  T \to S$ is a $d$-preoriented
  $S$-stack (and so in particular universally cocontinuous). If $Y$ is defined by the pullback square
  \csquare{Y}{X}{T}{S,}{\psi}{g}{f}{\phi}
  then there is a natural \emph{pushforward}
  map for closed relative differential forms
  \[ \int_{[T]} \colon \Apcl_{T}(Y,s) \to \Apcl_{S}(X, s-d).\]
\end{proposition}
\begin{proof}
  By Corollary~\ref{cor:relDRpb} there is for such a pullback square a
  natural equivalence
  \[\Gamma_{S}(\phi_{*}\mathcal{O}_{T} \otimes_{S} \DR_{S}(X)) \isoto
  \DR(Y/T).\]
  Combining this with $[T] \colon \phi_{*}\mathcal{O}_{T} \to
  \mathcal{O}_{S}[-d]$ we get a map
  \[ \DR(Y/T) \simeq \Gamma_{S}(\phi_{*}\mathcal{O}_{T}
  \otimes_{S} \DR_{S}(X)) \to \Gamma_{S}(\mathcal{O}_{S}[-d]
  \otimes_{S} \DR_{S}(X)) \simeq \DR(X/S)[-d]\]
  in $\GMM_{k}$. Applying $\Map_{\GMM_{k}}(k(p)[-p-s],\blank)$ we get
  a morphism \[\Apcl_{T}(Y, s) \to \Apcl_{S}(X, s-d),\] as required.
\end{proof}

\begin{remark}\label{rmk:pushfwddesc}
  In the situation of Proposition~\ref{propn:pushforwardexists}
  consider $\omega \in \Apcl_{T}(Y, s)$, whose underlying $p$-form
  corresponds under the equivalence of Lemma~\ref{lem:pbdiffform} to a
  map
  $\omega' \colon \mathcal{O}_{S}[-s] \to \phi_{*}\mathcal{O}_{T}
  \otimes f_{*}\Lambda^{p}\mathbb{L}_{X/S}$. Then the underlying
  $p$-form of $\int_{[T]}\omega$ is adjoint to the composite
\[ \mathcal{O}_{S}[-s] \xto{\omega'} \phi_{*}\mathcal{O}_{T}
  \otimes f_{*}\Lambda^{p}\mathbb{L}_{X/S} \xto{[T] \otimes \id}
  \mathcal{O}_{S}[-d] \otimes f_{*}\Lambda^{p}\mathbb{L}_{X/S} \simeq f_{*}\Lambda^{p}\mathbb{L}_{X/S}[-d];\]
  this follows immediately from
  Remark~\ref{rmk:pbderhamunderlying}. Alternatively, if $X$ is an
  Artin $S$-stack and $\omega$ corresponds to a map $\omega''
  \colon \Lambda^{p}\mathbb{T}_{X/S} \to
  f^{*}\phi_{*}\mathcal{O}_{T}[s]$ then $\int_{[T]}\omega$
  corresponds to the composite
  \[ \Lambda^{p}\mathbb{T}_{X/S} \xto{\omega''}
  f^{*}\phi_{*}\mathcal{O}_{T}[s] \xto{f^{*}[T]}
  f^{*}\mathcal{O}_{S}[s-d] \simeq \mathcal{O}_{X}[s-d].\]
\end{remark}

\begin{remark}\label{rmk:intftr}
  Let $\dSt_{S}^{\txt{geom}}$ denote the full subcategory of
  $\dSt_{S}$ spanned by the geometric $S$-stacks. Then we have a
  functor
  \[ \POr^{\op}_{S,d} \times \dSt_{S}^{\txt{geom},\op} \to \Fun(\Delta^{1},
  \GMC_{k}) \]
  taking $(\phi \colon T \to S, [T], X \to S)$ to the natural
  morphism
  \[ \DR(X \times_{S} T / T) \isofrom
  \Gamma_{S}(\phi_{*}\mathcal{O}_{T} \otimes_{S} \DR_{S}(X)) \to
  \DR(X/S).\]
  Composing with the functor $\Map_{\GMM_{k}}(k(p)[-p-s], \blank)$ we
  get a functor
  \[\POr^{\op}_{S,d} \times \dSt_{S}^{\txt{geom},\op} \to
  \Fun(\Delta^{1}, \mathcal{S})\]
  taking $(\phi \colon T \to S, [T], X \to S)$ to the
  pushforward map $\Apcl_{T}(X \times_{S}T, s) \to \Apcl_{S}(X, s-d)$,
  and similarly for non-closed forms.  If we let
  $\mathcal{X} \to \POr^{\op}_{S,d} \times
  \dSt_{S}^{\txt{geom},\op}$
  denote the left fibration for the presheaf
  $(T \to S, X \to S) \mapsto \Apcl_{T}(X \times_{S}T, s)$, then this
  corresponds to a functor
  $\mathcal{X} \to \dSt^{\txt{geom}}_{S/\Apcl_{S}(s-d)}$.
\end{remark}

\begin{remark}\label{rmk:pfbasechange}
  The pushforward construction is also compatible with base change:
  Suppose we have a commutative cube   \[
    \begin{tikzcd}
      Y' \arrow{rr} \arrow{dr}{\eta} \arrow{dd} & & X'
      \arrow{dr}{\xi} \arrow{dd} \\
      & Y \arrow[crossing over]{rr} & & X\arrow{dd} \\
      T' \arrow{rr}{\phi'}  \arrow{dr}{\tau}& & S' \arrow{dr}{\sigma} \\
      & T \arrow[leftarrow,crossing over]{uu} \arrow{rr}{\phi} && S
    \end{tikzcd}
  \]
  where all faces are cartesian, $T$ is a $d$-preoriented $S$-stack, and $X$ is
  a geometric $S$-stack. Then we have a commutative diagram
  \[
    \begin{tikzcd}
      \DR(Y/T) \arrow{r}{\eta^{*}} & \DR(Y'/T') \\
      \Gamma_{S}(\phi_{*}\mathcal{O}_{T} \otimes_{S}
      \DR_{S}(X)) \arrow{u}{\sim} \arrow{r} \arrow{d} &
\Gamma_{S'}(\phi'_{*}\mathcal{O}_{T'} \otimes_{S'}
\DR_{S'}(X')) \arrow{u}{\sim} \arrow{d}  \\
      \Gamma_{S}(\mathcal{O}_{S}[-d] \otimes_{S} \DR_{S}(X)) \arrow{r}
      \arrow{d}{\sim} & \Gamma_{S'}(\mathcal{O}_{S'}[-d] \otimes_{S'} \DR_{S'}(X))
      \arrow{d}{\sim} \\
      \DR(X/S)[-d] \arrow{r}{\xi^{*}}
       & \DR(X'/S')[-d] 
    \end{tikzcd}
  \]
  where the top square is from Remark~\ref{rmk:drpbbasechange},
  the middle comes from the commutative square
 \nolabelcsquare{\phi_{*}\mathcal{O}_{T}}{\mathcal{O}_{S}[-d]}{\sigma_{*}\phi'_{*}\mathcal{O}_{T'}}{\sigma_{*}\mathcal{O}_{S'}[-d]}
 and the projection formula transformation for $\sigma^{*} \dashv
 \sigma_{*}$, and the bottom follows from unwinding the definition of
 $\sigma^{*}$ in Remark~\ref{rmk:DRfunctor}. 
 Applying  $\Map_{\GMM_{k}}(k(p)[-p-s],\blank)$, we get a commutative
 square
 \[
   \begin{tikzcd}
     \Apcl_{T}(Y,s) \arrow{r}{\int_{[T]}} \arrow{d}{\eta^{*}} &
     \Apcl_{S}(X,s-d) \arrow{d}{\xi^{*}} \\
     \Apcl_{T'}(Y',s) \arrow{r}{\int_{[T']}} & \Apcl_{S'}(X',s-d).
   \end{tikzcd}
   \]
\end{remark}

\section{The AKSZ Construction on Differential Forms}\label{subsec:AKSZdiff}
Suppose $T$ is a $d$-preoriented $S$-stack and $X$ is a geometric
$S$-stack such that $X^{T}_{S} \to S$ is again geometric (where
$X^{T}_{S}$ denotes the internal Hom in $\dSt_{S}$). Then the
pushforward construction of the previous section gives a map
\[ \int_{[T]} \colon \Apcl_{T}(X^{T}_{S}
\times_{S} T, s) \to \Apcl_{S}(X^{T}_{S},s-d).\]
Combining this with the pullback of differential forms along the
evaluation morphism $\txt{ev} \colon X^{T}_{S} \times_{S} T \to X$
we have a map
\[ \txt{aksz}_{T} := \int_{[T]}\txt{ev}^{*} \colon \Apcl_{S}(X,s)
\to \Apcl_{S}(X^{T}_{S}, s-d).\]
This is what we refer to as the \emph{AKSZ construction}, as it is a
derived algebro-geometric analogue of the symplectic structures on
mapping spaces constructed in \cite{AKSZ}.

\begin{proposition}\label{propn:akszpair}
  Suppose $X$ is an Artin $S$-stack and $T$ is an
  $\mathcal{O}$-compact $S$-stack such that
  $f \colon X^{T}_{S} \to S$ is also an Artin
  $S$-stack. Then for $\omega \in \Apcl_{S}(X, s)$ the underlying $p$-form
  of $\int_{[T]}\txt{ev}^{*}\omega$ corresponds (under the
  natural equivalence
  $\mathbb{T}_{X^{T}_{S}} \isoto \pi_{*}\txt{ev}^{*}\mathbb{T}_{X/S}$
  of Proposition~\ref{propn:mapscotgt}) to the composite
  \[
    \begin{split}
    \Lambda^{p} \pi_{*}\txt{ev}^{*}\mathbb{T}_{X/S} \to
  \pi_{*}(\Lambda^{p} \txt{ev}^{*}\mathbb{T}_{X/S}) & 
  \xto{\pi_{*}\txt{ev}^{*}\omega}
  \pi_{*}\txt{ev}^{*}\mathcal{O}_{X}[s] \\ & \simeq
  f^{*}\sigma_{*}\mathcal{O}_{T}[s] \xto{f^{*}[T]}
  f^{*}(\mathcal{O}_{S}[-d])[s] \simeq
  \mathcal{O}_{X^{T}_{S}}[s-d],      
    \end{split}
\]
  where the first map is the adjoint of the equivalence
  \[\pi^{*}(\Lambda^{p} \pi_{*}\txt{ev}^{*}\mathbb{T}_{X/S}) \simeq
  \Lambda^{p}\pi^{*}\pi_{*}\txt{ev}^{*}\mathbb{T}_{X/S}\] and the third
  map uses the base change equivalence for the pullback square
  \csquare{X^{T}_{S} \times_{S} T}{X^{T}_S}{T}{S.}{\pi}{}{f}{\sigma}
\end{proposition}
\begin{proof}
  Combine the descriptions of Proposition~\ref {propn:mapscotgt} and
  Remark~\ref{rmk:pushfwddesc}.
\end{proof}

Our goal for the remainder of this section is to show that the AKSZ
construction is functorial, in 
the sense that it gives a functor 
\[ (\POr_{S,d}^{\op} \times \dSt^{\txt{Art}}_{S/\Apcl_{S}(s)})^{\txt{good}} \longrightarrow
\dSt^{\txt{Art}}_{S/\Apcl_{S}(s-d)},\]
where $(\POr_{S,d}^{\op} \times
\dSt^{\txt{Art}}_{S/\Apcl_{S}(s)})^{\txt{good}}$ denotes
the full subcategory of
$\POr_{S,d}^{\op} \times
\dSt^{\txt{Art}}_{S/\Apcl_{S}(s)}$
spanned by those pairs $((T,[T]),X)$ such that $X^{T}_{S}$ is again
an Artin $S$-stack. This
takes a bit of work because of the ``bivariant'' naturality of
evaluation maps: given a map $T \to T'$ over $S$ and $X$ over $S$ we
get a commutative diagram
\[
\begin{tikzcd}
   X^{T}_{S} \times_{S} T \arrow{dr} \\
X^{T'}_{S} \times_{S} T \arrow{u}{X^{\sigma}_{S} \times_{S}T}
\arrow{d}[left]{X^{T'}_{S}\times_{S} \sigma} \arrow{r} & X\\
 X^{T'}_{S} \times_{S} T'  \arrow{ur}
\end{tikzcd}
\]
along which we want to pull back differential forms on $X$, giving
\[
\begin{tikzcd}
  & \Apcl_{T}(X^{T}_{S} \times_{S} T, s) \arrow{d}{(X^{\sigma}_{S}
    \times_{S} T)^{*}}\\
\Apcl_{S}(X, s)   \arrow{ur} \arrow{dr} \arrow{r} &
\Apcl_{T}(X^{T'}_{S} \times_{S} T, s) \\
 & \Apcl_{T'}(X^{T'}_{S} \times_{S} T', s). \arrow{u}[right]{(X^{T'}_{S}
   \times_{S} \sigma)^{*}} 
\end{tikzcd}
\]
Next, if $\sigma \colon T \to T'$ is a morphism of $d$-preoriented stacks over
$S$, the functoriality of Remark~\ref{rmk:intftr} gives a commutative diagram of pullback and pushforward morphisms
\[
  \begin{tikzcd}
    \Apcl_{T}(X^{T}_{S} \times_{S} T, s) \arrow{r}{\int_{[T]}} \arrow{d}[left]{(X^{\sigma}_{S}
    \times_{S} T)^{*}} &
    \Apcl_{S}(X^{T}_{S},s-d) \arrow{d}{(X^{\sigma}_{S})^{*}} \\
    \Apcl_{T}(X^{T'}_{S} \times_{S} T,s) \arrow{r}{\int_{[T]}}  &
    \Apcl_{S}(X^{T'}_{S},s-d) \\
    \Apcl_{T'}(X^{T'}_{S} \times_{S} T',s). \arrow{u}{(X^{T'}_{S}
      \times_{S} \sigma)^{*}} \arrow{ur}[below right]{\int_{[T']}}
\end{tikzcd}
\]
Combining the two diagrams we get a big diagram from which we can
extract the commutative triangle
\[
\begin{tikzcd}
  {} & \Apcl_{S}(X^{T}_{S},s-d) \arrow{dd}{(X^{\sigma}_{S})^{*}} \\
  \Apcl_{S}(X,s) \arrow{ur} \arrow{dr} \\
   & \Apcl_{S}(X^{T'}_{S},s-d)
\end{tikzcd}
\]
that relates the AKSZ constructions for $T$ and $T'$ --- this gives
the required functoriality for an object of $\dSt^{\txt{Art}}_{S}$
and a morphism in $\POr_{S,d}$.

\begin{remark}\label{rmk:akszmor}
  Note that this argument shows that for any morphism
  $\tau \colon T \to T'$ in $\POr_{S,d}$, the AKSZ map
  $\Apcl_{S}(X,s) \to \Apcl_{S}(X^{T'}_{S},s-d)$ is equivalent to the
  composite
  \[ \Apcl_{S}(X,s) \to \Apcl_{T}(X^{T'}_{S} \times_{S} T,s)
    \xto{\int_{[T]}} \Apcl_{S}(X^{T'}_{S},s-d).\]
\end{remark}

To generalize this construction, we use the results on
parametrized adjunctions from \cite{paradj}. From the cartesian
closed \icat{} $\dSt_{S}$, applying (the dual of) \cite{paradj}*{Construction 4.2.1} as in
\cite{paradj}*{Example 4.2.4} gives a functor
\[ \dSt_{S} \times \Tw^{r}(\dSt_{S}) \to \dSt_{S}^{\Delta^{1}},\]
taking $(X, T \xto{\tau} T')$ to the map $X^{T'}_{S} \times_{S} T \to X$
adjoint to $X_{S}^{T'} \xto{X_{S}^{\tau}} X_{S}^{T}$. Moreover, at the terminal
object $S \in \dSt_{S}$, this takes $\tau$ to $T \simeq S^{T'}_{S}
\times_{S}T \to S$ and so induces a functor
\[ \dSt_{S} \times \Tw^{r}(\dSt_{S}) \to \dSt^{\Delta^{1}\times
    \Delta^{1}} \]
taking $(X, \tau)$ to the commutative square
\[
  \begin{tikzcd}
    X_{S}^{T'} \times_{S} T \arrow{r} \arrow{d} & X \arrow{d} \\
    T \arrow{r} & S.
  \end{tikzcd}
\]
Composing with $\Apcl_{\blank}(\blank) \colon \dSt^{\Delta^{1},\op} \to \mathcal{S}$ we
get a functor
\[ \dSt_{S}^{\op} \times \Tw^{r}(\dSt_{S})^{\op}
  \to \mathcal{S}^{\Delta^{1}},\]
which takes $(X,\tau)$ to the pullback morphism on differential forms
\[ \Apcl_{S}(X) \to \Apcl_{T}(X^{T'} \times_{S}T).\]
On the other hand, if we write $(\dSt_{S}
\times\dSt_{S}^\op)^{\txt{good}}$ for the full subcategory of $\dSt_{S}
\times\dSt_{S}^\op$ spanned by the pairs $(X,T)$ such that $X^{T}_{S}$ is
an Artin $S$-stack, then we can define a functor
\[ (\dSt_{S}^{\op} \times\dSt_{S})^{\txt{good}} \times_{\dSt_{S}} \Tw^{r}(\dSt_{S})^{\op}
  \times_{\dSt_{S}^{\op}} \POr_{S,d}^{\op} \to (\dSt_{S}^{\op} \times
  \POr_{S,d}^{\op})^{\txt{good}},\]
taking $(X, T \to T', [T])$ to $(X^{T'}_{S}, T, [T])$, which we can
combine with the naturality of the pushforward construction from
Remark~\ref{rmk:intftr}  to get a functor
\[ (\dSt_{S}^{\op} \times\dSt_{S})^{\txt{good}} \times_{\dSt_{S}}
  \Tw^{r}(\dSt_{S})^{\op} \times_{\dSt_{S}^{\op}} \POr_{S,d}^{\op} \to
  \mathcal{S}^{\Delta^{1}},\] taking $(X, T \to T', [T])$ to the
pushforward map
$\Apcl_{T}(X^{T'} \times T,s) \xto{\int_{[T]}} \Apcl_{S}(X^{T'})$.

We can now combine these two functors into a single functor
\[ (\dSt_{S}^{\op} \times\dSt_{S})^{\txt{good}} \times_{\dSt_{S}}
  \Tw^{r}(\dSt_{S})^{\op} \times_{\dSt_{S}^{\op}} \POr_{S,d}^{\op} \to
  \mathcal{S}^{\Delta^{2}},\]
and then compose to get a functor
\[ (\dSt_{S}^{\op} \times\dSt_{S})^{\txt{good}} \times_{\dSt_{S}}
  \Tw^{r}(\dSt_{S})^{\op} \times_{\dSt_{S}^{\op}} \POr_{S,d}^{\op} \to
  \mathcal{S}^{\Delta^{1}},\]
taking $(X, T \to T', [T])$ to the composite map \[\Apcl_{S}(X,s) \to
\Apcl_{T}(X^{T'}_{S} \times_{S} T,s) \to
\Apcl_{S}(X^{T'},s-d).\]

This functor takes the maps
$(X, T \to T', [T]) \to (X, T'' \to T', [T''])$ lying over $\id_{X}$
and $\id_{T'}$ to equivalences in $\mathcal{S}^{\Delta^{1}}$. We can
now apply \cite{paradj}*{Corollary 4.2.6} to understand the
localization at these morphisms. This result gives a functor from
\[
  (\dSt_{S}^{\op} \times\dSt_{S})^{\txt{good}} \times_{\dSt_{S}}
  \Tw^{r}(\dSt_{S})^{\op} \times_{\dSt_{S}^{\op}} \POr_{S,d}^{\op}\]
to the \icat{}
\[ 
  (\dSt_{S}^{\op} \times\dSt_{S})^{\txt{good}} \times_{\dSt_{S}}
  \POr_{S,d} \simeq (\dSt_{S}^{\op} \times \POr_{S,d})^{\txt{good}}    
\]
that exhibits the target as the localization we want. The functor we
defined above thus factors uniquely through a functor
\[  (\dSt_{S}^{\op} \times \POr_{S,d})^{\txt{good}}
  \to \mathcal{S}^{\Delta^{1}},\]
taking $(X, T, [T])$ to the AKSZ morphism
\[ \Apcl_{S}(X, s) \to \Apcl_{S}(X^{T}_{S}, s-d).\]

Now consider the morphism of right fibrations corresponding to this
natural transformation of presheaves. Over $0$, we get the composite
\[(\POr_{S,d} \times \dSt_{S}^{\op})^{\txt{good}} \to \dSt_{S}^{\op} \xto{\Apcl_{S}(s)}
\mathcal{S},\] with associated right fibration \[(\POr_{S,d}^{\op} \times
\dSt_{S/\Apcl_{S}(s)})^{\txt{good}} \to (\POr_{S,d}^{\op} \times
\dSt_{S})^{\txt{good}},\] and over $1$ we get the right fibration for the
composite
\[ (\POr_{S,d} \times \dSt_{S}^{\op})^{\txt{good}} \xto{\txt{exp}} \dSt_{S}^{\op}
  \xto{\Apcl_{S}(s-d)} \mathcal{S} \] where $\txt{exp}$ denotes the
functor taking $(T, X)$ to $X_{S}^{T}$; the corresponding
fibration is then the pullback of $\dSt_{S/\Apcl_{S}(s-d)} \to \dSt_{S}$ along
$\txt{exp}$, so our map of right fibrations corresponds to a
commutative square \csquare{(\POr_{S,d}^{\op} \times
  \dSt_{S/\Apcl_{S}})^{\txt{good}}}{\dSt_{S/\Apcl_{S}(s-d)}}{(\POr_{S,d}^{\op}
  \times \dSt_{S})^{\txt{good}}}{\dSt_{S}}{\txt{aksz}_{S,s,d}^{p}}{}{}{\txt{exp}} In other
words, given a $d$-preoriented $S$-stack $T$ and an Artin $S$-stack
 $X$, we have functorially associated to every closed
$s$-shifted $p$-form $\omega \in \Apcl_{S}(X,s)$ the $(s-d)$-shifted $p$-form
$\int_{[T]}\txt{ev}^{*}\omega$ on $X^{T}_{S}$.

\begin{remark}\label{rmk:akszbasechange}
  The AKSZ construction is compatible with base change, in the
  following sense: Suppose given a commutative diagram
  \[
    \begin{tikzcd}
    T' \arrow{r} \arrow{d}{\tau} & S' \arrow{d}{\sigma} & X' \arrow{l}
    \arrow{d}{\xi} \\
    T \arrow{r} & S & X \arrow{l}
    \end{tikzcd}
  \]
  where $T$ is a $d$-preoriented $S$-stack, $X$ is an Artin
  $S$-stack, and the squares are cartesian. Then we have a commutative
  diagram
\[
    \begin{tikzcd}
      X & X^{T}_{S} \times_{S} T \arrow{l} \arrow{r} & X^{T}_{S} \\
      X' \arrow{u} & \sigma^{*}X^{T}_{S} \times_{S'} T' \arrow{l}
      \arrow{r} \arrow{u}& \sigma^{*}X^{T}_{S} \arrow{u} \\
       & X'^{T'}_{S'} \times_{S'}T' \arrow{ul} \arrow{u} \arrow{r} &
       X'^{T'}_{S'} \arrow{u},
     \end{tikzcd}
   \]
   which induces on differential forms a commutative diagram
   \[
     \begin{tikzcd}
       \Apcl_{S}(X,s) \arrow{r} \arrow{d} & \Apcl_{T}(X^{T}_{S}
       \times_{S} T,s) \arrow{d} \arrow{r} & \Apcl_{S}(X^{T}_{S},s-d)
       \arrow{d} \\
       \Apcl_{S'}(X',s) \arrow{dr}\arrow{r}  & \Apcl_{T'}(\sigma^{*}X^{T}_{S}
       \times_{S'} T',s) \arrow{d} \arrow{r} & \Apcl_{S'}(\sigma^{*}X^{T}_{S},s-d)
       \arrow{d} \\       
        & \Apcl_{T'}(X'^{T'}_{S'} \times_{S'} T',s) \arrow{r} & \Apcl_{T'}(X'^{T'}_{S'},s-d).
     \end{tikzcd}
   \]
   Here the upper right square commutes by 
  Remark~\ref{rmk:pfbasechange}, the bottom right square commutes by
  Remark~\ref{rmk:intftr}, and the left part of the diagram commutes
  by the naturality of pullback of
  differential forms. Looking at the morphism on the outer boundary of
  this diagram, we get a commutative square
  \[
    \begin{tikzcd}
      \Apcl_{S}(X,s) \arrow{d} \arrow{r} &
      \Apcl_{S}(X^{T}_{S},s-d) \arrow{d} \\
      \Apcl_{S'}(X',s) \arrow{r} & \Apcl_{S'}(X'^{T'}_{S'},s-d).
    \end{tikzcd}
  \]
  where the horizontal maps are the AKSZ morphisms and the vertical
  ones are pullback morphisms.
\end{remark}

\section{The AKSZ Construction on Iterated Spans}\label{subsec:akszonspans}
Suppose $X$ is an $s$-presymplectic $S$-stack; we write
$\POr_{S,d}^{X\txt{-good}}$ for the full subcategory of $\POr_{S,d}$
spanned by the preoriented $S$-stacks $(T,[T])$ such that $X^{T}_{S}$
is an Artin $S$-stack. Then the AKSZ functor (for 2-forms) from the previous
section restricts to a functor
\[ \txt{aksz}^{X}_{S,d} \colon \POr_{S,d}^{X\txt{-good},\op} \to
  \PSymp_{S,s-d}.\]
Our goal in this section is to prove that this induces a symmetric
monoidal functor
\[ \Cospan_{n}(\POr_{S,d}^{X\txt{-good}}) \to \Span_{n}(\PSymp_{S,s-d})\]
on iterated spans.

As a first step we check that $\txt{aksz}^{X}_{S,s,d}$ induces a
functor on iterated spans, which follows from the following observation:
\begin{proposition}
  Let $X$ be an $s$-presymplectic $S$-stack. The \icat{}
  $\POr_{S,d}^{X\txt{-good}}$ has pushouts, and the AKSZ functor
  $\POr_{S,d}^{X\txt{-good},\op} \to \PSymp_{S,s-d}$ preserves fibre
  products.
\end{proposition}
\begin{proof}
  To see that $\POr_{S,d}^{X\txt{-good}}$ has pushouts, we will show that it is closed
  under pushouts in $\POr_{S,d}$. Pushouts in $\POr_{S,d}$ are
  detected in $\dSt_{S}$ by Proposition~\ref{propn:POrpushout}, so to
  see this it suffices to show that the full subcategory of $\dSt_{S}$
  on $S$-stacks $T$ such that $X^{T}_{S}$ is an Artin $S$-stack is
  closed under pushouts. This holds since
  $X^{(\blank)}_{S} \colon \dSt_{S}^{\op} \to \dSt_{S}$ preserves
  limits, and Artin $S$-stacks are closed under finite limits by
  Proposition~\ref{propn:geomfinlim}.

  It remains to see that $\txt{aksz}_{S,d}^{X}$ preserves fibre
  products.
  Since pullbacks in $\PSymp_{s-d}$ are detected in
  $\dSt_{S}$ by Lemma~\ref{lem:PSymppullback}, it suffices to prove
  that the composite $\POr_{S,d}^{X\txt{-good},\op} \to
  \dSt_{S}$ preserves pullbacks. But this factors as the composite of the
  forgetful functor $\POr_{S,d}^{X\txt{-good},\op} \to \dSt^{\op}$,
  which we just saw preserves pullbacks, and the internal Hom functor
  $X_{S}^{(\blank)} \colon \dSt_{S}^{\op} \to \dSt_{S}$, which preserves all
  limits.
\end{proof}

\begin{corollary}\label{cor:AKSZspan}
  Let $X$ be an $s$-presymplectic $S$-stack. Then 
  $\txt{aksz}_{S,d}^{X}$ induces a morphism of $n$-uple Segal spaces
  \[ \COSPAN_{n}(\POr_{S,d}^{X\txt{-good}}) \to
    \SPAN_{n}(\PSymp_{S,s-d}),\]
  and hence a functor 
  \[ \txt{AKSZ}_{n,X}^{S,d} \colon \Cospan_{n}(\POr_{S,d}^{X\txt{-good}}) \to
    \Span_{n}(\PSymp_{S,s-d}).\]
  of $(\infty,n)$-categories. \qed
\end{corollary}

Our next goal is to show that this functor is symmetric monoidal,
which we do by checking that it is compatible with the identifications
we used to define the symmetric monoidal structures in
\S\ref{subsec:itlagsymmon}. The following is the key observation we need:
\begin{proposition}
  Let $X$ be an $s$-presymplectic $S$-stack. Then
  $\txt{aksz}_{S,d}^{X}$ induces a functor
  \[ (\txt{aksz}_{S,d}^{X})_{\emptyset,\emptyset/} \colon (\POr^{X\txt{-good}}_{S,d})_{\emptyset,\emptyset/} \to (\PSymp_{S,s-d})_{/S,S},\]
  since $X^{\emptyset}_{S} \simeq S$. There is a natural commutative
  square
  \[
    \begin{tikzcd}[column sep=huge]
      (\POr^{X\txt{-good}}_{S,d})_{\emptyset,\emptyset/} \arrow{r}{(\txt{aksz}_{S,d}^{X})_{\emptyset,\emptyset/}}
      \arrow{d}{\sim} &  (\PSymp_{S,s-d})_{/S,S} \arrow{d}{\sim} \\
      \POr^{X\txt{-good}}_{S,d+1} \arrow{r}{\txt{aksz}_{S,d}^{X}} &\PSymp_{S,s-d-1},
    \end{tikzcd}
  \]
  where the vertical morphisms are the equivalences of
  Corollary~\ref{cor:deloopPSymp} and
  Proposition~\ref{propn:deloopPOr}. 
\end{proposition}

\begin{proof}
  Given an object $T \in
  (\POr^{X\txt{-good}}_{S,d})_{\emptyset,\emptyset/}$, the AKSZ
  construction gives a commutative diagram
  \[
    \begin{tikzcd}
      {} & \Atcl_{S}(S,s-d) \arrow{d} \\
      \Atcl_{S}(X,s) \arrow{r}{\txt{aksz}_{T}}
      \arrow{ur}{\txt{aksz}_{\emptyset}}
      \arrow{dr}[below left]{\txt{aksz}_{\emptyset}}& \Atcl_{S}(X^{T}_{S},s-d)\\
      & \Atcl_{S}(S,s-d) \arrow{u}
    \end{tikzcd}
  \]
  Unwinding the definitions, we see that we are required to show that
  the induced morphism \[\Atcl_{S}(X,s) \to \Atcl_{S}(X^{T}_{S},s-d-1)\]
  to the pullback (via Proposition~\ref{propn:Atclloop}) is
  $\txt{aksz}_{T'}$ where $T'$ denotes $T$ equipped with the
  $(d+1)$-orientation induced by the maps $\emptyset \to T$ as in
  Proposition~\ref{propn:deloopPOr}. The naturality of pullback and
  pushforward of differential forms, as discussed in \S\ref{subsec:AKSZdiff}, shows that
  each morphism $\emptyset \to T$ induces a
  commutative diagram
  \[
    \begin{tikzcd}
      {} & \Atcl_{\emptyset}(X_{S}^{\emptyset} \times_{S}\emptyset, s) \arrow{d}{\sim}
      \arrow{r}{\sim} & \Atcl_{S}(X_{S}^{\emptyset}, s-d) \arrow{d} \\
      \Atcl_{S}(X,s) \arrow{r} \arrow{ur} \arrow{dr} &
      \Atcl_{\emptyset}(X^{T}_{S} \times_{S} \emptyset) \arrow{r} &
      \Atcl_{S}(X^{T}_{S},s-d)\arrow[equals]{d} \\
      & \Atcl_{T}(X^{T}_{S} \times_{S} T,s) \arrow{r} \arrow{u} &
      \Atcl_{S}(X^{T}_{S},s-d).
    \end{tikzcd}
  \]
  We can therefore identify the diagram above with the outer square in
  the commutative
  diagram
  \[
    \begin{tikzcd}
    {} & \Atcl_{\emptyset}(X^{T}_{S}\times_{S} \emptyset,s) \arrow{dr}
    \\
    \Atcl_{S}(X,s)\arrow{ur} \arrow{dr} \arrow{r} &
    \Atcl_{T}(X^{T}_{S} \times_{S} T, s) \arrow{u} \arrow{d}
    \arrow{r}& \Atcl_{S}(X^{T}_{S},s-d) \\
     & \Atcl_{\emptyset}(X^{T}_{S}\times_{S} \emptyset,s). \arrow{ur}
   \end{tikzcd}
 \]
 This in turn arises from a commutative diagram of de Rham complexes:
 \[
   \begin{tikzcd}[column sep=small]
     {} & \DR(X^{T}_{S} \times_{S} \emptyset/\emptyset) & \arrow{l}{\sim} \Gamma_{S}(0
     \otimes_{S} \DR_{S}(X^{T}_{S})) \arrow{dr} \\
     \DR(X/S) \arrow{dr}\arrow{ur} \arrow{r} & \DR(X^{T}_{S} \times_{S}T/T)
     \arrow{u} \arrow{d} &
     \Gamma_{S}(\sigma_{*}\mathcal{O}_{T}
     \otimes_{S}\DR_{S}(X^{T}_{S}))  \arrow{u} \arrow{d}
     \arrow{l}{\sim} \arrow{r} & \DR_{S}(X^{T}_{S})[-d]\\
     {} & \DR(X^{T}_{S} \times_{S} \emptyset/\emptyset) & \arrow{l}{\sim} \Gamma_{S}(0
     \otimes_{S} \DR_{S}(X^{T}_{S})), \arrow{ur}
   \end{tikzcd}
 \]
 where we see that the right-most part arises by composing with the
 diagram of preorientations
 \[
   \begin{tikzcd}
     0 \arrow{dr} \\
     \sigma_{*}\mathcal{O}_{T} \arrow{r} \arrow{u} \arrow{d} &
     \mathcal{O}_{S}[-d] \\
     0. \arrow{ur}
   \end{tikzcd}
 \]
 It follows that the induced map to the fibre product
 \[ \DR(X/S) \to \DR(X^{T}_{S}/S)[-d-1] \]
 is the composite
 \[ \DR(X/S) \to \DR(X^{T}_{S} \times_{S} T/T) \isofrom \Gamma_{S}(\sigma_{*}\mathcal{O}_{T}
   \otimes_{S}\DR_{S}(X^{T}_{S})) \to \DR(X^{T}_{S}/S)[-d-1],\]
 where the third morphism arises from the induced map
 $\sigma_{*}\mathcal{O}_{T} \to \mathcal{O}_{S}[-d-1]$,
 which is exactly the $(d+1)$-preorientation on $T$. 
\end{proof}

\begin{corollary}
  Let $X$ be an $s$-presymplectic $S$-stack. Then there is a
  commutative square
  \[
    \begin{tikzcd}[column sep=huge]
      \Cospan_{n}(\POr_{S,d}^{X\txt{-good}})(\emptyset,\emptyset)
      \arrow{r}{\txt{AKSZ}_{S,d}^{X}(\emptyset,\emptyset)}
      \arrow{d}{\sim} & \Span_{n}(\PSymp_{S,s-d})(S,S) \arrow{d}{\sim}
      \\
      \Cospan_{n-1}(\POr_{S,d+1}^{X\txt{-good}})
      \arrow{r}{\txt{AKSZ}_{S,d+1}^{X}}
      & \Span_{n-1}(\PSymp_{S,s-d-1}),
    \end{tikzcd}
  \]
  where the vertical equivalences are those of
  Corollary~\ref{cor:SpanSSeq}. \qed
\end{corollary}

\begin{corollary}\label{cor:AKSZsymmon}
  Let $X$ be an $s$-presymplectic $S$-stack. Then the functor of
  $(\infty,n)$-categories
  \[\txt{AKSZ}_{S,d}^{X} \colon \Cospan_{n}(\POr_{S,d}^{X\txt{-good}})
  \to \Span_{n}(\PSymp_{S,s-d})\] has a natural symmetric monoidal
  structure, with respect to the symmetric monoidal structures of
  Corollary~\ref{cor:SpanPSympSM}. \qed
\end{corollary}

\begin{remark}
  For $K \in \mathcal{S}_{\fin}$, the Betti stack $K_{B} \times S$ is
  $X$-good for any Artin $S$-stack $X$, since we have
  \[ X_{S}^{K_{B} \times S} \simeq \lim_{K} X \]
  (where the limit is computed in $\dSt_{S}$); since Artin $S$-stacks
  are closed under finite limits, this means $X_{S}^{K_{B} \times S}$
  is again an Artin $S$-stack when $K$ is finite.
  Combining the AKSZ construction with the functor
  $\POr_{d}^{\mathcal{S}} \xto{(\blank)_{B} \times S} \POr_{S,d}$ from
  \S\ref{subsec:orspc} we therefore get for any $s$-presymplectic
  $S$-stack $X$ a symmetric monoidal functor
  \[ \Cospan_{n}(\POr_{d}^{\mathcal{S}}) \to \Span_{n}(\PSymp_{S,s-d}).\]
\end{remark}

\section{Non-Degeneracy of the AKSZ Construction}\label{subsec:aksznondeg}
Suppose $X$ is an Artin $S$-stack and $T$ is an $\mathcal{O}$-compact
$d$-preoriented $S$-stack such that $X^{T}_{S}$ is also an Artin
$S$-stack. If $\omega$ is an $s$-shifted relative 2-form on $X$, then
it follows from Proposition~\ref{propn:akszpair} that the morphism
\[ \mathbb{T}_{X^{T}_{S}/S} \to \mathbb{L}_{X^{T}_{S}/S}[s-d] \]
corresponding to the AKSZ form $\int_{[T]}\txt{ev}^{*}\omega$ is given
by the composite
\[ \pi_{*}\txt{ev}^{*}\mathbb{T}_{X/S} \xto{\pi_{*}\txt{ev}^{*}\omega}
  \pi_{*}\txt{ev}^{*}\mathbb{L}_{X/S}[s] \longrightarrow
  (\pi_{*}\txt{ev}^{*}\mathbb{T}_{X/S})^{\vee}[s-d],\] where the
second morphism arises from the preorientation of
$\pi \colon X^{T}_{S}\times_{S}T \to X^{T}_{S}$ pulled back from
$T \to S$. If $\omega$ is symplectic then the first morphism is an
equivalence, and if $T$ is $d$-oriented then the second morphism is an
equivalence. Thus we see that if $X$ is symplectic and $T$ is
oriented, then $X^{T}_{S}$, equipped with the AKSZ 2-form, is again
symplectic. Our goal in this section is to generalize this
observation, to see that the symmetric monoidal
AKSZ functor we constructed above restricts to subcategories of
oriented and symplectic $S$-stacks.

For this we need to impose the additional hypothesis of
$\mathcal{O}$-compactness (see Definition~\ref{defn:Ocompact}) on
preoriented stacks, which requires some notation:
\begin{notation}\label{not:Orcpt}
  We write $\POrc_{S,d}$ for the full subcategory of $\POr_{S,d}$
  consisting of the $\mathcal{O}$-compact $d$-preoriented
  $S$-stacks. This is closed under pushouts in $\POr_{S,d}$ by
  Propositions~\ref{propn:UCCcolim}(v) and \ref{propn:POrpushout}, so
  we have a sub-$n$-uple Segal space $\COSPAN_{n}(\POrc_{S,d})
  \subseteq \COSPAN_{n}(\POr_{S,d})$. We can then define $\ORcSd_{n}$
  as the intersection of $\OR^{S,d}_{n}$ and
  $\COSPAN_{n}(\POrc_{S,d})$; we also define the $(\infty,n)$-category
  $\OrcSd_{n}$ similarly. For $X$ an $s$-presymplectic $S$-stack we
  also similarly define subobjects containing only the $X$-good
  objects, with analogous notation.
\end{notation}

\begin{proposition}
  Let $X$ be an $s$-symplectic $S$-stack. Then the morphism
  \[ \COSPAN_{n}(\POr_{S,d}^{X\txt{-good}}) \longrightarrow
    \SPAN_{n}(\PSymp_{S,s-d})\]
  of $n$-uple Segal spaces from Corollary~\ref{cor:AKSZspan}, given by
  the AKSZ construction, restricts to a morphism
  \[ \OR_{n}^{\txt{cpt},X\txt{-good},S,d} \longrightarrow \LAG_{n}^{S,s-d}.\]
\end{proposition}
\begin{proof}
  Suppose $\Phi \colon \Sp^{n,\op} \to \POr_{S,d}^{\txt{cpt},X\txt{-good}}$ is
  an oriented $n$-uple cospan, and let $X_{S}^{\Phi} \colon \Sp^{n} \to
  \PSymp_{S,s-d}$ denote its image under the AKSZ construction. We
  also abbreviate $T := \Phi(\infty)$ and write $t_{i}\colon \Phi(i)
  \to T$ for the unique map to the terminal object in the diagram and
  $\bar{t}_{i}$ for $X_{S}^{t_{i}} \colon X_{S}^{T} \to X_{S}^{\Phi(i)}$. Our
  goal is then to prove that the induced diagram $\widehat{X_{S}^{\Phi}}
  \colon \tSp{}^{n} \to \QCoh(X_{S}^{T})$ is non-degenerate.

  For $i \in \Sp^{n,\op}$ we have a commutative diagram
  \[
    \begin{tikzcd}
      {} & X_{S}^{T} \times_{S} T \arrow{dr}{\txt{ev}} \\
      X_{S}^{T}\times_{S} \Phi(i) \arrow{ur}{\id \times_{S} t_{i}}
      \arrow{r}{\bar{t}_{i}\times_{S} \id}  \arrow{d}{\pi'_{i}}& X_{S}^{\Phi(i)} \times_{S}
      \Phi(i) \arrow{r}[swap]{\txt{ev}_{i}} \arrow{d}{\pi_{i}} & X \\
      X_{S}^{T} \arrow{r}{\bar{t}_{i}} & X_{S}^{\Phi(i)},
    \end{tikzcd}
  \]
  where $\pi_{i}$ and $\pi'_{i}$ are projections and $\txt{ev}_{i}$ is
  the evaluation map. By Proposition~\ref{propn:mapscotgt} and base
  change we then have a natural equivalence
  \[
    \begin{split}
    \widehat{X_{S}^{\Phi}}(i) &\simeq
    \bar{t}_{i}^{*}\mathbb{T}_{X^{\Phi(i)}/S}
    \simeq \bar{t}_{i}^{*}\pi_{i,*}\txt{ev}_{i}^{*}\mathbb{T}_{X/S}
    \simeq \pi'_{i,*}(\bar{t}_{i} \times_{S} \id)^{*} \txt{ev}_{i}^{*}
    \mathbb{T}_{X/S} \\
    & \simeq \pi'_{i,*} (\id \times_{S} t_{i})^{*}\txt{ev}^{*}
    \mathbb{T}_{X/S}
    \simeq \Gamma_{X^{T}_{S}}(\id \times_{S} t_{i})^{*}\txt{ev}^{*}\mathbb{T}_{X/S}.      
    \end{split}
\]
  Moreover, the morphism $\widehat{X^{\Phi}_{S}}(i) \to
  \widehat{X^{\Phi}_{S}}(i^{\vee})$ corresponds under this equivalence to
  the composite
  \[
    \begin{split}
    \Gamma_{X^{T}_{S}}(\id \times_{S} t_{i})^{*}\txt{ev}^{*}\mathbb{T}_{X/S}
    & \isoto \Gamma_{X^{T}_{S}}(\id \times_{S}
    t_{i})^{*}\txt{ev}^{*}\mathbb{L}_{X/S}[s] \\
    & \to (\Gamma_{X^{T}_{S}}(\id \times_{S}
    t_{i})^{*}\txt{ev}^{*}\mathbb{T}_{X/S})^{\vee}[s-d],      
    \end{split}
\]
  where the first morphism is the equivalence obtained from the
  $s$-symplectic form on $X$. We thus see that the diagram
  $\widehat{X_{S}^{\Phi}}$ is naturally equivalent to the diagram
  $\widehat{X_{S}^{T} \times_{S} \Phi}_{\txt{ev}^{*}\mathbb{T}_{X/S}}$ obtained by applying
  Corollary~\ref{cor:cospanEdiag} to the $n$-uple cospan $X_{S}^{T} \times_{S}
  \Phi$ and the quasicoherent sheaf $\txt{ev}^{*}\mathbb{T}_{X/S}$ on
  $X_{S}^{T} \times_{S} T$. This diagram is non-degenerate since by assumption
  $X_{S}^{T} \times_{S} \Phi$ is weakly oriented and
  $\txt{ev}^{*}\mathbb{T}_{X/S}$ is dualizable.
\end{proof}

\begin{corollary}\label{cor:AKSZOrLag}
  Let $X$ be an $s$-symplectic $S$-stack. Then the symmetric monoidal
  functor
  \[ \txt{AKSZ}_{n,X}^{S,d} \colon
    \Cospan_{n}(\POr_{S,d}^{X\txt{-good}})
    \longrightarrow\Span_{n}(\PSymp_{S,s-d}) \]
  from Corollary~\ref{cor:AKSZsymmon} restricts to a symmetric
  monoidal functor
  \[ \Or_{n}^{\txt{cpt},X\txt{-good},S,d} \longrightarrow \Lag_{n}^{S,s-d}\]
  of $(\infty,n)$-categories.\qed
\end{corollary}

\begin{remark}
  \enlargethispage{2ex}
  The Betti stack functor $\Or_{n}^{\mathcal{S},d} \to \Or_{n}^{S,d}$
from Corollary~\ref{cor:BettiOrn}  factors through
$\Or_{n}^{\txt{cpt},X\txt{-good},S,d}$ for any $X$, so we have a
composite symmetric monoidal functor
\[ \Or_{n}^{\mathcal{S},d} \longrightarrow \Lag_{n}^{S,s-d}\]
for any $s$-symplectic $S$-stack $X$.
\end{remark}

\chapter{From Cobordisms to Cospans}\label{sec:cobcospan}
In this chapter we will construct a symmetric monoidal functor
\[ \Bord_{d,d+n} \to \Cospan_{n}(\mathcal{S}) \] from an
$(\infty,n)$-category of bordisms to $n$-fold cospans in the \icat{}
$\mathcal{S}$ of spaces. This extracts from a $k$-bordism with corners
a $k$-fold cospan that encodes the homotopy types of the bordism and
its sources and targets in all $k$ directions.

In \S\ref{sec bordism background} we recall the definition of the
$(\infty,n)$-categories $\Bord_{d,d+n}$ from \cite{CS}. We then sketch
the construction of the functor in \S\ref{subsec:cutidea} before we
actually construct it in \S\ref{subsec:cut}. In
\S\ref{subsec:TFTspan} we apply our symmetric monoidal functor to
define unoriented TFTs valued in $\Span_{n}(\mathcal{C})$; this
allows us to prove that the induced action of the orthogonal group
$O(n)$ on the space $\mathcal{C}^{\simeq}$ of fully dualizable objects
in $\Span_{n}(\mathcal{C})$ is trivial for any \icat{} $\mathcal{C}$
with finite limits.

\section{Higher Categories of Bordisms}\label{sec bordism background}
In this section we review the definition of higher categories of
bordisms from \cite{CS}. We will consider both the unoriented and
oriented versions --- the latter will be used in \S\ref{sec:orcob}.

Recall that an $m$-dimensional \emph{bordism} is defined to be an $m$-dimensional
manifold with boundary together with a decomposition of the boundary
into a disjoint union of two closed $(m-1)$-dimensional manifolds,
which we call the incoming and outgoing boundaries.  If
the boundaries themselves have boundaries, we have a bordism of
bordisms, a {\em 2-bordism}. Iterating this procedure, an
$m$-dimensional $k$-bordism is intuitively an $m$-dimensional manifold
with corners, for which we have $k$ ``directions'' in which we have a
decomposition of the boundary into an incoming and an outgoing part
(see \cite{Laures, MadsenBoekstedt} for a precise definition). By
appropriate versions of Whitney's embedding theorem for manifolds with
corners, we can assume that (compositions of) $k$-bordisms come
together with an embedding into $\R^\infty$, in which $k$ directions
are singled out to record the data of where the incoming and outgoing
boundaries of the (composed) bordisms are. 

Varying $k$, these structures can be organized into an
$(\infty,n)$-category: roughly speaking, objects are closed
$d$-dimensional manifolds and $k$-morphisms for $1\leq k\leq n$ are
given by $(d+k)$-dimensional $k$-bordisms. 

\begin{warning}
  So far in this paper we have worked with $(\infty,n)$-categories in
  the setting of $n$-fold Segal objects in the \icat{} $\mathcal{S}$
  of spaces. To describe the $(\infty,n)$-category $\Bord_{d,d+n}$ we
  will instead define an $n$-fold Segal object in the model category
  $\sSet$ of simplicial sets. This induces an $n$-fold Segal object in
  $\mathcal{S}$ by inverting the weak equivalences between simplicial
  sets. To distinguish these two settings, we will denote the $n$-fold
  Segal object in $\sSet$ by $\sBord_{d,d+n}$.
\end{warning}

The construction of $\sBord_{d,d+n}$ is based on an auxiliary \icat{}
of intervals which records the data of collars along which a
composition of $k$-bordisms can be decomposed. We will now briefly
recall this --- for more details, see
\cite{CS}*{\S 4}.
\begin{definition}
  Let $T(\Int)$ be the category internal to simplicial sets defined as
  follows:
  \begin{itemize}
  \item The objects (\ie{} 0-simplices in the simplicial set of
    objects) are pairs $(a,b)$ of real numbers satisfying $a<b$. We
    think of such a pair as an open interval in $\R$.
  \item The simplicial set of objects is given by taking smooth
    simplices of the submanifold of $\R^2$ consisting of pairs $(a,b)$
    where $a<b$.
  \item There is exactly one morphism (\ie{} 0-simplex in the
    simplicial set of morphisms) from an open interval
    $I_0=(a_0, b_0)$ to an open interval $I_1=(a_1, b_1)$ if
    $a_0\leq a_1$ and $b_0\leq b_1$, and no morphisms otherwise. One
    should think of this morphism as a pair of ordered intervals which
    are closed in $(a_0, b_1)$, and we indicate the resulting partial
    order on intervals by writing $I_0 \leq I_1$.
  \item The simplicial set of morphisms is again obtained by taking
    smooth simplices, now of the submanifold in $\R^4$ given by
    elements that satisfy the above inequalities.
  \end{itemize}
  Taking the nerve of $T(\Int)$ we obtain a complete Segal space
  $\Int$. The $l$-simplices in the $k$th space $\Int_k$ are smooth
  families of ordered $(k+1)$-tuples of intervals
  $I_0(t) \leq \cdots \leq I_k(t)$ varying over points $t$ in the
  extended $l$-simplex $|\Delta^l|_e$. 
  Furthermore, if $I_{j} = (a_{j}, b_{j})$ then
  we denote by $B(I_0\leq \cdots \leq I_k) := (a_0, b_k)$ the (family
  of) open interval(s) which is the convex hull of the union of the
  intervals.
\end{definition}

\begin{notation}
  Let $\Int^{n} \colon \Dnop \to \sSet$ be the $n$-uple Segal
  simplicial set given by \[\Int^{n}_{\ind{j}}:= \Int_{j_{1}} \times
    \cdots \times \Int_{j_{n}}\] where $\ind{j} =
  (j_{1},\ldots,j_{n})$. An $l$-simplex of $\Int^{n}_{\ind{j}}$ is
  then given by an $l$-simplex  $I^i_0\leq\cdots\leq I^i_{j_i}$ in
  $\Int_{j_{i}}$ for $i = 1,\ldots,n$. We will often denote this data
  simply as 
  \[\oul{I}=(I^i_0\leq\cdots\leq
    I^i_{j_i})_{1\leq i \leq n}.\]
  If $I^{i}_{j} = (a^{i}_{j}, b^{i}_{j})$, then we denote by 
    \[B(\oul{I}) := B(I^{1}) \times \cdots \times B(I^{n}) = (a^1_0, b^1_{j_1})\times\cdots \times (a^{n}_0,
      b^{n}_{j_{n}})\]
    the $|\Delta^{l}|_{e}$-indexed family of products of open
    intervals that gives the ``box'' surrounding the intervals
    $\oul{I}$. We regard $B(\oul{I})$ as a subspace of $\RR^{n} \times
    |\Delta^{l}|_{e}$.
\end{notation}

We can now summarize the definition of $\sBord_{d,d+n}$ from
\cite{CS}. This proceeds in two steps: we first consider a version where the bordisms are embedded
in a fixed finite-dimensional vector space, and then take a colimit of
these.
\begin{definition}\label{defn:BORDV}
  Let $V$ be a finite-dimensional $\mathbb{R}$-vector space, and let
  $d, n,l$ be integers with $n,l\geq0$ (so $d$ is allowed to be
  negative). For $\ind{j} = (j_{1},\ldots,j_{n}) \in \Dnop$, the
  simplical set $\sBORD_{d,d+n, \ind j}^{V}$ has as its set
  $\sBORD_{d,d+n,\ind{j},l}^{V}$ of $l$-simplices the collection of
  pairs $(M, \oul{I})$ where
  \begin{itemize}
  \item $\oul{I}$ is an $l$-simplex in $\Int^{n}$,
  \item $M$ is a closed and bounded $(d+n+l)$-dimensional submanifold of
    $V\times B(\oul{I})$,
  \end{itemize}
  such that:
  \begin{enumerate}[(1)]
  \item The composition
    $\pi: M \hookrightarrow V\times B(\oul{I})\twoheadrightarrow
    B(\oul{I})\subset \mathbb{R}^{n}\times |\Delta^l|_e$ is a proper
    map.
  \item\label{defn:Bord_cond_simplex} The composition $\pi'$ of $\pi$ with the projection onto
    $|\Delta^l|_e$ is a submersion $M\xto{\pi'} |\Delta^l|_e$ which is
    trivial outside the closed set $|\Delta^l|\subset |\Delta^l|_e$,
    and is a family of oriented smooth manifolds, as in \cite{GRW}.
  \item\label{defn:Bord_cond_grid} For every $S\subseteq\{1,\ldots, n\}$, let
    $p_S \colon M\xrightarrow{\pi}B(\oul{I})\xrightarrow{\pi_S}\R^S\times
    |\Delta^l|_e$ be the composition of $\pi$ with the projection
    $\pi_S$ onto the $S$-coordinates. Then for every collection
    $\{k_i\}_{i\in S}$, where $0\leq k_i \leq j_i$, the map $p_S$ does
    not have critical values in $(I_{k_i})_{i\in S}$.
  \end{enumerate}
  These conditions imply that the fiber $M_t$ of $\pi$ at
  $t \in |\Delta^l|_e$ is a $0$-simplex in $\sBORD^{V}_{d, d+n,\ind{j}}$;
  denote the projection at $t$ by
  $\pi_t \colon M_t \to B(\oul{I}(t) )$. The cosimplicial structure of
  the extended simplices $|\Delta^\bullet|_e$ can be used to define a
  simplicial set $\sBORD^{V}_{d,d+n, \ind j}$. Moreover, by varying over
  $\ind{j}\in \simp^{n,\op}$ using the $n$-fold simplicial structure
  of $\Int^{n}$ we can make $\sBORD^{V}_{d,d+n}$ an $n$-fold simplicial
  object in $\sSet$ --- the $n$-fold simplicial structure arises from
  forgetting or
  repeating intervals and possibly cutting off part of the submanifold
  $M$ by taking the preimage under $\pi$ of the new box around the
  intervals.
\end{definition}

\begin{theorem}[\cite{CS}]
  $\sBORD_{d,d+n}^{V}$ is an $n$-uple Segal space.\qed
\end{theorem}

\begin{definition}
  Taking the (filtered) colimit over all finite-dimensional vector spaces $V$ in a
  given infinite dimensional vector space, say $\R^\infty$, we define
  \[\sBORD_{d,d+n} := \colim_{V\subset\R^\infty} \sBORD_{d,d+n}^{V}.\]
  Then $\sBORD_{d,d+n}$ is an $n$-uple Segal space (since these are
  closed under filtered homotopy colimits).
\end{definition}

\begin{remark}
  Our notational convention in this article is slightly different from
  that in \cite{CS}: $d$ denotes the bottom dimension, $n$ the
  category number, and hence $d+n$ the top dimension of manifolds that
  appear.
\end{remark}

\begin{remark}\label{rmk:trivialfib}
  For $\ind{i} = \ind{0} = (0,\ldots,0)$, an $l$-simplex of
  $\sBORD_{d,d+n,\ind{0}}$ is either $\emptyset$ or a proper surjective submersion $\pi'\colon M \to
  |\Delta^{l}|_{e}$, which is a trivial fibration by
  Ehresmann's fibration theorem.
\end{remark}

Following \cite{GMTW} and \cite{GRW}, we can add tangential
structures. Here we will only consider orientations:
\begin{variant}
  The $n$-uple Segal space $\sBORDor_{d,d+n}$ of oriented bordisms is
  defined in exactly the same way, except that an $l$-simplex of
  $\sBORD^{\txt{or},V}_{d,d+n,\ind{j}}$ is now a triple $(M, \oul{I},
  f)$, where the additional item $f$ is a fibrewise orientation of the fibration
  $\pi' \colon M \to |\Delta^{l}|_{e}$.
\end{variant}

\begin{remark}\label{rem:fiberwiseorientation}
  Let us explain what a fibrewise orientation is in this context. Let
  $\theta:BSO(d+n)\to BO(d+n)$ be a fibration corresponding to the
  inclusion of the subgroup, and $\gamma$ the canonical vector bundle
  on $BO(n+d)$. Then a fiberwise orientation of $\pi' \colon M \to
  |\Delta^{l}|_{e}$ is a morphism $f \colon M \to BSO(d+n)$ and a bundle
  map \[\ker(d\pi'\colon TM\longrightarrow
    \pi'^*T|\Delta^l|_e) \xrightarrow{\overline{f}} \theta^*\gamma,\]
  over $f$, by which we mean a commutative square
  \[
    \begin{tikzcd}
      \ker d\pi' \arrow{r}{\overline{f}} \arrow{d} & \theta^{*}\gamma \arrow{d} \\
      M \arrow{r}{f} & BSO(d+n)
    \end{tikzcd}
  \]
  such that $\overline{f}$ restricts to  an isomorphism on fibres
  \[ (\ker d\pi')_{m}\isoto (\theta^{*}\gamma)_{f(m)} \]
  for each $m \in M$, and so exhibits $\ker d\pi'$ as the pullback
  $f^{*}\theta^{*}\gamma$. Homotopically, this data is equivalent to
  a choice of orientation on the bundle $\ker d \pi'$, or a lift of
  the map $M \to BO(d+n)$ classifying this bundle to $BSO(d+n)$.
  Note that since
  \[TM \cong \ker d\pi' \oplus \pi'^*T|\Delta^l|_e,\] if we
  choose an orientation on $|\Delta^l|_e$ then the fiberwise
  orientation induces an orientation on $M$ itself.
\end{remark}

\begin{definition}
  We write $\sBord_{d,d+n}^{\por}$ for the underlying $n$-fold Segal
  simplicial set of $\sBORD_{d,d+n}^{\por}$. More explicitly, this can
  be obtained by replacing condition (3) in
  Definition~\ref{defn:BORDV} by the stronger version
  \begin{enumerate}
  \item[($3'$)] For every $S\subseteq\{1,\ldots, n\}$, let
    $p_S:M\xrightarrow{\pi}B(\oul{I})\xrightarrow{\pi_S}\R^S\times
    |\Delta^l|_e$ be the composition of $\pi$ with the projection
    $\pi_S$ onto the $S$-coordinates. Then for every $1\leq i\leq n$
    and $0\leq k\leq j_i$, at every $x\in p_{\{i\}}^{-1}(I_k)$, the
    map $p_{\{i,\ldots, n\}}$ is submersive.
\end{enumerate}
\end{definition}

\begin{remark}
  For any $d$ and $n\geq 0$, there is a natural map
  \[ \sBord_{d-1, d+n}^{\por}(\emptyset, \emptyset) \to \sBord_{d,
      d+n}^{\por}\]
given by forgetting the data of intervals in the $(n+1)$th direction, and this is a weak
  equivalence. Thus the ``tower'' $\sBord_{d-r, d+n}^{\por}$ for
  $r\geq 0$ endows $\Bord_{d,d+n}^{\por}$ with a symmetric monoidal
  structure.
\end{remark}

\begin{remark}
  The $n$-fold Segal space $\Bord_{d,d+n}^{\por}$ is generally not
  complete; we write $\widehat{\Bord}_{d,d+n}^{\por}$ for its
  completion. This is a symmetric monoidal $(\infty,n)$-category,
  since completion preserves products (see \cite{spans}*{Lemma
    7.10}).
\end{remark}

\begin{definition}
  An {\em $n$-extended (oriented) $(d+n)$-dimensional topological
    field theory} is a symmetric monoidal functor out of the bordism
  $(\infty,n)$-category $\widehat{\Bord}_{d,d+n}^{\por}$. For $d=0$ we
  refer to this simply as an \emph{extended $n$-dimensional (oriented)
    topological field theory}.
\end{definition}

\begin{remark}
  By the universal property of completion, to define an $n$-extended
  (oriented) $(d+n)$-dimensional topological field theory it suffices
  to define a symmetric monoidal functor of $n$-fold Segal spaces out
  of $\Bord_{d,d+n}^{\por}$. Since our goal in this paper is to
  construct such functors, we use the shorter notation for the
  uncompleted version, differing from the convention chosen in
  \cite{CS}.
\end{remark}

\section{Cutting Bordisms: Idea}\label{subsec:cutidea}
In this section we will sketch our construction of a symmetric
monoidal functor of $(\infty,n)$-categories
$\Bord_{d,d+n} \to \Cospan_{n}(\mathcal{S})$. The basic idea is
simple: To a closed $d$-manifold $M$ we assign its homotopy type,
and to a $(d+1)$-dimensional cobordism $X$ from $M$ to $N$ we assign
the cospan
\[ M \hookrightarrow X \hookleftarrow N\] of boundary
inclusions. Composing cobordisms by gluing gives in particular a
homotopy pushout of topological spaces, and so corresponds to a
composition of
cospans. Making this into a functor out of our particular model of the
bordism $(\infty,n)$-category takes some work, however.

Since we have a strict model for the $n$-uple Segal space
$\BORD_{d,d+n}$, we want to construct a functor to a strict model for
$\COSPAN_{n}(\mathcal{S})$. As we discuss in \S\ref{subsec:spanmodel},
we can describe the latter using the strict $n$-uple simplicial space
$\simp^{n,\op} \to \sSet$ taking $\ind{j}$ to
$\mathrm{N}(\Fun(\bbS^{\ind{j},\op}, \Top)^{W})$, \ie{} the nerve of the
subcategory of $\Fun(\bbS^{\ind{j},\op}, \Top)$ containing only the
natural weak equivalences. We might thus hope to construct our functor
by defining a natural family of functors
\[ \sBORD_{d,d+n,\ind{j},l} \times \Delta^{l} \times
  \bbS^{\ind{j},\op} \to \Top,\]
and then checking that the morphisms in the
$\Delta^{l}$-direction are weak equivalences and the functors in the
$\bbS^{\ind{j},\op}$-direction are homotopy left Kan
extensions from $\bbL^{\ind{j},\op}$. Unfortunately there are several
problems with this idea, which force us to work with slightly more
complicated diagram shapes, as we will now explain:

\begin{problem}[Composing in the spatial direction]
From a 1-simplex in $\sBORD_{d,d+n,\ind{0}}$ we can extract a
submersion $M \to |\Delta^{1}|_{e}$, from which we can obtain
a cospan
\[ M_{0} \hookrightarrow M_{0,1} \hookleftarrow M_{1}\] by taking
preimages over the faces of $|\Delta^{1}| \subseteq
|\Delta^{1}|_{e}$. By Remark~\ref{rmk:trivialfib} the map
$M_{0,1} \to |\Delta^{1}|$ is a trivial fibration, so both maps in
this cospan are homotopy equivalences. However, there is no canonical
way to extract from this a map from $M_0$ to $M_{1}$, as this requires
choosing a homotopy inverse (or a trivialization of the fibration,
which can be done using (a variant of) the Morse lemma).

Similarly, from a 2-simplex we
get a submersion $M \to |\Delta^{2}|_{e}$, from which taking
pre-images of faces gives a diagram
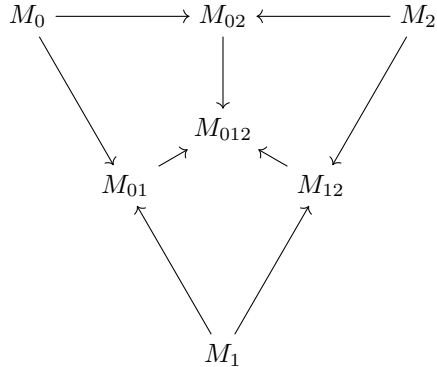
\[
\begin{tikzpicture}[scale=1.5]
\draw (0,0) node (012) {$M_{012}$}  (0,-2) node (1) {$M_{1}$};
\path (0,-2) arc (-90:30:2) node (2) {$M_{2}$};
\path (0,-2) arc (270:150:2) node
(0) {$M_{0}$}; 
\path (0) -- node (01) {$M_{01}$} (1) -- node (12) {$M_{12}$} (2) -- node (02) {$M_{02}$} (0);
\draw[->] (0) -- (01);
\draw[->] (1) -- (01);
\draw[->] (0) -- (02);
\draw[->] (2) -- (02);
\draw[->] (2) -- (12);
\draw[->] (1) -- (12);
\draw[->] (01) -- (012);
\draw[->] (02) -- (012);
\draw[->] (12) -- (012);
\end{tikzpicture}
\]
where all the morphisms are homotopy equivalences. In general, from an
$l$-simplex of $\sBORD_{d,d+n,\ind{0}}$ we obtain a functor
$\txt{sd}(\Delta^{l}) \to \Top$ that takes all morphisms to weak
equivalences, where $\txt{sd}(\Delta^{l}):= \simp_{\txt{inj}/[l]}$ is
the partially ordered set of faces of $\Delta^{l}$. This suggests that
we should define compatible functors
\[ \sBORD_{d,d+n,\ind{j},l} \times \txt{sd}(\Delta^{l}) \times
  \bbS^{\ind{j},\op} \to \Top,\]
which will induce morphisms of simplicial sets \[\sBORD_{d,d+n,\ind{j}} \to
\txt{example}^{1}\mathrm{N}\Fun(\bbS^{\ind{j},\op}, \Top)^{W},\] where
$\txt{example}^{1}$ denotes the adjoint to subdivision. As is well-known,
the natural map $K \to \Ex^{1}K$ is a weak equivalence for any
simplicial set $K$, so this determines the data we want on the
$\infty$-category level.  
\end{problem}

\begin{problem}[Degeneracies in the spatial direction]
Unfortunately, the morphisms we considered above are not quite natural
for maps in $\simp$. In particular, we don't get a map of simplicial
sets $\sBORD_{d,d+n,\ind{0}} \to \Ex^{1}\Nrv\Top$ because our proposed
map is not compatible with degeneracies: if $M$ is a 0-simplex of
$\sBORD_{d,d+n,\ind{0}}$, then the degenerate 1-simplex on $M$ is the
trivial fibration $M \times |\Delta^{1}|_{e} \to |\Delta^{1}|_{e}$,
which gives a map $\sd(\Delta^{1}) \to \Top$ we can depict as
\[ M \hookrightarrow M \times
  |\Delta^{1}|_{e} \hookleftarrow M.\]
On the other hand, the degeneracy on $M$ in $\Ex^{1}\Nrv\Top$ 
is the constant diagram
\[ M \xlongequal{\phantom{x}} M \xlongequal{\phantom{x}} M.\]
We can circumvent this issue by working with semi-simplicial sets
(\ie{} presheaves on $\Dinj$)
instead of simplicial sets: if $i$ denotes the inclusion $\Dinj \hookrightarrow \simp$, we \emph{do} have a map of
semisimplicial sets
\[ \sBORD_{d,d+n,\ind{j}}^{\semi} \to (\Ex^{1}\Nrv\Fun(\bbS^{\ind{j},\op},
  \Top))^{\semi}. \]
This is enough since semi-simplicial sets also present the \icat{} of
spaces, as we discuss in \S\ref{subsec:spanmodel}.
\end{problem}

\begin{remark}
	In the cobordism direction, this problem does not appear: the degeneracy maps in $\sBORD_{d,d+n}$ simply double the intervals indicating where to cut, so the underlying manifold stays the same. In effect, the degeneracy on an object is given by an identity cospan on a width-zero cobordism.
\end{remark}

In conclusion, we have the following procedure that we will follow to
define a symmetric monoidal functor $\BORD_{d,d+n} \to
\COSPAN_{n}(\mathcal{S})$:
\begin{proc}\label{proc:recipe_for_functor}\mbox{}
\begin{enumerate}[(1)]
\item Define functors $\sBORD_{d,d+n,\ind{j},l}\times \sd(\Delta^{l})
  \times \bbS^{\ind{j},\op} \to \Top$, strictly natural in $l \in
  \Dinj^{\op}$, $\ind{j} \in \simp^{n,\op}$. In fact, they will take values in a strict category $\Mfd$ of manifolds with corners  and smooth maps, and we later take the composite with the forgetful map $\Mfd\to\Top$.
\item Check that morphisms in $\sd(\Delta^{l})$ are taken to weak equivalences
  in $\Top$.
\item Check that the functors from $\bbS^{\ind{j},\op}$ are homotopy
  left Kan extended from $\bbL^{\ind{j},\op}$.
\item Check that the functors are compatible with delooping.
\end{enumerate}
\end{proc}

\section{Cutting Bordisms: Details}\label{subsec:cut} 
In this section we carry out the details of the construction we
sketched in the previous section. We will thus define a functor
that ``cuts'' bordisms into cospans of spaces
 \[\cut \colon \BORD_{d, d+n} \longrightarrow \COSPAN_n(\mathcal{S})\]
 following \ref{proc:recipe_for_functor}, by finding a
compatible family of functors
\[ \scut_{\ind{j},l} \colon \sBORD_{d,d+n,\ind{j},l} \times
  \bbS^{\ind{j},\op} \times \sd(\Delta^{l}) \to \Mfd \] 
 strictly natural in $l \in
  \Dinj^{\op}$, $\ind{j} \in \simp^{n,\op}$. Here $\Mfd$ is the strict category of manifolds with corners and smooth maps.

  By construction, the manifolds in $\sBORD_{d,d+n,\ind{j},l}$ all come
  equipped with collars. To define the functor to cospans we can
  either keep the collars or remove them. Keeping the collars has the
  conceptual advantage that the diagrams that appear in compositions
  are actually strict pushouts of manifolds with corners.  However,
  this causes problems when we want to add orientations in the next
  section, as the manifolds don't have the right dimension. We
  therefore need the slightly more complex construction where we
  remove the collars; we still get pushouts of topological spaces, and
  these are all homotopy pushouts (and so pushouts in the \icat{} of
  spaces). Since a manifold with or without its
  collar both represent the same homotopy type, we can choose the one more
  convenient for us.  We will construct both versions of the cutting
  functor, implementing the two variants outlined above: In the first,
  a composed bordism is decomposed by cutting it along one line into
  two pieces, so that these pieces do not overlap. In the second
  version, we cut the bordism along two parallel lines, so that the
  pieces overlap on a little collar.

  Before we delve into what the cutting procedure does to bordisms, we
  must first understand how to extract the relevant grid data. Namely,
  we first define two families of maps
\[\grid_{\ind{j},l}, \gridT_{\ind{j},l} \colon  \Int^n_{\ind{j},l} \times \bbS^{\ind{j},\op}
  \times \sd(\Delta^{l}) \to \Mfd,\]
 together with a natural transformation to 
the composite
\[ \Int^n_{\ind{j},l} \times \bbS^{\ind{j},\op} \times \sd(\Delta^{l})
  \to \Int^{n}_{\ind{j},l} \xto{B} \Mfd.\] For $[l] \in \Dinj^{\op}$,
$\ind{j} \in \simp^{n,\op}$, recall that an element
$(M, \oul{I}) \in \sBORD_{d,d+n,\ind{j},l}$ includes the data of a
morphism $ \pi \colon M \to B(\oul{I})$ in $\Mfd$, where the ``box''
$B(\oul{I})$ is a subspace of $\RR^{n} \times |\Delta^{l}|_{e}$. We
will define the space $\scut_{\ind{j},l}(M,\oul{I},\xi,\tau)$ for
$\xi \in \bbS^{\ind{j},\op}, \tau \in \sd(\Delta^{l})$ as the preimage
$\pi^{-1}\grid_{\ind{j},l}(\oul{I}, \xi, \tau)$ in $M$, and similarly
for the second variant.
 
The following diagrams illustrate the difference between the functors
$\grid$ and $\gridT$ for a composition: 
$$
\begin{tikzpicture}[scale=0.9]

\draw (-5,0) -- (0,0);

\draw (0, -0.25) node [anchor=north] {\tiny $b_2$};
\draw (0,0) arc (0:30:0.5);
\draw (0,0) arc (0:-30:0.5);

\draw (-5, -0.25) node [anchor=north] {\tiny $a_0$};
\draw (-5,0) arc (0:30:-0.5);
\draw (-5,0) arc (0:-30:-0.5);

\draw (-4.1, 0.25) -- (-4, 0.25) -- (-4, -0.25) node[anchor=north] {\tiny $b_0$} -- (-4.1, -0.25);

\draw (-2.9, 0.25) -- (-3, 0.25) -- (-3, -0.25) node[anchor=north] {\tiny $a_1$} -- (-2.9, -0.25);

\draw (-2.1, 0.25) -- (-2, 0.25) -- (-2, -0.25) node[anchor=north] {\tiny $b_1$} -- (-2.1, -0.25);

\draw (-0.9, 0.25) -- (-1, 0.25) -- (-1, -0.25) node[anchor=north] {\tiny $a_2$} -- (-0.9, -0.25);

\fill[blue] (-4.7, 0.03) node[inner sep=0] (A0) {} circle (0.15em);
\fill[blue] (-2.3, 0.03) node[inner sep=0] (B0) {} circle (0.15em);
\draw[blue, thick] (A0)  -- node[anchor=south] {\tiny }  (B0) ;

\fill[red] (-2.7, -0.03) node[inner sep=0] (A1) {} circle (0.15em);
\fill[red] (-0.3, -0.03) node[inner sep=0] (B1) {} circle (0.15em);
\draw[red, thick] (A1) -- node[anchor=south] {\tiny }  (B1);

\draw (1,0) node {$\overset{\gridT}{\longmapsto}$};

\begin{scope}[shift={(7.5,0)}]

\begin{scope}[yshift=-1cm]
\fill[violet] (-2.3, 0) node[anchor=north] {} circle (0.15em);
\fill[violet] (-2.7, 0) node[anchor=north] {} circle (0.15em);
\draw[violet, thick] (-2.3, 0) -- node (M) {} node[anchor=north] {\tiny } (-2.7, 0);
\draw (-2.7, 0) node[inner sep=0] (L) {};
\draw (-2.3, 0) node[inner sep=0] (R) {};
\end{scope}

\begin{scope}[shift={(-3,-1)}]
\fill[blue] (-2.3, 0) node[anchor=north] {} circle (0.15em);
\fill[blue] (-2.7, 0) node[anchor=north] {} circle (0.15em);
\draw[blue, thick] (-2.3, 0) -- node (M0) {} node[anchor=north] {\tiny } (-2.7, 0);
\draw (-2.7, 0) node[inner sep=0] (L0) {};
\draw (-2.3, 0) node[inner sep=0] (R0) {};
\end{scope}

\begin{scope}[shift={(3,-1)}]
\fill[red] (-2.3, 0) node[anchor=north] {} circle (0.15em);
\fill[red] (-2.7, 0) node[anchor=north] {} circle (0.15em);
\draw[red, thick] (-2.3, 0) -- node (M2) {} node[anchor=north] {\tiny } (-2.7, 0);
\draw (-2.7, 0) node[inner sep=0] (L2) {};
\draw (-2.3, 0) node[inner sep=0] (R2) {};
\end{scope}

\begin{scope}[xshift=-0.5cm]
\fill[blue] (-4.7, 0) node[anchor=north] {} circle (0.15em);
\draw[blue, thick] (-4.7,0) -- node (blueM) {} node[anchor=north] {\tiny }  (-2.3,0);
\fill[blue] (-2.3, 0) node[anchor=north] {} circle (0.15em);
\draw (-2.7, 0) node[inner sep=0] (BL) {};
\draw (-2.3, 0) node[inner sep=0] (BR) {};
\draw (-4.7, 0) node[inner sep=0] (Bend) {};
\path (-4.3, 0) -- node (BM) {} (-4.7, 0);
\path (-2.3, 0) -- node (BM2) {} (-2.7, 0);
\end{scope}

\begin{scope}[xshift=0.5cm]
\fill[red] (-2.7, 0) node[anchor=north] {} circle (0.15em);
\draw[red, thick] (-2.7,0) -- node (redM) {} node[anchor=north] {\tiny }  (-0.3,0);
\fill[red] (-0.3, 0) node[anchor=north] {} circle (0.15em);
\draw (-2.7, 0) node[inner sep=0] (RL) {};
\draw (-2.3, 0) node[inner sep=0] (RR) {};
\draw (-0.3, 0) node[inner sep=0] (Rend) {};
\path (-0.3, 0) -- node (RM) {} (-0.7, 0);
\path (-2.3, 0) -- node (RM2) {} (-2.7, 0);
\end{scope}

\begin{scope}[yshift= 1cm]
\draw[blue, dotted] (-2.3, 0) -- (BR) -- (R);
\draw[blue, dotted] (-2.7, 0) -- (BL) -- (L);
\draw[blue, dotted] (-4.7, 0) -- (Bend) -- (L0);
\draw[blue, dotted] (-4.3, 0) -- (-4.8, -1) -- (R0);

\draw[red, dotted] (-2.3, 0) -- (RR) -- (R);
\draw[red, dotted] (-2.7, 0) -- (RL) -- (L);
\draw[red, dotted] (-0.3, 0) -- (Rend) -- (R2);
\draw[red, dotted] (-0.7, 0) -- (-0.2, -1) -- (L2);

\draw[very thick] (-4.7,0) -- node (topM) {} node[anchor=south] {\tiny }  (-0.3,0);
\fill[] (-4.7, 0) node[anchor=north] {} circle (0.15em);
\fill[] (-0.3, 0) node[anchor=north] {} circle (0.15em);
\path (-4.3, 0) -- node (topL) {} (-4.7, 0);
\path (-0.3, 0) -- node (topR) {} (-0.7, 0);
\path (-4.7,0) -- node (topM1) {}  (-2.3,0);
\path (-0.7,0) -- node (topM2) {}  (-2.3,0);
\end{scope}

\draw[right hook->] (M0) -- (BM);
\draw[right hook->] (M.north east) -- (RM2);
\draw[left hook->] (M.north west) -- (BM2);
\draw[left hook->] (M2) -- (RM);
\draw[right hook->]  (blueM) -- (topM1);
\draw[left hook->]  (redM) -- (topM2);
\end{scope}
\end{tikzpicture}
$$
$$
\begin{tikzpicture}[scale=0.9]

\draw (-5,0) -- (0,0);

\draw (0, -0.25) node [anchor=north] {\tiny $b_2$};
\draw (0,0) arc (0:30:0.5);
\draw (0,0) arc (0:-30:0.5);

\draw (-5, -0.25) node [anchor=north] {\tiny $a_0$};
\draw (-5,0) arc (0:30:-0.5);
\draw (-5,0) arc (0:-30:-0.5);

\draw (-4.1, 0.25) -- (-4, 0.25) -- (-4, -0.25) node[anchor=north] {\tiny $b_0$} -- (-4.1, -0.25);

\draw (-2.9, 0.25) -- (-3, 0.25) -- (-3, -0.25) node[anchor=north] {\tiny $a_1$} -- (-2.9, -0.25);

\draw (-2.1, 0.25) -- (-2, 0.25) -- (-2, -0.25) node[anchor=north] {\tiny $b_1$} -- (-2.1, -0.25);

\draw (-0.9, 0.25) -- (-1, 0.25) -- (-1, -0.25) node[anchor=north] {\tiny $a_2$} -- (-0.9, -0.25);

\fill[blue] (-4.5, 0.03) node[inner sep=0] (A0) {} circle (0.15em);
\fill[blue] (-2.5, 0.03) node[inner sep=0] (B0) {} circle (0.15em);
\draw[blue, thick] (A0) -- node[anchor=south] {\tiny }  (B0);

\fill[red] (-2.5, -0.03) node[inner sep=0] (A1) {} circle (0.15em);
\fill[red] (-0.5, -0.03) node[inner sep=0] (B1) {} circle (0.15em);
\draw[red, thick] (A1) -- node[anchor=south] {\tiny }  (B1);

\draw (1,0) node {$\overset{\grid}{\longmapsto}$};

\begin{scope}[shift={(7.5,0)}]

\begin{scope}[yshift=-1cm]
\fill[violet] (-2.5, 0) node[anchor=north] {} circle (0.15em);
\fill[violet] (-2.5, 0) node[anchor=north] {} circle (0.15em);
\draw[violet, thick] (-2.5, 0) -- node (M) {} node[anchor=north] {\tiny } (-2.5, 0);
\draw (-2.5, 0) node[inner sep=0] (L) {};
\draw (-2.5, 0) node[inner sep=0] (R) {};
\end{scope}

\begin{scope}[shift={(-3,-1)}]
\fill[blue] (-2.5, 0) node[anchor=north] {} circle (0.15em);
\fill[blue] (-2.5, 0) node[anchor=north] {} circle (0.15em);
\draw[blue, thick] (-2.5, 0) -- node (M0) {} node[anchor=north] {\tiny } (-2.5, 0);
\draw (-2.5, 0) node[inner sep=0] (L0) {};
\draw (-2.5, 0) node[inner sep=0] (R0) {};
\end{scope}

\begin{scope}[shift={(3,-1)}]
\fill[red] (-2.5, 0) node[anchor=north] {} circle (0.15em);
\fill[red] (-2.5, 0) node[anchor=north] {} circle (0.15em);
\draw[red, thick] (-2.5, 0) -- node (M2) {} node[anchor=north] {\tiny } (-2.5, 0);
\draw (-2.5, 0) node[inner sep=0] (L2) {};
\draw (-2.5, 0) node[inner sep=0] (R2) {};
\end{scope}

\begin{scope}[xshift=-0.5cm]
\fill[blue] (-4.5, 0) node[anchor=north] {} circle (0.15em);
\draw[blue, thick] (-4.5,0) -- node (blueM) {} node[anchor=north] {\tiny }  (-2.5,0);
\fill[blue] (-2.5, 0) node[anchor=north] {} circle (0.15em);
\draw (-2.5, 0) node[inner sep=0] (BL) {};
\draw (-2.5, 0) node[inner sep=0] (BR) {};
\draw (-4.5, 0) node[inner sep=0] (Bend) {};
\path (-4.5, 0) -- node (BM) {} (-4.5, 0);
\path (-2.5, 0) -- node (BM2) {} (-2.5, 0);
\end{scope}

\begin{scope}[xshift=0.5cm]
\fill[red] (-2.5, 0) node[anchor=north] {} circle (0.15em);
\draw[red, thick] (-2.5,0) -- node (redM) {} node[anchor=north] {\tiny }  (-0.5,0);
\fill[red] (-0.5, 0) node[anchor=north] {} circle (0.15em);
\draw (-2.5, 0) node[inner sep=0] (RL) {};
\draw (-2.5, 0) node[inner sep=0] (RR) {};
\draw (-0.5, 0) node[inner sep=0] (Rend) {};
\path (-0.5, 0) -- node (RM) {} (-0.5, 0);
\path (-2.5, 0) -- node (RM2) {} (-2.5, 0);
\end{scope}

\begin{scope}[yshift= 1cm]
\draw[blue, dotted] (-2.5, 0) -- (BR) -- (R);
\draw[blue, dotted] (-2.5, 0) -- (BL) -- (L);
\draw[blue, dotted] (-4.5, 0) -- (Bend) -- (L0);

\draw[red, dotted] (-2.5, 0) -- (RR) -- (R);
\draw[red, dotted] (-2.5, 0) -- (RL) -- (L);
\draw[red, dotted] (-0.5, 0) -- (Rend) -- (R2);

\draw[very thick] (-4.5,0) -- node (topM) {} node[anchor=south] {\tiny }  (-0.5,0);
\fill[] (-4.5, 0) node[anchor=north] {} circle (0.15em);
\fill[] (-0.5, 0) node[anchor=north] {} circle (0.15em);
\path (-4.5, 0) -- node (topL) {} (-4.5, 0);
\path (-0.5, 0) -- node (topR) {} (-0.5, 0);
\path (-4.5,0) -- node (topM1) {}  (-2.5,0);
\path (-0.5,0) -- node (topM2) {}  (-2.5,0);
\end{scope}

\draw[right hook->] (M0) -- (BM);
\draw[right hook->] (M.north east) -- (RM2);
\draw[left hook->] (M.north west) -- (BM2);
\draw[left hook->] (M2) -- (RM);
\draw[right hook->]  (blueM) -- (topM1);
\draw[left hook->]  (redM) -- (topM2);
\end{scope}
\end{tikzpicture}
$$

\begin{definition}
  Given a sequence of intervals $I = I_{0}\leq \cdots \leq I_{k}$ where
  $I_{j} = (a_{j},b_{j})$, we define two smaller versions of $B(I)$ by
  \[ B'(I) = [\frac{a_{0}+b_{0}}{2}, \frac{a_{k}+b_{k}}{2}],\qquad B''(I) = [\frac{2a_{0}+b_{0}}{3}, \frac{a_{k}+2b_{k}}{3}],   \]
  both being closed intervals in $\RR$.
  If  $(I(t))_{t \in |\Delta^{l}|_{e}}$ is an $l$-simplex of $\Int$,
  we regard $B'(I)$ and $B''(I)$ as families of closed intervals
  parametrized by the usual simplex $|\Delta^l|$ and the extended
  simplex $|\Delta^l|_{e}$, respectively. For $\oul{I} = (I^{1},\ldots,I^{n})$
  an $l$-simplex of $\Int^{n}_{\ind{j}}$ we write
  \[ B'(\oul{I}) := \prod_{i=1}^{n} B'(I^{i}), \qquad B''(\oul{I}) :=
    \prod_{i=1}^{n} B''(I^{i}). \] We regard these as subspaces of
  $|\Delta^{l}| \times \RR^{n}$ and $|\Delta^{l}|_{e} \times \RR^{n}$,
  respectively, where we have
  $B'(\oul{I}) \subseteq B''(\oul{I}) \subseteq B(\oul{I})$.
\end{definition}

\begin{remark}\label{rmk:bbSint}
  Recall from \cite{spans}*{Remark 5.4} that we can view $\bbS^{n}$ as
  the partially ordered set $\simp_{\xint/[n]}^{\op}$ of inert
  maps to $[n]$ in $\simp$ --- the inert map
  $\xi \colon [m] \hookrightarrow [n]$ corresponds to the pair
  $(\xi(0), \xi(m))$. 
\end{remark}

\begin{definition}\label{defn:grid}
We define
\[\grid_{\ind{j},l}, \gridT_{\ind{j},l} \colon  \Int^n_{\ind{j},l} \times \bbS^{\ind{j},\op}
    \times \sd(\Delta^{l}) \longrightarrow \Mfd\]
    as follows. Using the identifications of Remark~\ref{rmk:bbSint}, for
  \begin{itemize}
  \item an object $\oul{I}\in \Int^n_{\ind{j},l}$,
  \item an inert map $\ind{i} \hookrightarrow
  \ind{j}$ in $\simp^{n}$ viewed as an object $\xi\in\bbS^{\ind{j},\op}$, and
  \item an injective map $[k]\hookrightarrow[l]$ in $\simp_{\inj/[l]}$ viewed an an object $\tau\in \sd(\Delta^{l})$,
  \end{itemize}
let
 \[\grid_{\ind{j},l}(\oul{I}, \xi, \tau) :=
    B'(\xi^{*}_{\ext}\tau^{*}_{\xint}\oul{I}) \hookrightarrow
    B(\oul{I}),\]
  \[\gridT_{\ind{j},l}(\oul{I}, \xi, \tau) :=
    B''(\xi^{*}_{\ext}\tau^{*}_{\xint}\oul{I}) \hookrightarrow
    B(\oul{I}),\]
  where we write $\tau^{*}_{\xint}$ for the simplicial structure map in
  the simplicial set $\Int^{n}_{\ind{j}}$ and $\xi^{*}_{\ext}$ for
  the structure map $\Int^{n}_{\ind{j}} \to \Int^{n}_{\ind{i}}$.

  It is clear from the definition that if we
  have maps $\xi \to \xi'$, $\tau \to \tau'$, then
  \[ B'(\xi^{*}_{\ext}\tau^{*}_{\xint}\oul{I}) \subseteq
    B'(\xi'^{*}_{\ext}\tau'^{*}_{\xint}\oul{I}),\]
  which makes $\grid_{\ind{j},l}$ a functor, together with a natural transformation to 
the composite
\[ \Int^n_{\ind{j},l} \times \bbS^{\ind{j},\op}
  \times \sd(\Delta^{l}) \to \Int^{n}_{\ind{j},l} \xto{B} \Mfd\]
   and similarly for
  $\gridT_{\ind{j},l}$. Since these functors are defined in terms of
  the structure maps for $\Int^{n}$, it is also immediate that they
  are functorial in $\ind{j} \in \simp^{n,\op}$ and $[l]\in \Dinj^{\op}$.  Moreover, the inclusions $B'(\oul{I}) \subseteq
  B''(\oul{I})$ induce a compatible family of natural transformations
  \[ \gamma_{\ind{j},l} \colon \grid_{\ind{j},l} \to \gridT_{\ind{j},l}.\]
\end{definition}

We now use the maps $\grid_{\ind{j},l}$ to define ``cutting''
functors:  
\begin{definition}\label{defn:cut}
  Keeping the notation from \ref{defn:grid}, we define functors
  \[ \scut_{\ind{j},l}, \scutT_{\ind{j},l} \colon
    \sBORD_{d,d+n,\ind{j},l}^{(\mathrm{or})} \times \bbS^{\ind{j},\op} \times
    \sd(\Delta^{l}) \longrightarrow \Mfd^{(\mathrm{or})}\]
  on objects as follows: for $(M, \oul{I})$ the manifold $M$ is
  equipped with a map $\pi \colon M \to B(\oul{I})$, and we set
  \[ \scut_{\ind{j},l}(M,\oul{I},\xi,\tau) :=
    \pi^{-1}(\grid_{\ind{j},l}(\oul{I},\xi,\tau)), \qquad
    \scutT_{\ind{j},l}(M,\oul{I},\xi,\tau) :=
    \pi^{-1}(\gridT_{\ind{j},l}(\oul{I},\xi,\tau)).\]
  Generalizing the Morse lemma similarly to \cite[Proposition 8.9]{CS}, these are   manifolds with faces. If we started with an oriented bordism, they are oriented as closed submanifolds of the oriented manifold $M$.

 Since
  $\grid$ and $\gridT$ are functorial and taking preimages is
  functorial, this
  assignment on objects extends to a pair of functors, which are
  moreover natural in $\ind{j}\in\simp^{n,\op}, [l] \in
  \Dinj^{\op}$. In addition, the natural
  transformations $\gamma_{\ind{j},l}$ induce a 
  compatible family of natural transformations
 \[ \Gamma_{\ind{j},l} \colon \scut_{\ind{j},l} \to \scutT_{\ind{j},l}.\]
\end{definition}

\begin{remark}
  Essentially, these maps extract the composable cobordisms  and their
  sources and targets in the composition. For $d=0, n=1$, an example
  of the circle as a composition of two half circles is illustrated as
  follows: 
  \begin{center}
   \includestandalone{pics/pic_cut_extracted_bordism}\\
   \includestandalone{pics/pic_cut_extracted_bordism_2}
  \end{center}
 \end{remark}
 \begin{remark}\label{remark:mfldcorners}
 Generalizing \cite{CS}*{Proposition 8.9}, for $l=0$ every image of some
  $\ind{\xi}=(\ind{\alpha}, \ind{\beta})\in \bbS^{\ind{j}}$ is an (oriented)
  cubical $n$-bordism. If for every $i$ we have that $\beta_i-\alpha_i =1 $, then the image
  is one of the (oriented) cubical $d$-dimensional $n$-bordisms in the
  composition. In the picture in the previous remark, the 
  bordisms being composed are depicted in the second line.
\end{remark}

\begin{notation}\label{notn:Mxitau}
If the context is clear, we will use the short-hand notation
 \[ M_{\xi,\tau} := \scut_{\ind{j},l}(M,\oul{I},\xi,\tau) = 
    \pi^{-1}B'(\xi^{*}_{\ext}\tau^{*}_{\xint}\oul{I}).\]
By the previous remark, this is an (oriented) manifold with corners of
dimension less than or equal to
  $d+l'+\sum \epsilon_{i}$, where
  \[ \epsilon_{i} :=
    \begin{cases}
      0, & j_{i}=0,\\
      1, & j_{i}>0.
    \end{cases}
  \]
  In the above, we required $\xi$ to correspond to an inert map
  $\ind{i} \hookrightarrow \ind{j}$ in $\simp^{n}$ and $\tau$ to an
  injective map $[k]\hookrightarrow[l]$ in
  $\simp_{\inj/[l]}$. However, the right-hand side makes sense for
  \emph{any} maps in $\simp^{n}$ and $\simp_{/[l]}$, since we just use
  the $(n+1)$-fold simplicial structure of $\Int^n$.  The short-hand
  notation turns out to be useful also in the more general case, and
  we observe the following: From the definition of $B'$ we see that if
  $\alpha \colon \ind{j}'' \to \ind{j}'$ is an \emph{active} injective
  morphism, then
  \[ M_{\xi\alpha,\tau} = M_{\xi,\tau}. \]
  This applies in particular to the unique active map $\alpha \colon
  \ind{\epsilon} \to \ind{j}'$. On the other hand, if $\xi$ factors
  as $\ind{j'} \xto{\alpha} \ind{k} \xto{\iota} \ind{j}$ with $\iota$
  inert and $\alpha$ active, then we have
  \[ M_{\xi,\tau} = M_{\iota,\tau} = \scut_{\ind{j},l}(M,
    \oul{I}, \iota, \tau). \]
\end{notation}

\begin{remark}\label{rmk:Misthesame}
  Given 
  $\ind{j}'' \xto{\xi'} \ind{j}' \xto{\xi} \ind{j}$ in
  $\simp^{n}$ and  $ [l''] \xto{\tau'} [l'] \xto{\tau} [l]$ in $\simp$, it
  is immediate from the definition that we have
  \[ (\xi^{*}_{\ext}\tau^{*}_{\xint}M)_{\xi',\tau'} =
    M_{\xi\xi',\tau\tau'}.\]
\end{remark}

\begin{proposition}\label{propn:cut_functor}
  The maps $\scut_{\ind{j},l}$ and $\scutT_{\ind{j},l}$ induce functors
  of $n$-uple Segal spaces
  \[  \BORD_{d,d+n} \longrightarrow \COSPAN_{n}(\mathcal{S}_{\fin}) \]
  which restrict to functors between the underlying $n$-fold Segal spaces
  \[  \Bord_{d,d+n} \longrightarrow \Cospan_{n}(\mathcal{S}_{\fin}) \]
  Furthermore, the natural transformations $\Gamma_{\ind{j},l}$
  induce a natural equivalence between them.
\end{proposition}
\begin{proof}
  The images
  $\scut_{\ind{j},l}(M,\oul{I},\xi,\tau) $ and
  $\scutT_{\ind{j},l}(M,\oul{I},\xi,\tau) $ are both compact manifolds with
  corners, by Remark \ref{remark:mfldcorners}. We can always smoothe corners to replace a compact manifold
  with corners by an ordinary compact manifold with boundary, up to
  homeomorphism. A standard argument using Morse theory (see for
  instance \cite[Ch. 7, Corollary 1.2]{Kosinski}) shows that a compact
  manifold with boundary has a handlebody decomposition relative to
  its boundary, which in particular means that the boundary inclusion
  is a finite relative cell complex. It follows that any compact
  manifold with corners has the homotopy type of a finite cell
  complex, and moreover that the inclusion of the boundary is a
  relative cell complex, and hence a Serre cofibration.
  
  Next we check the last statement, namely, that the natural
  transformations $\Gamma_{\ind{j},l}$ induce a natural
  equivalence. To see this, note that not only are the inclusions
  $B'(\oul{I}) \subseteq B''(\oul{I})$ deformation retracts, but we
  have even more: Using a Morse lemma
  argument, the submersivity condition on the projection
  $\pi$
  in \ref{defn:BORDV}\ref{defn:Bord_cond_grid}
  implies that the inclusions of the preimages
  \[ \Gamma_{\ind{j},l}(M,\oul{I},\xi,\tau) \colon \scut_{\ind{j},l}(M,\oul{I},\xi,\tau) \to \scutT_{\ind{j},l}(M,\oul{I},\xi,\tau).\]
  are deformation retracts, and so in particular homotopy equivalences. 
  Thus, we only need to check the steps (2) and (3) in \ref{proc:recipe_for_functor} for $\scut$.
  
 For (2), this follows from \ref{defn:Bord_cond_simplex} in
 \ref{defn:BORDV}: Given $\tau \to \tau'$ in $\sd(\Delta^{l})$, the
 submersivity of the projection $\pi'$ implies, again by using a Morse
 lemma argument, that the induced map on preimages under $\pi$ is a
 homotopy equivalence.
  
 For (3), we need to check that the images
 $\scut_{\ind{j},l}(M,\oul{I})$ indeed are ($l$-simplices of) cartesian
 diagrams. Observe that $\grid_{\ind{j},l}(\oul{I})$ is an
 ($l$-simplex of) cartesian diagrams in the category of topological
 spaces, and the same is true for the preimages under $\pi$. Furthermore,
 the pushout squares that appear are all homotopy pushouts, since the
 maps are  boundary inclusions for manifolds with corners, and hence
 are Serre cofibrations.
\end{proof}

\begin{definition}
We denote the functors from \ref{propn:cut_functor} by $\cut_{n,d}$
and ${\cutT}_{n,d}$. If $d=0$, we shorten the notation to $\cut_{n}$
and ${\cutT}_{n}$.
\end{definition}

Finally, we need to check (4) in \ref{proc:recipe_for_functor} and construct a delooping of the functors $\cut_{n,d}$ and ${\cutT}_{n,d}$.

\begin{construction}
	The cutting functors 
\[ \cutT_{n,d;r}:=\cutT_{n+r,d-r}\colon \Bord_{d-r,d+n} \longrightarrow \Cospan_{n+r}(\mathcal{S}),\]
\[ \cut_{n,d;r}:=\cut_{n+r,d-r}\colon \Bord_{d-r,d+n} \longrightarrow \Cospan_{n+r}(\mathcal{S}),\]
for $r\geq 0$
 induce functors
 \[ (\cutT_{n, d; r+1})_{\emptyset,\emptyset} \colon \Bord_{d-r-1,d+n}(\emptyset,\emptyset) \xrightarrow{\cutT_{n, d;r+1}}
\Cospan_{n+r+1}(\mathcal{S})(\emptyset,\emptyset),\]
\[ (\cut_{n, d; r+1})_{\emptyset,\emptyset} \colon \Bord_{d-r-1,d+n}(\emptyset,\emptyset) \xrightarrow{\cut_{n, d;r+1}}
\Cospan_{n+r+1}(\mathcal{S})(\emptyset,\emptyset).\]
\end{construction}

By inspection, we immediately get from the definitions that both
cutting functors are compatible with the deloopings of $\Bord_{d,d+n}$
from \cite{CS}*{\S 7.2} and those of $\Cospan_{n}(\mathcal{S})$ from
\cite{spans}*{Proposition 12.1}, in the following sense:
\begin{proposition}
  For every $d,n,\ind{j},l$, we have commutative squares
  \[
    \begin{tikzcd}
      \sBord_{d-1,d+n}(\emptyset,\emptyset)_{\ind{j},l} \times
      \bbS^{\ind{j}} \times \sd(\Delta^{l}) \arrow{r} \arrow{d}{\simeq} \arrow{dr}{(\scut_{n,d;1})_{\emptyset, \emptyset}}&
      \sBord_{d-1,d+n,(1,\ind{j}),l} \times \bbS^{(1,\ind{j})} \times
      \sd(\Delta^{l}) \arrow{d}{\scut_{n,d;1,(1,\ind{j}),l}} \\
      \sBord_{d,d+n,\ind{j},l} \times
      \bbS^{\ind{j}} \times \sd(\Delta^{l})  \arrow[swap]{r}{\scut_{n,d,\ind{j},l}} & \Top,
    \end{tikzcd}
  \]
  and similarly for $\cutT$.
\end{proposition}

\begin{corollary}
  For every $d,n$ we have a commutative square of $n$-fold Segal spaces
  \[
    \begin{tikzcd}[column sep=large]
      \Bord_{d-1,d+n}(\emptyset,\emptyset) \arrow{r}{(\cut_{n,d;1})_{\emptyset, \emptyset}} \arrow{d}{\simeq} &
      \Cospan_{n+1}(\mathcal{S})(\emptyset,\emptyset) \arrow{d}{\simeq} \\
      \Bord_{d,d+n} \arrow{r}{\cut_{n,d}}& \Cospan_{n}(\mathcal{S}),
    \end{tikzcd}
  \]
and similarly for $\cutT$.
\end{corollary}

\begin{corollary}\label{cor:cutsymmon}
  The cutting functors $\cut_{n,d}$ and $\cutT_{n,d}$ have natural symmetric
  monoidal structures coming from the functors $\cut_{n,d;r}$ and $\cutT_{n,d;r}$; therefore they are (extended) unoriented topological field theories.
\end{corollary}

\section{Extended TFTs in Spans}\label{subsec:TFTspan}
Combining Corollary~\ref{cor:cutsymmon} with the functoriality of
spans discussed in Remark~\ref{rmk:spanfun}, we obtain the following
family of unoriented TFTs:
\begin{corollary}\label{cor:SpanunorTFT}
  Let $\mathcal{C}$ be an \icat{} with finite limits. Then for every
  object $x \in \mathcal{C}$ there is an $n$-dimensional unoriented
  extended TFT obtained as the composite
  \[ \Bord_{0,n} \xto{\cut_{n}} \Cospan_{n}(\mathcal{S}_{\fin}) \to
    \Span_n(\mathcal{C}),\] where the second functor arises from the
  limit-preserving functor
  \[ x^{(\blank)} \simeq \lim_{(\blank)} x \colon \mathcal{S}_{\fin}^{\op}
    \to \mathcal{C} \]
  given by the
  natural cotensoring of $\mathcal{C}$ over $\mathcal{S}_{\fin}$. \qed
\end{corollary}

\begin{example}
  If we take $\mathcal{C}$ to be the \icat{} $\mathcal{S}$, then the
  cotensoring $X^{T}$ is simply the mapping space $\Map(T, X)$. For
  every space $X$ we thus have a family of unoriented TFTs
  $\Bord_{0,n} \to \Span_{n}(\mathcal{C})$ that to a closed manifold
  $M$ assign the mapping space $\Map(M, X)$ (where we also write $M$
  for the underlying homotopy type of the manifold), to a cobordism
  $M$ from $M_{0}$ to $M_{1}$ assigns the span
  \[ \Map(M_{0}, X) \from \Map(M,X) \to \Map(M_{1}, X) \]
  given by restricting maps to the boundary, and so forth.
\end{example}

Recall that for any symmetric monoidal $(\infty,n)$-category
$\mathcal{X}$, the Cobordism Hypothesis implies that there is a
canonical (homotopical) $O(n)$-action on the $\infty$-groupoid of
fully dualizable objects in $\mathcal{X}$. In the
$(\infty,n)$-category $\Span_{n}(\mathcal{C})$ it was proved in
\cite{spans}*{\S 12} that every object is fully dualizable, so that
the space of fully dualizable objects is the underlying
$\infty$-groupoid $\mathcal{C}^{\simeq}$ of $\mathcal{C}$. We will now
show that Corollary~\ref{cor:SpanunorTFT} implies that this
$O(n)$-action is trivial:
\begin{corollary}\label{cor:Ontrivact}
  The $O(n)$-action on $\mathcal{C}^{\simeq}$, viewed as the space of
  fully dualizable objects in $\Span_{n}(\mathcal{C})$, is trivial.
\end{corollary}

For the proof we use the following observation, which was explained to
us by Tyler Lawson:
\begin{lemma}\label{lem:hGsecttriv}
  Suppose $G$ is a topological group and the space $X$ has a
  homotopical $G$-action. If the natural map
  $X^{h G} \to X$ (where $X^{h G}$ is the space of homotopy fixed
  points) has a section, then the $G$-action on $X$ is trivial.
\end{lemma}
\begin{proof}
  Let $BG$ be the classifying space of $G$. Then a $G$-action on $X$
  is the same thing as a functor $BG \to \mathcal{S}$ whose value on
  the point of $BG$ is $X$, or equivalently a space over $BG$ whose
  fibre is $X$. Let us write $\overline{X} \to BG$ for the space over
  $BG$ corresponding to a $G$-action. Then the homotopy fixed points $X^{h G}$ are the limit of
  the corresponding functor, or the space of sections of $\overline{X}
  \to BG$.

  The functor $(\blank)^{h G}$ is hence right adjoint to the constant
  diagram functor $\mathcal{S} \to \Fun(BG, \mathcal{S})$, or to
  $\blank \times BG \colon \mathcal{S} \to \mathcal{S}_{/BG}$ (which
  corresponds to the trivial $G$-action on a space). The
  canonical map $f \colon X^{h G} \to X$ is the map on fibres from the
  counit map
  \[ X^{h G} \times BG \xto{\epsilon} \overline{X} \] over $BG$. If we have a
  section $\sigma \colon X \to X^{h G}$ then by adjunction we have a map
  \[ X \times BG \xto{\sigma \times \id} X^{h G} \times BG \xto{\epsilon}
    \overline{X}\]
  over $BG$. On fibres, we get the composite $X \xto{\sigma}
  X^{hG} \xto{f} X$, which by assumption is equivalent to the
  identity. Since equivalences in $\mathcal{S}_{/BG} \simeq \Fun(BG,
  \mathcal{S})$ are detected fibrewise, it follows that the map
  $X \times BG \to \overline{X}$ is an equivalence. In other words,
  the $G$-action on $X$ is equivalent to the trivial action, as required.
\end{proof}

\begin{proof}[Proof of Corollary~\ref{cor:Ontrivact}]
  By Lemma~\ref{lem:hGsecttriv} it suffices to show that the canonical
  map $(\mathcal{C}^{\simeq})^{h O(n)} \to \mathcal{C}^{\simeq}$ has a
  section. By the unoriented version of the Cobordism Hypothesis, we
  can identify $(\mathcal{C}^{\simeq})^{h O(n)}$ with
  $\Fun^{\otimes}(\Bord_{0,n}, \Span_{n}(\mathcal{C}))$, and under
  this equivalence our map corresponds to evaluation at the point in
  $\Bord_{0,n}$. We have a family of TFTs in $\Span_{n}(\mathcal{C})$ defined
  as the composite
  \[
    \begin{split}
\mathcal{C}^{\simeq} \to
    \Map^{\txt{flim}}(\mathcal{S}_{\fin}^{\op}, \mathcal{C}) & \to 
    \Map^{\otimes}(\Cospan_{n}(\mathcal{S}_{\fin}),
    \Span_{n}(\mathcal{C})) \\ & \xto{\cut_{n}^{*}}
    \Fun^{\otimes}(\Bord_{0,n}, \Span_{n}(\mathcal{C})),
    \end{split}
  \]
    where $\Map^{\txt{flim}}(-,-)$ means finite limit preserving maps.
  Unravelling the definition we see that the composite
  \[ \mathcal{C}^{\simeq} \to     \Fun^{\otimes}(\Bord_{0,n},
    \Span_{n}(\mathcal{C}))\xto{\txt{ev}_{\txt{pt}}}
    \mathcal{C}^{\simeq}\]
  is the identity.  
\end{proof}

Applying the Cobordism Hypothesis again, we deduce:
\begin{corollary}
  For any map $\xi \colon X \to BO(n)$, the space of
  $(X,\xi)$-structured $n$-dimensional extended TFTs valued in
  $\Span_{n}(\mathcal{C})$ is $\Map(X, \mathcal{C}^{\simeq})$. 
\end{corollary}
\begin{proof}
  By the Cobordism Hypothesis the space of
  such TFTs is identified with the space $\Map_{/BO(n)}(X, \mathcal{C}^{\simeq}
  \times BO(n))$, since the $O(n)$-action is trivial.
\end{proof}

\begin{remark}
  As a special case, we get that $n$-dimensional unoriented TFTs
  valued in $\Span_{n}(\mathcal{S})$ correspond to maps $BO(n) \to
  \mathcal{S}^{\simeq}$, or equivalently spaces over $BO(n)$, which we
  can view as spaces equipped with an $n$-dimensional vector bundle.
\end{remark}

\chapter{From Oriented Cobordisms to Oriented Cospans}\label{sec:orcob}
Our goal in this chapter is to upgrade our construction of the cutting
functor $\cut_{n} \colon \Bord_{0,n} \to
\Cospan_{n}(\mathcal{S}_{\fin})$ from the
previous chapter to obtain a symmetric monoidal functor
\[ \orcut_{n} \colon \Bordor_{0,n} \longrightarrow \Or_{n}^{\mathcal{S},0} \]
from the oriented bordism $(\infty,n)$-category to the
$(\infty,n)$-category of oriented iterated cospans of spaces (which we
introduced in \S \ref{subsec:orspc}). Combining this with the Betti
stack functor $\Or_{n}^{\mathcal{S},0} \xto{(\blank)_{B} \times S}
\Or_{n}^{S,0}$ from Corollary~\ref{cor:BettiOrn} and the AKSZ functor
$\Or_{n}^{S,0} \xto{\txt{AKSZ}_{n,X}^{S,d}} \Lag_{n}^{S,s}$ from
Corollary~\ref{cor:AKSZOrLag} for an $s$-symplectic $S$-stack $X$, we
obtain a family of extended TFTs valued in $\Lag_{n}^{S,s}$,
\[ \Bordor_{0,n} \xrightarrow{\orcut_{n}}
\Or_{n}^{\mathcal{S},0}  \xrightarrow{(\blank)_{B} \times S}
\Or_{n}^{S,0} \xrightarrow{\txt{AKSZ}_{n,X}^{S,d}} \Lag_{n}^{S,s}.\]

We will first define $\orcut_{n}$ as a symmetric monoidal functor
\[ \Bordor_{0,n} \to \Cospan_{n}(\POr_{0}^{\mathcal{S}});\] we give an
informal description of this construction in \S\ref{sec:coborcospidea}
and then the full details in
\S\S\ref{subsec:somediags}--\ref{subsec:orsymmmon}: the key point is
to construct a symmetric monoidal functor from $\Bordor_{0,n}$ to
spans in $\Mod_{\k/\k}$ that encodes orientations, and to define this
as a strict functor we first define certain diagrams of chain
complexes in \S\ref{subsec:somediags}, which we then use to define the
target of this orientation functor in \S\ref{sec:cobpreordet} before
we use integration of differential forms to actually define the
functor itself in \S\ref{subsec:integral} and then show that it is
symmetric monoidal in \S\ref{subsec:orsymmmon}.  We then prove in
\S\ref{sec:cobnondeg} that this functor satisfies the non-degeneracy
condition required to land in the sub-$(\infty,n)$-category
$\Or_{n}^{\mathcal{S},0}$, which boils down to Poincar\'e--Lefschetz
duality for manifolds with boundary, and then use the resulting
construction of oriented TFTs to deduce some information about TFTs
valued in $\Lag_{n}^{S,s}$ more generally.

\section{From Cobordisms to Preoriented Spaces: Idea}\label{sec:coborcospidea}
Our overall goal in this chapter is to construct a family of functors
\[ \BORD_{d,d+n}^{\txt{or}} \to \COSPAN_{n}(\POr_{d}^{\mathcal{S}}),\]
from which we can extract a symmetric monoidal functor
\[ \Bord_{d,d+n}^{\txt{or}} \to \Cospan_{n}(\POr_{d}^{\mathcal{S}}).\]
Since we have by definition a pullback square of \icats{}
\[\begin{tikzcd}
\POr^{\mathcal{S}}_d \arrow{r} \arrow{d} & \mathcal{S}_{\fin}
\arrow{d}{C^{*}}\\
(\Mod_{\k/\k[-d]})^{\op} \arrow{r} & \Mod_{\k}^{\op},
\end{tikzcd}
\]
and $\COSPAN_{k}(\blank)$ preserves limits by Lemma~\ref{lem:Spanlim}, it suffices to construct a 
commutative square\footnote{Some care is necessary with opposites
  here: Note that while we have $\Span_{k}(\mathcal{C})
  \simeq \Cospan_{n}(\mathcal{C}^{\op})$, it is not true that $\SPAN_{n}(\mathcal{C})$
  is $\COSPAN_{n}(\mathcal{C}^{\op})$ --- rather it is the levelwise
  opposite.}
\[
\begin{tikzcd}
  \BORDor_{d,d+n} \arrow{r} \arrow{d} & \COSPAN_{n}(\mathcal{S}_{\fin})
  \arrow{d}{\COSPAN_{n}(C^{*})} \\
  \COSPAN_{n}(\Mod_{\k/\k[-d]}^{\op}) \arrow{r} & \COSPAN_{n}(\Mod^{\op}_{\k}).
\end{tikzcd}
\]
The top horizontal morphism will be the composite of the functor
$\cut_{d,n}$  we constructed
in \S\ref{subsec:cut} with the forgetful functor from oriented cobordisms to unoriented cobordisms,
\[\BORDor_{d,d+n} \longrightarrow
\BORD_{d,d+n} \xrightarrow{\cut_{d,n}} \COSPAN_{n}(\mathcal{S}_{\fin}).\]
 We are thus required to define the left vertical
functor, which should take a closed oriented $d$-manifold $M$ to the morphism
$C^{*}(M) \to \k[-d]$ dual to the fundamental class of $M$.

To define this functor we will again use the definition of
$\sBORDor_{d,d+n}$ as a strict $n$-uple Segal object in simplicial
sets. We will describe $\Mod_{\k}$ in terms of the model category
$\Ch_{\k}$ of cochain complexes, and in order to get a strict model
for the dual fundamental classes we will model $C^{*}(M)$ by the de Rham
complex of differential forms $\Omega^{*}(M; \k)$. This forces us to take
$\k = \RR$ (or an extension of $\RR$, such as $\CC$), which is in any
case the relevant choice in the physical examples.
The basic idea is then to assign to a closed oriented $d$-manifold $M$ the
morphism
\[ \Omega^{*}(M; \k) \to \k[-d] \]
given by $\omega \mapsto \int_{M} \omega$ for
$\omega \in \Omega^{d}(M; \k)$ (and $0$ in other degrees). There are
a number of complications that arise from extending this to a strict
functor, and we now explain step by step our strategies to deal with them
below:

\begin{problem}[Unoriented issues]
  We encounter the same issues with compositions and degeneracies in the spatial
  direction as in \S\ref{subsec:cutidea}. This
  suggests that, as we did there, we should aim to define a family of
  morphisms of semi-simplicial sets 
  \[ \sBORDor_{d,d+n,\ind{j}} \to \Ex^{1}\Nrv\Fun(\bbSj,
    \Ch_{\k/\k[-d]})^{\op},\]
  natural in $\ind{j} \in \Dnop$, 
  or equivalently a family of functors\footnote{Again we need to be careful with opposites: We set $Y=\Fun(\bbS^{\ind{i}},
    \Ch_{\k/\k[-d]})$ in $\Hom(\Delta^l, \Ex^1(Y^\op) ) = \Hom(\sd(\Delta^l), Y^\op ) = \Hom (\sd(\Delta^l)^\op, Y)$.}
  
  \[ \sBORDor_{d,d+n,\ind{j},l} \times \bbSj \times \sdDl^{\op}
    \to \Ch_{\k/\k[-d]},\]
   strictly natural in $l\in \Dinjop$ and
  $\ind{j} \in \Dnop$.
 However, it turns out that we can't actually define a family of
 functors to $\Ch_{\k/\k[-d]}$, as we explain next.
\end{problem}

\begin{problem}[Commutativity over {$\k[-d]$}] \label{prob:commutativity}
To an oriented cobordism $\Sigma$ from $M$ to $N$ we want to
assign a commutative diagram
\[
\begin{tikzcd}
  \Omega^{*}(M; \k) \arrow{dr}[swap]{\int_{M}} & \Omega^{*}(\Sigma; \k) \arrow{l} \arrow{r}
  \arrow{d}
  & \Omega^{*}(N; \k) \arrow{dl}{\int_{N}} \\
  & \k[-d], 
\end{tikzcd}
\]
but unfortunately this diagram does not  commute, as the maps
going through $\Omega^{*}(M;\k)$ and $\Omega^{*}(N;\k)$ are not
\emph{equal}. However, if we view $\k[-d]$ as a constant diagram and replace it by a quasi-isomorphic one, we can define a commutative diagram
\[
\begin{tikzcd}
  \Omega^{*}(M; \k) \arrow{d}{\int_{M}} & \Omega^{*}(\Sigma; \k) \arrow{l} \arrow{r}
  \arrow{d}{I}
  & \Omega^{*}(N; \k) \arrow{d}{\int_{N}} \\
  \k[-d] & Q_{1}[-d] \arrow{l} \arrow{r} & \k[-d].
\end{tikzcd}
\]
as follows.
The cochain complex $Q_{1}$ is 
\[ \k^{2} \to \k,\]
in degrees $0$ and $1$,
with differential $d(x,y) = y-x$, and the horizontal maps are
\[\begin{tikzcd}[column sep=large]
  \k[-d] & Q_{1}[-d] \arrow[swap]{l}{(x,y)\mapsto x} \arrow{r}{(x,y)\mapsto y} & \k[-d],
\end{tikzcd}\] which are quasi-isomorphisms. The bottom row is thus
constant up to quasi-isomorphism. The vertical middle map
$I \colon \Omega^{*}(\Sigma; \k) \to Q_{1}[-d]$ is given, in degrees
$d$ and $d+1$, respectively, by
\[\omega \mapsto (\int_{M}\omega|_{M}, \int_{N} \omega|_{N}), \qquad
  \omega \mapsto \int_{\Sigma}\omega,\] and $0$ otherwise. This is a
morphism of cochain complexes since\footnote{More precisely, this is
  true in the non-degenerate case. In the degenerate case, where
  $\Sigma = N = M$, the left-hand side is $0$ for degree reasons and
  on the right-hand side the two terms cancel.}  by Stokes's theorem
we have, for $\omega \in \Omega^{d}(\Sigma; \k)$,
\[ I(d\omega) = \int_{\Sigma} d\omega = \int_{N} \omega|_{N} -
  \int_{M} \omega|_{M} = d \, I(\omega).\]

To generalize this, observe that since
\[\begin{tikzcd}
\Mod_{\k/\k[-d]} \arrow{r}\arrow{d} & \Mod_{\k}^{\Delta^1} \arrow{d}{ev_1}\\
\Delta^0 \arrow{r}{\k[-d]} & \Mod_{\k}
\end{tikzcd}\]
is a pullback square of \icats{}, we obtain a pullback square
\[
  \begin{tikzcd}
\COSPAN_{n}(\Mod^{\op}_{\k/\k[-d]}) \arrow{r} \arrow{d} & \COSPAN_{n}((\Mod_{\k}^{\Delta^1})^{\op}) \arrow{d}{ev_1} \\
* \simeq \COSPAN_{n}(\Delta^{0,\op}) \arrow{r}{\k[-d]} & \COSPAN_{n}(\Mod^{\op}_{\k}) ,
\end{tikzcd}
\]
where the lower horizontal arrow is given by picking out the constant
diagram with value $\k[-d]$. So, instead of defining a to functor
$\Chkd$ we can therefore try to define a functor to
$\Ch_{\k}^{\Delta^{1}}$, such that evaluating at $1$ we get a diagram
equivalent to the constant diagram at $\k[-d]$, rather than a functor
to $\Chkd$. This would involve chain complexes $Q_n\simeq \k[-d]$, which
together form a suitable non-constant diagram equivalent to the
constant one. Unfortunately, the situation turns out to be even more
complicated, because the diagrams required to get a strict functor
must also depend on the semisimplicial variable: we will have to
define functors
\[ \sBORDor_{d,d+n,\ind{j},l} \times \bbSj \times \sdDl^{\op} \times [1]
  \to \Chk \]
where the restriction to $1 \in [1]$ is given by a diagram
\[ \Pjl[-d] \colon \bbSj \times \sdDl^{\op} \to \Chk \]
that is quasi-isomorphic to the constant diagram $\k[-d]$. We will
explain the details below, for now we just note one further wrinkle in
the construction:
\end{problem}

\begin{problem}[Compositions]
  When we compose cobordisms our fix in Problem
  \ref{prob:commutativity} means that we consider a diagram of cochain
  complexes of shape $\bbS^{2}$ over
\[
\begin{tikzcd}
{}   &  & Q_{2}[-d]\arrow{dl} \arrow{dr} \\
 & Q_{1}[-d]\arrow{dl} \arrow{dr} & & Q_{1}[-d]\arrow{dl} \arrow{dr} \\
\k[-d] & & \k[-d] & & \k[-d],
\end{tikzcd}
\]
where $Q_{2}$ is a cochain complex of the form $\k^{3} \to \k^{2}$,
which allows us to encode the integrals over the parts of the
composite cobordism. However, if we na\"ively try to extract the
composite span from this, we obtain a diagram over
\[ \k[-d] \longleftarrow Q_{2}[-d] \longrightarrow \k[-d],\] but here
we ought to have $Q_{1}$ in place of $Q_{2}$. To fix this we simply
compose with a canonical map of the form
\[
  \begin{tikzcd}
    \k[-d] \arrow[equals]{d} & Q_{2}[-d] \arrow{d} \arrow{l} \arrow{r} & \k[-d]
    \arrow[equals]{d} \\
    \k[-d]  & Q_{1}[-d] \arrow{l} \arrow{r} & \k[-d],
  \end{tikzcd}
\]
where the middle map adds up the contributions from the two parts of
the composed bordism.
\end{problem}

\section{Some Diagrams of Chain Complexes}\label{subsec:somediags}
As explained in the previous section, we are going to define maps
\[ \mathcal{O}_{\ind j,l} \colon \sBORDor_{d,d+n,\ind{j},l} \times
  \bbS^{\ind j} \times \sdDl^{\op} \times [1] \to \Ch_{\k},\]
where the restriction $\mathcal{O}_{\ind j,l}(0)$ to $0 \in
[1]$ is the composite
\[ \sBORDor_{d,d+n,\ind j,l} \times 
  \bbS^{\ind j,\op} \times \sdDl \to \sBORD_{d,d+n,\ind{j},l}  \times
  \bbS^{\ind j,\op} \times \sdDl \!\xto{\scut_{\ind j,l}} \Mfd \xto{\Omega^{*}}
  \Ch_{\k}^{\op},\]
with $\scut_{\ind j,l}$ the functor defined in
Definition~\ref{defn:cut}. On the other hand, the functor $\mathcal{O}_{\ind j,l}(1)$ given by
restriction to $1 \in [1]$ will be a composite
\[ \sBORDor_{d,d+n,\ind j,l} \times \bbS^{\ind j} \times \sdDl^{\op}
  \longrightarrow \bbS^{\ind j} \times \sdDl^{\op} \xto{P_{\ind
      j,l}[-d]} \Ch_{\k},\] for a certain family of functors
$P_{\ind j,l}\colon \bbS^{\ind j} \times \sdDl^{\op} \to \Chk$.

Our goal in this section is to define these diagrams $\Pjl$. We
will do so using the natural (co)chain complexes associated to certain
simple semisimplicial sets, for which we first need to introduce some
notation:
\begin{notation}
  If $K$ is a semisimplicial set, we write $C(K)$ for the associated
  cochain complex over $\k$, given by
  \[ C(K)^{-b} = \k K_{b}, \]
  with differential $d = \sum_{i} (-1)^{i}d_{i}$. 
\end{notation}
  
\begin{notation}
  Let $\Delta^{l}_{\nd,b}$ denote the set of $b$-dimensional
  non-degenerate simplices of $\Delta^{l}$, \ie{} the set of strictly
  increasing sequences
  \[ [i_{0}i_{1}\cdots i_{b}] = 0 \leq i_{0} < i_{1} < \cdots < i_{b} \leq l.\]
  These form a semisimplicial set $\Delta^{l}_{\nd,\bullet}$, where
  the $s$th face map $d_{s}$ simply omits the $s$th element, \ie{}
  \[ d_{s}[i_{0}\cdots i_{b}] = [i_{0} \cdots i_{s-1} i_{s+1} \cdots
    i_{b}].\]
  This semisimplicial set is nothing but the representable presheaf
  $\Hom_{\Dinj}(\blank, [l])$, and so determines a functor
  \[\Dnd^{\bullet} \colon \Dinj \longrightarrow \ssSet.\]
  Composing with the forgetful functors $\sdDl\cong \simp_{\inj/[l]} \to \Dinj$ we
  obtain functors
  \[ \mathcal{C}_{l}\colon \sdDl \to \Dinj
    \xto{C(\Dnd^{\bullet})} \Ch.\]
  For $\phi \colon [l] \to [k]$ in $\Dinj$, we have a commutative
  triangle
  \[
    \begin{tikzcd}
      \sdDl \arrow{dd}[swap]{\sd(\phi)} \arrow{dr}{\mathcal{C}_{l}} \\
      & \Ch_{\k} \\
      \sd(\Delta^{k}), \arrow{ur}[swap]{\mathcal{C}_{k}}
    \end{tikzcd}
  \]
  since $\sd(\phi)$ just corresponds to the functor $\simp_{\inj/[l]}
  \to \simp_{\inj/[k]}$ given by composition with $\phi$.
\end{notation}

\begin{remark}
  Note that the cochain complex $C(\Dnd^{l})$ is isomorphic to
  the \emph{normalized chains} on the simplicial set $\Delta^{l}$
  (viewed as a cochain complex). In particular, this means that
  $C(\Dnd^{l})$ is contractible. Moreover, it makes
  $C(\Dnd^{\bullet})$ a functor from $\simp$ to $\Ch_{\k}$
  rather than just from $\Dinj$, though we will not make use of this.
\end{remark}

\begin{definition}
  Let \[N_{l} := C(\Dnd^{l})^{\vee} \]
  denote the dual cochain complex of $C(\Dnd^{l})$,
  which is isomorphic to the normalized \emph{co}chains of
  $\Delta^{l}$. This gives a semisimplicial cochain complex $N_{\bullet}$ using the
  functoriality of the normalized (co)chains. 
  Composing with the forgetful functors $\sdDl\cong \simp_{\inj/[l]} \to \Dinj$ we
  obtain functors
  \[ \mathcal{N}_{l} := \mathcal{C}_{l}^{\vee} \colon
    \sdDl^{\op} \to \Ch_{\k}.\]
\end{definition}

\begin{warning}\label{warn:sign}
  We will deviate from the Koszul sign rule and use the geometric sign
  convention for both $C(\Dnd^{l})$ and its dual $N_{l}$ (and
  similarly for other duals). This is harmless, but will lead to a
  4-periodic sign showing up below.
\end{warning}

\begin{notation}
  Let $\Delta^{l}_{\spi}$ denote the \emph{spine} of $\Delta^{l}$,
  meaning the simplicial set
  \[ \Delta^{l}_{\spi} :=  \Delta^{1} \amalg_{\Delta^{0}} \cdots
    \amalg_{\Delta^{0}} \Delta^{1}.\]
  We write $\Dndspi^{l}$
  for the semisimplicial set of non-degenerate simplices of
  $\Delta^{l}_{\spi}$, which we can also define as a colimit
  \[ \Dndspi^{l} := \Dnd^{1} \amalg_{\Dnd^{0}} \cdots
    \amalg_{\Dnd^{0}} \Dnd^{1}\]
  of semisimplicial sets. Thus
  \[ \Delta^{l}_{\nd,\spi,b} =
    \begin{cases}
      \{[0],[1],\ldots,[l]\}, & b = 0,\\
      \{[01],[12],\ldots,[(l-1)l]\}, & b = 1,\\
      \emptyset, & b > 1,
    \end{cases}
  \]
  with the only non-trivial face maps being
  \[ d_{0}[(s-1)s] = s, \quad d_{1}[(s-1)s] = s-1.\]
  If $\phi \colon [i] \to [j]$ is a morphism in $\Dinj$, then the
  induced map $\phi_{*} \colon \Dnd^{i} \to \Dnd^{j}$
  restricts to a map $\Delta^{i}_{\nd,\spi} \to
  \Delta^{j}_{\nd,\spi}$ \IFF{} $\phi$ is inert; thus
  $\Delta^{\bullet}_{\nd,\spi}$ gives a functor $\Dint \to \ssSet$ from the subcategory
  $\Dint$ of inert maps.
\end{notation}

\begin{remark}
  It will sometimes be convenient to think of the elements of
  $\Delta^{l}_{\nd,\spi}$ in degrees $0$ and $1$ as the inert maps
  from $[0]$ and $[1]$ to $[i]$, respectively. For an inert map $[i]
  \to [j]$, the map $\phi_{*} \colon \Dndspi^{i} \to
  \Dndspi^{j}$ is then simply given by composition.
\end{remark}
  
\begin{remark}
  The associated cochain
  complex is given by
  \[
    C(\Dndspi^{l})^{-b} =
    \begin{cases}
      \k \{[0],\ldots,[l]\} \cong \k^{l+1}, & b = 0,\\
      \k \{[01],\ldots,[(l-1)l]\} \cong \k^{l}, & b = 1,\\
      0, & b>1, b<0,
    \end{cases}
  \]
  with differential \[d[(s-1)s] = [s] - [s-1].\] Alternatively, if we
  think of the generators in degree $-1$ as inert maps $\sigma \colon
  [1] \to [l]$, we have
  \[ d\sigma = \sigma(1)-\sigma(0).\]
  We can also identify $C(\Dndspi^{l})$ with the
  normalized chains on $\Delta^{l}_{\spi}$.
\end{remark}

\begin{definition}
  We can extend $C(\Dndspi^{\bullet})$ to a functor
  \[C^{\bullet}_{\spi} \colon \simp
    \to \Chk\]
  as follows: For $\phi \colon [i] \to [j]$ in $\simp$ we define
  $\phi_{*} \colon C(\Dndspi^{i}) \to
  C(\Dndspi^{j})$ on generators by
  \begin{align*}
 \phi_{*}[s] &= [\phi(s)], \\
  \phi_{*}[(s-1)s] &=
                     [\phi(s-1)(\phi(s-1)+1)] + [(\phi(s-1)+1)(\phi(s-1)+2)]\\
     & \phantom{=} + \cdots + [(\phi(s)-1)\phi(s)],
  \end{align*}
  where this sum is defined to be $0$ if $\phi(s-1)=\phi(s)$.
  This is indeed a chain map, since we have
  \[
    \begin{split}
d \phi_{*}[(s-1)s] & = [\phi(s)] - [\phi(s)-1] + \cdots + [\phi(s-1)+1]
    - [\phi(s-1)] \\ & = [\phi(s)]-[\phi(s-1)] = \phi_{*}d[(s-1)s].
    \end{split}
  \]
  (Note that if $\phi$ is inert, then $\phi_{*}[(s-1)s]$ is always a
  generator, and indeed $\phi_{*}$ is induced by the map of
  semisimplicial sets $\Dndspi^{i} \to
  \Dndspi^{j}$ described above.)
\end{definition}

\begin{remark}
  Recall that the partially ordered set $\bbS^{j,\op}$ can be
  described as the category $\simp_{\xint/[j]}$ of inert maps to $[j]$
  (with $(a,b) \in \bbS^{j,\op}$ corresponding to the inert inclusion
  $\{a,a+1,\ldots,b\} \hookrightarrow [j]$). The forgetful functor
  $\simp_{\xint/[j]} \to \simp$ corresponds under this equivalence to
  a functor
  \[ u_{j}\colon \bbS^{j,\op} \to \simp \]
  that takes $(a,b)$ to $[b-a]$. For a morphism $\phi \colon [j] \to
  [k]$ in $\simp$, we have a natural transformation
  \[ \eta_{\phi} \colon u_{j} \to u_{j'}\circ \bbS(\phi)\]
  of functors $\bbS^{j,\op} \to \simp$, given at $(a,b)$ by the
  morphism
  \[ \{a,a+1,\ldots,b\} \to \{\phi(a),\phi(a)+1,\ldots,\phi(b)\} \]
  obtained by restricting $\phi$. These natural transformations are
  functorial in the sense that $\eta_{\phi'\phi} = \eta_{\phi'}\bbS(\phi)
  \circ \eta_{\phi}$. In other words, they make $u_{(\blank)}$ a
  lax natural transformation $\bbS^{\bullet,\op} \to
  \simp$ of functors $\simp \to \Cat$.
\end{remark}

\begin{definition}
  We define functors
  \[ \mathcal{C}_{j}^{\spi} \colon \bbS^{j,\op} \to \Chk \]
  as the composites
  \[ \bbS^{j,\op} \xto{u_{j}} \simp \xto{C^{\bullet}_{\spi}} \Chk.\]
  The natural transformations $\eta_{\phi}$ determine natural
  transformations $\eta_{\phi}^{\spi} \colon \mathcal{C}_{j}^{\spi}
  \to \mathcal{C}^{\spi}_{k}$ for $\phi \colon [j] \to [k]$, which
  make $\mathcal{C}^{\spi}_{(\blank)}$ a lax natural transformation
  $\bbS^{\bullet,\op} \to \Chk$ of functors $\simp \to \Cat$.
\end{definition}

\begin{definition}
  We write 
  \[N^{\spi}_{l} := C(\Dndspi^{l})^{\vee} \]
  for the dual cochain complex of $C(\Dndspi^{l})$,
  which is isomorphic to the normalized \emph{co}chains of
  $\Delta^{l}_{\spi}$. Taking the duals of the functors
  $\mathcal{C}^{\spi}$ we also get
  functors
  \[ \mathcal{N}^{\spi}_{j} := (\mathcal{C}^{\spi}_{j})^{\vee} \colon
    \bbS^{j} \to \Chk.\]
  The natural transformation $\eta_{\phi}^{\spi}$ for $\phi \colon [j] \to
  [k]$ induces a natural transformation $\nu_{\phi} \colon
  \mathcal{N}^{\spi}_{k} \circ \bbS(\phi)\to \mathcal{N}^{\spi}_{j}$;
  these make $\mathcal{N}^{\spi}_{(\blank)}$ an oplax natural
  transformation $\bbS^{\bullet} \to \Chk$ of functors $\Dop \to \Cat$.
\end{definition}

\begin{remark}
  Explicitly, $N^{\spi}_{l}$
  is the cochain complex
  \[ \k^{l+1} \to \k^{l}\]
  in degrees 0 and 1 with differential $d(x_{0},\ldots,x_{l}) =
  (x_{1}-x_{0},\ldots,x_{l}-x_{l-1})$. For $\phi
  \colon [n] \to [m]$ in $\simp$,
  the induced map $N^{\spi}(\phi) \colon N^{\spi}_{m} \to N^{\spi}_{n}$ 
  is given in degree $0$ by
  $(x_{0},\ldots,x_{m}) \mapsto (x_{\phi(0)},\ldots, x_{\phi(n)})$ and
  in degree $1$ by \[(a_{1},\ldots,a_{m}) \mapsto
  \left(\sum_{\phi(0)<i\leq\phi(1)}a_{i}, \ldots,
    \sum_{\phi(n-1)<i\leq\phi(n)} a_{i}\right).\]
  This is indeed a cochain map, since we have
  \[
    \begin{split}
\phi^{*}d(x_{0},\ldots,x_{m}) & = \left(\sum_{\phi(0)<i\leq
      \phi(1)} x_{i}-x_{i-1},\ldots\right) =
  (x_{\phi(1)}-x_{\phi(0)},\ldots) \\ & = d\phi^{*}(x_{0},\ldots,x_{m}).
    \end{split}
  \]
  
\end{remark}

\begin{definition}
  For $\ind{j} = (j_{1},\ldots,j_{n}) \in \simp^{n}$ we define
  \[\mathcal{N}^{\spi}_{\ind{j}} := \mathcal{N}^{\spi}_{j_{1}} \otimes
    \cdots \otimes \mathcal{N}^{\spi}_{j_{n}} \colon \bbSj
    \to \Chk. \]
  The natural transformations $\nu_{\phi}$ for $\phi$ in $\simp$
  combine to give natural transformations $\nu^{n}_{\Phi} \colon
  \mathcal{N}^{\spi}_{\ind{j}} \circ \bbS(\Phi) \to
  \mathcal{N}^{\spi}_{\ind{i}}$ for $\Phi \colon \ind{i} \to \ind{j}$
  in $\simp^{n}$; these make $\mathcal{N}^{\spi}_{(\blank)}$ an oplax
  natural transformation $\bbS^{(\blank)} \to \Chk$ of functors $\Dnop
  \to \Cat$.
\end{definition}

\begin{example}
  Let us spell this out in detail for $\ind{j}=(1,1)$. The cochain
  complex
  $C_\bullet(\Dnd^{1})=C_\bullet(\Dndspi^{1})$ is
  given by
  \[\k \longrightarrow \k^2=\k\{x,y\}, \quad 1\longmapsto y-x,\]
  in degrees $-1$ and $0$, and its square is
  \[ \arraycolsep=1pt%
    \begin{array}{ccccc}
	\k\{1\otimes 1\} &\longrightarrow &\k\{x\otimes 1,1\otimes x, y\otimes 1, 1\otimes y\} &\longrightarrow & \k\{x\otimes x, x\otimes y, y\otimes x, y\otimes y\}\\
	1\otimes 1 & \longmapsto & (y-x)\otimes 1 + 1\otimes (y-x),\\
	&& 1\otimes x & \longmapsto & (y-x)\otimes x\\
	&& x\otimes 1 & \longmapsto & -x\otimes (y-x)\\
	&& 1\otimes y & \longmapsto & (y-x)\otimes y\\
	&& y\otimes 1 & \longmapsto & -y\otimes (y-x).
\end{array}\]
Note that this is precisely the cellular chain complex associated to the square with the cells labelled as follows:
\begin{center}
\includestandalone{pics/pic_square_cell_complex}
\end{center}
 This cochain complex is the dual of the value of
 $\mathcal{N}^{\spi}_{(1,1)}((0,1),(0,1))$, which is thus given by
 \[ \k^{4} \to \k^{4} \to \k,\]
 with the induced dual differentials.
\end{example}

\begin{definition}\label{defn:Pjl}
  Finally, we define
  $\Pjl \colon \bbSj \times \sdDl^{\op} 
   \to \Chk$ as
  \[\Pjl := \mathcal{N}^{\spi}_{\ind{j}} \otimes
    \mathcal{N}_{l}.\]
  For morphisms $\Phi \colon \ind{i} \to \ind{j}$ in $\simp^{n}$ and
  $\psi \colon [k] \to [l]$ in $\Dinj$, the triangle
  \[
    \begin{tikzcd}
      \bbSj \times \sd(\Delta^{k})^{\op} \arrow{dd}[swap]{\id \times
        \sd(\psi)} \arrow{dr}{P_{\ind{j},k}} \\
      & \Chk \\
      \bbSj \times \sdDl^{\op} \arrow{ur}[swap]{\Pjl}
    \end{tikzcd}
  \]
  commutes, while we have a natural morphism in the triangle
  \[
    \begin{tikzcd}
      \bbS^{\ind{i}} \times \sdDl^{\op}
      \arrow{dd}[swap]{\bbS(\Phi) \times \id} \ar[dr, "P_{\ind{i},l}"{name=B}] \\
      & \Chk \\
      \bbSj \times \sdDl^{\op}.
      \arrow{ur}[swap]{\Pjl}
      \arrow[to=B,from=3-1,Rightarrow,shorten=5mm,"\alpha_{\Phi,l}"]
    \end{tikzcd}
  \]
  These two functorialities are independent, so that
  \[ \alpha_{\Phi,l} \circ (\id \times \sd(\psi)) = (\id \times
    \sd(\psi)) \circ \alpha_{\Phi,k}.\]
  This makes $P_{(\blank,\blank)}$ an oplax natural transformation
  $\bbS^{(\blank)} \times \sd(\Delta^{\bullet}) \to \Chk$ of functors $\Dnop \times
  \Dinjop \to \Cat$ (which is actually strict in the second variable).
\end{definition}

\begin{remark}
  Note that neither $\mathcal{N}^{\spi}_{\ind{j}}$ nor $\Pjl$
  are given by cochains on products of semisimplicial sets, since the
  (co)chain functor is only lax monoidal.
\end{remark}

\begin{remark}\label{rmk:gens}
  We will often think of an object $\xi \in \bbS^{\ind{j}} \times
  \sd(\Delta^{l})$ as a pair $(\sigma, \tau)$ where $\sigma =
  (\sigma_{1},\ldots,\sigma_{n})$ is a list of inert morphisms
  $\sigma_{i} \colon [s_{i}] \to [j_{i}]$ and $\tau$ is an injective
  morphism $[t] \to [l]$.
  By definition, the dual cochain complex $\Pjl(\xi)^{\vee}$ is
  then the tensor
  product
  $C(\Dndspi^{s_{1}}) \otimes \cdots \otimes
  C(\Dndspi^{s_{n}}) \otimes C(\Dnd^{t})$. In
  degree $-b$ we thus have have the free $\k$-module
  \[ (\Pjl(\xi)^{\vee})^{-b} \cong \k \mathcal{G}_{b}^{\xi},\]
  where we define
  \[ \mathcal{G}_{b}^{\xi} := \coprod_{c_{1}+\cdots+c_{n}+a = b}
    \mathcal{G}^{\xi}_{c_{1},\ldots,c_{n},a},\]
  \[ \mathcal{G}^{\xi}_{c_{1},\ldots,c_{n},a} :=  
    \Delta^{s_{1}}_{\nd,\spi,c_{1}} \times \cdots \times
    \Delta^{s_{n}}_{\nd,\spi,c_{n}} \times \Delta^{t}_{\nd,a},\]
  for $c_{i} \in \{0,1\}, a \in \mathbb{N}$.
  Note that we can think of the elements of
  $\Delta^{s}_{\nd,\spi,c}$ as the inert maps $[c] \to [s]$ 
  and the elements of $\Delta^{t}_{\nd,a}$ as the injective maps $[a]
  \to [t]$ in $\simp$. Thus a generator $\Gamma \in
  \mathcal{G}^{\sigma}_{c_{1},\ldots,c_{n},a}$ corresponds to a pair
  $(\gamma, \delta)$ with $\gamma = (\gamma_{1},\ldots,\gamma_{n})$ a
  list of inert maps $[c_{i}] \to [s_{i}]$ and $\delta
  \colon [a] \to [t]$ an injective map in $\simp$. With this notation,
  the
  differential in $\Pjl(\xi)^{\vee}$ is given on generators by
   \[ \begin{split}
      d \Gamma & = (d \gamma_{1}) \otimes \gamma_{2} \otimes \cdots
      \otimes \gamma_{n} \otimes \delta + (-1)^{c_{1}} \gamma_{1}
      \otimes d\gamma_{2} \otimes \cdots \otimes \gamma_{n} \otimes
      \delta \\
       & \phantom{=} +  \cdots + (-1)^{c_{1}+\cdots+c_{n-1}} \gamma_{1} \otimes \cdots \otimes
       d \gamma_{n} \otimes \delta\\
       & \phantom{=} + (-1)^{c_{1}+\cdots+c_{n}}
       \gamma_{1} \otimes \cdots \otimes \gamma_{n} \otimes d \delta
       \\
       & = \sum_{\substack{i=1,\ldots,n\\ c_{i}=1}}
       (-1)^{c_{1}+\cdots+c_{i}} \gamma_{1} \otimes \cdots \otimes
       \gamma_{i-1} \otimes (\gamma_{i}(1)-\gamma_{i}(0)) \otimes
       \gamma_{i+1} \otimes \cdots \otimes \gamma_{n} \otimes \delta\\
       & \phantom{=} + (-1)^{b-a} \sum_{k=1}^{l} (-1)^{k} \gamma_{1}
       \otimes \cdots \otimes \gamma_{n}\otimes (\delta \circ d_{k})\\
       & = \sum_{\substack{i=1,\ldots,n\\ c_{i}=1}}
       (-1)^{c_{1}+\cdots+c_{i}} (\Gamma \circ (d_{0}^{(i)},\id) -
        \Gamma \circ (d_{1}^{(i)},\id)) \\
         & \phantom{=} + (-1)^{b-a} \sum_{k=1}^{l}
       (-1)^{k} \Gamma \circ (\id, d_{k}),
        \end{split}\]
     where 
  \[ d_{0}^{(i)} = (\id,\ldots,\id,d_{0},\id,\ldots,\id) \]
  with $d_{0}$ in the $i$th position, and similarly for $d_{1}^{(i)}$.
\end{remark}

\begin{lemma}
  $\Pjl$ is quasi-isomorphic to the constant diagram with value $\k$.
\end{lemma}
\begin{proof}
  The functor $\Pjl$ takes every morphism in
  $\bbSj \times \sdDl^{\op}$ to a
  quasi-isomorphism, and this category is weakly contractible.
\end{proof}

\section{A Model for Cospans in $\Chkd$}\label{sec:cobpreordet}
In this section we will explain why constructing functors
\[ \mathcal{O}_{\ind j,l} \colon \sBORDor_{d,d+n,\ind{j},l} \times
  \bbS^{\ind j} \times \sdDl^{\op} \times [1] \to \Ch_{\k}\] as
discussed above actually leads to a functor of $n$-uple \icats{}
\[\BORDor_{d,d+n} \to \COSPAN_{n}(\Modkd^{\op}).\]
We first define the strict object that is actually the target of our
functor, and then relate this to $\COSPAN_{n}(\Modkd^{\op})$.

\begin{notation}
  Let $\Fun'(\bbSj \times
    \sdDl, \Chk)$ denote the full subcategory of $\Fun(\bbSj \times
    \sdDl, \Chk)$ containing those functors $F \colon \bbSj \times
    \sdDl \to \Chk$ such that $F(\sigma, \blank)$ takes every morphism
    in $\sdDl$ to a quasi-isomorphism for every $\sigma \in
    \bbSj$. We also write $\Fun'(\bbSj \times
    \sdDl, \Chk)_{/\Pjl[-d]}$ for the corresponding overcategory for
    $\Pjl[-d]$.
\end{notation}

\begin{construction}
  The strict naturality of $\Pjl$ in $l$ means that we can
  define a family of semisimplicial sets $\Xdj$ as
  \[ \Xdjl := \{ F \in \Fun'(\bbSj \times
    \sdDl, \Chk)_{/\Pjl[-d]}\},\]
  with the semisimplicial structure maps given by restriction.
  The oplax structure in $\ind{j}$ then makes this a functor
  $\Xd_{(\blank)} \colon \Dnop \to \ssSet$, with the map
  $\Xdj \to \Xd_{\ind{i}}$ corresponding to
  $\Phi \colon \ind{i} \to \ind{j}$ in $\simp^{n}$ given by
  restricting along $\bbS(\Phi)$ and composing with the natural maps
  $\Pjl \circ \bbS(\Phi) \to P_{\ind{i},l}$.
\end{construction}

\begin{remark}
  The semisimplicial set $\Xd_{\ind{j}}$ can be defined as a
  pullback of semisimplicial sets
  \[
    \begin{tikzcd}
      \Xd_{\ind{j}} \arrow{r} \arrow{d} &
      (\Ex^{1}\mathrm{N}\,\Fun(\bbSj \times [1],
      \Chk)^{(W)})^{\semi} \arrow{d} \\
      \Delta^{0,\semi}\arrow{r}{\Pj[-d]} & (\Ex^{1}\mathrm{N}\,\Fun(\bbSj,
      \Chk)^{(W)})^{\semi},
    \end{tikzcd}
  \]
  where the superscript $(W)$ means we only include those morphisms
  that are natural weak equivalences.
\end{remark}

\begin{definition}
  We define another semisimplicial set by
  \[ \Ydj := (\Ex^{1}\mathrm{N}\,\Fun(\bbSj, \Chkd)^{(W)})^{\semi}.\]  
  There are canonical natural transformations $\gamma_{\ind{j},l} \colon \k
  \to \Pjl$ (where we write $\k$ for the constant diagram
  with value $\k$), induced by the degeneracy maps
  \[ \k = N_{0} \to N_{l},\qquad \k = N_{0}^{\spi} \to N_{l}^{\spi}.\]
  Composition with these give a natural transformation $\Gamma^{(d)}_{(\blank)} \colon
  \Yd_{(\blank)} \to
  \Xd_{(\blank)}$.
\end{definition}

\begin{remark}
Unfortunately we do not know if the maps $\Gamma_{\ind{j}}^{(d)}$ are weak
equivalences. However, we can relate $\Xdj$ to diagrams in $\Chkd$ by
a more indirect route:  
\end{remark}

\begin{construction}\label{constr:tXj}
We can view $\Xdj$ as the
semisimplicial set of 0-simplices in a semisimplicial simplicial set
$\tXd_{\ind{j},\bullet}$ given by
\[ \mathrm{N}\,\Fun'(\bbSj \times \sd(\Delta^{(\blank)}),
  \Chk)^{(W)}_{/P_{\ind{j},(\blank)}[-d]}.\]
 We then have a natural
commutative square of semisimplicial simplicial sets
\[
  \begin{tikzcd}
    \disc \Ydj
    \arrow{d}{\Gamma_{\ind{j}}^{(d)}} \arrow{r} & \tYdj \arrow{d}{\widetilde{\Gamma}_{\ind{j}}^{(d)}}\\
  \disc \Xdj \arrow{r}     & \tXdj,
\end{tikzcd}
\]
where $\disc$ indicates ``constant in the simplicial direction'',
$\tYdj$ is the semisimplicial simplicial set
\[ \tYdj := \mathrm{N}\,\Fun'(\bbSj \times
  \sdDl, \Chkd)^{(W)},\]
and $\widetilde{\Gamma}^{(d)}_{\ind{j}}$ is again given by composition with
the maps $\gamma_{\ind{j},l}$ from above.  
\end{construction}

\begin{proposition}\label{propn:Pdiagweqs}\ 
  \begin{enumerate}[(i)]  \item The semisimplicial structure maps for both
    $\tYdjl$ and
    $\tXdjl$ are all weak equivalences.    
  \item The map $\widetilde{\Gamma}^{(d)}_{\ind{j}} \colon \tYdj \to
    \tXdj$ is a weak equivalences for all
    $\ind{j}$.
  \item The map $\disc \Ydj \to \tYd_{\ind{j}}$ induces
    an equivalence in the \icat{} $\mathcal{S}$.
  \end{enumerate}
\end{proposition}

For the proof we need the following observation:
\begin{lemma}\label{lem:diagdiscwe}
  Let $K$ be a simplicial set such that the degeneracy
  $K \to K^{\Delta^{l}}$ is a weak equivalence for all $l$ (or,
  equivalently, the structure maps in the bisimplicial set
  $K^{\Delta^{\bullet}}$ are all weak equivalences). Then the natural
  morphism of bisimplicial sets
  \[ \disc K \to K^{\Delta^{\bullet}}\]
  (given in each degree by the inclusion of the $0$-simplices)
  induces a weak equivalence
  \[ \hocolim \disc K \isoto \hocolim K^{\Delta^{\bullet}}.\]
\end{lemma}
\begin{proof}
  Recall that in bisimplicial sets the homotopy colimit, modelled as
  the usual coend via the Reedy model structure (where all objects are
  cofibrant), is isomorphic to the functor $\diag$ given
  by restriction along the diagonal $\Dop \to \Dop \times \Dop$.
  Here $\diag\disc K \cong K$, while
  \[ (\diag K^{\Delta^{\bullet}})_{l} = \Hom(\Delta^{l}\times
    \Delta^{l}, K), \] with the map
  $K_{l}\to (\diag K^{\Delta^{\bullet}})_{l}$ given by composition
  with the projection $\Delta^{l} \times \Delta^{l}\to
  \Delta^{l}$. But (since we can permute the two coordinates) this is
  the same as the map on diagonals for the morphism
  $\const K \to K^{\Delta^{\bullet}}$ induced by the degeneracies $[l]
  \to [0]$, where $\const K$ denotes the constant bisimplicial set
  with value $K$. This is a levelwise weak equivalence by assumption,
  hence \[\diag \const K \to \diag K^{\Delta^{\bullet}}\] is a weak
  equivalence (since $\diag$ is a left Quillen functor).
\end{proof}

\begin{example}\label{ex:diagdiscwemodcat}
  The hypothesis of Lemma~\ref{lem:diagdiscwe} is satisfied if $K$ is
  a Kan complex, but also if it is the nerve
  $\mathrm{N}\mathbf{M}^{(W)}$ where $\mathbf{M}$ is a model category
  and $\mathbf{M}^{(W)}$ its subcategory containing only the weak
  equivalences: We may identify
  $(\mathrm{N}\mathbf{M}^{(W)})^{\Delta^{l}}$ with the nerve
  $\mathrm{N}\,\Fun'([l], \mathbf{M})^{(W)}$, where
  $\Fun'([l], \mathbf{M})^{(W)}$ is the subcategory of
  $\Fun([l], \mathbf{M})$ whose objects are the functors that take all
  morphisms in $[l]$ to weak equivalences and whose morphisms are the
  natural weak equivalences. By Corollary~\ref{cor:relcatspanmodel}
  this models the $\infty$-groupoid $\Map(\|[l]\|,
  \mathbf{M}[W^{-1}])$, which is equivalent to
  $\mathbf{M}[W^{-1}]^{\simeq}$ since $[l]$ is weakly contractible.
\end{example}

\begin{proof}[Proof of Proposition~\ref{propn:Pdiagweqs}]
  From Corollary~\ref{cor:relcatspanmodel} we know that the simplicial set
  $\tYdjl$ models the $\infty$-groupoid
  \[ \Map(\|\sdDl\|, \Fun(\bbSj,
    \Chkd))\simeq \Map(\bbSj, \Chkd),\]
  which shows that the semisimplicial structure maps are weak
  equivalences. On the other hand,
  $\tXdjl$ models
  \[ \Fun(\|\sdDl\| \times \bbSj,
    \Chk)_{/\Pjl[-d]}^{\simeq}.\]
  Since the map $\gamma_{\ind{j},l}$ is a quasi-isomorphism, composing
  with it gives a left Quillen equivalence, and hence an equivalence
  of \icats{}. This shows that $\Gamma^{(d)}_{\ind{j},l}$ is a weak
  equivalence, \ie{} (ii), and also proves (i).

  To prove (iii), it suffices to prove that the natural map
  \[ \disc \mathrm{N}\,\Fun(\bbSj,
    \Chkd)^{(W)} \to \Fun'(\bbSj \times
    \Delta^{(\blank)}, \Chkd)^{(W)} \]
  gives a weak equivalence in $\mathcal{S}$, since \[\mathrm{N}\,\Fun(\bbSj,
  \Chkd)^{(W)} \to \Ex^{1} \mathrm{N}\,\Fun(\bbSj,
  \Chkd)^{(W)}\] is a weak equivalence, as is the corresponding map
  \[ \mathrm{N}\Fun'(\bbSj \times
  \Delta^{(\blank)}, \Chkd)^{(W)} \to \tYd_{\ind{j}}\]
  (levelwise). Since these maps of semisimplicial simplicial sets
  are restricted from maps of bisimplicial sets, this follows from 
  Lemma~\ref{lem:diagdiscwe} applied to the simplicial set
  $\mathrm{N}\,\Fun(\bbSj, \Chkd)^{(W)}$, which
  satisfies the required hypothesis by Example~\ref{ex:diagdiscwemodcat}.
\end{proof}

\begin{remark}\label{rmk:functortoXdj}
  In the next section, we will define natural maps of
  semisimplicial sets
  \[\sBORDorsemi_{d,d+n,\ind{j}} \to \Xdj\] that fit in
  commutative diagrams
  \[
    \begin{tikzcd}
      {} & \Xdj \arrow{dd}\\
      \sBORDorsemi_{d,d+n,\ind{j}} \arrow{ur} 
      \arrow{dr}[swap]{\Omega^{*} \circ \scut} &  & \Ydj \arrow{dl} \arrow{ul} \\
      & (\Ex^{1} \mathrm{N}
      \Fun(\bbSj, \Chk)^{(W)})^{\semi}.
    \end{tikzcd}
  \]
  Since we can combine these with the diagram of semisimplicial
  simplicial sets from Construction~\ref{constr:tXj}, Proposition~\ref{propn:Pdiagweqs} implies
  that this induces natural commutative diagrams of $\infty$-groupoids
  \[
    \begin{tikzcd}
          \BORDor_{d,d+n,\ind{j}}  \arrow{r} \arrow{dr}[swap]{C^{*} \circ
            \cut} &
          \Map(\bbSj, \Mod_{\k/\k[-d]}) \arrow{d}\\
           & \Map(\bbSj, \Mod_{\k}),
    \end{tikzcd}
  \]
  and thus a commutative square
  \[
    \begin{tikzcd}
      \BORDor_{d,d+n} \arrow{rr} \arrow{d}{\cut} & & \oCOSPAN_{n}(\Mod_{\k/\k[-d]}^{\op})
      \arrow{d} \\
      \COSPAN_{n}(\mathcal{S}) \arrow{r}{C^{*}} & \COSPAN_{n}(\Mod_{\k}^{\op})
      \arrow[hookrightarrow]{r} & \oCOSPAN_{n}(\Mod_{\k}^{\op})
    \end{tikzcd}
  \]
  in $\Fun(\Dnop, \mathcal{S})$. Since the forgetful functor
  $\Mod_{\k/\k[-d]} \to \Mod_{\k}$ is conservative and preserves
  pullbacks, a diagram $\bbSj \to \Mod_{\k/\k[-d]}$ is cartesian
  \IFF{} the composite $\bbSj \to \Mod_{\k/\k[-d]} \to \Mod_{\k}$ is
  cartesian; this means that we also have a pullback square
  \[
    \begin{tikzcd}
      \COSPAN_{n}(\Mod_{\k/\k[-d]}^{\op}) \arrow[hookrightarrow]{r} \arrow{d} &
      \oCOSPAN_{n}(\Mod_{\k/\k[-d]}^{\op})
    \arrow{d} \\
    \COSPAN_{n}(\Mod_{\k}^{\op})
    \arrow[hookrightarrow]{r} & \oCOSPAN_{n}(\Mod_{\k}^{\op})
    \end{tikzcd}
  \]
  in $\Fun(\Dnop, \mathcal{S})$. Thus our functor above
  factors through $\COSPAN_{n}(\Mod_{\k/\k[-d]}^{\op})$ automatically, giving a
  commutative square
  \[
    \begin{tikzcd}
      \BORDor_{d,d+n} \arrow{r} \arrow{d}{\cut} & \COSPAN_{n}(\Mod_{\k/\k[-d]}^{\op})
      \arrow{d} \\
      \COSPAN_{n}(\mathcal{S}) \arrow{r}{C^{*}} & \COSPAN_{n}(\Mod_{\k}^{\op})
    \end{tikzcd}
  \]
  of $n$-uple Segal spaces.
\end{remark}

\section{The Integration Functor}\label{subsec:integral}
In this section we will use integration of differential forms to
define a family of maps of semisimplicial sets
\[ \sBORDorsemi_{d,d+n,\ind{j}} \to \Xdj,\] which, as we explained in
the previous section, induce a functor of $n$-uple \icats{}
\[ \BORDor_{d,d+n} \to \COSPAN_{n}(\Ch_{\k/\k[-d]}^{\op}).\]

\begin{notation}
  For
  $(M, \oul{I}=(I^i_0\leq\ldots\leq I^i_{j_i})_{1\leq i \leq k}, f)\in
  \sBORD_{d,d+n,\ind{j},l}^{\txt{or}}$, recall the short-hand notation
  for the pieces of the bordism from
  \ref{notn:Mxitau},
  \[M_{\sigma,\tau} :=
    \pi^{-1}B'(\sigma^{*}_{\ext}\tau^{*}_{\xint}\oul{I}).\]
  This is well-defined
  for \emph{any} $\sigma\colon \ind{i} \to \ind{j}$ in $\simp^{n}$ and
  $\tau\colon [k]\to [l]$ in $\simp$.

  Given $\xi = (\sigma,\tau) \in \bbSj \times \sdDl^{\op}$, we will
  view $\sigma$ as an inert morphism $\ind{i} \to \ind{j}$ and $\tau$
  as an injective morphism $\tau \colon [k] \to [l]$ via the isomorphisms
  $\bbS^{\ind{j},\op} \cong (\Dint^{n})_{/\ind{j}}$ and
  $\sdDl \cong (\Dinj)_{/[l]}$. We will then write
  \[ M_{\xi} := M_{\sigma,\tau}.\]
\end{notation}

\begin{definition}\label{defn:O,iota}
  For $\xi = (\sigma,\tau) \in \bbSj \times \sdDl^{\op}$
  and $(M, \oul{I},f) \in \sBORDor_{d,d+n,\ind{j},l}$ we wish to define
  a morphism of cochain complexes
  \[ \mathcal{O}_{\ind{j},l}(M, \oul{I}, f, \xi) \colon
    \Omega^{*}(M_{\xi}) \longrightarrow
    \Pjl(\xi)[-d],\] where we have
  \[\cut(M,\oul{I},\xi) = M_{\sigma,\tau} = M_{\xi}.\]
  This will be obtained as
  the adjoint of a morphism
  \[ \Pjl(\xi)^{\vee} \otimes \Omega^{*}(M_{\xi})
    \longrightarrow \k[-d],\] which we define as follows: First, since
  $\k[-d]$ is concentrated in cohomological degree $d$, such a map
  corresponds to a linear map
\[ \iota_{\ind{j},l,\xi} \colon ( \Pjl(\xi)^{\vee} \otimes \Omega^{*}(M_{\xi}) )^{d} \to \k \]
with the property that $\iota_{\ind{j},l,\xi}(d\zeta) = 0$ for $\zeta \in
(\Pjl(\xi)^{\vee} \otimes \Omega^{*}(M_{\xi}))^{d-1}$. Here we
have
\[ (\Pjl(\xi)^{\vee} \otimes \Omega^{*}(M_{\xi})  )^{d} \cong
\bigoplus_{t} (\Pjl(\xi)^{\vee})^{-t} \otimes \Omega^{d+t}(M_{\xi}),\]
where $(\Pjl(\xi)^{\vee})^{-t}$ is the free $\k$-module on
 generators $\Gamma \in \mathcal{G}^{\xi}_{t}$ as in 
 Remark~\ref{rmk:gens}. To define such a
linear map it then suffices to define linear maps 
\[ \iota_{\ind{j},l,\xi}^{\Gamma} \colon \Omega^{d+t}(M_{\xi})
  \to \k\] for each generator $\Gamma \in \mathcal{G}^{\xi}_{t}$. We
define this linear map in turn as the composite
\[ \Omega^{d+t}(M_{\xi}) \to \Omega^{d+t}(M_{\xi\Gamma}) \xto{(-1)^{\epsilon(t)}\int_{M_{\xi\Gamma}}}
  \k,\]
where we view $\Gamma$ as a pair $(\gamma, \delta)$ with $\gamma$ an inert
map in $\simp^{n}$ and $\delta$ an injective map in $\simp$,
with $\xi\Gamma = (\sigma\gamma, \tau\delta)$ indicating composition,
the first map is restriction of differential forms to the
submanifold $M_{\xi\Gamma}$, and the sign is given by
\[ \epsilon(t) =
  \begin{cases}
    0, & t \equiv 0,1 \pmod{4} \\
    1, & t \equiv 2,3 \pmod{4}
  \end{cases}
\]
(so that $(-1)^{\epsilon(t)}=(-1)^{t-1}(-1)^{\epsilon(t-1)}$).
\end{definition}

\begin{remark}
  As in Warning~\ref{warn:sign}, the sign $(-1)^{\epsilon(t)}$ appears
  here because we have not followed the Koszul sign convention for the
  duals $\Pjl(\xi)^{\vee}$; this means that the tensor-Hom
  adjunction on cochain complexes involves a 4-periodic sign, and the
  signs $(-1)^{\epsilon(t)}$ exactly cancel these, so that the adjoint
  map $\Omega^{*}(M_{\xi}) \to \Pjl(\xi)[-d]$ does not involve
  a sign.
\end{remark}

The following proposition tells us that the previous definition does
in fact define a cochain map:
\begin{proposition}\label{propn:iotachainmap}
  The linear map $\iota_{\ind{j},l,\xi}$ defined above satisfies
  $\iota_{\ind{j},l,\xi}(d \zeta) = 0$.
\end{proposition}

The proof reduces to an application of Stokes's Theorem; for this we
first need to identify the oriented boundary of the manifold involved:
\begin{lemma}\label{lem:bdryisd}
  Given a non-degenerate element $(M, \oul{I}, f)\in
  \sBORDor_{d,d+n,\ind{j},l}$ where $\ind{j}=(j_{1},\ldots,j_{n})$
  with each $j_{i}$ either $0$ or $1$ and $\iota = (\id,\id) \in
  \bbS^{\mathbf{j}} \times \sdDl^{\op}$ the initial object,
  the oriented boundary of $M_{\iota}$ is
  \[
    \begin{split}
\partial M_{\iota} & \cong
    \bigcup_{i,j_{i}\neq 0} (-1)^{j_{1}+\cdots+j_{i-1}}
    \left(M_{(d_{0}^{(i)},\id)} \cup -M_{(d_{1}^{(i)},\id)}\right) \\
    & \phantom{\cong} \cup (-1)^{j_{1}+\cdots+j_{n}} \bigcup_{t=0}^{l} (-1)^{t}M_{(\id,d_{t})},
    \end{split}
  \]
  where 
  \[ d_{0}^{(i)} = (\id,\ldots,\id,d_{0},\id,\ldots,\id) \]
  with $d_{0}$ in the $i$th position, and similarly for $d_{1}^{(i)}$.
\end{lemma}
\begin{proof}
   By definition, the manifold $M_{\iota}$ lives over $B'(\oul{I})$,
   which is diffeomorphic to $\prod_{i,j_{i}\neq 0} |\Delta^{1}| \times
   |\Delta^{l}|$.
   The boundary of $M_{\iota}$ is therefore the subspace that lives
   over the boundary of $\prod_{i,j_{i}\neq 0} |\Delta^{1}| \times
   |\Delta^{l}|$. Using the standard description of the oriented
   boundary of a product, we get
   \[
     \begin{split}
     \partial\left(\prod_{i,j_{i}\neq 0} |\Delta^{1}| \times
   |\Delta^{l}|\right)  & \cong \bigcup_{i,j_{i}\neq 0}
                          (-1)^{j_{1}+\ldots+j_{i-1}}\left(\prod_{i'<i,j_{i'}\neq 0} |\Delta^{1}|\right) \times \partial |\Delta^{1}| \\
                        & \phantom{\cong \bigcup_{i,j_{i}\neq 0}
                          (-1)^{j_{1}+\ldots+j_{i-1}}}
                          \times
 \left(\prod_{i'> i,j_{i'}\neq 0} |\Delta^{1}|\right) \times |\Delta^{l}|  \\
  & \phantom{\cong} \cup (-1)^{j_{1}+\ldots+j_{i-1}}
  \prod_{i,j_{i}\neq 0} |\Delta^{1}| \times \partial |\Delta^{l}|.
     \end{split}
   \]
   Since the orientation of $M_{\iota}$ is compatible with that of
   $B'(\oul{I})$, we get the required decomposition by taking the
   preimage of this.
\end{proof}

\begin{proof}[Proof of Proposition~\ref{propn:iotachainmap}]
  It suffices to consider $\zeta$ of the form $\Gamma \otimes \omega$
  with $\omega \in \Omega^{d+t-1}(M_{\xi})$ and $\Gamma \in
  \mathcal{G}^{\xi}_{c_{1},\ldots,c_{n},a}$ a
  generator of $(\Pjl(\xi)^{\vee})^{-t}$. Now
  $d \zeta = d\Gamma \otimes \omega + (-1)^{t} \Gamma \otimes d
  \omega$, where $d \Gamma$ is a signed sum of generators of
  $(\Pjl(\xi)^{\vee})^{-(t-1)}$.
  By definition, we have
  \[\iota_{\ind{j},l,\xi}(\Gamma \otimes d\omega)  =
    (-1)^{\epsilon(t)}  \int_{M_{\xi\Gamma}}
    d\omega|_{M_{\xi\Gamma}},\]
  while the formula for $d\Gamma$ in
  Remark~\ref{rmk:gens} gives
  \[
    \begin{split}
    \iota_{\ind{j},l,\xi}(d\Gamma \otimes \omega) & =
\sum_{\substack{i=1,\ldots,n\\ c_{i}=1}}
       (-1)^{c_{1}+\cdots+c_{i}} (\iota_{\ind{j},l,\xi}(\Gamma
        (d_{0}^{(i)},\id) \otimes \omega) -
       \iota_{\ind{j},l,\xi}(\Gamma (d_{1}^{(i)},\id)
       \otimes \omega)) \\
       & \phantom{=} + (-1)^{t-a} \sum_{k=1}^{l}
       (-1)^{k} \iota_{\ind{j},l,\xi}(\Gamma (\id, d_{k})
       \otimes \omega) \\
       &= (-1)^{\epsilon(t-1)}( \sum_{\substack{i=1,\ldots,n\\ c_{i}=1}}
       (-1)^{c_{1}+\cdots+c_{i}}
       \left(\int_{M_{\xi\Gamma(d_{0}^{(i)},\id)}} \omega -
       \int_{M_{\xi\Gamma(d_{1}^{(i)},\id)}} \omega\right)
       \\
       & \phantom{= (-1)^{\epsilon(t-1)}} + (-1)^{t-a} \sum_{k=1}^{l}
       (-1)^{k} \int_{M_{\xi\Gamma (\id, d_{k})}} \omega) 
    \end{split}
  \]
  If $M_{\xi\Gamma}$ is non-degenerate, the description of $\partial M_{\xi\Gamma}$ from
  Lemma~\ref{lem:bdryisd} lets us identify this with the integral of
  $\omega$ over $\partial M_{\xi \Gamma}$, giving
\[
  \begin{split}
\iota_{\ind{j},l,\xi}(d \zeta) & = \iota_{\ind{j},l,\xi}(d
  \Gamma \otimes \omega) + (-1)^{t} \iota_{\ind{j},l,\xi}(\Gamma
  \otimes d \omega) \\ & = (-1)^{\epsilon(t-1)} \int_{\partial
    M_{\xi\Gamma}} \omega + (-1)^{t}(-1)^{\epsilon(t)}
  \int_{M_{\xi\Gamma}} d\omega \\ & = 0,
  \end{split}
\]
since we have $\int_{M_{\xi\Gamma}} d\omega = \int_{\partial
  M_{\xi\Gamma}} \omega$ by Stokes's Theorem, and
\[ (-1)^{\epsilon(t-1)}+(-1)^{t}(-1)^{\epsilon(t)} = 0\]
using  $(-1)^{\epsilon(t)} = (-1)^{t-1}(-1)^{\epsilon(t-1)}$.

In the degenerate case, we have $\int_{M_{\xi\Gamma}} d\omega = 0$
since the dimension of $M_{\xi\Gamma}$ is smaller than that of
$d\omega$, while in the formula for $\iota_{\ind{j},l,\xi}(d\Gamma
\otimes \omega)$ all the terms vanish for the same reason except
potentially for 
those of the form
\[ \pm \left(\int_{M_{\xi\Gamma(d_{0}^{(i)},\id)}} \omega -
    \int_{M_{\xi\Gamma(d_{1}^{(i)},\id)}} \omega\right) \]
where the $i$th direction is degenerate. But in this case the manifolds 
$M_{\xi\Gamma(d_{0}^{(i)},\id)}$ and $M_{\xi\Gamma(d_{1}^{(i)},\id)}$
are equal and so the two integrals cancel. Thus
$\iota_{\ind{j},l,\xi}(d\Gamma \otimes \omega) = 0$, as required.
\end{proof}

\begin{lemma}
  The construction of Definition~\ref{defn:O,iota} determines a functor
  \[\mathcal{O}_{\ind{j},l} \colon
    \sBORDor_{d,d+n,\ind{j},l} \times \bbSj \times
  \sdDl^{\op}  \times [1] \to \Chk.\]
\end{lemma}
\begin{proof}
  By definition we have functoriality when we restrict to $0,1 \in
  [1]$, so it suffices to check that for a morphism $\lambda \colon
  \xi \to \eta$ in $\bbS_{\sd}^{\ind{j}} \times \sdDl^{\op}$
  (which means we have $\eta = \xi \lambda$)
  we have commutative squares
  \[
    \begin{tikzcd}
      \Omega^{*}(M_{\xi}) \arrow{r} \arrow{d} &
      \Omega^{*}(M_{\eta}) \arrow{d} \\
      \Pjl(\xi)[-d] \arrow{r} & \Pjl(\eta)[-d].
\end{tikzcd}
\]
By duality this amounts to the diagram
\[
\begin{tikzcd}
  {} & \Omega^{*}(M_{\xi}) \otimes \Pjl(\eta)^{\vee} \arrow{dl}
  \arrow{dr} \\
\Omega^{*}(M_{\xi}) \otimes \Pjl(\xi)^{\vee} \arrow{dr}  & &
\Omega^{*}(M_{\eta}) \otimes \Pjl(\eta)^{\vee} \arrow{dl} \\
 & \k[-d]
\end{tikzcd}
\]
commuting. To check that it does, it suffices to consider
$\omega \in \Omega^{d+t}(M_{\xi})$ and a generator $\Gamma$
of $(\Pjl(\eta)^{\vee})^{-t}$. The left composite takes
$\omega \otimes \Gamma$ to
\[ \iota_{\ind{j},l,\xi}(\lambda\Gamma \otimes \omega) =
  (-1)^{\epsilon(t)} \int_{M_{\xi\lambda\Gamma}} \omega|_{M_{\xi\lambda\Gamma}},\]
while the right composite takes this to
\[ \iota_{\ind{j},l,\eta}(\Gamma \otimes \omega|_{M_{\eta}}) =
  (-1)^{\epsilon(t)}
  \int_{M_{\eta\Gamma}} \omega|_{M_{\eta\Gamma}}. \]
Since we have $M_{\xi\lambda\Gamma} = M_{\eta\Gamma}$, these integrals
are indeed equal.
\end{proof}

\begin{lemma}
  The functors $\mathcal{O}_{\ind{j},l}$ are natural in
  $[l] \in \Dinjop$.
\end{lemma}
\begin{proof}
  We know from the naturality of $\scut_{\ind{j},l}$ and 
  $\Pjl$  in $l$ that $\mathcal{O}_{\ind{j},l}(0)$ and
  $\mathcal{O}_{\ind{j},l}(1)$ are natural in $l$. It therefore
  suffices to check that for $\phi \colon [k] \to [l]$ in $\Dinj$,
  $\xi = (\sigma,\tau) \in \bbSj \times \sdDl^{\op}$, and $(M, \oul{I}, f)\in
  \sBORDor_{d,d+n,\ind{j},l}$, 
  the diagram
  \[
    \begin{tikzcd}
      \Omega^{*}(M_{\xi}) \arrow{r}\arrow{d} &
      \Omega^{*}(M_{\xi'}) \arrow{d} \\
      \Pjl(\xi)[-d] \arrow{r} & P_{\ind{j},k}(\xi')[-d]
    \end{tikzcd}
  \]
  commutes, where $\xi' = \xi(\id,\phi)=(\sigma,\tau\phi)$. Taking duals, this is
  equivalent to the square
  \[
  \begin{tikzcd}
    & P_{\ind{j},k}(\xi')^{\vee} \otimes \Omega^{*}(M_{\xi})
    \arrow{dl} \arrow{dr} \\
    \Pjl(\xi)^{\vee} \otimes \Omega^{*}(M_{\xi})
    \arrow{dr} & &  P_{\ind{j},k}(\xi')^{\vee} \otimes
    \Omega^{*}(M_{\xi'}) \arrow{dl} \\
    & \k[-d]
  \end{tikzcd}
\]
commuting. To see this it suffices to consider a generator $\Gamma$ of
$(P_{\ind{j},k}(\xi')^{\vee})^{-t}$ and $\omega \in
\Omega^{d+t}(M_{\xi})$. Since we have an identification $P_{\ind{j},k}(\xi')^{\vee} \cong
\Pjl(\xi)^{\vee} $ under which $\Gamma$ corresponds to
$\Gamma' = (\id,\phi)\Gamma$,
the composite on the left is
\[ (-1)^{\epsilon(t)}\int_{M_{\xi\Gamma'}} \omega|_{M_{\xi\Gamma'}}, \]
while that on the right is
\[ (-1)^{\epsilon(t)}\int_{M_{\xi'\Gamma}}
  \omega|_{M_{\xi'\Gamma}}. \]
Here $\xi'\Gamma = \xi(\id,\phi)\Gamma = \xi \Gamma'$ so these are
indeed equal.
\end{proof}

This implies that the maps $\mathcal{O}_{\ind{j},l}$ for varying
$l$ determine a morphism of semisimplicial sets
\[ \sorcut_{\ind{j}} \colon \sBORDor_{d,d+n,\ind{j}} \to
  \Xdj.\]

\begin{lemma}
  For any morphism $\Phi \colon \ind{i} \to \ind{j}$ in $\simp^{n}$,
  the square
  \[
  \begin{tikzcd}[column sep=large]
    \sBORDor_{d,d+n,\ind{j}} \arrow{d}{\Phi^{*}} \arrow{r}{\sorcut_{\ind{j}}} 
    & \Xdj \arrow{d}{\Phi^{*}} \\
    \sBORDor_{d,d+n,\ind{i}}  \arrow{r}{\sorcut_{\ind{i}}} 
    & \Xd_{\ind{i}}
  \end{tikzcd}
  \]
  commutes.
\end{lemma}
\begin{proof}
  For $\sigma \in \bbS^{\ind{i}}$, viewed as an inert map $\ind{a} \to
  \ind{i}$, the object $\sigma' = \bbS(\Phi)(\sigma)$ corresponds to
  the inert part of the active-inert factorization of $\Phi \circ
  \sigma \colon \ind{a} \to \ind{j}$. 
Given $\tau \in \sdDl$ and $(M, \oul{I},f) \in
\sBORDor_{d,d+n,\ind{j}}$, we have by definition
\[ (\Phi^{*}M)_{(\sigma,\tau)} = M_{(\sigma',\tau)}.\]
  Unwinding the definitions, we must check
  that the following diagram commutes:
  \[
    \begin{tikzcd}
      \Omega^{*}(M_{(\sigma',\tau)}) \arrow{r} \arrow[equals]{d}& \Pjl(\sigma',\tau) \arrow{d}{\alpha_{\Phi,l}} \\
      \Omega^{*}((\Phi^{*}M)_{(\sigma,\tau})) \arrow{r} &
      P_{\ind{i},l}(\sigma,\tau), 
    \end{tikzcd}
  \]
  where $\alpha_{\Phi,l}$ is the map from Definition~\ref{defn:Pjl}.
  Equivalently, the dual square
  \[
    \begin{tikzcd}
      \Omega^{*}(M_{(\sigma',\tau)}) \otimes P_{\ind{i},l}(\sigma,\tau)^{\vee}
      \arrow{r}{\id \otimes \alpha_{\Phi,l}^{\vee}} \arrow[equals]{d} &
      \Omega^{*}(M_{(\sigma',\tau)}) \otimes \Pjl(\sigma',\tau)^{\vee}
      \arrow{d} \\
      \Omega^{*}((\Phi^{*}M)_{(\sigma,\tau)}) \otimes
      P_{\ind{i},l}(\sigma,\tau)^{\vee} \arrow{r} & \k[-d]
    \end{tikzcd}
  \]
  should commute. If $\Phi$ is inert then this is immediate since  $\sigma' = \Phi \circ
  \sigma$ and $\alpha_{\Phi,l}^{\vee}$ is an isomorphism. By
  functoriality in $\Phi$ it therefore suffices to check this when
  $\Phi$ is active, for which in turn it is enough to consider the
  cases where $\Phi$ is given by an inner face map or a degeneracy in one coordinate and
  by identities in the others; for notational convenience we'll just check
  the case where the first coordinate is non-trivial.

  First suppose $\Phi = (d_{l},\id,\ldots,\id)$ for
  $d_{l} \colon [i_{1}] \to [i_{1}+1]$.  If
  $\sigma = (\sigma_{1},\sigma_{+})$ then
 \[\sigma' =
   \begin{cases}
     (\sigma'_{1}, \sigma_{+}), & \sigma_{1}(0) \leq l-1 < l \leq \sigma_{1}(a_{1}), \\
     (d_{l} \circ \sigma_{1}, \sigma_{+}), & \text{otherwise,}
   \end{cases}
 \]
 where $\sigma'_{1}\colon [a_{1}+1] \to [i_{1}+1]$ is the inert map
 determined by setting $\sigma'_{1}(0)=\sigma_{1}(0)$.
  Given a generator $\Gamma = (\gamma,\delta) = (\gamma_{1},\ldots,\gamma_{n},\delta)$ in $
  (P_{\ind{i},l}(\sigma,\tau)^{\vee})^{-t}$ and writing $\Gamma' =
  (\gamma_{2},\ldots,\gamma_{n},\delta)$, we have
  \[\alpha_{\Phi,l}^{\vee}\Gamma =
    \begin{cases}
      (\gamma_{1}^{-},\Gamma') + (\gamma_{1}^{+},\Gamma'), & \sigma_{1}\gamma_{1} =
      \rho_{l}, \\
      (\gamma'_{1},\Gamma'), & \text{otherwise}.
    \end{cases}
  \]
  where $\sigma'_{1}\gamma_{1}^{-} = \rho_{l},
  \sigma'_{1}\gamma_{1}^{+}=\rho_{l+1}$, $\sigma'_{1}\gamma'_{1} =
  d_{l}\sigma_{1}\gamma_{1}$, and $\rho_{i}$ denotes the inclusion of
  the subset $\{i-1,i\}$.
  We abbreviate $\Upsilon := (\sigma_{+},\tau)\Gamma'$. Given
  $\omega \in \Omega^{d+t}(M_{\sigma',\tau})$, there are thus two
  cases two consider: if either the image of $\sigma_{1}$ doesn't
  contain both $l-1$ and $l$ or $\sigma_{1}\gamma_{1}\neq \rho_{l}$
  then the required identity
\[ \int_{M_{(\sigma'_{1}\gamma'_{1}, \Upsilon)}} \omega
  = \int_{(\Phi^{*}M)_{(\sigma_{1}\gamma_{1}, \Upsilon)}}
  \omega \]
holds since $\sigma'_{1}\gamma'_{1} = d_{l}\sigma_{1}\gamma_{1}$ and so the two
manifolds are identical. In the other case, where
$\sigma_{1}\gamma_{1}=\rho_{l}$, we instead have
\[ (\Phi^{*}M)_{(\rho_{l}, \Upsilon)} =
M_{(\rho_{l},\Upsilon)} \cup
M_{(\rho_{l+1},\Upsilon)},\]
so that 
\[ \int_{(\Phi^{*}M)_{(\sigma_{1}\gamma_{1}, \Upsilon)}}
  \omega = \int_{M_{(\rho_{l},\Upsilon)}} \omega +
  \int_{M_{(\rho_{l+1},\Upsilon)}} \omega
  = \int_{M_{(\sigma'_{1}\gamma^{-}_{1}, \Upsilon)}}
  \omega + \int_{M_{(\sigma'_{1}\gamma^{+}_{1}, \Upsilon)}} \omega,\]
as required.

Now consider the case where $\Phi = (s_{l},\id,\ldots,\id)$ for $s_{l}
\colon [i_{1}] \to [i_{1}-1]$.
Then
 \[\sigma' =
   \begin{cases}
     (\sigma'_{1}, \sigma_{+}), & \sigma_{1}(0) \leq l < l+1 \leq \sigma_{1}(a_{1}) \\
     (s_{l} \circ \sigma_{1}, \sigma_{+}), & \text{otherwise,}
   \end{cases}
 \]
 where $\sigma'_{1} \colon [a_{1}-1]\to [i_{1}-1]$ is the inert map
 determined by setting $\sigma'_{1}(0)=\sigma_{1}(0)$. Given a generator
 $\Gamma$ in $(P_{\ind{i},l}(\sigma,\tau)^{\vee})^{-t}$ as before, we
 have
  $\alpha_{\Phi,l}^{\vee}\Gamma = (\gamma_{1}', \Gamma')$
 where $\sigma_{1}'\gamma_{1}' = s_{l}\sigma_{1}\gamma_{1}$ if
 $\sigma_{1}\gamma_{1} \neq \rho_{l+1}$, while in the case where
 $\sigma_{1}\gamma_{1} = \rho_{l+1}$ we have
 $\alpha_{\Phi,l}^{\vee}\Gamma = 0$.

 There are again two cases to
 consider: if $\sigma_{1}\gamma_{1} \neq \rho_{l+1}$ then we have
 \[ \int_{M_{(\sigma'_{1}\gamma'_{1}, \Upsilon)}} \omega
  = \int_{(\Phi^{*}M)_{(\sigma_{1}\gamma_{1}, \Upsilon)}}
  \omega \]
since $\sigma'_{1}\gamma'_{1} = s_{l}\sigma_{1}\gamma_{1}$, which
means the two manifolds are identical. On the other hand, if 
$\sigma_{1}\gamma_{1} = \rho_{l+1}$ then as required we have
\[ \int_{(\Phi^{*}M)_{(\sigma_{1}\gamma_{1}, \Upsilon)}}
  \omega = 0\]
for degree reasons: since $s_{l}\sigma_{1}\gamma_{1}$ is constant
(at $l$) the dimension of the manifolds we integrate over cannot equal
the degree of $\omega$.
\end{proof}

\begin{remark}
  In the simplest case, where $n = 1$ and we consider the inner face
  map $d_{1} \colon [1] \to [2]$, we have shown that the square
  \[
    \begin{tikzcd}
      \Omega^{*}(X) \arrow[equals]{r} \arrow{d} & \Omega^{*}(X) \arrow{d} \\
      N^{\spi}_{2}[-d] \arrow{r} & N^{\spi}_{1}[-d],
    \end{tikzcd}
  \]
  commutes for a composition $X = X_{1} \cup X_{2}$ in the bordism
  direction, because we have
  $\int_{X} = \int_{X_{1}} + \int_{X_{2}}$.
\end{remark}

Putting the preceding results together, we have shown that we have
defined a natural family of maps of semisimplicial sets
\[ \sBORDorsemi_{d,n,\ind{j}} \xto{\sorcut_{\ind{j}}} \Xdj, \]
where by construction we have a commutative square
\[
  \begin{tikzcd}
    \sBORDorsemi_{d,n,\ind{j}} \arrow{r}{\sorcut_{\ind{j}}} \arrow{d}{\scut_{\ind{j}}} & \Xdj
    \arrow{d} \\
    (\Ex^{1}\mathrm{N}\,\Fun(\bbSj, \Mfd)^{(W)})^{\semi} \arrow{r}{\Omega^{*}} & (\Ex^{1}\mathrm{N}\,\Fun(\bbSj, \Chk)^{(W)})^{\semi}.
  \end{tikzcd}
  \]
As we saw in Remark~\ref{rmk:functortoXdj}, this implies that we get a
functor of $n$-uple Segal spaces:
\begin{proposition}
  The maps $\sorcut_{\ind{j},l}$  induce functors
  \[\orcut_{d,n}\colon \BORDor_{d,d+n} \to
    \COSPAN_{n}(\POr_{d}^{\mathcal{S}}) \]
  of $n$-uple Segal spaces, and hence functors
  \[\orcut_{d,n}\colon \Bordor_{d,d+n} \to
    \Cospan_{n}(\POr_{d}^{\mathcal{S}}) \]
  of $n$-fold Segal spaces. \qed
\end{proposition}

\section{Symmetric Monoidal Structure}\label{subsec:orsymmmon}
In this section we will show that the functors $\orcut_{d,n}$ we
defined above also have a symmetric monoidal structure, which is moreover compatible
with that we previously constructed on the functors $\cut_{d,n}$. We
prove this by showing that $\orcut_{d,n}$ is a delooping of
$\orcut_{d+1,n-1}$. Using ``reduced'' versions of our strict models
from before, corresponding to passing from $n$-uple to $n$-fold Segal
spaces, we will construct commutative diagrams of semisimplicial sets
of the following form:
\[
  \begin{tikzcd}[column sep=-0.2cm]
    \sBordorsemi_{d,d+n,\ponej}(\emptyset,\emptyset) \arrow{rr}
    \arrow{d}{\sorcut^{\pred}_{d,n,\ponej}(\emptyset,\emptyset)} \arrow[bend right=80]{ddd}[near start]{\Omega^{*}\scut_{d,n,\ponej}(\emptyset,\emptyset)}& &
    \sBordorsemi_{d+1,d+n,\ind{j}} \arrow{d}[swap]{\sorcut^{\pred}_{d+1,n-1,\ind{j}}} \arrow[bend
    left=80]{ddd}[near start,swap]{\Omega^{*}\scut_{d+1,n-1,\ind{j}}} \\
    \Xdr_{\ponej}(0,0)  \arrow[bend right=60]{dd} & & \Xdpr_{\ind{j}} \arrow{ll} \arrow[bend left=60]{dd}\\
    \Ydr_{\ponej}(0,0) \arrow{u} \arrow{d} \arrow{r}{\sim} &
    \Zd_{\ponej} \arrow{dl} \arrow{ul}
    &  \Ydpr_{\ind{j}} \arrow{d} \arrow{l}{\sim} \arrow{u} \\
    \Ex^{1}\mathrm{N}\,\Fun_{(\txt{red})}(\bbS^{\ponej},
    \Chk)^{(W)}_{(0,0)}& &
    \Ex^{1}\mathrm{N}\,\Fun_{(\txt{red})}(\bbSj, \Chk)^{(W)} \arrow{ll}{\cong}.
  \end{tikzcd}
\]
natural in $\ind{j}$.
To see that this gives a symmetric monoidal structure, we extend the
middle part of the diagram to a commutative diagram of
semisimplicial simplicial sets of the following form:
\[
  \begin{tikzcd}
    \disc \mathfrak{Y}_{\ind{j}}^{(d+1,\txt{red})} \arrow{rr}[near start]{(\sim)} \arrow{dr}
    \arrow{dd}[swap]{(\sim)} & &
    \widetilde{\mathfrak{Y}}_{\ind{j}}^{(d+1,\txt{red})} \arrow{dd}[near start]{\sim} \arrow{dr}{\sim} \\
    & \disc \mathfrak{X}_{\ind{j}}^{(d+1,\txt{red})} \arrow[crossing over]{rr} & &
    \tX_{\ind{j}}^{(d+1,\txt{red})}  \\
   \disc \mathfrak{Z}_{\ponej}^{(d)} \arrow{rr}[near start]{(\sim)} & &
   \widetilde{\mathfrak{Z}}_{\ponej}^{(d)} \arrow{ddr}{\sim} \\
   \disc \mathfrak{Y}_{\ponej}^{(d,\txt{red})}(0,0) \arrow{u}{(\sim)} \arrow{rr}[near start]{(\sim)} \arrow{dr} & &
   \widetilde{\mathfrak{Y}}_{\ponej}^{(d,\txt{red})}(0,0) \arrow{dr}[swap]{\sim} \arrow{u}[swap]{\sim} \\
    & \disc \mathfrak{X}_{\ponej}^{(d,\txt{red})}(0,0) \arrow{rr}
    \arrow[crossing over,leftarrow]{uuu}  \arrow[crossing
    over,leftarrow]{uul}& &
   \tX_{\ponej}^{(d,\txt{red})}(0,0) \arrow[crossing over,leftarrow]{uuu}{\sim}
  \end{tikzcd}
\]
In this diagram we will show that the maps labelled $\sim$ are levelwise weak
equivalences and those labelled $(\sim)$ give weak equivalences on
geometric realization.

Combining the two diagrams, we get on the \icatl{} level the following
commutative diagram of $(n-1)$-fold Segal spaces:
\[
  \begin{tikzcd}
  \Bordor_{d,d+n}(\emptyset,\emptyset) \arrow{r}{\sim}
  \arrow{d}{\orcut_{d,n}}  \arrow[bend right=85]{dd}[swap]{\cut_{d,n}} &
  \Bordor_{d+1,d+n} \arrow{d}[swap]{\orcut_{d+1,n-1}}  \arrow[bend left=85]{dd}{\cut_{d+1,n-1}}\\
  \Cospan_{n}(\POr_{d}^{\mathcal{S}})(\emptyset,\emptyset)
  \arrow{r}{\sim}  \arrow{d} & \Cospan_{n-1}(\POr_{d+1}) \arrow{d} \\
  \Cospan_{n}(\mathcal{S})(\emptyset,\emptyset) \arrow{r}{\sim} &
  \Cospan_{n-1}(\mathcal{S}).
\end{tikzcd}
\]
Taken together for all $d$ and $n$, this encodes the desired symmetric
monoidal structure on $\orcut$ and moreover shows that this is
compatible with that on the unoriented cutting functors.

We will now define the objects that appear in our strict diagrams
above. For this we use the following notation:
\begin{definition}
  $\Fun_{(\txt{red})}(\bbSj, \Chk)$ is the full
  subcategory of $\Fun(\bbSj, \Chk)$ spanned by the
  functors that are \emph{reduced} in the sense that they take the
  morphisms in $R_{\ind{j}}$ from Definition~\ref{def:redbbSj} to weak
  equivalences.
\end{definition}

\begin{remark}
  By Corollary~\ref{cor:relcatspanmodel}, the simplicial set
  $\Nrv \Fun_{\pred}(\bbSj, \Chk)^{(W)}$ models the
  sub-$\infty$-groupoid of $\Map(\bbSj, \Mod_{\k})$ spanned
  by the reduced functors, as does the weakly equivalent simplicial
  set $\Ex^{1} \Nrv \Fun_{\pred}(\bbSj, \Chk)^{(W)}$. The
  \igpd{} $\Cospan_{n}(\Mod_{\k})_{\ind{j}}$ is in turn contained in
  this as the sub-\igpd{} of cartesian reduced diagrams by
  Remark~\ref{rmk:Spanasredftr}.
\end{remark}

\begin{definition}
  We define
  $\Fun_{\pred}(\bbS^{\ponej}, \Chk)_{(0,0)}$ to be
  the full subcategory of the category
  $\Fun_{\pred}(\bbS^{\ponej}, \Chk)$ spanned by
  those functors $\bbS^{1}\times \bbSj \to \Chk$ whose restriction to $\{(0,0), (1,1)\} \times \bbSj$ is
  constant at $0$. We define a functor
  \[Z_{\ind{j}} \colon \Fun_{\pred}(\bbSj, \Chk) \to
    \Fun_{\pred}(\bbS^{\ponej}, \Chk)_{(0,0)}\] by
  taking the right Kan extension along the inclusion
  \[ \bbSj\times \sdDl \cong \{(0,1)\} \times
    \bbSj \times \sdDl \hookrightarrow
    \bbS^{\ponej}\times \sdDl,\] which amounts to
  adding two copies of the unique map to the terminal object (which is
  the
  constant functor with value $0$). This is clearly an isomorphism of
  categories, with inverse the restriction to
  $\{(0,1)\} \times \bbSj$, and is moreover natural in
  $\ind{j}$.
\end{definition}

\begin{remark}
  The restriction
  \[\Fun(\bbS^{\ponej}, \Chk)^{(W)} \to (\Fun(\bbSj,
  \Chk)^{(W)})^{\times 2}\] to
  $\{(0,0),(1,1)\} \times \bbSj$ is a cocartesian fibration, and the fibre
  at $(F,G)$ can be identified with the category
  $(\Fun(\bbSj, \Chk)_{/F,G})^{(W)}$, with the cocartesian pushforward
  functors given by composition. Since we only consider weak
  equivalences in the base, these functors induce weak
  equivalences on classifying spaces. It therefore follows from
  Quillen's Theorem B~\cite{QuillenK} that the simplicial set
  $\Nrv \Fun_{\pred}(\bbS^{\ponej}, \Chk)_{(0,0)}^{(W)}$
  models the \igpd{} $\Map(\bbS^{\ponej}, \Mod_{\k})_{(0,0)}$
  that is the fibre of the corresponding restriction
  \[ \Map(\bbS^{\ponej}, \Mod_{\k}) \to \Map(\bbSj,
    \Mod_{\k})^{\times 2} \]
  at two copies of the
  constant functor with value $0$. Moreover, the restriction of the
  isomorphism $Z_{\ind{j}}$ and its inverse model the equivalence
  \[ \Map(\bbSponej,
    \Mod_{\k})_{(0,0)} \isoto \Map(\bbSj, \Mod_{\k}) \]
  obtained by restricting to $\{(0,1)\} \times \bbSj$.
\end{remark}

\begin{definition}
  We define \[ \Fun_{\pred}(\bbSj, \Chkd), \quad
  \Fun_{\pred}(\bbS^{\ponej}, \Chkd)_{(0,0)}\] as the full
  subcategories of $\Fun(\bbSj, \Chkd)$ and
  $\Fun(\bbS^{\ponej}, \Chkd)$ containing those objects whose
  compositions with the forgetful functor to $\Chk$ lie in
  $\Fun_{\pred}(\bbSj, \Chk)$ and
  $\Fun_{\pred}(\bbS^{\ponej}, \Chk)_{(0,0)}$, respectively. We
  then define 
  \[ \Ydrj := \Ex^{1} \Nrv \Fun_{\pred}(\bbSj, \Chkd)^{(W)},\]
  \[ \Ydr_{\ponej}(0,0) := \Ex^{1} \Nrv
    \Fun_{\pred}(\bbS^{\ponej}, \Chkd)_{(0,0)}^{(W)}. \]
\end{definition}

\begin{remark}
  We would like to use the functor $Z_{\ind{j}}$ to define maps of
  simplicial sets \[\Ydpr_{\ind{j}} \to \Ydr_{\ponej}(0,0)\] that model
  the equivalence
  \[   \Map(\bbS^{\ponej},
    \Mod_{\k/\k[-d-1]})_{(0,0)} \isoto \Map(\bbSj, \Modkd),\]
  which uses that $\k[-d-1]$ is the pullback of
  \[ 0 \to \k[-d] \from 0.\]
  However, we can't model this homotopy pullback
  just using the chain complexes $\k[-d]$ and $\k[-d-1]$, which forces
  us to instead define a zig-zag of functors, as follows:
\end{remark}

\begin{definition}
  Let $Q_{\ind{j}} \colon \bbS^{\ponej} \to \Chk$ denote the composite
  \[ \bbS^{1} \times \bbSj \xto{\txt{proj}} \bbS^{1}
    \xto{\mathcal{N}^{\spi}_{1}} \Chk.\]
  We then define \[\Fun_{\pred}(\bbS^{\ponej},
    \Chk)_{/Q_{\ind{j}[-d]},(0,0)} \subseteq \Fun(\bbS^{\ponej},
    \Chk)_{/Q_{\ind{j}}[-d]}\] to be the full subcategory spanned by the objects whose image in $\Fun(\bbS^{\ponej},
    \Chk)$ lies in $\Fun_{\pred}(\bbS^{\ponej},
    \Chk)_{(0,0)}$. Now consider the natural transformations
  \[ Z_{0}(\k[-1]) \xto{\ell} \mathcal{N}^{\spi}_{1} \xfrom{r} \k\] depicted
  as
  \[
    \begin{tikzcd}
      0  \arrow{r} & \k[0] & \k[0] \arrow[equals]{l} \\
      \k[-1]  \arrow{u} \arrow{d} \arrow{r}{\ell_{(0,1)}} & N_{1}^{\spi}\arrow{u} \arrow{d}  & \k[0]\arrow[equals]{u} \arrow[equals]{d} 
      \arrow{l}[swap]{r_{(0,1)}} \\
      0 \arrow{r} & \k[0]  & \k[0] \arrow[equals]{l}
    \end{tikzcd}
  \]
  where $\ell_{(0,1)} \colon \k[-1] \to N_{1}^{\spi}$ is simply the
  identity in cohomological degree $1$ and
  $r_{(0,1)}\colon \k[0] \to N_{1}^{\spi}$ is the diagonal
  $\k \to \k \oplus \k$ in cohomological degree $0$. Here $r$ is a
  natural weak equivalence, while $\ell$ exhibits $\k[-1]$ as the
  homotopy pullback of two copies of the zero map $0 \to \k[0]$.

  From $\ell$ and $r$ we get
  natural transformations
  \[ Z_{\ind{j}}(\k[-d-1]) \xto{\ell_{\ind{j}}[-d]} Q_{\ind{j}}[-d]
    \xfrom{r_{\ind{j}}[-d]} \k[-d],\]
  which in turn induce functors
  \[
    \begin{tikzcd}
      \Fun_{\pred}(\bbSj, \Ch_{\k/\k[-d-1]}) \arrow{d}{Z'_{\ind{j}}} \\
      \Fun_{\pred}(\bbS^{\ponej},
      \Chk)_{/Q_{\ind{j}[-d]},(0,0)} \\
      \Fun_{\pred}(\bbS^{\ponej},
    \Chkd)_{(0,0)}, \arrow{u}{R_{\ind{j}}}
    \end{tikzcd}
  \]
  where $Z'_{\ind{j}}$ is given by applying $Z_{\ind{j}}$ and then
  composing with $\ell_{\ind{j}}$, while $R_{\ind{j}}$ is given by
  composition with $r_{\ind{j}}$. If we define
  \[ \Zd_{\ponej} := \Ex^{1} \Nrv \Fun_{\pred}(\bbS^{\ponej},
    \Chk)_{/Q_{\ind{j}},(0,0)}^{(W)},\]
  then $Z'_{\ind{j}}$ and $R_{\ind{j}}$ induce natural maps of
  simplicial sets
  \[ \Ydpr_{\ind{j}} \xto{Z'_{\ind{j}}} \Zd_{\ponej}
    \xfrom{R_{\ind{j}}} \Ydr_{\ponej}(0,0),\]
  both of which are weak equivalences.
\end{definition}
Passing to \igpds{}, this zig-zag of weak equivalences models the
equivalence
\[ \Map_{\pred}(\bbSj, \Mod_{\k/\k[-d-1]}) \simeq
  \Map_{\pred}(\bbS^{\ponej}, \Modkd)_{(0,0)},\]
which gives the equivalence
\[ \Cospan_{n}(\Mod_{\k/\k[-d-1]})_{\ind{j}} \simeq
  \Cospan_{n+1}(\Modkd)(0,0)_{\ind{j}} \]
upon restricting to the cartesian diagrams.

So far we have constructed the bottom part of the first diagram
above. Next, it is convenient to consider the second diagram, starting
with the following definitions:
\begin{definition}
  We define $\Fun_{\pred}(\bbSj \times \sdDl, \Chk)$ to be the full
  subcategory of the category $\Fun(\bbSj \times \sdDl, \Chk)$ containing the
  functors that take the morphisms in $R_{\ind{j}} \times \sdDl$ to
  weak equivalences. Next, we define
  \[ \Fun_{\pred}(\bbS^{\ponej} \times \sdDl, \Chk)_{(0,0)} \subseteq
\Fun_{\pred}(\bbS^{\ponej} \times \sdDl, \Chk) \]
   to be the
  full subcategory
  spanned by those functors
  $\bbS^{1}\times \bbSj \times \sdDl \to \Chk$ such that the
  restriction to $\{(0,0), (1,1)\} \times \bbSj \times \sdDl$ is
  constant at $0$. As in the case $l = 0$, we have an isomorphism of
  categories
  \[ Z_{\ind{j},l} \colon \Fun_{\pred}(\bbSj \times \sdDl, \Chk) \to
    \Fun_{\pred}(\bbS^{\ponej} \times \sdDl, \Chk)_{(0,0)} \]
  that adds two copies of the unique map to the terminal object.
\end{definition}

\begin{definition}
  Let $\Fun_{\pred}(\bbSj \times \sdDl, \Chkd)$ and
  $\Fun_{\pred}(\bbS^{\ponej} \times \sdDl, \Chkd)_{(0,0)}$ be the full
  subcategories of $\Fun(\bbSj \times \sdDl, \Chkd)$ and
  $\Fun(\bbS^{\ponej} \times \sdDl, \Chkd)$ containing those objects whose
  compositions with the forgetful functor to $\Chk$ lie in the full subcategories
  $\Fun_{\pred}(\bbSj \times \sdDl, \Chk)$ and
  $\Fun_{\pred}(\bbS^{\ponej} \times \sdDl, \Chk)_{(0,0)}$, respectively. We
  then define families of bisimplicial sets $\tYdrj$ and $\tYdr_{\ponej}(0,0)$ by
  \[ \tYdrjl := \Nrv \Fun_{\pred}(\bbSj \times \sdDl, \Chkd)^{(W)},\]
  \[ \tYdr_{\ponej,l}(0,0) := \Nrv
    \Fun_{\pred}(\bbS^{\ponej} \times \sdDl, \Chkd)_{(0,0)}^{(W)}. \]
  Note that taking the underlying simplicial sets of 0-simplices we
  get $\Ydrj$ and $\Ydr_{\ponej}(0,0)$, respectively.
\end{definition}

\begin{definition}
  Let $Q_{\ind{j},l} \colon \bbS^{\ponej} \times \sdDl \to \Chk$ denote the
  composite
  \[ \bbS^{\ponej} \times \sdDl \to \bbS^{\ponej} \xto{Q_{\ind{j}}} \Chk.\]
  We define $\Fun_{\pred}(\bbS^{\ponej} \times \sdDl, \Chk)_{/Q_{\ind{j},l}[-d],(0,0)}$ as the full
  subcategory of the slice category 
  $\Fun(\bbS^{\ponej} \times \sdDl, \Chk)_{/Q_{\ind{j},l}[-d]}$ containing those objects whose
  images in $\Fun(\bbS^{\ponej} \times \sdDl, \Chk)$ lie in
  $\Fun_{\pred}(\bbS^{\ponej} \times \sdDl, \Chk)_{(0,0)}$. We
  then define a family of bisimplicial sets $\tZd_{\ponej}$ by
  \[ \tZd_{\ponej,l} := \Nrv
    \Fun_{\pred}(\bbS^{\ponej} \times \sdDl, \Chk)_{/Q_{\ind{j},l}[-d],(0,0)}^{(W)}. \]
  Note that taking the underlying simplicial set of 0-simplices we
  get  $\Zd_{\ponej}$.  
\end{definition}

\begin{definition}
  Composing with the projection $\bbS^{\ponej}\times \sdDl \to
  \bbS^{\ponej}$, the natural transformations  $\ell_{\ind{j}}$ and $r_{\ind{j}}$ induce
  natural transformations
  \[ Z_{\ind{j},l}(\k[-d-1]) \xto{\ell_{\ind{j},l}[-d]} Q_{\ind{j},l}[-d]
    \xfrom{r_{\ind{j},l}[-d]} \k[-d],\]
  which in turn induce functors
  \[
    \begin{tikzcd}
      \Fun_{\pred}(\bbSj \times \sdDl, \Ch_{\k/\k[-d-1]})
      \arrow{d}{Z'_{\ind{j},l}} \\
      \Fun_{\pred}(\bbS^{\ponej} \times \sdDl,
      \Chk)_{/Q_{\ind{j}[-d]},(0,0)} \\
      \Fun_{\pred}(\bbS^{\ponej}\times \sdDl,
    \Chkd)_{(0,0)}, \arrow{u}{R_{\ind{j},l}}
    \end{tikzcd}
    \]
  where $Z'_{\ind{j},l}$ is given by applying $Z_{\ind{j},l}$ and then
  composing with $\ell_{\ind{j},l}[-d]$, while $R_{\ind{j},l}$ is given by
  composition with $r_{\ind{j},l}[-d]$. These give  natural maps of
  bisimplicial sets
  \[ \tYdpr_{\ind{j},l} \xto{Z'_{\ind{j},l}} \Zd_{\ponej,l}
    \xfrom{R_{\ind{j},l}} \Ydr_{\ponej,l}(0,0),\]
  both of which are weak equivalences.
\end{definition}

\begin{definition}
  We define the categories
  \[ \Fun_{\pred}(\bbSj \times \sdDl, \Chk)_{/\Pjl[-d]} \subseteq \Fun(\bbSj \times \sdDl, \Chk)_{/\Pjl[-d]},\] 
  \[ \Fun_{\pred}(\bbS^{\ponej} \times \sdDl, 
  \Chk)_{/P_{\ponej,l}[-d],(0,0)} \subseteq \Fun(\bbS^{\ponej} \times \sdDl, \Chk)_{/P_{\ponej,l}[-d]},\]  as the full
  subcategories of containing
  those objects whose source lies in
  \[ \Fun_{\pred}(\bbSj \times \sdDl, \Chk), \quad
  \Fun_{\pred}(\bbS^{\ponej} \times \sdDl, \Chk)_{(0,0)},\] respectively. We
  then define families of semisimplicial simplicial sets $\tXdrj$ and
  $\tXdr_{\ponej}(0,0)$ by
  \[ \tXdrjl := \Nrv \Fun_{\pred}(\bbSj \times \sdDl, \Chk)_{/\Pjl[-d]}^{(W)},\]
  \[ \tXdr_{\ponej,l}(0,0) := \Nrv
    \Fun_{\pred}(\bbS^{\ponej} \times \sdDl, \Chk)_{/P_{\ponej,l}[-d],(0,0)}^{(W)}. \]
  We define $\Xdrjl$ and $\Xdr_{\ponej,l}(0,0)$ to be the underlying
  semisimplicial sets of 0-simplices of these semisimplicial
  simplicial sets.
\end{definition}

\begin{definition}
  By definition, using the decomposition $\bbS^{\ponej} \times \sdDl
  \cong \bbS^{1} \times (\bbSj \times \sdDl)$ we have natural isomorphisms
  \[ P_{\ponej,l} \cong \mathcal{N}_{1}^{\spi}
    \otimes \Pjl,\]
  \[ Z_{\ind{j},l}(\Pjl) \cong Z_{0}(\k) \otimes \Pjl.\]
  Tensoring with the map $\ell \colon Z_{0}(\k[-1]) \to
  \mathcal{N}^{\spi}_{1}$ from above therefore gives natural
  transformations
  \[ \zeta^{(d)}_{\ind{j},l} \colon Z_{\ind{j},l}(\Pjl[-d-1]) \to
    P_{\ponej, l}[-d]. \]
  Using this we get a natural transformation
  \[ \Fun_{\pred}(\bbSj \times \sdDl, \Chk)_{/\Pjl[-d-1]} \to 
    \Fun_{\pred}(\bbS^{\ponej} \times \sdDl,
    \Chk)_{/P_{\ponej,l}[-d],(0,0)}\]
  by applying $Z_{\ind{j},l}$ and then composing with
  $\zeta^{(d)}_{\ind{j},l}$.
\end{definition}

Tensoring with the natural map $\k \to \Pjl$ we now get a commutative
diagram
\[
  \begin{tikzcd}
    Z_{\ind{j},l}(\Pjl[-d-1]) \arrow{rr} & & P_{\ponej,l}[-d] \\
    Z_{\ind{j},l}(\k[-d]) \arrow{u} \arrow{r} & Q_{\ind{j},l}[-d]
    \arrow{ur} & \k[-d] \arrow{u} \arrow{l}
  \end{tikzcd}
\]
of functors $\bbS^{\ponej} \times \sdDl \to \Chk$. Composing with
this, we get a natural commutative diagram of semisimplicial
simplicial sets
\[
  \begin{tikzcd}
    \tX^{(d+1),\txt{red}}_{\ind{j},l} \arrow{rr} & & \tXdr_{\ponej,l}(0,0)
    \\
    \tY^{(d+1),\txt{red}}_{\ind{j},l} \arrow{r} \arrow{u}& \tZd_{\ponej,l}\arrow{ur} &
    \tYdr_{\ponej,l}(0,0), \arrow{l} \arrow{u} 
  \end{tikzcd}
\]
where moreover all the maps are weak equivalences. This is the right-hand
face of the second diagram above. We get the bottom part of the first
diagram by taking $0$-simplices, and we get the entire second diagram using
the counit maps for the adjunction between $\disc$ and the $0$-simplex
functor. Here the maps
\[ \disc \Ydpr_{\ind{j}} \to \tYdpr_{\ind{j}},\]
\[ \disc \Zd_{\ponej} \to \tZd_{\ponej},\]
\[ \disc \Ydr_{\ponej}(0,0) \to \tYdr_{\ponej}(0,0)\] all give weak
equivalences on geometric realizations by applying
Lemma~\ref{lem:diagdiscwe} as in the proof of
Proposition~\ref{propn:Pdiagweqs}.

We are left with showing the following:
\begin{proposition}
We have a commutative square
\[
  \begin{tikzcd}
    \sBordor_{d,d+n,\ponej}(\emptyset,\emptyset) \arrow{r}
    \arrow{d}[swap]{\sorcut^{\pred}_{d,n,\ponej}(\emptyset,\emptyset)} &
    \sBordor_{d+1,d+n,\ind{j}} \arrow{d}{\sorcut^{\pred}_{d+1,n-1,\ind{j}}} \\
    \Xdr_{\ponej}(0,0)   & \Xdpr_{\ind{j}}. \arrow{l}
  \end{tikzcd}
\]
\end{proposition}
\begin{proof}
Unwinding all the definitions, this amounts to checking that for
\[ (M,\oul{I},f) \in \sBordor_{d,d+n,\ponej}(\emptyset,\emptyset)\] and $\xi
= (\sigma_{0},\xi')
\in \bbS^{1}\times \bbS^{\ind{j}} \times \sdDl$, the following square commutes:
\[
  \begin{tikzcd}
    Z_{\ind{j},l}(\Omega^{*}(M_{(\id_{[1]},\blank)})(\xi) \arrow{d} \arrow[equals]{r} &
    \Omega^{*}(M_{\xi}) \arrow{d} \\
    (Z_{\ind{j},l}P_{\ind{j},l}[-d-1])(\xi) \arrow{r}{\zeta_{\ind{j},l}^{(d)}(\xi)} & P_{\ponej,l}(\xi)[-d]
  \end{tikzcd}
\]
If $\sigma_{0} \in \bbS^{1}$ is $(0,0)$ or $(1,1)$ the top row is zero, and there
is nothing to check. In the remaining case, where $\sigma_{0}$ is
$(0,1)$, we want to show that the square
\[
  \begin{tikzcd}
    \Omega^{*}(M_{(\id_{[1]},\xi')}) \arrow{d} \arrow[equals]{r} &
    \Omega^{*}(M_{(\id_{[1]},\xi')}) \arrow{d} \\
    P_{\ind{j},l}(\xi')[-d-1] \arrow{r} & P_{\ponej,l}(\id_{[1]},\xi')[-d]
  \end{tikzcd}
\]
commutes, or dually that the triangle
\[
    \begin{tikzcd}
  \Omega^{*}(M_{(\id_{[1]},\xi')}) \otimes
  P_{\ponej,l}(\id_{[1]},\xi')^{\vee} \arrow{rr} \arrow{dr} & & 
  \Omega^{*}(M_{(\id_{[1]},\xi')}) \otimes P_{\ind{j},l}(\xi')^{\vee}[1]
  \arrow{dl} \\
  & \k[-d]
\end{tikzcd}
\]
commutes. The horizontal map is the identity on
$\Omega^{*}(M_{(\id_{[1]},\xi')}) \otimes P_{\ind{j},l}(\xi')^{\vee}$
tensored with the dual of the map $\ell_{(0,1)} \colon \k[-1] \to
N_{1}^{\spi}$ given by the identity in cohomological degree $1$. The
dual $\ell_{(0,1)}^{\vee} \colon C(\Dndspi^{1}) \to \k[1]$
is then given by the identity in cohomological degree $-1$ and $0$
elsewhere. Consider a generator $\Gamma$ in
$(P_{\ind{j},l}(\xi')^{\vee})^{-t}$, a generator $\alpha$ in
$C(\Dndspi^{1})^{-s}$ ($s=0$ or $1$) and $\omega \in
\Omega^{d+t+s}(M_{(\id_{[1]},\xi)})$. If $s = 0$ then
$\ell_{(0,1)}^{\vee}(\alpha) = 0$ and we must show that
\[ \int_{M_{(\alpha,\xi'\Gamma)}} \omega = 0;\]
this is true since $M_{(\alpha,\xi'\Gamma)} = \emptyset$ by
assumption. On the other hand, if $s = 1$ and $\alpha$ is the unique
generator in degree $1$, corresponding to $\id_{[1]}$, then both maps
take $\omega \otimes (\alpha, \Gamma)$ to
$\int_{M_{(\id_{1},\xi'\Gamma)}} \omega$, as required.
\end{proof}  

Putting everything together and passing to the \icatl{} level, we have proved:
\begin{corollary}\label{cor:orcutsymmon}
  The functor
  \[\orcut_{n} := \orcut_{n,0} \colon \Bord_{d,d+n} \to
  \Cospan_{n}(\POr_{d}^{\mathcal{S}})\] has a natural symmetric
  monoidal structure coming from the functors $\orcut_{n,d+r}$.\qed
\end{corollary}

\section{Non-Degeneracy and TFTs in Lagrangian
  Correspondences}\label{sec:cobnondeg}

In this section our main goal is to prove that the symmetric
monoidal functors
\[ \orcut_{d,n} \colon \Bordor_{d,d+n} \to
  \Cospan_{n}(\POr_{d}^{\mathcal{S}}) \] we just defined factor
through the subobject $\Or_{n}^{\mathcal{S},d}$, which allows us to
define topological field theories via the AKSZ construction by
composing with the symmetric monoidal functors we have previously
constructed. The key point to check is the following:
\begin{proposition}
  For every $(M,\oul{I},f) \in \BORDor_{d,d+n,\ind{j}}$, where each $j_{i}$
  is $0$ or $1$, the image
  $\orcut_{d,d+n}(M,\oul{I},f) \in
  \COSPAN_{n}(\POr_{d}^{\mathcal{S}})_{\ind{j}}$ is oriented.
\end{proposition}
\begin{proof}
  Let $m=\sum_{i}j_{i}$; then $\orcut_{d,d+n}(M,\oul{I},f)$ is a
  diagram
  $p \colon \Sp^{m} \to \POr_{d}^{\mathcal{S}}$, where $p(-\infty)$ is
  (the homotopy type of) the manifold $M$ together with its
  (pre)orientation. By Definition~\ref{def:orcospS}, we must show that
  for any $A \in 
  \CAlg(\k)$ and any dualizable object $\mathcal{E} \in \Fun(M,
  \Mod_{A})$ (\ie{} a local system of dualizable $A$-modules), the
  induced diagram
  \[ \hat{p}_{A,\mathcal{E}} \colon \tSp{}^{n} \to \Mod_{A} \]
  is non-degenerate. Using the criteria of
  Propositions~\ref{propn:nondeghalf} and \ref{propn:Spklimfibseq}
  this is equivalent to the induced commutative square
  \[
    \begin{tikzcd}
          \hat{p}_{A,\mathcal{E}}(-\infty) \arrow{r} \arrow{d} &
          \lim_{I \in \Sp^{n,\circ}} \hat{p}_{A,\mathcal{E}}(I)
          \arrow{d} \\
          0  \arrow{r} & \hat{p}_{A,\mathcal{E}}(\infty)[1-m]
    \end{tikzcd}
  \]
  being cartesian. Here we can identify
  $\hat{p}_{A,\mathcal{E}}(-\infty)$ with $C^{*}(M; \mathcal{E}^{\vee})$ and
  $\hat{p}_{A,\mathcal{E}}(\infty)$ with $C_{*}(M; \mathcal{E})[-d]$,
  while if we write $M_{I}$ for the appropriate value of the cutting
  functor we have
  \[ \lim_{I \in \Sp^{n,\circ}} \hat{p}_{A,\mathcal{E}}(I) \simeq
    \lim_{I \in \Sp^{n,\circ}} C^{*}(M_{I}; \mathcal{E}^{\vee}|_{M_{I}}). \]
  Here we can identify $\colim_{I \in \Sp^{n,\circ}} M_{I}$ with the
  homotopy type of the \emph{boundary} of $M$: if we compute this
  colimit in $\Top$ we get precisely the boundary, but this is an
  iterated pushout along maps that are relative cell complexes, hence
  this is also the homotopy colimit. We then have
  \[
    \begin{split}
\lim_{I \in \Sp^{n,\circ}} C^{*}(M_{I}; \mathcal{E}^{\vee}|_{M_{I}}) & 
    \simeq
    \lim_{I \in \Sp^{n,\circ}} \lim_{q \in M_{I}} \mathcal{E}^{\vee}_{q} \\
    & \simeq \lim_{q \in \colim_{I \in \Sp^{n,\circ}} M_{I}}
    \mathcal{E}^{\vee}_{q} \\ &  \simeq \lim_{q \in \partial M} \mathcal{E}^{\vee}_{q} \simeq C^{*}(\partial M; \mathcal{E}^{\vee}|_{\partial M}),
    \end{split}
  \]
  where the second equivalence holds since we can first rewrite the iterated
  limit as a limit over the right fibration $\mathcal{F} \to \Sp^{n,\circ}$ for the functor $I
  \mapsto M_{I}$ and then note that the resulting functor from
  $\mathcal{F}$ takes all morphisms to equivalences and hence factors
  through the coinitial map
  \[ \mathcal{F} \to \|\mathcal{F}\| \simeq \colim_{I \in
      \Sp^{n,\circ}} M_{I}.\]
  We can thus identify the square as
  \[
    \begin{tikzcd}
          C^{*}(M; \mathcal{E}^{\vee}) \arrow{r} \arrow{d} &
          C^{*}(\partial M; \mathcal{E}^{\vee}|_{\partial M})
          \arrow{d} \\
          0  \arrow{r} & C_{*}(M; \mathcal{E})[1-m-d]
    \end{tikzcd}
  \]
  Here $M$ is an oriented $m+d$-manifold with corners, but we can
  smoothe out the corners and replace $M$ by an oriented $m+d$-manifold
  with boundary $M'$ that gives rise to an equivalent square
  \[
    \begin{tikzcd}
          C^{*}(M'; \mathcal{E}^{\vee}) \arrow{r} \arrow{d} &
          C^{*}(\partial M'; \mathcal{E}^{\vee}|_{\partial M'})
          \arrow{d} \\
          0  \arrow{r} & C_{*}(M'; \mathcal{E})[1-m-d].
    \end{tikzcd}
  \]
  This is cartesian by Poincar\'e--Lefschetz duality, as we discussed
  in Example~\ref{ex:cobordnondeg}.
\end{proof}

\begin{corollary}
  The morphism \[\orcut_{d,d+n} \colon \BORDor_{d,d+n} \to
  \COSPAN_{n}(\POr_{d}^{\mathcal{S}})\] takes values in
  $\OR_{n}^{d,\mathcal{S}}$ and so induces a symmetric monoidal functor
  \[ \Bordor_{d,d+n} \longrightarrow \Or_{n}^{d,\mathcal{S}} \]
  of $n$-fold Segal spaces.\qed
\end{corollary}

Combining this with the symmetric monoidal functors from
Corollaries~\ref{cor:BettiOrn} and \ref{cor:AKSZOrLag}, we get:
\begin{corollary}
  Suppose $X$ is an $s$-symplectic derived $S$-stack. Then there is a
  $n$-dimensional oriented extended TFT obtained as the composite
  \[ \Bordor_{0,n} \xto{\orcut_{n}} \Or_{n}^{\mathcal{S},0}
    \xto{(\blank)_{B} \times S} \Or_{n}^{\txt{cpt},X\txt{-good},S,0}
    \xto{\AKSZ_{n,X}^{S,0}} \Lag_{n}^{S,s} \]
  for every $n$.\qed
\end{corollary}

Since this is natural in the space $\Lag_{0}^{S,s}$ of $s$-symplectic
$S$-stacks, by the same proof as that of Corollary~\ref{cor:Ontrivact}
we also have:
\begin{corollary}
  The $SO(n)$-action on $\Lag_{0}^{S,s}$, viewed as the space of fully
  dualizable objects in $\Lag_{n}^{S,s}$, is trivial.\qed
\end{corollary}

\begin{corollary}
  The $O(n)$-action on $\Lag_{0}^{S,s}$, viewed as the space of fully
  dualizable objects in $\Lag_{n}^{S,s}$, is just the $\ZZ/2$-action
  given by $(X,\omega) \mapsto (X,-\omega)$, taking each object to its
  dual. \qed
\end{corollary}

\appendix


\chapter{Background on Higher Categories}\label{sec:cat}

In this appendix we review some background on higher categories: We
first recall iterated Segal spaces in \S\ref{subsec:segsp} and higher
categories of iterated spans in \S\ref{subsec:spans}. Then we discuss
how to describe higher categories of spans in the \icat{} associated
to a model category in \S\ref{subsec:spanmodel}, and finally in
\S\ref{subsec:twarr} we recall twisted arrow \icats{} and provide a
description of the left adjoint to this construction.

\section{Iterated Segal Spaces}\label{subsec:segsp}
In this section we briefly recall the definition of \emph{iterated
  Segal spaces} (originally introduced by Barwick in
\cite{BarwickThesis}), which is the model for $(\infty,n)$-categories
we use in this paper. For more details and motivation, see for
instance the more extensive discussion in \cite{spans}*{\S 2}.

\begin{definition}
  Suppose $\mathcal{C}$ is an \icat{} with finite limits. A
  \emph{category object} in $\mathcal{C}$ is a simplicial object $X_{\bullet}
  \colon \simp^{\op} \to \mathcal{C}$ such that the natural maps
  \[ X_{n} \to X_{1} \times_{X_{0}} \cdots \times_{X_{0}} X_{1} \]
  induced by the maps $\sigma_{i}\colon [0] \to [n]$ sending $0$ to
  $i$ and $\rho_{i} \colon [1] \to [n]$ sending $0$ to $i-1$ and $1$
  to $i$ are equivalences in $\mathcal{C}$ for all $n$. We
  write $\Cat(\mathcal{C})$ for the full subcategory of
  $\Fun(\simp^{\op}, \mathcal{C})$ spanned by the category objects. We
  refer to a category object in the
  \icat{} $\mathcal{S}$ of spaces as a (\emph{1-uple}) \emph{Segal space}.
\end{definition}

\begin{definition}
  We say a morphism $\phi \colon [n] \to [m]$ in $\simp$ is
  \emph{inert} if $\phi(i) = \phi(0)+i$ for all $i$, and \emph{active}
  if $\phi(0)=0$ and $\phi(n)=m$. The active and inert maps form a
  (strict) factorization system on $\simp$, \ie{} every morphism
  factors uniquely as an active map followed by an inert map. We write
  $\Dint$ for the subcategory of $\simp$ containing only the inert maps.
\end{definition}

\begin{remark}
  Let $\Del$ be the full subcategory of $\Dint$ spanned by the
  objects $[0]$ and $[1]$, \ie{}
  \[ [0] \rightrightarrows [1]. \]
  It is easy to see that a simplicial object $X \colon \Dop \to \mathcal{C}$ is
  a category object \IFF{} $X|_{\Dintop}$ is a right Kan extension of
  its restriction $X|_{\Delop}$.
\end{remark}

\begin{definition}
  An \emph{$n$-uple category object} in an \icat{} $\mathcal{C}$ (with
  finite limits) is inductively defined to be a category object in the
  \icat{} of $(n-1)$-uple category objects. We
  write \[\Cat^{n}(\mathcal{C}) := \Cat(\Cat^{n-1}(\mathcal{C}))\]
  for the \icat{} of $n$-uple category objects in $\mathcal{C}$. We
  refer to an $n$-uple category object in $\mathcal{S}$ as an
  \emph{$n$-uple Segal space}.
\end{definition}

\begin{remark}\label{app:cat:notation Del}
  If we write $\Del^{n} := (\Del)^{\times n}$, then an $n$-uple
  simplicial object $X \colon \Dnop \to \mathcal{C}$ is an $n$-uple
  category object \IFF{} $X|_{\Dintnop}$ is a right Kan extension of
  its restriction to $X|_{\Delnop}$. In other words, $X$ is an
  $n$-uple category object precisely if the canonical map
  \[ X(\ind{i}) \to \lim_{\DnelopI} X \]
  is an equivalence for all $\ind{i} \in \Dnop$, where $\DelnI :=
  \Del^{n}  \times_{\Dint^{n}} (\Dint^{n})_{/\ind{i}}$.
\end{remark}

\begin{definition}
  Suppose $\mathcal{C}$ is an \icat{} with finite limits. A
  \emph{1-fold Segal object} in $\mathcal{C}$ is just a category
  object in $\mathcal{C}$. For $n > 1$ we inductively define an
  \emph{$n$-fold Segal object} in $\mathcal{C}$ to be an $n$-uple
  category object $X$ such that
  \begin{enumerate}[(i)]
  \item the $(n-1)$-uple category object
    $X_{0,\bullet,\ldots,\bullet}$ is constant,
  \item the $(n-1)$-uple category object
    $X_{k,\bullet,\ldots,\bullet}$ is an $(n-1)$-fold Segal
    object for all $k$.
  \end{enumerate}
  We write $\Seg_{n}(\mathcal{C})$ for the full subcategory of
  $\Cat^{n}(\mathcal{C})$ spanned by the $n$-fold Segal objects. When
  $\mathcal{C}$ is the \icat{} $\mathcal{S}$ of spaces, we refer to
  $n$-fold Segal objects in $\mathcal{S}$ as \emph{$n$-fold Segal spaces}.
\end{definition}

\begin{definition}
  If $X$ is an $n$-fold Segal space, we refer to the elements of
  $X_{0,\ldots,0}$ as the \emph{objects} of $X$, and the elements of
  $X_{1,\ldots,1,0,\ldots,0}$ (with $i$ 1's) as the
  \emph{$i$-morphisms}. If $x,y$ are objects of $X$, we denote by
  $X(x,y)$ the $(n-1)$-fold Segal space of morphisms from $x$ to $y$,
  given by the pullback square of $(n-1)$-fold Segal spaces
  \nolabelcsquare{X(x,y)}{X_1}{\{(x,y)\}}{X_0 \times X_0.}  
\end{definition}

As far as the algebraic structure of compositions and identities is
concerned, $n$-uple Segal spaces are a model for $n$-uple \icats{} and
$n$-fold Segal spaces model $(\infty,n)$-categories. However, to get
the correct \icats{} of the latter we must invert the appropriate
class of ``fully faithful and essentially surjective'' maps. This can
be done by restricting to the \emph{complete} ones, as proved by
Rezk~\cite{RezkCSS} for $n = 1$ and Barwick~\cite{BarwickThesis} for
$n > 1$. We will not recall the definition here, as it will not play
much of a role in this paper --- see \eg{} \cite{spans}*{\S 5} or
\cite{LurieGoodwillie}.

We will make use of the observation that from an $n$-uple Segal space
we can extract an underlying $n$-fold Segal space:
\begin{proposition}[\cite{spans}*{Proposition 4.12}]
  Let $\mathcal{C}$ be an \icat{} with finite limits. Then the
  inclusion $\Seg_{n}(\mathcal{C}) \hookrightarrow
  \Cat^{n}(\mathcal{C})$ has a right adjoint $U_{\txt{fold}} \colon
  \Cat^{n}(\mathcal{C}) \to \Seg_{n}(\mathcal{C})$.\qed
\end{proposition}

Next, we recall the definition of iterated monoids, which specialize
to define iterated monoidal structures on $n$-uple and $n$-fold Segal
spaces:
\begin{definition}
  Suppose $\mathcal{C}$ is an \icat{} with finite products. Then an
  (\emph{associative}) \emph{monoid} in $\mathcal{C}$ is a functor $X
  \colon \Dop \to \mathcal{C}$ such that the natural maps
  \[ X_{n} \to \prod_{i=1}^{n} X_{1},\]
  induced by the inert maps $\rho_{i} \colon [1] \to [n]$, are
  equivalences. (Equivalently $X$ is a category object such that
  $X_{0} \simeq *$.) We write $\Mon(\mathcal{C})$ for the full
  subcategory of $\Fun(\Dop, \mathcal{C})$ spanned by the
  monoids. This is again an \icat{} with finite products, so we can 
  inductively define $\Mon^{n}(\mathcal{C}) :=
  \Mon(\Mon^{n-1}(\mathcal{C}))$; we refer to the objects of
  $\Mon^{n}(\mathcal{C})$ as \emph{$n$-uple
    monoids}. Equivalently, an $n$-uple monoid is a functor
  $X \colon \Dnop \to \mathcal{C}$ such
  that the natural maps
  \[ X_{k_{1},\ldots,k_{n}} \to \prod_{i_{1}=1}^{k_{1}} \cdots
  \prod_{i_{n}=1}^{k_{n}} X_{1,\ldots,1}\] are equivalences.
\end{definition}

\begin{definition}
  Evaluating at $[1]$ in the first coordinate, we get a diagram
  \[ \cdots \to \Mon^{n}(\mathcal{C}) \to \Mon^{n-1}(\mathcal{C}) \to
  \cdots \to \Mon(\mathcal{C}) \to \mathcal{C}.\]
  We let $\Mon^{\infty}(\mathcal{C})$ be the limit of this sequence of
  \icats{}, and call its objects \emph{$\infty$-uple
    monoids}.
\end{definition}

\begin{remark}
  If $\mathcal{C}$ is an \icat{} with finite products, then $n$-uple
  monoids can be identified with $E_{n}$-algebras in $\mathcal{C}$
  (with respect to the cartesian product) and $\infty$-uple monoids
  can be identified with $E_{\infty}$-algebras in $\mathcal{C}$. This
  follows immediately from the Additivity Theorem for
  $E_{n}$-algebras, \ie{} \cite{HA}*{Theorem 5.1.2.2}; this is spelled
  out in more detail in \cite{spans}*{Proposition 10.11}.
\end{remark}

\begin{definition}
  We call an $m$-uple monoid in $\Cat^{n}(\mathcal{S})$ or
  $\Seg_{n}(\mathcal{S})$ an
  \emph{$m$-uply monoidal} $n$-uple or $n$-fold Segal space. For $n =
  \infty$ we will generally use \emph{symmetric monoidal} instead of
  \emph{$\infty$-uply monoidal}.
\end{definition}

\begin{remark}
  By definition, $m$-uply monoidal $n$-fold Segal spaces are certain
  $m$-uple simplicial objects in $n$-uple simplicial spaces. Viewed as
  $(n+m)$-uple simplicial objects we see that they are precisely the
  $(n+m)$-fold Segal spaces $X$ such that
  $X_{1,\ldots,1,0,\ldots,0} \simeq *$ when there are at most $m$
  1's. If $X$ is an $n$-fold Segal space and $x$ is an object of $X$,
  then we can extract a monoid $\Omega_{x}X$ in $(n-1)$-fold Segal
  spaces, which exhibits $X(x,x)$ as a monoidal $(n-1)$-fold Segal
  space. This process can be iterated, giving an $m$-uple monoid
  $\Omega^{m}_{x}X$ in $(n-m)$-fold Segal spaces, which exhibits an
  $m$-uply monoidal structure on the endomorphisms of the identity
  $(m-1)$-morphism of $x$. (See \cite{spans}*{\S 10} for more
  detail.\footnote{Note that this section was corrected in the final arXiv
  version of the paper.})
  In other words, if we have a sequence $(X_{i},x_{i})$ where $X_{i}$
  is an $(n+i)$-fold Segal space and $x_{i}$ is an object of $X_{i}$
  such that
  $(X_{i}(x_{i},x_{i}), \id_{x_{i}}) \simeq (X_{i-1}, x_{i-1})$, then
  $X_{0}$ inherits an $n$-uply monoidal structure (with $x_{0}$ as
  unit). If we can iterate this ``delooping'' process infinitely many
  times, we similarly obtain a symmetric monoidal structure. We will
  use this construction repeatedly to obtain symmetric monoidal
  $(\infty,n)$-categories and symmetric monoidal functors between
  them.
\end{remark}

\section{Higher Categories of Spans}\label{subsec:spans}
In this section we will briefly recall the construction of higher
categories of iterated spans from \cite{spans}, which is the basis for
the construction of higher categories of symplectic and oriented
stacks in this paper. We begin by discussing the notions of $n$-fold
and $n$-uple spans, and making explicit the relation between them
(which was not discussed in \cite{spans}):

\begin{definition}\label{defn:foldspan}
  We define a partially ordered set $\lsp^{n}$ as follows: First, let
  $\lsp^{n,\circ}$ denote the partially ordered set with objects
  $A_{n}$ and $B_{n}$ ($n = 1,\ldots,n$) and with ordering defined by:
  \begin{itemize}
  \item $A_{n} \leq A_{k}$ and $B_{n} \leq B_{k}$ \IFF{} $n \geq k$,
  \item $A_{n} \leq B_{k}$ and $B_{n} \leq A_{k}$ \IFF{} $n > k$.
  \end{itemize}
  Then we set $\lsp^{n} := (\lsp^{n,\circ})^{\triangleleft}$ (\ie{} we
  add an initial object). An \emph{$n$-fold span} in an \icat{}
  $\mathcal{C}$ is a functor $\lsp^{n} \to \mathcal{C}$.
\end{definition}

\begin{examples}\ 
  \begin{enumerate}[(i)]
  \item $\lsp^{0}$ is the terminal category.
  \item $\lsp^{1}$ can be depicted as
    
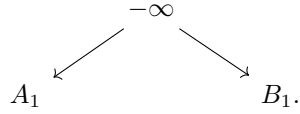
\[
      \begin{tikzcd}
        {} & -\infty \arrow{dl} \arrow{dr} \\
        A_{1} & & B_{1}.
      \end{tikzcd}
    \]
  \item $\lsp^{2}$ can be depicted as
    
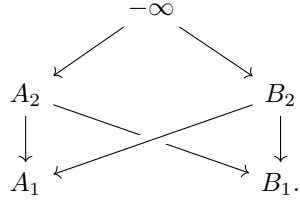
\[
      \begin{tikzcd}
         & -\infty \arrow{dl} \arrow{dr} \\
        A_{2} \arrow{d} \arrow{drr}& & B_{2} \arrow[crossing over]{dll} \arrow{d} \\
        A_{1} & &  B_{1}.
      \end{tikzcd}
    \]
  \end{enumerate}
\end{examples}
Although in the context of TFTs we are primarily interested in
$n$-fold spans, it turns out that to define higher categories of spans
it is convenient to work with a slightly different version of higher
spans, which is inductively defined:
\begin{definition}\label{defn:uplespan}
  We write $\Sp^{n} := (\lsp^{1})^{\times n}$. An \emph{$n$-uple span}
  in an \icat{} $\mathcal{C}$ is a functor $\Sp^{n} \to \mathcal{C}$.
\end{definition}

\begin{notation}
  It is convenient to denote the initial object
  $(-\infty,\ldots,-\infty)$ of $\Sp^{n}$ simply by $-\infty$.
  We define $\Sp^{n,\circ}$ to be the full subcategory of $\Sp^{n}$
  containing all objects \emph{except} this initial object. Then
  $\Sp^{n} \simeq (\Sp^{n,\circ})^{\triangleleft}$ since there are no
  non-trivial maps to $-\infty$ in $\Sp^{n}$.
\end{notation}

\begin{example}
  $\Sp^{2}$ can be depicted as
  \[
    \begin{tikzcd}
      \bullet &  \bullet \arrow{l} \arrow{r} & \bullet \\
      \bullet \arrow{u} \arrow{d} &  \bullet \arrow{l} \arrow{r}
      \arrow{u} \arrow{d}  & \bullet\arrow{u} \arrow{d} \\
      \bullet &   \bullet \arrow{l} \arrow{r} & \bullet. \\
    \end{tikzcd}
  \]
  Observe that $\lsp^{2}$ can be obtained by contracting the
  horizontal edges in the top and bottom rows, and so a 2-fold span
  can be viewed as a 2-uple span where these edges are sent to
  equivalences. We will now prove that the same statement holds also
  for $n \geq 2$, using the following functor:
\end{example}
\begin{definition}\label{defn:jn}
  We define a functor
  $j^{\circ}_{n} \colon \Sp^{n,\circ} \to \lsp^{n,\circ}$ as follows:
  Given $X = (I_{1},\ldots,I_{n})$ in $\Sp_{n}^{\circ}$, let $k$ be
  the least index such that $I_{k} \neq -\infty$. Then
  $j^{\circ}_{n}(X) = A_{k}$ if $I_{k} = A_{1}$ and
  $j^{\circ}_{n}(X) = B_{k}$ if $I_{k} = B_{1}$ --- this respects the
  partial ordering, and so gives a functor. We then write $j_{n}$ for
  the induced functor
  $(j_{n}^{\circ})^{\triangleleft} \colon \Sp^{n}\to \lsp^{n}$.
\end{definition}

\begin{definition}\label{defn:redspan}
  Let $R_{n}$ denote the set of maps in $\Sp^{n}$ that are taken to
  identities in $\lsp^{n}$. We say an $n$-uple span $\Phi \colon
  \Sp^{n} \to \mathcal{C}$ is \emph{reduced} if $\Phi$ takes the
  morphisms in $R_{n}$ to equivalences in $\mathcal{C}$.
\end{definition}

\begin{remark}
  Consider a morphism $f = (f_{1},\ldots,f_{n})$ in $\Sp^{n}$ with
  $f_{i}\colon X_{i}\to Y_{i}$ a morphism in $\Sp^{1}$. Then $f$
  lies in $R_{n}$ precisely when
  \begin{itemize}
  \item $f_{1}$ is an identity morphism and $X_{1} \neq -\infty$, or
  \item $f_{1}=\id_{-\infty}$, $f_{2}$ is an identity morphism, and
    $X_{2}\neq -\infty$, or
  \item \ldots,
  \item $f_{1}=\ldots=f_{n-2}=\id_{-\infty}$, $f_{n-1}$ is an identity
    morphism, and $X_{n-1} \neq -\infty$, or
  \item $f_{1}=\ldots=f_{n-1}=\id_{-\infty}$ and $f_{n}$ is an
    identity morphism.
  \end{itemize}
  This gives the following inductive description of reduced $n$-uple
  spans:
\end{remark}

\begin{lemma}
  A diagram $\Phi \colon \Sp^{n} \to \mathcal{C}$ is reduced \IFF{}
  the following conditions hold:
    \begin{enumerate}[(i)]
  \item $\Phi(A_{1},\blank) \colon \Sp^{n-1} \to \mathcal{C}$ is
    constant,
  \item $\Phi(B_{1},\blank) \colon \Sp^{n-1} \to \mathcal{C}$ is
    constant,
  \item $\Phi(-\infty,\blank) \colon \Sp^{n-1}\to \mathcal{C}$ is
    reduced. \qed
  \end{enumerate}
\end{lemma}

\begin{lemma}\label{lem:foldspanloc}\ 
  \begin{enumerate}[(i)]
  \item For $X \in \lsp^{n}$, let $\Sp^{n}_{X}$ denote
  the subcategory of $\Sp^{n}$ containing the objects $I$ such that
  $j_{n}(I) = X$. Then $\Sp^{n}_{X}$ is weakly contractible for
  every $X$.
  \item For every $X \in \lsp^{n}$, the inclusion $\Sp^{n}_{X}\hookrightarrow
    \Sp^{n}_{/X} := \Sp^{n} \times_{\lsp^{n}} \lsp^{n}_{/X}$ is cofinal.
  \item The functor $j^{\circ}_{n} \colon \Sp^{n,\circ} \to \lsp^{n,\circ}$ is coinitial.
  \item For any \icat{} $\mathcal{C}$, the functor
    \[ j_{n}^{*} \colon \Fun(\lsp^{n}, \mathcal{C}) \to \Fun(\Sp^{n},
      \mathcal{C})\]
    is fully faithful with image the full subcategory of reduced
    $n$-uple spans.
  \item $j_{n}$ exhibits $\lsp^{n}$ as the
  localization $\Sp^{n}[R_{n}^{-1}]$.
  \end{enumerate}
\end{lemma}
\begin{proof}
  Both (i) and (ii) are trivial if $X$ is the initial object; we will
  prove the statements for $X = A_{k}$, the case of $B_{k}$ being
  exactly the same. In this case we see that $\Sp^{n}_{A_{k}}$ is
  the partially ordered set of
  $I = (I_{1},\ldots, I_{n}) \in \Sp^{n}$ such that
  $I_{s} = -\infty$ for $s = 1,\ldots,k-1$ and
  $I_{k}= A_{1}$. Thus $\Sp^{n}_{A_{k}}$ is isomorphic to
  $\Sp^{n-k}$, which is weakly contractible since it has an initial
  object; this proves (1). To prove (2), it suffices by \ThmA{} to
  show that for every
  $J = (J_{1},\ldots,J_{n}) \in \Sp^{n}_{/A_{k}}$,
  the category $\Sp^{n}_{A_{k},J/}$ is weakly contractible. If
  $j_{n}(J) = A_{k}$, then this trivially has an initial object. On
  the other hand, if $j_{n}(J) < A_{k}$ then $J$ must satisfy
  $J_{s} = -\infty$ for $s = 1,\ldots,k$, and
  $\Sp^{n}_{A_{k},J/}$ is isomorphic to $(\Sp_{n-k})_{J'/}$ where
  $J' = (J_{k+1},\ldots,J_{n})$; this again has
  an initial object and thus is weakly contractible.

  The proof of (ii) also shows that the inclusion
  $\Sp^{n}_{X} \to \Sp^{n,\circ}_{/X}$ is cofinal, and together with
  (i) this implies that $\Sp^{n,\circ}_{/X}$ is weakly contractible
  for every $X \in \lsp^{n}$. Hence $j_{n}^{\circ}$ is coinitial by
  \ThmA{}.

  To prove (iv), we will show that for any reduced $n$-uple span
  $\Phi \colon \Sp^{n} \to \mathcal{C}$, the left Kan
  extension \[j_{n,!}\Phi \colon \lsp^{n} \to \mathcal{C}\] exists,
  and the unit map $\Phi \to j_{n}^{*}j_{n,!}\Phi$ is an equivalence.

  The (pointwise) left Kan extension $j_{n,!}\Phi$ exists \IFF{} for
  every $X \in \lsp^{n}$, the colimit of $\Phi$ over
  $\Sp^{n}_{/X}$ exists in $\mathcal{C}$. If $\Phi$ is a reduced $n$-uple span, then $\Phi$ is constant on
  $\Sp^{n}_{X}$ for every $X$, so (i) implies that the colimit of
  $\Phi$ over $\Sp^{n}_{X}$ exists, and for every $I \in
  \Sp^{n}_{X}$ the canonical map $\Phi(I) \to \colim
  \Phi|_{\Sp^{n}_{X}}$ is an equivalence. Now (ii) implies that
  $\colim \Phi|_{\Sp^{n}_{X}} \simeq \colim \Phi|_{\Sp^{n}_{/X}}$
  --- in particular the latter colimit exists in $\mathcal{C}$, and so the
  Kan extension $j_{n,!}\Phi$ exists. Moreover, the component of $\Phi
  \to j_{n}^{*}j_{n,!}\Phi$ at $I$ is the map $\Phi(I) \to \colim
  \Phi|_{\Sp^{n}_{/j_{n}(I)}}$, which is an equivalence. This proves
  (iv), of which (v) is a reformulation.
\end{proof}

To define higher categories where the $i$-morphisms are $i$-fold
spans, we can thus first consider the structure formed by $i$-uple
spans and then restrict to the reduced ones. This strategy was carried
out in \cite{spans}, and we now briefly review the construction:

\begin{definition}\label{defn:bbS}
  Let $\bbS^{n}$ be the partially ordered set of pairs $(i,j)$ with $0
  \leq i \leq j \leq n$, where $(i,j) \leq (i',j')$ when $i \leq i'$
  and $j' \leq j$. We write $\bbL^{n}$ for the subcategory of pairs
  $(i,j)$ where $j-i \leq 1$, and for $\ind{i} =
  (i_{1},\ldots,i_{k})$ we abbreviate $\bbS^{\ind{i}} :=
  \bbS^{i_{1}} \times \cdots \times \bbS^{i_{k}}$ and
  $\bbL^{\ind{i}} := \bbL^{i_{1}} \times \cdots \times
  \bbL^{i_{k}}$. Note that $\lsp^{1} \simeq \Sp^{1} \simeq \bbS^{1}$
  and $\Sp^{n} \simeq \bbS^{1,\ldots,1}$. 
\end{definition}

\begin{definition}
  For varying $\mathbf{i}$ the categories
  $\bbS^{\mathbf{i}}$ combine to a functor $\simp^{\times k} \to
  \CatI$, and if $\mathcal{C}$ is an \icat{} we
  define $\oSPAN^{+}_{k}(\mathcal{C}) \to (\simp^{\op})^{\times k}$ to be
  the cocartesian fibration associated to the functor
  $\Fun(\bbS^{(\blank)}, \mathcal{C})$.
\end{definition}

\begin{definition}
  If $\mathcal{C}$ is an \icat{} with pullbacks we say a functor
  $\bbS^{\mathbf{n}} \to \mathcal{C}$ is \emph{cartesian} if it is a
  right Kan extension of its restriction to $\bbL^{\mathbf{n}}$. We
  define $\SPAN^{+}_{k}(\mathcal{C})$ to be the full subcategory of
  $\oSPAN^{+}_{k}(\mathcal{C})$ spanned by the cartesian diagrams in
  $\Fun(\bbS^{\mathbf{n}}, \mathcal{C})$ for all $\mathbf{n}$.
\end{definition}

We then have:
\begin{theorem}[\cite{spans}*{Corollary
    5.12 and Proposition 5.14}]
  The restricted projection
  \[\SPAN^{+}_{n}(\mathcal{C}) \to \simp^{n,\op}\] is a cocartesian
  fibration, and the corresponding functor $\simp^{n,\op} \to \CatI$
  is an $n$-uple category object. As a consequence, the underlying left fibration corresponds to an $n$-uple Segal space $\SPAN_{n}(\mathcal{C})$. \qed
\end{theorem}

\begin{definition}
  We let $\Span_{n}(\mathcal{C})$ be the underlying $n$-fold Segal
  space of $\SPAN_{n}(\mathcal{C})$. Regarding \icats{} as complete
  Segal spaces, we may also view $\SPAN^{+}_{n}(\mathcal{C})$ as an
  $(n+1)$-uple Segal space; we let $\Span_{n}^{+}(\mathcal{C})$ be its
  underlying $(n+1)$-fold Segal space.
\end{definition}

\begin{definition}\label{def:redbbSj}
  Let us define $R_{\ind{j}}$ to be the set of maps $f=(f_{1},\ldots,f_{n})$ in
  $\bbS^{\ind{j}}$ such that either $f$ is an identity morphism or
  \begin{itemize}
  \item $f_{1}$ is an identity morphism at an object of the form
    $(a_{1},a_{1}) \in \bbS^{j_{1}}$, or
  \item $f_{1}$ is an identity morphism at an object of the form
    $(a_{1},b_{1})$, $b_{1} \neq a_{1}$, and $f_{2}$ is an identity morphism at an object of the form
    $(a_{2},a_{2}) \in \bbS^{j_{2}}$, or
  \item \ldots,
  \item $f_{1},\ldots,f_{n-2}$ are identity morphisms at objects of
    the form $(a_{i},b_{i}) \in \bbS^{j_{i}}$ with $a_{i} \neq b_{i}$,
    and $f_{n-1}$ is an identity morphisms at an object of the form
    $(a_{n-1},a_{n-1})\in \bbS^{j_{n-1}}$.
  \end{itemize}
  We say a functor $\bbS^{\ind{j}} \to \mathcal{C}$ is \emph{reduced}
  if it takes the morphisms in $R_{\ind{j}}$ to equivalences.
\end{definition}
\begin{remark}\label{rmk:Spanasredftr}
  We can identify the \icat{} $\Span_{n}^{+}(\mathcal{C})_{\ind{j}}$ with the full
  subcategory of $\Fun(\bbS^{\ind{j}}, \mathcal{C})$ spanned by
  functors that are both cartesian and reduced. Similarly,
  $\Span_{n}(\mathcal{C})_{\ind{j}}$ is the subspace of
  $\Map(\bbS^{\ind{j}}, \mathcal{C})$ spanned by the cartesian and
  reduced functors.
\end{remark}

We have the following results from \cite{spans}:

\begin{proposition}[\cite{spans}*{Corollary 8.5}]
  For any \icat{} $\mathcal{C}$ with pullbacks, the $n$-fold Segal space $\Span_{n}(\mathcal{C})$ is complete. \qed
\end{proposition}

\begin{proposition}
  For objects, $x,y \in \mathcal{C}$, the $(n-1)$-fold Segal space
  of maps $\Span_{n}(\mathcal{C})(x,y)$ is naturally equivalent to
  $\Span_{n-1}(\mathcal{C}_{/x,y})$. 
\end{proposition}
\begin{proof}
  This is essentially \cite{spans}*{Proposition 8.3}, except that as
  stated this makes the assumption that $\mathcal{C}$ has all finite
  limits. However, this is not required for the proof, provided the
  slice $\mathcal{C}_{/x \times y}$ is replaced by the double slice
  $\mathcal{C}_{/x,y}$.
\end{proof}

As in \cite{spans}*{Proposition 12.1}, this has the following consequence:
\begin{corollary}
  If $\mathcal{C}$ also has a terminal object then the
  $(\infty,n)$-category $\Span_{n}(\mathcal{C})$ has a natural
  symmetric monoidal structure.
\end{corollary}

\begin{remark}\label{rmk:spanmondeloop}
  More relevant for the present paper is a slight variant of this
  symmetric monoidal structure: If we have a sequence of pointed
  \icats{} $(\mathcal{C}_{n}, c_{n})$ where $\mathcal{C}_{n}$ has
  pullbacks, and equivalences
  $\mathcal{C}_{n} \simeq (\mathcal{C}_{n+1})_{/c_{n+1}, c_{n+1}}$
  under which $c_{n}$ corresponds to the degenerate span
  \[ c_{n+1} \xfrom{\id} c_{n+1} \xto{\id} c_{n+1},\] then we get
  equivalences
  $\Span_{n+k}(\mathcal{C}_{k})(c_{k},c_{k}) \simeq
  \Span_{n+k-1}(\mathcal{C}_{k-1})$, under which $c_{k-1}$ corresponds
  to the identity map of $c_{k}$, which equips
  $\Span_{n}(\mathcal{C}_{0})$ with a symmetric monoidal structure
  with unit~$c_{0}$.
\end{remark}

\begin{remark}\label{rmk:spanfun}
  It is clear from the constructions of $\SPAN_{n}(\mathcal{C})$ and $\Span_{n}(\mathcal{C})$ that
  they are natural in the \icat{} $\mathcal{C}$: if
  $F \colon \mathcal{C} \to \mathcal{D}$ is a functor that preserves
  pullbacks then it induces natural morphisms
  \[ \SPAN_{n}(F) \colon \SPAN_{n}(\mathcal{C}) \to \SPAN_{n}(\mathcal{D}),\qquad
    \Span_{n}(F) \colon \Span_{n}(\mathcal{C}) \to
    \Span_{n}(\mathcal{D}).\]
  We thus get functors $\SPAN_{n}(\blank)$ and $\Span_{n}(\blank)$
  from the \icat{} of \icats{} with pullbacks and functors that
  preserve these to $n$-uple \icats{} and $(\infty,n)$-categories.
Moreover, for $x,y \in \mathcal{C}$ the
  induced functor
  \[ \Span_{n-1}(\mathcal{C}_{/x, y}) \simeq
    \Span_{n}(\mathcal{C})(x,y) \to \Span_{n}(\mathcal{D})(Fx,Fy)
    \simeq \Span_{n-1}(\mathcal{D}_{/Fx, Fy})\] can be identified
  with $\Span_{n-1}(F_{/x,y})$ where $F_{/x,y}$ is the functor
  $\mathcal{C}_{/x, y} \to \mathcal{D}_{/Fx, Fy}$ induced by $F$. Thus
  if $F$ also preserves the terminal object (\ie{} $F$ preserves
  all finite limits) then $\Span_{n}(F)$ is a symmetric
  monoidal functor. Thus we obtain a functor $\Span_{n}(\blank)$ from
  the \icat{} of \icats{} with finite limits and functors that
  preserve these to symmetric monoidal $(\infty,n)$-categories.
\end{remark}

\begin{lemma}\label{lem:Spanlim}
  The functors $\SPAN_{n}(\blank)$ and $\Span_{n}(\blank)$, viewed as
  functors from \icats{} with finite limits to $n$-uple \icats{} and
  $(\infty,n)$-categories, respectively, both preserve limits.
\end{lemma}
\begin{proof}
  This follows from the definitions, since limits in $n$-uple \icats{}
  and $(\infty,n)$-categories are both computed in
  $\Fun(\simp^{n,\op}, \mathcal{S})$, where limits are in turn given
  pointwise, while finite limits in the \icat{} of \icats{} with finite
  limits are computed in $\CatI$ by \cite{HTT}*{Lemma 5.4.5.5}.
\end{proof}

\section{Spans in Model Categories}\label{subsec:spanmodel}
In \S\S\ref{sec:cobcospan} and \ref{sec:orcob} we will construct
several functors from $(\infty,n)$-categories of cobordisms to various
$(\infty,n)$-categories of (co)spans. The cobordism
$(\infty,n)$-categories are defined using functors of ordinary
categories $\Dnop \to \sSet$, and the goal of this section is to
give a similarly strict model for $(\infty,n)$-categories of spans
valued in a model category, which we will use to define these
functors. We start by recalling a result of Cisinski:

\begin{theorem}[Cisinski]\label{thm:modcatfun}
  If $\mathbf{M}$ is a cofibrantly generated model category, and
  $\mathbf{M}[W^{-1}]$ denotes the \icat{} obtained by inverting the
  class $W$ of weak equivalences in $\mathbf{M}$, then for every small category
  $\mathbf{I}$ the natural map
  \[ \Fun(\mathbf{I},\mathbf{M})[W^{-1}_{\mathbf{I}}] \to \Fun(\mathbf{I},
  \mathbf{M}[W^{-1}])\]
  is an equivalence of \icats{}, where $W_{\mathbf{I}}$ denotes the
  class of natural weak equivalences, \ie{} natural transformations given objectwise by weak equivalences.
\end{theorem}
\begin{proof}
  This is a special case of \cite{CisinskiCat}*{Theorem 7.9.8}, since
  for a cofibrantly generated model category $\mathbf{M}$ the
  projective model structure on the functor
  category $\Fun(\mathbf{I}, \mathbf{M})$ exists for any small
  category $\mathbf{I}$.
\end{proof}
\begin{remark}
  Under stronger assumptions on $\mathbf{M}$, namely that it is a
  combinatorial simplicial model category, this result is due to Lurie
  in \cite{HTT}*{Corollary A.3.4.14}.
\end{remark}

Passing to the underlying spaces of these \icats{}, we obtain:
\begin{corollary}
  If $\mathbf{M}$ is a cofibrantly generated model category,
  $\mathbf{I}$ is a small category, and $\Fun(\mathbf{I},
  \mathbf{M})^{W}$ denotes the subcategory of $\Fun(\mathbf{I},
  \mathbf{M})$ where the morphisms are the weak equivalences, then the
  natural map
  \[ \|\Fun(\mathbf{I}, \mathbf{M})^{W}\| \to \Map(\mathbf{I},
  \mathbf{M}[W^{-1}])\]
  is an equivalence in $\mathcal{S}$.\qed
\end{corollary}

Recall from \S\ref{subsec:spans} that the $n$-uple Segal space
$\SPAN_{n}(\mathcal{C})$ is given by the diagram
$\Dnop \to \mathcal{S}$ taking $\ind{i} \in \Dnop$ to
$\Map_{\txt{cart}}(\bbS^{\ind{i}}, \mathcal{C})$, where
$\Map_{\txt{cart}}(\bbS^{\ind{i}}, \mathcal{C})$ is the subspace of
connected components in $\Map(\bbS^{\ind{i}}, \mathcal{C})$
corresponding to the cartesian functors. To construct a
functor to $\SPAN_{n}(\mathcal{C})$ it is therefore sufficient to
construct a natural transformation of functors $\Dnop \to \mathcal{S}$
with target the diagram $\oSPAN_{n}(\mathcal{C})$ given by
$\Map(\bbS^{(\blank)}, \mathcal{C})$ and then check the
\emph{condition} that this diagram lands in $\SPAN_{n}(\mathcal{C})$.
Now if $\mathcal{C}$ is modelled by a cofibrantly generated model
category $\mathbf{C}$, we can describe $\oSPAN_{n}(\mathcal{C})$ using a strict
diagram of simplicial sets defined in terms of $\mathbf{C}$:
\begin{corollary}\label{cor:SPANsset}
  Let $\mathbf{M}$ be a cofibrantly generated model category. Then the
  $n$-simplicial space $\oSPAN_{n}(\mathbf{M}[W^{-1}])$ arises from
  the functor $\Dnop \to \sSet$ given by
  \[ \mathbf{i} \mapsto \mathrm{N}\,\Fun(\bbS^{\mathbf{i}}, \mathbf{M})^{W}\]
  under composition with the equivalence $\sSet[\mathrm{WHE}^{-1}]
  \isoto \mathcal{S}$, where $\mathrm{WHE}$ denotes the class of weak
  homotopy equivalences.\qed
\end{corollary}

The following variant of Theorem~\ref{thm:modcatfun} will also be useful:
\begin{proposition}
  Suppose $\mathbf{M}$ is a cofibrantly generated model category and
  $(\mathbf{I}, U)$ is a small relative category. Let
  $\Fun_{(U,W)}(\mathbf{I}, \mathbf{M})$ denote the full subcategory
  of $\Fun(\mathbf{I}, \mathbf{M})$ spanned by the functors that take
  the morphisms in $U$ to weak equivalences. Then we have equivalences
  of \icats{}
  \[ \Fun_{(U,W)}(\mathbf{I}, \mathbf{M})[W_{\mathbf{I}}^{-1}] \isoto
  \Fun_{(U)}(\mathbf{I}, \mathbf{M}[W^{-1}]) \isofrom
  \Fun(\mathbf{I}[U^{-1}], \mathbf{M}[W^{-1}]),\]
  where $\Fun_{(U)}(\mathbf{I}, \mathbf{M}[W^{-1}])$ denotes the
  \icat{} of functors taking the maps in $U$ to equivalences.
\end{proposition}
\begin{proof}
  The first equivalence follows from Theorem~\ref{thm:modcatfun}
  since the two subcategories correspond under the equivalence
  $\Fun(\mathbf{I}, \mathbf{M})[W_{\mathbf{I}}^{-1}] \isoto
  \Fun(\mathbf{I}, \mathbf{M}[W^{-1}])$. The second equivalence
  follows from the universal property of $\mathbf{I}[U^{-1}]$.
\end{proof}

\begin{corollary}\label{cor:relcatspanmodel}
  Suppose $\mathbf{M}$ is a cofibrantly generated model category and
  $(\mathbf{I}, U)$ is a small relative category. Then there is a natural
  equivalence of spaces
  \[ \|\Fun_{(U,W)}(\mathbf{I}, \mathbf{M})^{W}\| \simeq
    \Map(\mathbf{I}[U^{-1}], \mathbf{M}[W^{-1}]).\]
\end{corollary}

We will need to consider a slight modification of the description of
Corollary~\ref{cor:SPANsset}: It turns out that the morphisms of
$n$-fold Segal spaces we want to define using diagrams of simplicial
sets do not behave nicely with respect to degeneracies. To circumvent
this problem we will instead work with semisimplicial sets, as we
will now briefly justify:
\begin{definition}
  Let $\Dinj$ denote the subcategory of $\simp$ containing
  only the injective morphisms, and let $i_{\semi} \colon \Dinj \to \simp$ be
  the inclusion. A \emph{semisimplicial set} is a
  presheaf of sets on $\Dinj$. We write $\ssSet$ for the category
  $\Fun(\Dinjop, \Set)$ of semisimplicial sets. The functor $i$
  induces an adjunction
  \[ i_{\semi,!} : \ssSet \rightleftarrows \sSet : i_{\semi}^{*},\]
  where $i_{\semi,!}$ ``freely adds degeneracies''. We also use the
  notation $X^{\semi} := i^{*}_{\semi}X$ for the underlying
  semisimplicial set of a simplicial set $X$.
\end{definition}

\begin{definition}
  We say a morphism $f \colon X \to Y$ in $\ssSet$ is a \emph{weak
    equivalence} if $|i_{\semi,!}f|$ is a weak homotopy equivalence of
  topological spaces, where $|\blank|$ denotes the geometric
  realization functor $\sSet \to \Top$, or equivalently if $i_{\semi,!}f$ is a weak
  equivalence of simplicial sets.
\end{definition}

\begin{remark}
  $|i_{\semi,!}(\blank)|$ is the left adjoint to the functor $\Top \to
  \ssSet$ induced by the restriction $\Dinj \to \simp \to \Top$ of the
  geometric simplices, and is therefore given by a coend over
  $\Dinj$. In particular, for a simplicial set $X$, the space
  $|i_{\semi,!}i_{\semi}^{*}X|$ is the classical ``fat geometric realization'' of
  $X$. This is well-known to be weakly equivalent to
  the usual geometric realization, which allows us to prove the following:
\end{remark}

\begin{proposition}\
  \begin{enumerate}[(i)]
  \item   For any $X \in \sSet$, the counit morphism $i_{\semi,!}i_{\semi}^{*}X \to X$ is a
    weak equivalence.
  \item For any $T \in \ssSet$, the unit morphism $T \to i_{\semi}^{*}i_{\semi,!}T$
    is a weak equivalence.
  \item The functors $i_{\semi,!}$ and $i_{\semi}^{*}$ both preserve weak equivalences.
  \end{enumerate}
\end{proposition}
\begin{proof}
  To prove (i), we observe that the geometric realization of the
  counit $|i_{\semi,!}i_{\semi}^{*}X| \to |X|$ is precisely the canonical map from
  the fat geometric realization of $X$ to the geometric
  realization. This is a weak homotopy equivalence by
  \cite[Proposition A.1]{SegalCatCohlgy}. To prove (ii), it is
  equivalent to show that $i_{\semi,!}T \to i_{\semi,!}i_{\semi}^{*}i_{\semi,!}T$ is a weak
  equivalence in $\sSet$. The composite
  \[i_{\semi,!}T \to i_{\semi,!}i_{\semi}^{*}i_{\semi,!}T \to i_{\semi,!}T\] with the counit map is the
  identity, and we already saw that the counit is a weak equivalence,
  so (ii) follows from the 2-of-3 property of weak equivalences. To
  prove (iii), we recall that $i_{\semi,!}$ creates weak equivalences by
  definition, and $i_{\semi}^{*}$ therefore preserves weak equivalences by (i)
  and the 2-of-3 property.
\end{proof}

\begin{warning}
  The category $\ssSet$ does \emph{not} admit a model structure such
  that the adjunction $i_{\semi,!} \dashv i_{\semi}^{*}$ is a Quillen equivalence
  (see
  \url{https://mathoverflow.net/questions/133051/}).
\end{warning}

\begin{corollary}
  For any small category $\mathbf{I}$, the adjunction
  \[ \Fun(\mathbf{I}, \ssSet) \rightleftarrows \Fun(\mathbf{I},
    \sSet) \] given by composition with $i_{\semi,!}$ and $i_{\semi}^{*}$ induces an
  equivalence of \icats{}
  \[ \Fun(\mathbf{I}, \ssSet)[W^{-1}] \simeq \Fun(\mathbf{I},
    \sSet)[W^{-1}], \]
  where $W$ denotes the relevant class of natural weak
  equivalences. In particular, the \icat{} $\Fun(\mathbf{I}, \ssSet)[W^{-1}]$ is
  equivalent to the functor \icat{} $\Fun(\mathbf{I}, \mathcal{S})$. \qed
\end{corollary}
\begin{proof}
  Since $i_{\semi,!}$ and $i_{\semi}^{*}$ both preserve weak equivalences, the
  induced adjunction on functor categories descends to an adjunction
  between the localized \icats{}. This is an equivalence since the
  unit and counit transformations are both natural weak equivalences.
\end{proof}

\section{Twisted Arrow $\infty$-Categories and their Left
  Adjoint}\label{subsec:twarr}
In this section we will prove some results about twisted arrow
\icats{}; in particular, we will provide a description of the left
adjoint to the twisted arrow construction. We first recall the
definition of the twisted arrow \icat{} of an \icat{}:
\begin{definition}\label{defn:Tw}
  Let $\epsilon^{r} \colon \simp \to \simp$ be the functor $[n] \mapsto
  [n] \star [n]^{\op}$. Then composition with $\epsilon^{r}$ induces a
  functor $\Tw^{r}:= \epsilon^{r,*} \colon \Fun(\Dop, \mathcal{S}) \to \Fun(\Dop,
  \mathcal{S})$; on simplicial sets this is known as the
  \emph{edgewise subdivision} functor. The natural inclusions $[n],
  [n]^{\op} \hookrightarrow [n] \star [n]^{\op}$ induce natural
  transformations $\Tw^{r}X \to X, X^{\op}$.
\end{definition}

\begin{proposition}[\cite{cois}*{Proposition A.2.3}]\label{propn:TwCicat}\
  \begin{enumerate}[(i)]
  \item If $\mathcal{C}$ is a Segal space, then so is
    $\Tw^{r}\mathcal{C}$.
  \item If $\mathcal{C}$ is a Segal space, then the morphism
    $\Tw^{r}\mathcal{C} \to \mathcal{C} \times
    \mathcal{C}^{\op}$ is a right fibration.
  \item If $\mathcal{C}$ is a complete Segal space, then so is $\Tw^{r}\mathcal{C}$.
  \end{enumerate}
\end{proposition}

\begin{remark}
  There are two possible conventions for the twisted arrow \icat{} ---
  we use the superscript $r$ as a reminder that we use the version
  that is a \emph{right} fibration (the other version being
  $\Tw^{\ell}(\mathcal{C}) := \Tw^{r}(\mathcal{C})^{\op}$, which can
  be defined using $\epsilon^{\ell}([n]) := [n]^{\op} \star [n]$).
\end{remark}

\begin{corollary}\label{cor:Tw!}
  Composition with $\epsilon^{r}$ induces a  functor $\Tw^{r} \colon \CatI \to
  \CatI$ with a left adjoint $\Tw^{r}_{!}$.
\end{corollary}
\begin{proof}
  Since $\Tw^{r}$ preserves complete Segal spaces, the adjunction
  \[ \epsilon^{r}_{!} : \Fun(\Dop, \mathcal{S}) \rightleftarrows
  \Fun(\Dop, \mathcal{S}) : \epsilon^{r,*}\]
  induces an adjunction 
  \[ \Tw^{r}_{!} : \CatI \rightleftarrows \CatI : \Tw^{r} \]
between the full subcategories of complete Segal spaces, where $\Tw^{r}$
is the restriction of $\epsilon^{r,*}$ and $\Tw^{r}_{!}$ is the composite of
$\epsilon^{r}_{!}$ and the localization functor from $\Fun(\Dop,
\mathcal{S})$ to $\CatI$.
\end{proof}

For $\mathcal{C}$ an \icat{} we refer to the \icat{}
$\Tw^{r} \mathcal{C}$ as the \emph{twisted arrow \icat{}} of
$\mathcal{C}$. If $\mathbf{C}$
is an ordinary category, it is easy to see that
$\Tw^{r}(\mathbf{C})$ is the usual twisted
arrow category of $\mathbf{C}$.

\begin{example}\label{ex:Twn}
  The category $\Tw^{r}[n]$ is equivalent to the partially ordered
  set $\bbS^{n}$ of Definition~\ref{defn:bbS}.
\end{example}

\begin{remark}
  There are natural inclusions $\mathcal{C}, \mathcal{C}^{\op}
  \hookrightarrow \Tw^{r}_{!} \mathcal{C}$, which correspond to the
  projections $\Tw^{r} \mathcal{D} \to \mathcal{D}, \mathcal{D}^{\op}$ in
  the sense that we have a commutative triangle
  \opctriangle{\Map(\Tw^r_! \mathcal{C}, \mathcal{D})}{\Map(\mathcal{C},
    \Tw^r \mathcal{D})}{\Map(\mathcal{C}, \mathcal{D}) \times
    \Map(\mathcal{C}^{\op}, \mathcal{D}),}{\sim}{}{}
  where the left diagonal is given by restricting along the inclusions
  $\mathcal{C}, \mathcal{C}^{\op} \hookrightarrow \Tw^{r}_{!}\mathcal{C}$,
  and the right diagonal to composing with the projections $\Tw^{r}
  \mathcal{D} \to \mathcal{D}, \mathcal{D}^{\op}$.
\end{remark}

Our goal in this section is to obtain the following explicit
description of $\Tw^{r}_{!}\mathcal{C}$ for an \icat{} $\mathcal{C}$:
\begin{proposition}\label{propn:Tw!desc}
  Let $\mathcal{C}$ be an \icat{}. Then the \icat{}
  $\Tw^{r}_{!}\mathcal{C}$ has a natural map to $\Delta^{1}$ such that
  \begin{enumerate}[(i)]
  \item $(\Tw^{r}_{!}\mathcal{C})_{0} \simeq \mathcal{C}$ (we denote the
    object of $\Tw^{r}_{!}\mathcal{C}$ corresponding to $x \in
    \mathcal{C}$ also by $x$),
  \item $(\Tw^{r}_{!}\mathcal{C})_{1} \simeq \mathcal{C}^{\op}$ (we denote the
    object of $\Tw^{r}_{!}\mathcal{C}$ corresponding to $x \in
    \mathcal{C}^{\op}$ by $x^{\vee}$),
  \item $\Map_{\Tw^{r}_{!}\mathcal{C}}(x, y^{\vee}) \simeq \| \mathcal{C}_{x,y/} \|$.
  \end{enumerate}
\end{proposition}

The proof uses a result of Ayala--Francis and Stevenson, which we need to recall. This
requires a bit of notation:
\begin{definition}
  Let $\txt{RCorr}$ denote the full subcategory of
  $\Fun(\Lambda^{2}_{0}, \CatI)$ spanned by correspondences
  \[
    \begin{tikzcd}
      {} & \mathcal{E} \arrow{dl}[above left]{p} \arrow{dr}{q} \\
      \mathcal{C} & & \mathcal{D}^{\op}
    \end{tikzcd}
  \]
  such that $(p,q) \colon \mathcal{E} \to \mathcal{C} \times
  \mathcal{D}^{\op}$ is a right fibration. 
  We can define a functor $\mathfrak{R} \colon \Cat_{\infty/\Delta^{1}} \to \txt{RCorr}$
  as follows: Given $\mathcal{E} \to \Delta^{1}$ we have the right
  fibration $\Tw^{r}\mathcal{E} \to \mathcal{E} \times
  \mathcal{E}^{\op}$, and we can take the pullback
  $(i_{0},i_{1}^{\op})^{*}\mathcal{E} \to \mathcal{E}_{0} \times
  \mathcal{E}^{\op}_{1}$ along the inclusions $i_{s} \colon \mathcal{E}_{s}
  \to \mathcal{E}$ of the fibres at $s \in \Delta^{1}$. It is easy to
  see that this has a
  left adjoint $\mathfrak{L}$, which takes a right fibration $\mathcal{E}
  \to \mathcal{C} \times \mathcal{D}^{\op}$ to the pushout
  $\Tw^{r}_{!}(\mathcal{E}) \amalg_{\mathcal{E} \amalg \mathcal{E}^{\op}}
  (\mathcal{C} \amalg \mathcal{D})$.
\end{definition}

\begin{theorem}[Ayala--Francis~\cite{AyalaFrancisFib}, Stevenson~\cite{Stevenson}]\label{thm:bifibeq}
  The adjunction $\mathfrak{L} \dashv \mathfrak{R}$ is an equivalence
  of \icats{}.
\end{theorem}

\begin{remark}
  More precisely, both Ayala--Francis and Stevenson only prove this statement fibrewise over
  $\CatI \times \CatI$. However, both proofs show that the unit and
  counit maps for this adjunction are equivalences, and hence also
  imply our statement. Another proof, using the
  representability of bimodules for (enriched) \icats{}, is given by
  Hinich in \cite{HinichYoneda}*{\S 8}.
\end{remark}

\begin{remark}\label{rmk:bifibmaps}
  Given $\pi \colon \mathcal{E} \to \Delta^{1}$, the right fibration
  $\mathfrak{R}(\pi)$ corresponds to the functor
  $\mathcal{E}_{0}^{\op} \times \mathcal{E}_{1} \to \mathcal{S}$ given
  by $(x,y) \mapsto \Map_{\mathcal{E}}(x,y)$ (as $\Tw^{r}\mathcal{E}$
  is the fibration for $\Map_{\mathcal{E}}(\blank, \blank)$ and
  pullback of fibrations corresponds to composition of functors).
  The equivalence $\mathfrak{R}\mathfrak{L} \simeq \id$ therefore has the
  following interpretation: If $(p,q) \colon \mathcal{E} \to
  \mathcal{C} \times \mathcal{D}^{\op}$ is the right fibration
  corresponding to
  a functor $\phi \colon \mathcal{C}^{\op} \times \mathcal{D} \to
  \mathcal{S}$, then the \icat{} $\mathfrak{L}(p,q)$ can be described
  as follows:
  \begin{itemize}
  \item There are equivalences $\mathcal{C} \simeq
    \mathfrak{L}(p,q)_{0}, \mathcal{D} \simeq \mathfrak{L}(p,q)_{1}$.
  \item We have $\Map_{\mathfrak{L}(p,q)}(c,d) \simeq \phi(c,d)$ for $c \in
    \mathcal{C} \simeq \mathfrak{L}(p,q)_{0}$, $d \in
    \mathcal{D} \simeq \mathfrak{L}(p,q)_{1}$.
  \end{itemize}
\end{remark}

\begin{proof}[Proof of Proposition~\ref{propn:Tw!desc}]
  For \icats{} $\mathcal{C}$ and $\mathcal{D}$ we have equivalences of
  spaces
  \[ \Map(\mathcal{C}, \Tw^{r}\mathcal{D})
    \simeq
    \left\{
        \begin{tikzcd}
          \mathcal{C} \arrow{r} \arrow{dr} & \Tw^{r}\mathcal{D}
          \arrow{d} \\
           & \mathcal{D} \times \mathcal{D}^{\op}
        \end{tikzcd}
      \right\}
          \simeq
    \left\{
        \begin{tikzcd}
          \mathcal{C} \arrow{r} \arrow{d}{\Delta} & \Tw^{r}\mathcal{D}
          \arrow{d} \\
          \mathcal{C} \times \mathcal{C} \arrow{r} & \mathcal{D} \times \mathcal{D}^{\op}.
        \end{tikzcd}
        \right\}
      \]
      Here $\Tw^{r}\mathcal{D} \to \mathcal{D} \times
      \mathcal{D}^{\op}$ is a cartesian fibration, so this is
      equivalent to the space of morphisms from the free cartesian
      fibration on $\Delta \colon \mathcal{C} \to \mathcal{C} \times
      \mathcal{C}$ (that preserve cartesian morphisms, but this is a
      vacuous condition since the target is a right fibration). This
      free fibration is the fibre product $(\mathcal{C} \times
      \mathcal{C})^{\Delta^{1}} \times_{\mathcal{C} \times
        \mathcal{C}} \mathcal{C}$, formed via $\txt{ev}_{1}$, with the
      fibration induced by $\txt{ev}_{0}$ \cite{freepres}*{Theorem 4.5}. This
      corresponds to the functor $(c,c') \mapsto
      \mathcal{C}_{c,c'/}$. Let $\mathcal{E} \xto{\phi} \mathcal{C} \times
      \mathcal{C}$ be the right fibration obtained by inverting all
      morphisms in the fibres, corresponding to the functor $(c,c')
      \mapsto \|\mathcal{C}_{c,c'/}\|$; this is the free right
      fibration on $\Delta$. We therefore have a natural equivalence
      \[
        \begin{split}
\Map(\mathcal{C}, \Tw^{r}\mathcal{D})
                  & \simeq
    \left\{
        \begin{tikzcd}
          \mathcal{E} \arrow{r} \arrow{d}{\phi} \pgfmatrixnextcell \Tw^{r}\mathcal{D}
          \arrow{d} \\
          \mathcal{C} \times \mathcal{C} \arrow{r} \pgfmatrixnextcell \mathcal{D} \times \mathcal{D}^{\op}.
        \end{tikzcd}
        \right\} \\ &
      \simeq
      \Map\left(\Tw^{r}_{!}\mathcal{E} \amalg_{\mathcal{E} \amalg
        \mathcal{E}^{\op}} (\mathcal{C} \amalg \mathcal{C}^{\op}), \mathcal{D}\right),
        \end{split}
  \]
  where the second equivalence follows from the adjunction
  $\Tw^{r}_{!} \dashv \Tw^{r}$.
  We thus have a natural identification
  \[\Tw_{!}^{r}\mathcal{C} \simeq \Tw_{!}^{r}(\mathcal{E}) \amalg_{\mathcal{E} \amalg
        \mathcal{E}^{\op}} (\mathcal{C} \amalg \mathcal{C}^{\op}) \simeq
      \mathfrak{L}(\phi).\]
    Applying Theorem~\ref{thm:bifibeq} as in
    Remark~\ref{rmk:bifibmaps}, this implies the required description
    of $\Tw_{!}^{r}\mathcal{C}$.
\end{proof}

\begin{remark}
  In particular, if for objects $x,y \in \mathcal{C}$ there exists a
  coproduct $x \amalg y$ in $\mathcal{C}$, then
  $\Map_{\Tw_{!}\mathcal{C}}(x, y^{\vee})$ is contractible, since the
  \icat{} $\mathcal{C}_{x,y/}$ has an initial object, and so is weakly
  contractible. Similarly, if $\mathcal{C}$ has a terminal object,
  then $\Map_{\Tw_{!}\mathcal{C}}(x, y^{\vee})\simeq *$ for all $x,y
  \in \mathcal{C}$.
\end{remark}

Let us spell out one particular case of
Proposition~\ref{propn:Tw!desc} that we will use in this paper:
\begin{corollary}\label{cor:Tw!poset}
  Suppose $\mathbf{P}$ is a partially ordered set such that for any
  $p,p' \in \mathbf{P}$, the partially ordered set $\{q : p \leq q, p'
  \leq q\}$ is either empty or has a minimal element (\ie{} $p$ and
  $p'$ have a coproduct). Then
  $\Tw_{!}\mathbf{P}$ is also a partially ordered set, and $p \leq
  (p')^{\vee}$ \IFF{} there exists an element $q$ such that $p \leq q$ and $p'
  \leq q$. \qed
\end{corollary}

\begin{example}\label{ex:epsTwn}
  The partially ordered set $\bbS^{n}$ (see Definition~\ref{defn:bbS})
  satisfies the hypothesis of Corollary~\ref{cor:Tw!poset}: If
  $\max\{i,i'\} > \min\{j,j'\}$ then there are clearly no pairs
  $(x,y)$ such that $(i,j) \leq (x,y)$ and $(i',j') \leq (x,y)$. On
  the other hand, if $\max\{i,i'\} \leq \min\{j,j'\}$ then
  $(\max\{i,i'\},\min\{j,j'\})$ is a minimal such object.  Thus
  $\Tw_{!}\bbS^{n}$ is equivalent to the partially ordered set
  $\tbbS^{n}$ with objects $(i,j)$ and $(i,j)^{\vee}$ with $0 \leq i
  \leq j \leq n$ where $(i,j) \leq (i',j')^{\vee}$ \IFF{}
  $\max\{i,i'\} \leq \min\{j,j'\}$.
\end{example}

\begin{example}\label{ex:epsTwnprod}
  Let $\mathbf{m} = (m_{1},\ldots,m_{n})$. Then $\tbbS^{\mathbf{m}}
  := \Tw_{!}\bbS^{\mathbf{m}}$ is also a partially ordered set
  --- there is a unique morphism
  $((i_{1},j_{1}),\ldots,(i_{n},j_{n})) \to
  ((i'_{1},j'_{1}),\ldots,(i'_{n},j'_{n}))^{\vee}$ if $\max\{i_{r},i'_{r}\}
  \leq \min\{j_{r},j'_{r}\}$ for all $r = 1,\ldots,n$, and no morphism
  otherwise. If $\mathbf{m} = (1,\ldots,1)$ we will also denote
  $\tbbS^{1,\ldots,1}$ by $\tSp{}^{n}$.
\end{example}

We now prove some further useful properties of the functor $\Tw^{r}_{!}$.

\begin{lemma}
  For any \icats{} $\mathcal{C}$ and $\mathcal{D}$ there is a natural
  equivalence
  \[ \Tw^{r}_{!}(\mathcal{C} \times \mathcal{D}) \simeq
    \Tw^{r}_{!}(\mathcal{C}) \times_{\Delta^{1}} \Tw^{r}_{!}(\mathcal{D})\]
\end{lemma}
\begin{proof}
  In the proof of Proposition~\ref{propn:Tw!desc} we saw that there is
  an equivalence
  $\Tw^{r}_{!}(\mathcal{C}) \simeq \mathfrak{L}(\phi_{\mathcal{C}})$,
  where $\phi_{\mathcal{C}} \colon \mathcal{E}_{\mathcal{C}} \to
  \mathcal{C} \times \mathcal{C}$ is the right fibration for the
  functor $\mathcal{C}^{op} \times \mathcal{C}^{\op} \to \mathcal{S}$
  given by $(c,c') \mapsto \|\mathcal{C}_{c,c'/}\|$. Since
  $\mathfrak{L}$ is an equivalence, it preserves products, and so
  $\Tw^{r}_{!}(\mathcal{C}) \times_{\Delta^{1}}
  \Tw^{r}_{!}(\mathcal{D})$ is equivalent to
  $\mathfrak{L}(\phi_{\mathcal{C}} \times \phi_{\mathcal{D}})$. Here
  $\phi_{\mathcal{C}}\times \phi_{\mathcal{D}}$ is the right fibration
  for the functor $(\mathcal{C}^{\op} \times
  \mathcal{D}^{\op})^{\times 2} \to \mathcal{S}$ taking $(c,d,c',d')$
  to $\|\mathcal{C}_{c,c'/}\| \times \|\mathcal{D}_{d,d'/}\|$. But it
  is a classical fact that the
  functor $\|\blank\| \colon \CatI \to \mathcal{S}$ preserves
  products, so we can identify $\phi_{\mathcal{C}} \times
  \phi_{\mathcal{D}}$ with $\phi_{\mathcal{C} \times \mathcal{D}}$. In
  other words, we have a natural equivalence
  \[ \Tw^{r}_{!}(\mathcal{C} \times \mathcal{D}) \simeq
    \mathfrak{L}(\phi_{\mathcal{C} \times \mathcal{D}}) \simeq
    \mathfrak{L}(\phi_{\mathcal{C}} \times \phi_{\mathcal{D}}) \simeq
    \Tw^{r}_{!}(\mathcal{C}) \times_{\Delta^{1}}
    \Tw^{r}_{!}(\mathcal{D}),\]
  as required.
\end{proof}

\begin{lemma}\label{lem:Tw!initial}
  For any \icat{} $\mathcal{C}$ we have a natural equivalence
  \[ \Tw^{r}_{!}(\mathcal{C}^{\triangleleft}) \simeq
    \Tw_{!}^{r}(\mathcal{C})^{\triangleleft\triangleright}. \]
\end{lemma}
\begin{proof}
  From our description of $\Tw^{r}_{!}$ it is clear that
  $\Tw^{r}_{!}(\mathcal{C}^{\triangleleft})$ has an initial and
  terminal object (both induced by the cone point). Thus the functor
  $\Tw^{r}_{!}\mathcal{C} \to
  \Tw^{r}_{!}(\mathcal{C}^{\triangleleft})$ induced by the inclusion
  $\mathcal{C} \hookrightarrow \mathcal{C}^{\triangleleft}$
  canonically extends to a functor
  $\Tw_{!}^{r}(\mathcal{C})^{\triangleleft\triangleright} \to
  \Tw^{r}_{!}(\mathcal{C}^{\triangleleft})$, and from the description
  of these \icats{} we see that this is fully faithful and essentially surjective.
\end{proof}

\begin{proposition}\label{propn:CtoTw!Ccoinit}\ 
  \begin{enumerate}[(i)]
  \item For any \icat{} $\mathcal{C}$, the inclusion
    $\mathcal{C} \to \Tw^{r}_{!}(\mathcal{C})$ is coinitial.
  \item If a functor $f \colon \mathcal{C} \to \mathcal{D}$ is
    coinitial, so is the induced functor
    \[\Tw^{r}_{!}(f) \colon
    \Tw^{r}_{!}\mathcal{C} \to \Tw^{r}_{!}\mathcal{D}.\]
  \end{enumerate}
\end{proposition}
\begin{proof}
  For (i), by \ThmA{} we need to show that for any
  $X \in \Tw^{r}_{!}(\mathcal{C})$, the \icat{}
  $\mathcal{C}_{/X} := \mathcal{C} \times_{\Tw^{r}_{!}(\mathcal{C})}
  \Tw^{r}_{!}(\mathcal{C})_{/X}$ is weakly contractible. If $X$ is the
  image of $x \in \mathcal{C}$ then this is obvious as
  $\mathcal{C}_{/X} \simeq \mathcal{C}_{/x}$ has a terminal
  object. Suppose therefore that $X$ is $x^{\vee}$ for some
  $x \in \mathcal{C}$. Then the forgetful functor
  $\mathcal{C}_{/x^{\vee}} \to \mathcal{C}$ is the right fibration for
  the functor $\mathcal{C}^{\op} \to \mathcal{S}$ taking
  $c \in \mathcal{C}$ to
  $\Map_{\Tw_{!}^{r}\mathcal{C}}(c, x^{\vee}) \simeq
  \|\mathcal{C}_{c,x/}\|$, where the equivalence is from
  Proposition~\ref{propn:Tw!desc}. The cartesian fibration for the
  functor $\mathcal{C}^{\op} \to \CatI$ taking $c$ to
  $\mathcal{C}_{c,x/}$ is given by
  $\mathcal{C}_{x/}\times_{\mathcal{C}} \mathcal{C}^{\Delta^{1}} \to
  \mathcal{C}$ (where the fibre product is via evaluation at 1 and the
  projection to $\mathcal{C}$ via evaluation at 0), so we have a
  coinitial functor
  \[\mathcal{C}_{x/}\times_{\mathcal{C}} \mathcal{C}^{\Delta^{1}} \to
    \mathcal{C}_{/x^{\vee}}\] 
  by \cite{cois}*{Lemma A.3.8}. Since coinitial functors
  are in particular weak homotopy equivalences, this implies that
  $\mathcal{C}_{/x^{\vee}}$ is weakly contractible precisely when
  $\mathcal{C}_{x/}\times_{\mathcal{C}} \mathcal{C}^{\Delta^{1}}$ is
  weakly contractible. Now applying \ThmA{} to the functor
  $\txt{ev}_{1} \colon \mathcal{C}^{\Delta^{1}} \to \mathcal{C}$ we
  see that this holds for all $x \in \mathcal{C}$ precisely when
  $\txt{ev}_{1}$ is cofinal. But $\txt{ev}_{1}$ is a cocartesian
  fibration whose fibres (\ie{} the slices $\mathcal{C}_{/c}$ for
  $c \in \mathcal{C}$) are all weakly contractible, hence
  $\txt{ev}_{1}$ 
  is cofinal for any \icat{} $\mathcal{C}$ by \cite{cois}*{Lemma A.3.5}. This proves (i).

  To prove (ii), observe that we have a commutative square
  \csquare{\mathcal{C}}{\mathcal{D}}{\Tw^{r}_{!}\mathcal{C}}{\Tw^{r}_{!}\mathcal{D}.}{f}{}{}{\Tw^{r}_{!}f}
  Here the vertical morphisms are coinitial by (i) and the top
  horizontal morphism is coinitial by assumption. Hence the bottom
  horizontal morphism is also coinitial by \cite{HTT}*{Proposition
    4.1.1.3(2)}.
\end{proof}

\chapter{Background on Derived Algebraic Geometry}\label{sec:dag}
In this appendix we review some background on derived algebraic
geometry, mainly following \cite{HAG2,PTVV,CPTVV}. Other useful
references on the subject include
\cite{SAG,GaitsgoryRozenblyum1,GaitsgoryRozenblyum2}.

We first look at some derived algebra: In \S\ref{subsec:derring} we
discuss derived (or homotopy-coherent) rings and modules over a field
of characteristic zero, and in \S\ref{subsec:algcotgt} we recall the
cotangent complex for a morphism of derived rings and its basic
properties. Then we move on to algebraic geometry in the broadest
sense and discuss \'etale sheaves and derived stacks in
\S\ref{subsec:derstack}. We next look at sheaves on derived stacks in
\S\ref{subsec:sheaves} and quasicoherent sheaves in
\S\ref{subsec:qcoh}. In \S\ref{subsec:affinemor} we recall affine
morphisms between derived stacks before we introduce geometric
morphisms and derived Artin stacks in \S\ref{subsec:geommor}. Base
change and the projection formula for quasicoherent sheaves play an
important role in the paper, and we review the proof of these results,
due to Ben-Zvi, Francis, and Nadler~\cite{BenZviFrancisNadler},
in \S\ref{subsec:basechg}. We then make a categorial digression in
\S\ref{subsec:natbasechg} where we discuss the naturality of
pushforward morphisms and use this to formulate the naturality of base
change and the projection formula. Turning back to geometry, we review
cotangent complexes for derived stacks in \S\ref{subsec:cotgt}. We end
by recalling the de Rham complex and differential forms on derived
stacks in \S\ref{subsec:deRham} and \S\ref{subsec:diffform},
respectively, following \cite{PTVV,CPTVV}.

\section{Derived Rings}\label{subsec:derring}
The basic idea of derived algebraic geometry is to replace ordinary
commutative rings with some notion of ``derived commutative
rings''. In general there are several non-equivalent ways of making
this notion precise:
\begin{enumerate}[(1)]
\item Connective\footnote{I.e.~non-negatively graded homologically or
    non-positively graded cohomologically.} commutative differential graded algebras --- this is
  perhaps the easiest notion to come to grips with, but it only really
  makes sense over a field of characteristic zero.
\item Simplicial commutative rings --- this is the notion closest to ordinary
  algebraic geometry, and leads to the what is usually called
  \emph{derived algebraic geometry} as in \cite{HAG2}.
\item Connective commutative ring spectra --- this is, unsurprisingly,
  the notion most relevant for applications in stable homotopy theory,
  and leads to what Lurie calls \emph{spectral algebraic geometry} \cite{SAG}.
\end{enumerate}
In this paper we will always be working over a fixed field $\k$ of
characteristic zero, and in this case these three notions\footnote{I.e.\
commutative algebras in connective (co)chain complexes of
$\k$-vector spaces, simplicial commutative $\k$-algebras, and connective
commutative algebras over the Eilenberg--MacLane spectrum $H\k$.} give the
same homotopy theory of derived rings. In most of this paper we work
entirely at the level of \icats{}, where there is no need
to choose any particular model, and we generally remain
agnostic whenever possible. 

\begin{definition}
  We denote by $\Mod_{\k}$ the symmetric monoidal \icat{} of
  \emph{$\k$-modules}, obtained either by inverting the
  quasi-isomorphisms between (unbounded) (co)chain complexes of
  $\k$-modules or as modules over $H\k$ in the \icat{} of
  spectra.
  Similarly, we write $\CAlg_{\k} := \CAlg(\Mod_{\k})$ for the
  \icat{} of \emph{commutative $\k$-algebras}, \ie{} commutative
  algebras in $\Mod_{\k}$; this can alternatively be obtained by
  inverting the quasi-isomorphisms between commutative dg-algebras
  over $\k$, or as commutative $H\k$-algebras in spectra. These are both
  symmetric monoidal \icats{} with respect to the derived tensor
  product of chain complexes, or equivalently the relative smash
  product over $H \k$.
\end{definition}

\begin{convention}
  We will to the greatest extent possible attempt to remain agnostic
  about homological versus cohomological grading conventions. In the
  topologist's convention an object $M$ of $\Mod_{\k}$ has
  \emph{homotopy groups} $\pi_{i}M$ ($i \in \mathbb{Z}$). If $M$ is
  modelled by a chain complex $\tilde{M}$ then
  $\pi_{i}M \cong H_{i}\tilde{M}$, and if $M$ is modelled by a cochain
  complex $\tilde{M}$ then $\pi_{i}M \cong H^{-i}\tilde{M}$. We say
  that $M$ is \emph{connective} if $\pi_{i}M = 0$ for $i < 0$ --- in
  cohomological terms, this means that $M$ is non-positively
  graded. The \emph{shifts} $M[n]$ can be defined model-independently
  by induction via the cartesian and cocartesian squares
  \nolabelcsquare{M[n]}{0}{0}{M[n+1],}
  starting with $M[0] := M$. Thus to a topologist $M[n]$ is
  the $n$-fold suspension $\Sigma^{n}M$ for $n > 0$ and the $n$-fold
  loop object $\Omega^{-n}M := \Sigma^{n}M$ for $n < 0$. This agrees
  with the standard homological and cohomological grading conventions
  for shifts
  in terms of (co)chain complexes.
\end{convention}

\begin{definition}
  We write $\Modconn_{\k}$ for the full subcategory of $\Mod_{\k}$
  spanned by the connective $\k$-modules, and
  $\CAlgconn_{\k} := \CAlg(\Modconn_{\k})$ for the full subcategory of
  $\CAlg_{\k}$ spanned by the connective commutative $\k$-algebras, or
  equivalently the \icat{} of commutative algebras in the symmetric
  monoidal \icat{} $\Modconn_{\k}$. (This can also be obtained by
  inverting the weak homotopy equivalences between simplicial
  commutative $\k$-algebras.)
\end{definition}

To do derived algebraic geometry over $\k$, we want to work with
\emph{connective} commutative $\k$-algebras, \ie{} the objects of
$\CAlgconn_{\k}$ --- many notions from algebraic geometry can be
extended to this setting, while it is still unclear to what extent it
is possible to ``do algebraic geometry'' with non-connective derived
rings. In particular, a number of definitions from commutative algebra
can easily be extended to connective derived algebras in terms of
homotopy (or (co)homology) groups:
\begin{definition}
  A morphism $f \colon A \to B$ in $\CAlgkc$ is \emph{strong} if for
  every $n$ the natural morphism
  $\pi_{n}(A) \otimes_{\pi_{0}(A)} \pi_{0}(B) \to \pi_{n}(B)$ is an
  isomorphism. We then say $f$ is \emph{flat}, \emph{faithfully flat},
  \emph{\'{e}tale}, or \emph{smooth} if $f$ is strong and the morphism
  $\pi_{0}f \colon \pi_{0}A \to \pi_{0}B$ is, respectively, flat,
  faithfully flat, \'{e}tale, or smooth in the ordinary sense.
\end{definition}

\begin{definition}
  We say a strong morphism $f \colon A \to B$ is \emph{surjective} if
  the induced morphism $\Spec \pi_{0}f \colon \Spec \pi_{0}B \to \Spec
  \pi_{0}A$ is surjective.
\end{definition}

The following is a useful finiteness condition for morphisms of
commutative $\k$-algebras:
\begin{definition}
  A morphism $f \colon A \to B$ in $\CAlgkc$ is \emph{finitely
    presented} if it is a compact object in $(\CAlgkc)_{A/}$. (See
  \cite{HAG2}*{2.2.2.4} for a more explicit criterion.)
\end{definition}

\section{Cotangent Complexes of Algebras}\label{subsec:algcotgt}
Every $A \in \CAlg_{\k}$ has a \emph{cotangent complex} $\mathbb{L}_{A}
\in \Mod_{A}$. This has the universal property that 
\[ \Map_{\Mod_{A}}(\mathbb{L}_{A}, M) \simeq \Map_{\CAlg_{\k/A}}(A,
A[M]),\]
where $A[M]$ is the module $A \oplus M$ equipped with the trivial
square-zero multiplication. Equivalently, maps from $\mathbb{L}_{A}$
to $M$ are $M$-valued \emph{derivations} on $A$, in that the space
$\txt{Der}_{\k}(A,M)$ of derivations can be defined as $\Map_{\CAlg_{\k/A}}(A,
A[M])$. More generally every morphism $A \to B \in \CAlg_{\k}$
has a \emph{relative cotangent complex} $\mathbb{L}_{B/A} \in
\Mod_{B}$, which satisfies
\[ \Map_{\Mod_{B}}(\mathbb{L}_{B/A}, M) \simeq \Map_{\CAlg_{A/B}}(B,
B[M]) =: \txt{Der}_{A}(B, M).\]
There are two (equivalent) approaches to defining these structures:
\begin{itemize}
\item Choosing a model (such as simplicial commutative $\k$-algebras
  or commutative differential graded $\k$-algebras) for $\CAlg_{\k}$,
  we can use the model to define the square-zero extensions
  $A[M]$. This determines the cotangent complex as the module
  representing the functor $\txt{Der}_{A}(A, \blank)$. See \cite[\S
  1.2.1]{HAG2}, \cite{PortaVezzosiSqZ}, or \cite{BasterraTAQ} for
  more details.
\item We may identify $\Mod_{A}$ with the stabilization of the \icat{}
  $\CAlg_{\k/A}$. Then we have a ``suspension spectrum''/``infinite
  loop space'' adjunction
  \[ \Sigma^{\infty}_{+} : \CAlg_{\k/A} \rightleftarrows \Mod_{A} :
  \Omega^{\infty} \]
  and we define $\mathbb{L}_{A} := \Sigma^{\infty}_{+}A$ and $A[M] :=
  \Omega^{\infty}M$. This is carried out in \cite{HA}*{\S 7.3.2}.
\end{itemize}

\begin{remark}\label{rmk:sqzpb}
  For $B \to C$ in $\CAlg_{A}$ and $M \in \Mod_{C}$ we have a natural
  pullback square \nolabelcsquare{B[M]}{C[M]}{B}{C.}  
  We can use this to reformulate the universal property of the
  cotangent complex: Let
  $\Mod \to \CAlg_{\k}$ denote the cocartesian (and cartesian)
  fibration corresponding to the functor $\CAlg_{\k} \to \LCatI$ taking
  $A$ to $\Mod_{A}$, then the cotangent complex $\mathbb{L}_{A}$ of
  $A$ has the universal property
  \[ \Map_{\Mod}(\mathbb{L}_{A}, M) \simeq \Map_{\CAlg_{\k}}(A,
  B[M]),\]
  for $M$ a $B$-module. To see this, it suffices to see we have an
  equivalence on the fibre over every map $f \colon A \to B$, where we indeed
  get 
  \[ \Map_{\Mod_{A}}(\mathbb{L}_{A}, f_{*}M) \simeq
  \Map_{\CAlg_{\k/A}}(A, A[f_{*}M]) \simeq \Map_{\CAlg_{\k/B}}(A,
  B[M]).\]
  where $f_{*}$ denotes the forgetful functor $\Mod_{B} \to \Mod_{A}$
  induced by $f$. Thus the cotangent complex gives a left adjoint to
  the square-zero extension functor $\Mod \to \CAlg_{\k}$, which takes
  $(A,M)$ to $A[M]$.
\end{remark}

\begin{lemma}
  For $A \to B \to C$ in $\CAlg_{\k}$, there is for $M \in \Mod_{C}$ a
  natural pullback square \nolabelcsquare{\Der_{B}(C, M)}{\Der_{A}(C,
    M)}{*}{\Der_{A}(B, M),} corresponding to a cofibre sequence
  \nolabelcsquare{\mathbb{L}_{B/A} \otimes_{B}
    C}{\mathbb{L}_{C/A}}{0}{\mathbb{L}_{C/B}} in $\Mod_{C}$.
\end{lemma}
\begin{proof}
  (See \eg{} \cite{HAG2}*{Proposition
    1.2.1.6}.) For $M \in \Mod_{C}$, consider the map
  \[\Map_{\CAlg_{A/C}}(C, C[M]) \to \Map_{\CAlg_{A/C}}(B,
  C[M])\]
  given by composition with $B \to C$.
  Using the pullback square \nolabelcsquare{B[M]}{C[M]}{B}{C} of
  Remark~\ref{rmk:sqzpb} we may identify the target of this map with
  $\Map_{\CAlg_{A/B}}(B, B[M])$, so we have a map
  $\Der_{A}(C, M) \to \Der_{A}(B, M)$. Moreover, the fibre of the map
  at the trivial derivation $B \to B[M]$, which corresponds to the
  composite $B \to C \to C[M]$ (with the second map again the trivial
  derivation) is precisely \[\Map_{\CAlg_{A,B//C}}(B, C[M]) \simeq
    \Map_{\CAlg_{B/C}}(B, C[M]) \simeq \Der_{B}(C, M),\]
  since $\CAlg_{A,B/} \simeq \CAlg_{B}$.
  This gives the desired pullback square, which induces the cofibre
  sequence of cotangent complexes via the Yoneda Lemma.
\end{proof}

\begin{remark}
  As in Remark~\ref{rmk:sqzpb} we can reformulate the universal
  property of $\mathbb{L}_{B/A}$ in terms of the \icat{} $\Mod$:
  Namely, for $M \in \Mod$ over $C \in \CAlg_{\k}$ there is a natural
  pullback square
  \nolabelcsquare{\Map_{\Mod}(\mathbb{L}_{B/A}, M)}{\Map_{\CAlg_{\k}}(B,
    C[M])}{\Map_{\CAlg_{\k}}(A, C)}{\Map_{\CAlg_{\k}}(A, C[M])}
  where the right vertical map is given by composition with $A \to B$
  and the lower horizontal map by composition with the inclusion $C
  \to C[M]$.
\end{remark}

\begin{remark}\label{rmk:sqznat}
  The square-zero extension functors are compatible with base change
  in the sense that for $A \to B$ in $\CAlg_{\k}$ and $M \in \Mod_{A}$
  we have a natural equivalence
  $B \otimes_{A} A[M] \simeq B[B \otimes_{A} M]$. In other words, we
  have commutative squares
  \csquare{\Mod_{A}}{\CAlg_A}{\Mod_B}{\CAlg_B}{A[\blank]}{B \otimes_A
    \blank}{B \otimes_A \blank}{B[\blank]} and so square-zero
  extension gives a natural transformation
  $\Mod_{(\blank)} \to \CAlg_{(\blank)}$ of functors
  $\CAlg_{\k} \to \LCatI$.
\end{remark}

\section{Derived Stacks}\label{subsec:derstack}
We now turn from algebra to geometry, and introduce the most general
class of algebro-geometric objects we consider, the \emph{derived
  stacks}. These are a special case of a more general notion of
\'etale sheaves, which we discuss first:
\begin{definition}
  Suppose $\mathcal{C}$ is a very large presentable \icat{}. An
  \emph{\'etale sheaf} valued in $\mathcal{C}$ is a functor
  $\mathcal{F} \colon \CAlgkc \to \mathcal{C}$ such that
  \begin{enumerate}[(1)]
  \item $\mathcal{F}$ takes finite 
    products in $\CAlgkc$ to products in $\mathcal{C}$.
  \item If $\phi \colon A \to B$ is a surjective \'etale morphism
    then $\mathcal{F}(A)$ is the limit of the cosimplicial diagram
    \[ \mathcal{F}(B) \rightrightarrows \mathcal{F}(B \otimes_{A} B)\,
      \cdots \]
    obtained by applying $\mathcal{F}$ to the dual \v{C}ech nerve of
    $\phi$.
  \end{enumerate}
  We write $\Shet(\mathcal{C})$ for the full
  subcategory of $\Fun(\CAlgkc, \mathcal{C})$ spanned by the
  \'etale sheaves.
\end{definition}

\begin{remark}
  The collection of surjective \'etale morphisms in $\CAlgkcop$
  is easily seen to satisfy the criteria of \cite{DAG7}*{Proposition
    5.1}. The \icat{} $\CAlgkcop$ therefore has a Grothendieck
  topology --- the \emph{\'etale topology} --- where the covering
  families are those families of morphisms
  $\{f_{i} \colon A \to A_{i}\}$, $i \in I$, that contain a finite
  collection of morphisms $f_{j}$, $j \in J \subseteq I$, such that
  $A \to \prod_{j \in J} A_{j}$ is a surjective \'etale
  morphism. Such families are called \emph{\'etale covering families}
  in \cite{HAG2}. By \cite{DAG7}*{Proposition 5.7}, the \'etale
  sheaves according to our definition are precisely the sheaves with
  respect to the \'etale topology.
\end{remark}

\begin{definition}\label{defn:dSt}
  A \emph{derived stack} is an \'etale sheaf valued in the \icat{}
  $\hS$ of large spaces. We write $\dSt := \Shet(\hS)$ for the \icat{}
  of derived stacks. For any derived stack $S$ we use the abbreviation
  $\dSt_{S} := \dSt_{/S}$ when we think of $S$ as a fixed \emph{base}
  stack, and refer to its objects as stacks \emph{over} $S$.
\end{definition}

\begin{proposition}\label{propn:subcan}
  The \v{C}ech conerve of a surjective \'etale morphism is a limit
  diagram in $\CAlgkc$. In other words, the presheaf
  \[ \Spec A := \Map_{\CAlgkc}(A, \blank) \colon \CAlgkc \to
  \mathcal{S}\]
  is a derived stack for all $A \in \CAlgkc$, \ie{} the \'etale
  topology on $\CAlgkcop$ is \emph{subcanonical}.
\end{proposition}
\begin{proof}
  This is the content of \cite{HAG2}*{Corollary 1.3.2.5} combined with
  \cite{HAG2}*{Lemma 2.1.1.1}, or \cite{SAG}*{Corollary
    D.6.3.4}.
\end{proof}

\begin{corollary}
  The Yoneda embedding $(\CAlgkc)^{\op} \hookrightarrow \Fun(\CAlgkc, \hS)$
  factors through $\dSt$.\qed
\end{corollary}

\begin{definition}
  We write $\Spec \colon (\CAlgkc)^{\op} \hookrightarrow \dSt$ for the Yoneda
  embedding, and write $\dAff$ for its essential image; the objects of
  $\dAff$ we refer to as \emph{affine derived stacks}.
\end{definition}

\begin{lemma}\label{lem:dStunivprop}
  The \icat{} $\dSt$ is an $\infty$-topos, and has the universal
  property that for every very large presentable \icat{}
  $\mathcal{C}$, the functor $\Fun(\dSt^{\op}, \mathcal{C}) \to
  \Fun(\CAlgkc, \mathcal{C})$ obtained by restriction along 
   $\Spec \colon \CAlgkcop \hookrightarrow \dSt$, restricts to
   an equivalence
   \[ \Fun^{R}(\dSt^{\op}, \mathcal{C}) \isoto \Shet(\mathcal{C}).\]
   where $\Fun^{R}(\mathcal{D}, \mathcal{C})$ denotes the full
   subcategory of $\Fun(\mathcal{D}, \mathcal{C})$ spanned by functors
   that preserve all large limits.
\end{lemma}

\begin{proof}
  By \cite{DAG7}*{Proposition 5.7} the \icat{} $\dSt$ is the \icat{}
  of sheaves with respect to a Grothendieck topology on $\CAlgkcop$,
  hence it is an $\infty$-topos by \cite{HTT}*{Proposition 6.2.2.7}.
  The universal property then follows from \cite{DAG5}*{Proposition
    1.1.12}.
\end{proof}

\begin{remark}\label{rmk:dstcolimaff}
  Note that the inverse to the restriction along $\Spec$ is given by
  right Kan extension. As a consequence, for any derived
  stack $Y$ the functor
  $\Map_{\dSt}(\blank, Y) \colon \dSt^{\op} \to \hS$, which clearly
  preserves limits, is right Kan extended from $\dAff^{\op}$. Thus for
  any derived stack $X$, we have
  \[ \Map_{\dSt}(X, Y) \simeq \lim_{\Spec A \to X \in
    \dAff_{/X}} \Map_{\dSt}(\Spec A, Y).\]
  By the Yoneda Lemma, this implies that for every $X$ we have a
  canonical equivalence
  \[ X \simeq \colim_{\Spec A \to X \in \dAff_{/X}} \Spec A.\]
\end{remark}

\begin{definition}\label{defn:betti}
  For $\Sigma \in \mathcal{S}$ we write $\Sigma_{B} \in \dSt$ for the colimit
  $\colim_{\Sigma} \Spec \k$ of the constant diagram with value $\Spec
  \k$, and call this the \emph{Betti stack} of $\Sigma$. Equivalently,
  $\Sigma_{B}$ is the sheafification of the constant presheaf on
  $\dAff$ with value $\Sigma$, \ie{} the constant sheaf with
  value $\Sigma$.
\end{definition}

\section{Sheaves on Derived Stacks and $\mathcal{O}_{X}$-Modules}\label{subsec:sheaves}
In this section we introduce sheaves on derived stacks and sheaves of $\mathcal{O}_{X}$-modules.
\begin{definition}
  For $X$ a derived stack, we write $\dAff_{/X} := \dAff \times_{\dSt}
  \dSt_{/X}$, and we say a morphism in $\dAff_{/X}$ is an \emph{\'etale
    cover} if its image in $\dAff$ is an \'etale cover. If
  $\mathcal{C}$ is a very large presentable \icat{}, we say a functor
  $\mathcal{F} \colon \dAff_{/X}^{\op} \to \mathcal{C}$ is an
  \emph{\'etale sheaf on $X$} if
  \begin{enumerate}[(1)]
  \item $\mathcal{F}$ takes finite coproducts in $\dAff_{/X}$ to
    products.
  \item $\mathcal{F}$ takes the \v{C}ech nerves of \'etale covers in
    $\dAff_{/X}$ to limit diagrams.
  \end{enumerate}
  We write $\Shet_{X}(\mathcal{C})$ for the full subcategory of
  $\Fun(\dAff_{/X}^{\op}, \mathcal{C})$ spanned by the \'etale sheaves.
\end{definition}

As a variant of Lemma~\ref{lem:dStunivprop}, we have:
\begin{lemma}
  Let $\mathcal{C}$ be a very large presentable \icat{}. The functor \[\Fun(\dSt_{/X}^{\op}, \mathcal{C}) \to
  \Fun(\dAff_{/X}^{\op}, \mathcal{C}),\] induced by composition with
  the inclusion $\dAff_{/X} \hookrightarrow \dSt_{/X}$, restricts to
  an equivalence
  \[ \Fun^{R}(\dSt_{/X}^{\op}, \mathcal{C}) \isoto \Shet_{X}(\mathcal{C}).\]
\end{lemma}

\begin{remark}
  As a consequence, by \cite{HA}*{Proposition 4.8.1.17} we can
  identify $\Shet_{X}(\mathcal{C})$ with the cocomplete tensor
  product $\dSt_{/X} \otimes \mathcal{C}$.
\end{remark}

\begin{lemma}\label{lem:shalg}
  Suppose $\mathcal{C}$ is a very large presentably symmetric monoidal
  \icat{}. Then $\Shet_{X}(\mathcal{C})$ inherits a symmetric monoidal
  structure as a localization of the pointwise tensor product on
  $\Fun(\dAff_{/X}^{\op}, \mathcal{C})$. Explicitly, the tensor
  product of two sheaves $F, G$ is the sheafification of the presheaf
  $p \in \dAff_{/X}^{\op} \mapsto F(p) \otimes G(p)$. This symmetric
  monoidal \icat{} has the universal property that for any \iopd{}
  $\mathcal{O}$ the \icat{}
  $\Alg_{\mathcal{O}}(\Shet_{X}(\mathcal{C}))$ can be naturally
  identified with
  \begin{enumerate}[(i)]
  \item the full subcategory of
    $\Alg_{\mathcal{O}}(\Fun(\dAff_{/X}^{\op}, \mathcal{C}))$
    spanned by the $\mathcal{O}$-algebras whose images in
    $\Fun(\dAff_{/X}^{\op}, \mathcal{C})$ lie in
    $\Shet_{X}(\mathcal{C})$ for all objects of $\mathcal{O}$,
  \item the \icat{} $\Shet_{X}(\Alg_{\mathcal{O}}(\mathcal{C}))$.
  \end{enumerate}
\end{lemma}
\begin{proof}
  By \cite{HA}*{Proposition 2.2.1.9}  the localized
  symmetric monoidal structure on $\Shet_{X}(\mathcal{C})$ exists
  since \'etale covers are stable under base change. This
  localization has the universal property that the functor
  \[ \Alg_{\mathcal{O}}(\Shet_{X}(\mathcal{C})) \to
  \Alg_{\mathcal{O}}(\Fun(\dAff_{/X}^{\op}, \mathcal{C}))\]
  given by composition with the lax symmetric monoidal inclusion
  \[\Shet_{X}(\mathcal{C}) \hookrightarrow \Fun(\dAff_{/X}^{\op},
  \mathcal{C}))\]
  is fully faithful with image the $\mathcal{O}$-algebras whose images
  in $\Fun(\dAff_{/X}^{\op}, \mathcal{C})$ lie in the full
  subcategory of sheaves.  This proves (i).

  To prove (ii), we first note that the pointwise tensor product on
  $\Fun(\mathcal{I}, \mathcal{C})$ for any \icat{} $\mathcal{I}$ and
  symmetric monoidal \icat{} $\mathcal{C}$ has
  the universal property that
  \[ \Alg_{\mathcal{O}}(\Fun(\mathcal{I}, \mathcal{C})) \simeq
    \Fun(\mathcal{I}, \Alg_{\mathcal{O}}(\mathcal{C}));\]
  for example, this follows from the universal property of the
  Boardman--Vogt tensor product of \iopds{} and its symmetry, since we
  can identify functors $\mathcal{I} \to \mathcal{C}$ with algebras in
  $\mathcal{C}$ for $\mathcal{I}$ viewed as an \iopd{} with only unary
  operations.
  
  Thus in particular the \icat{} $\Alg_{\mathcal{O}}(\Fun(\dAff_{/X}^{\op}, \mathcal{C}))$ is
  naturally equivalent to $\Fun(\dAff_{/X}^{\op},
  \Alg_{\mathcal{O}}(\mathcal{C}))$. 
  Combining this with (i), we see that
  $\Alg_{\mathcal{O}}(\Shet_{X}(\mathcal{C}))$ is naturally equivalent
  to the full subcategory of
  $\Fun(\dAff_{/X}^{\op}, \Alg_{\mathcal{O}}(\mathcal{C}))$ spanned
  by functors whose images in $\Fun(\dAff_{/X}^{\op}, \mathcal{C})$
  are sheaves. Since these forgetful functors, given by evaluation at
  the objects of $\mathcal{O}_{\angled{1}}$, jointly detect limits,
  this is precisely the full subcategory
  $\Shet_{X}(\Alg_{\mathcal{O}}(\mathcal{C}))$ of \'etale sheaves
  valued in $\Alg_{\mathcal{O}}(\mathcal{C})$.
\end{proof}

\begin{definition}
  Let
  $\mathcal{O}_{X} \colon \dAff_{/X}^{\op} \to \dAff^{\op}\simeq
  \CAlgkc$
  denote the forgetful functor that takes $\Spec A \to X$ to $A$. This
  is an \'etale sheaf on $X$ valued in $\CAlgkc$ by
  Proposition~\ref{propn:subcan}; since the limit diagrams appearing
  in the definition of \'etale sheaves are also limit diagrams in
  $\CAlgk$, we may also view $\mathcal{O}_{X}$ as an \'etale sheaf
  on $X$ valued in $\CAlgk$. By Lemma~\ref{lem:shalg}, we may
  equivalently view $\mathcal{O}_{X}$ as a commutative algebra in
  $\Shet_{X}(\Modconn_{\k})$ or $\Shet_{X}(\Mod_{\k})$.
\end{definition}

\begin{definition}
  A \emph{sheaf of $\mathcal{O}_{X}$-modules} is an
  $\mathcal{O}_{X}$-module in $\Shet_{X}(\Mod_{\k})$; we write
  \[\Mod_{\mathcal{O}_{X}} :=
    \Mod_{\mathcal{O}_{X}}(\Shet_{X}(\Mod_{\k}))\]
  for the \icat{} of these.
\end{definition}

\begin{remark}\label{rmk:OXpsh}
  Equivalently, by Lemma~\ref{lem:shalg} applied to the left module
  operad, a sheaf of
  $\mathcal{O}_{X}$-modules is a module
  over $\mathcal{O}_{X}$ in $\Fun(\dAff_{/X}^{\op}, \Mod_{\k})$ whose
  underlying presheaf is a sheaf. 
\end{remark}

We now give an alternative description of sheaves of
$\mathcal{O}_{X}$-modules that will be used in the next section to
compare two definitions of quasicoherent sheaves:
\begin{proposition}\label{propn:OXpshsect}
  Let $\mathcal{M} \to \CAlgk$ denote the cocartesian fibration
  associated to the functor $\CAlgk \to \CatI$ that sends $A$ to
  $\Mod_{A}$, and for $X$ a derived stack let $\mathcal{M}_{X}$ denote
  the pullback $\mathcal{M} \times_{\CAlgk} \dAff^{\op}_{/X}$. Then an
  $\mathcal{O}_{X}$-module in $\Fun(\dAff_{/X}^{\op}, \Mod_{\k})$ is
  equivalently a section $\dAff^{\op}_{/X} \to \mathcal{M}_{X}$.
\end{proposition}
\begin{proof}
  Let $\txt{LM}$ denote the \iopd{} for left modules; then
  $\mathcal{M}$ can be identified with the \icat{}
  $\Alg_{\txt{LM}}(\Mod_{\k})$. 
  Thus the forgetful functor $\mathcal{M} \to \CAlg_{\k}$ is
  the restriction induced by a map of operads. The pointwise tensor
  product on $\Fun(\dAff_{/X}^{\op}, \Mod_{\k})$ has the property that
  for any \iopd{} $\mathcal{O}$ the \icat{} of $\mathcal{O}$-algebras
  is equivalent to $\Fun(\dAff_{/X}^{\op},
  \Alg_{\mathcal{O}}(\Mod_{\k}))$. Thus the \icat{} of
  $\mathcal{O}_{X}$-modules is the pullback
  \nolabelcsquare{\Mod_{\mathcal{O}_{X}}(\Fun(\dAff_{/X}^{\op},
    \Mod_{\k}))}{\Fun(\dAff_{/X}^{\op},
    \mathcal{M})}{\{\mathcal{O}_{X}\}}{\Fun(\dAff_{/X}^{\op},
    \CAlgk).}
  But this is clearly the same as the \icat{} of sections of the
  projection $\mathcal{M}_{X} \to \dAff_{/X}^{\op}$.
\end{proof}

\begin{corollary}\label{cor:OXmodsection}
  There is a pullback square of \icats{}
  \nolabelcsquare{\Mod_{\mathcal{O}_{X}}}{\Fun_{/\dAff_{/X}^{\op}}(\dAff_{/X}^{\op},
    \mathcal{M}_{X})}{\Shet_{X}(\Mod_{\k})}{\Fun(\dAff_{/X}^{\op},
    \Mod_{\k}).}
  In other words, a sheaf of $\mathcal{O}_{X}$-modules is a section of
  $\mathcal{M}_{X}$ whose underlying presheaf of $\k$-modules is a sheaf.
\end{corollary}
\begin{proof}
  By Remark~\ref{rmk:OXpsh} we have a pullback square
  \nolabelcsquare{\Mod_{\mathcal{O}_{X}}}{\Mod_{\mathcal{O}_{X}}(\Fun(\dAff_{/X}^{\op},
    \Mod_{\k}))}{\Shet_{X}(\Mod_{\k})}{\Fun(\dAff_{/X}^{\op},
    \Mod_{\k})}
  The result now follows by combining this square with the description of the \icat{} $\Mod_{\mathcal{O}_{X}}(\Fun(\dAff_{/X}^{\op},
    \Mod_{\k}))$ in Proposition~\ref{propn:OXpshsect}.
\end{proof}

\section{Quasicoherent Sheaves}\label{subsec:qcoh}
We now introduce quasicoherent sheaves on derived stacks;
these can be defined as a certain class of sheaves of
$\mathcal{O}_{X}$-modules:
\begin{definition}\label{defn:qcohOXmod} 
  Let $X$ be a derived stack. We say a sheaf $\mathcal{M}$ of
  $\mathcal{O}_{X}$-modules is \emph{quasicoherent} if for any
  commutative triangle
  \opctriangle{\Spec A}{\Spec B}{X,}{f}{p}{q}
  the induced map 
  \[A \otimes_{B} q^{*}\mathcal{M} \to p^{*}\mathcal{M}\]
  in $\Mod_{A}$ is an equivalence.
\end{definition}
This is the natural generalization of one of the usual definitions of
quasicoherent sheaves on ordinary schemes. On the other hand, we have
the following common definition of the \icat{} of quasicoherent sheaves:
\begin{definition}\label{defn:QCoh}
  There is a functor $\txt{CAlg}_{\k} \to \PrL$ that to a ring $A$
  assigns the \icat{} $\Mod_{A}$ of $A$-modules and to a map
  $f \colon A \to B$ assigns the functor
  $f_{!} \colon \Mod_{A} \to \Mod_{B}$ given by
  $B \otimes_{A} \blank$. Restricted to connective algebras this
  satisfies a very strong form of descent (\cf{} \cite{SAG}*{\S D}), and
  in particular is an \'etale sheaf valued in $\PrL$. Thus it
  determines a limit-preserving functor $\QCoh \colon \dSt^{\op} \to
  \PrL$ by Lemma~\ref{lem:dStunivprop}. By
  Remark~\ref{rmk:dstcolimaff} the \icat{} $\QCoh(X)$ is given by
  \[\QCoh(X) \simeq \lim_{\Spec A \to X \in \dAff_{/X}}
  \Mod_{A},\]
  and we call this the \icat{} of \emph{quasicoherent sheaves} on
  $X$. Since the functors $f_{!}$ for $f$ a map of rings are strong
  monoidal with respect to the tensor products of modules, and the
  forgetful functor from presentably symmetric monoidal \icats{} to
  $\PrL$ detects limits, we can equally well think of $\QCoh$ as a
  limit-preserving functor from derived stacks to presentably
  symmetric monoidal \icats{}.  For a map $f \colon X \to Y$ of
  derived stacks we write $f^{*} \colon \QCoh(Y) \to \QCoh(X)$ for the
  induced symmetric monoidal functor, and $f_{*}$ for its right
  adjoint (which is lax symmetric monoidal).
\end{definition}
At first glance, these two definitions seem rather different. We will
now show that they are in fact equivalent:
\begin{proposition}\label{propn:qcohisOXmod}
  For any derived stack $X$, the \icat{} $\QCoh(X)$ is naturally
  equivalent to the full subcategory of $\Mod_{\mathcal{O}_{X}}$
  spanned by the quasicoherent sheaves of $\mathcal{O}_{X}$-modules in
  the sense of Definition~\ref{defn:qcohOXmod}.
\end{proposition}
Using the notation of Corollary~\ref{cor:OXmodsection}, by a dual version
of \cite{HTT}*{Corollary 3.3.3.2} we may identify the limit
\[\QCoh(X) \simeq \lim_{\Spec A \to X \in \dAff_{/X}^{\op}}
  \Mod_{A}\] with the \icat{} of \emph{cocartesian} sections of the
fibration $\mathcal{M}_{X} \to \dAff_{/X}^{\op}$. This is a full
subcategory of the \icat{} of all sections, of which
$\Mod_{\mathcal{O}_{X}}$ is also a full subcategory by
Corollary~\ref{cor:OXmodsection}. Moreover, an element of
$\Mod_{\mathcal{O}_{X}}$ gives a cocartesian section precisely when it
is quasicoherent in the sense of Definition~\ref{defn:qcohOXmod}.  It
thus suffices to show that the functor
$\dAff_{/X}^{\op} \to \Mod_{\k}$ underlying a cocartesian section is
necessarily a sheaf. To see this we use the following technical
observation:
\begin{proposition}
  Suppose $F \colon \mathcal{C} \to \PrL$ is a limit-preserving
  functor, with $\pi \colon \mathcal{E} \to \mathcal{C}$ the
  associated cocartesian (and cartesian) fibration, and
  $\overline{p} \colon \mathcal{I}^{\triangleleft} \to \mathcal{C}$ is
  a limit diagram. Then for $X \in F(\overline{p}(-\infty))$, the
  cocartesian lift of $\mathcal{E}$ along $\overline{p}$ is a
  $\pi$-limit diagram. In other words, the diagram
  $\mathcal{I}^{\triangleleft} \to F(\overline{p}(-\infty))$ obtained
  by the cartesian pullback of this along $\overline{p}$ is a limit
  diagram.
\end{proposition}
\begin{proof}
  The first statement is part of \cite[Lemma 5.17]{DAG7}, and the
  second follows from \cite[Propositions 4.3.1.9 and 4.3.1.11]{HTT}.
\end{proof}

\begin{definition}\label{def:Qfib}
  We write $\mathcal{Q} \to \dSt^{\op}$ for the cocartesian fibration
  associated to the functor $\QCoh(\blank)$ given by pullbacks. For $S
  \in \dSt$ we also write $\mathcal{Q}_{S} \to \dSt_{S}^{\op}$ for the
  pullback $\mathcal{Q} \times_{\dSt^{\op}} \dSt_{S}^{\op}$.
\end{definition}

Since $\QCoh(\blank)$ takes colimits in $\dSt$ to limits of \icats{},
we obtain the following special case of this observation:
\begin{corollary}\label{cor:qcohrellim}
  Suppose $\overline{p} \colon \mathcal{I}^{\triangleright} \to \dSt$
  is a colimit diagram and $\mathcal{E}$ is an object of
  $\QCoh(\overline{p}(\infty))$. Then the cocartesian lift of
  $\mathcal{E}$ along $\overline{p}^{\op}$ is a limit diagram in
  $\mathcal{Q}$. In other words, the diagram
  $(\mathcal{I}^{\triangleright})^{\op} \to
  \QCoh(\overline{p}(\infty))$ obtained as the cartesian pullback of
  this along $\overline{p}^{\op}$ is a limit diagram. (Informally,
  $\mathcal{E}$ is the limit in $\QCoh(\overline{p}(\infty))$ of
  $f_{i,*}f_{i}^{*}\mathcal{E}$ where $f_{i} \colon \overline{p}(i)
  \to \overline{p}(\infty)$ is the image of the unique map $i \to
  \infty$.) \qed
\end{corollary}

In particular, we have:
\begin{corollary}
  For any derived stack $X$ and $\mathcal{E} \in \QCoh(X)$ we have a natural equivalence
  \[ \mathcal{E} \simeq \lim_{p \in \dAff_{/X}^{\op}} p_{*}p^{*}\mathcal{E}.\]
\end{corollary}

\begin{proof}[Proof of Proposition~\ref{propn:qcohisOXmod}]
  It remains to show that if $\mathcal{F} \colon \dAff_{/X}^{\op}
  \to \mathcal{M}_{X}$ is a cocartesian section, then $\mathcal{F}$
  gives a sheaf of $\k$-modules. Suppose
  \opctriangle{\Spec A}{\Spec B}{X}{f}{p}{q}
  is an \'etale cover. Then we must show that $\mathcal{F}(q)$ is
  the limit of $\mathcal{F}$ applied to the \v{C}ech nerve of
  $f$. Since $\mathcal{F}$ is cocartesian, the resulting diagram
  $\bar{f} \colon \simp^{\triangleleft} \to \mathcal{M}_{X}$ is the cocartesian
  pushforward of $\mathcal{F}(q)$ along the \v{C}ech
  nerve of $f$ viewed as a diagram $\simp^{\triangleleft} \to
  \dAff_{/X}^{\op}$. Now this diagram gives a limit diagram in
  $\dSt^{\op}$, and so by Corollary~\ref{cor:qcohrellim} we see that
  $\mathcal{F}(q)$ is the limit of the cartesian pullback of
  $\mathcal{F} \circ \bar{f}$ to $\Mod_{B}$. Since the forgetful
  functor $\Mod_{B} \to \Mod_{\k}$ preserves limits, this completes the proof.
\end{proof}

We also record the following observation about $\QCoh$:
\begin{proposition}
  Suppose $\bar{p} \colon \mathcal{I}^{\triangleright} \to \dSt$ is a
  colimit diagram. Then a morphism $\mathcal{E} \to \mathcal{F}$ in
  $\QCoh(\bar{p}(\infty))$ is an equivalence \IFF{} the map
  $f_{i}^{*}\mathcal{E} \to f_{i}^{*}\mathcal{F}$ is an equivalence
  for all $i \in \mathcal{I}$, where $f_{i}$ denotes the image under
  $\bar{p}$ of the unique map $i \to \infty$.
\end{proposition}
\begin{proof}
  This is a special case of \cite[Lemma 5.17]{DAG7}.
\end{proof}

As a useful special case, we have:
\begin{corollary}\label{cor:qcoheqonpts}
  For any $S$-stack $X$, a morphism $\mathcal{E} \to \mathcal{F}$
  in $\QCoh(X)$ is an equivalence \IFF{}
  $\overline{p}^{*}\mathcal{E} \to \overline{p}^{*}\mathcal{F}$ is an equivalence in
  $\QCoh(X_{p})$ for all maps $p \colon \Spec A \to S$, where
  $\overline{p} \colon X_{p} \to X$ is the base change of $p$ to $X$. \qed
\end{corollary}

\begin{remark}
  In particular, taking $S = X$, we see that $\mathcal{E} \to
  \mathcal{F}$ is an equivalence \IFF{} $p^{*}\mathcal{E} \to
  p^{*}\mathcal{F}$ is an equivalence for all $p \colon \Spec A \to X$.
\end{remark}

\begin{corollary}\label{cor:colimeqonpoints}
  For any $S$-stack $X$ and \icat{} $\mathcal{I}$, we have:
  \begin{enumerate}[(i)]
  \item a diagram $\overline{f} \colon \mathcal{I}^{\triangleright}
    \to \QCoh(X)$ is a colimit diagram \IFF{}
    $\overline{p}^{*}\overline{f}$ is a colimit diagram in
    $\QCoh(X_{p})$ for all maps $p \colon \Spec A \to S$,
  \item if $\mathcal{I}$ is finite, a diagram $\overline{f} \colon
    \mathcal{I}^{\triangleleft} \to \QCoh(X)$ is a limit diagram
    \IFF{} $\overline{p}^{*}\overline{f}$ is a limit diagram in
    $\QCoh(X_{p})$ for all maps $p \colon \Spec A \to S$,
  \end{enumerate}
\end{corollary}
\begin{proof}
  This follows from Corollary~\ref{cor:qcoheqonpts} since the functors
  $\overline{p}^{*}$ preserve colimits and finite limits.
\end{proof}

\begin{definition}\label{defn:Qotimes}
  The functor $\dAff^{\op} \to \CAlg(\PrL)$ that to a ring $A$ assigns
  $\Mod_{A}$ as a (presentably) symmetric monoidal \icat{} also
  satisfies descent (since limits in $\CAlg(\PrL)$ are detected in
  $\PrL$), and so extends to a limit-preserving functor
  $\QCoh^{\otimes} \colon \dSt^{\op} \to \CAlg(\PrL)$. Forgetting from
  $\CAlg(\PrL)$ to $\CAlg(\CatI)$, we can view the result as a functor
  $\dSt^{\op} \times \xF_{*} \to \CatI$; we write
  $\mathcal{Q}^{\otimes} \to \dSt^{\op} \times \xF_{*}$ for the
  corresponding cocartesian fibration, which encodes the symmetric
  monoidal \icats{} $\QCoh(X)^{\otimes}$ for $X \in \dSt^{\op}$
  together with the symmetric monoidal functoriality via pullback.
  For $S \in \dSt$, we  also write $\mathcal{Q}_{S}^{\otimes}$ for the
  pullback $\mathcal{Q}^{\otimes} \times_{\dSt^{\op}} \dSt_{S}^{\op}$.
\end{definition}

\begin{remark}
  With a bit more care, we can identify $\QCoh(X)$ as a symmetric
  monoidal full subcategory of $\Mod_{\mathcal{O}_{X}}$. In
  particular, $\mathcal{O}_{X}$ is the unit for the symmetric monoidal
  structure on $\QCoh(X)$.
\end{remark}

\section{Affine Morphisms and Relative Spec}\label{subsec:affinemor}
The classical notion of affine morphisms of schemes has a natural
generalization to derived stacks:
\begin{definition}\label{defn:affmor}
  A morphism $f \colon X \to Y$ of derived stacks is \emph{affine}
  \IFF{} for every map $p \colon \Spec A \to Y$ from an affine stack,
  the fibre product $X \times_{Y} \Spec A$ is again affine.
\end{definition}
For ordinary schemes, affine morphisms can be described as the
relative Spec's of quasicoherent sheaves of commutative algebras. Our
goal in this section is to prove the analogous equivalence for
derived stacks, where the latter has the following natural definition:
\begin{definition}
  The \icat{} of \emph{quasicoherent sheaves of commutative algebras}
  on a derived stack $X$ is $\QCAlg(X) := \CAlg(\QCoh(X))$. Since
  $\CAlg(\blank)$ preserves limits of (symmetric monoidal) \icats{},
  we may identify this \icat{} with
  \[ \lim_{\dAff_{/X}^{\op}} \CAlg(\Mod_{A}) \simeq
  \lim_{\dAff_{/X}^{\op}} (\CAlg_{\k})_{A/}.\]
  We say a quasicoherent sheaf $\mathcal{F} \in \QCoh(X)$ is
  \emph{connective} if $p^{*}\mathcal{F} \in \Mod_{A}$ is connective
  for all points $p \colon \Spec A \to X$, and denote the full
  subcategory of $\QCoh(X)$ spanned by the connective objects by
  $\QCohconn(X)$; we have
  $\QCohconn(X) \simeq \lim_{\dAff_{/S}^{\op}} \Modconn_{A}$.
  Similarly, we say a quasicoherent sheaf of commutative algebras is
  \emph{connective} if the underlying quasicoherent sheaf of 
  modules is connective. We write $\QCAlgconn(X)$ for the full
  subcategory $\QCAlg(X)$ spanned by the connective objects; then
  $\QCAlgconn(X) \simeq \CAlg(\QCohconn(X))$.
\end{definition}

\begin{definition}
  We say $\mathcal{F} \in \QCoh(X)$ is \emph{$n$-connective} if
  $\mathcal{F}[-n]$ is connective, and write $\QCoh^{c[n]}(X)$
  for the full subcategory of $\QCoh(X)$ spanned by the $n$-connective
  objects.
\end{definition}

\begin{proposition}\label{propn:affisqcalg}
  For a derived stack $S$, we write $\dSt_{/S}^{\txt{aff}}$ for the
  full subcategory of $\dSt_{/S}$ spanned by the affine
  morphisms. Then there is a natural equivalence
  \[ \dSt_{/S}^{\txt{aff}} \isoto \QCAlgconn(S)^{\op},\]
  where the quasicoherent sheaf of commutative algebras $\mathcal{A}$
  associated to an affine morphism $X \to S$ is characterized by
  a natural equivalence $X_{p} \simeq \Spec p^{*}\mathcal{A}$, while
  the affine morphism corresponding to $\mathcal{A}$ is given by
  \[ \colim_{p \in \dAff_{/S}} \Spec p^{*}\mathcal{A} \to  \colim_{p \in \dAff_{/S}} \Spec p^{*}\mathcal{O}_{S} \simeq S.\]
\end{proposition}
\begin{proof}
For a morphism $X \to S$ and a point $p \colon \Spec A \to S$, we
write $X_{p}$ for the pullback of $X$ along $p$. Then since colimits
in $\dSt$ are universal, we have a pullback square
\nolabelcsquare{\colim_{p} X_p}{X}{\colim_{p} \Spec A}{S,}
where the colimit is over $p \colon \Spec A \to S$ in
$\dAff_{/S}$. The lower horizontal map is an equivalence, hence we
have $X \simeq \colim_{p} X_{p}$. This observation can be enhanced to
an equivalence of \icats{}: by \cite{HTT}*{Theorem 6.1.3.9}
\[ \dSt_{/S} \isoto \lim_{\dAff_{/S}^{\op}} \dSt_{/\Spec A},\]
where the functor takes $X \to S$ to its fibres $X_{p} \to \Spec A$
and the inverse is given by taking the colimit in $\dSt$. Restricting
to the full subcategory $\dSt_{/S}^{\txt{aff}}$ on the left
corresponds to restricting to the full subcategory $\lim_{\dAff_{/S}^{\op}}
\dAff_{/\Spec A}$ on the right, giving the required equivalence
\[ \dSt^{\txt{aff}}_{/S} \isoto  \lim_{\dAff_{/S}^{\op}}
\dAff_{/\Spec A} \simeq \lim_{\dAff_{/S}^{\op}} (\CAlgkc)_{A/}^{\op}
\simeq \QCAlgconn(S)^{\op}.\qedhere\]
\end{proof}

\begin{definition}
  We write $\Spec_{S} \colon \QCAlgconn(S)^{\op} \to \dSt_{/S}$ for the
  inverse of this equivalence, which takes $\mathcal{A} \in
  \QCAlgconn(S)$ to $\colim_{p\in \dAff_{/S}} \Spec
  p^{*}\mathcal{A}$, and refer to this as the \emph{relative Spec}
  functor over $S$.
\end{definition}
Then an immediate reformulation of Proposition~\ref{propn:affisqcalg}
gives:
\begin{corollary}
  A morphism $X \to S$ is affine \IFF{} it is of the form
  $\Spec_{S}\mathcal{A} \to S$ for some $\mathcal{A} \in
  \QCAlgconn(S)$.\qed
\end{corollary}

\begin{remark}
  For $\mathcal{A} \in \QCAlgconn(S)$, the equivalence of
  Proposition~\ref{propn:affisqcalg} implies that for every $p
  \colon \Spec A \to S$ we have a natural pullback square
  \csquare{\Spec p^* \mathcal{A}}{\Spec_S \mathcal{A}}{\Spec
    A}{S.}{}{}{}{p}
  This generalizes to the following observation:
\end{remark}

\begin{proposition}\label{propn:pbrelspec}
  Given a morphism $f \colon T \to S$ in $\dSt$ and $\mathcal{A} \in
  \QCAlgconn(S)$, we have a natural pullback square
  \csquare{\Spec_T f^* \mathcal{A}}{\Spec_S \mathcal{A}}{T}{S.}{}{}{}{f}
\end{proposition}
\begin{proof}
  Since colimits in $\dSt$ are universal, we have a pullback square
  \csquare{\colim_{p \in \dAff_{/T}} (\Spec_S
      \mathcal{A})_{fp}}{\Spec_S \mathcal{A}}{T}{S.}{}{}{}{f}
  Here we have a natural equivalence $(\Spec_S
  \mathcal{A})_{fp} \simeq \Spec (fp)^{*}\mathcal{A}$, giving
  \[ \colim_{p \in \dAff_{/T}} (\Spec_S \mathcal{A})_{fp} \simeq
\colim_{p \in \dAff_{/T}} \Spec p^{*}f^{*}\mathcal{A} \simeq
\Spec_{T}f^{*}\mathcal{A},\]
as required.
\end{proof}

\begin{corollary}\label{cor:maptorelspec}
  For $f \colon X \to S$ in $\dSt$ and $\mathcal{A} \in
  \QCAlgconn(S)$ we have natural equivalences
  \[ \Map_{/S}(X, \Spec_{S} \mathcal{A}) \simeq
  \Map_{\QCAlg(X)}(f^{*}\mathcal{A}, \mathcal{O}_{X}) \simeq
  \Map_{\QCAlg(S)}(\mathcal{A}, f_{*}\mathcal{O}_{X}).\]
\end{corollary}
\begin{proof}
  By Proposition~\ref{propn:pbrelspec} we have a natural equivalence
  \[ \Map_{/S}(X, \Spec_{S}\mathcal{A}) \simeq \Map_{/X}(X,
  \Spec_{X}f^{*}\mathcal{A}).\]
  Now the equivalence between $\dSt_{/X}^{\txt{aff}}$ and
  $\QCAlgconn(X)^{\op}$ of Proposition~\ref{propn:affisqcalg} gives
  \[ \Map_{/X}(X,
  \Spec_{X}f^{*}\mathcal{A}) \simeq \Map_{\QCAlgconn(X)^{\op}}(f^{*}\mathcal{A}, \mathcal{O}_{X}).\]
  Combining these equivalences gives the first equivalence, and the
  second follows from the adjunction $f^{*} \dashv f_{*}$.
\end{proof}

\begin{remark}
  Let $\QCAconn \to \dSt^{\op}$ be the cocartesian fibration for
  $\QCAlgconn(\blank)$ (with respect to the left adjoint pullback
  functors --- this is also a cartesian fibration with respect to
  pushforward). The relative $\Spec$ functors for all derived stacks
  can be combined to a single functor, giving a commutative triangle
  \opctriangle{\QCAconnop}{\Fun(\Delta^{1}, \dSt)}{\dSt.}{\Spec}{}{\txt{ev}_1}
  Then Proposition~\ref{propn:pbrelspec} says that $\Spec$ preserves
  cartesian morphisms. Moreover, Corollary~\ref{cor:maptorelspec} says
  that $\Spec$ is the right adjoint to a functor $\Fun(\Delta^{1},
  \dSt) \to \QCAconnop$ that takes $f \colon X \to S$ to
  $f_{*}\mathcal{O}_{X}$. It follows that $\Spec$, viewed as a functor
  $\QCAconnop \to \dSt$ (by composing with $\txt{ev}_{0}$)
  is right adjoint to a functor $\dSt \to \QCAconnop$ taking
  $X$ to $\mathcal{O}_{X}$.
\end{remark}

\begin{corollary}\label{cor:pfOSpecSA}
  Let $\pi$ denote the
  projection $\Spec_{S} \mathcal{A} \to S$ where $\mathcal{A}$ is an object of $\QCAlgconn(S)$.
  Then we have a natural
  equivalence
  \[ \pi_{*}\mathcal{O}_{\Spec_{S}\mathcal{A}} \simeq \mathcal{A}\]
  in $\QCAlgconn(S)$.
\end{corollary}
\begin{proof}
  By Corollary~\ref{cor:maptorelspec} we have for $\mathcal{B} \in
  \QCAlgconn(S)$  a natural equivalence
  \[ \Map_{\QCAlgconn(S)}(\mathcal{B},
    \pi_{*}\mathcal{O}_{\Spec_{S} \mathcal{A}}) \simeq
    \Map_{/S}(\Spec_{S}\mathcal{A}, \Spec_{S}\mathcal{B}).\]
  From the equivalence of
  Proposition~\ref{propn:affisqcalg} we get
  \[ \Map_{/S}(\Spec_{S}\mathcal{A}, \Spec_{S}\mathcal{B}) \simeq
  \Map_{\QCAlgconn(S)}(\mathcal{B}, \mathcal{A}).\]
  By the Yoneda Lemma, these natural equivalences imply that
  $\pi_{*}\mathcal{O}_{\Spec_{S}\mathcal{A}}$ is naturally equivalent to $\mathcal{A}$.
\end{proof}

\section{Geometric Morphisms and Artin Stacks}\label{subsec:geommor}
We now recall the inductive definition of a \emph{geometric} derived
stack --- this is a subclass of derived stacks to which many notions
of algebraic geometry can be extended, which is typically not possible
in complete generality.
\begin{definition}\ 
  \begin{enumerate}[(1)]
  \item A stack is \emph{$(-1)$-geometric} if it is affine.
  \item A morphism of stacks $Y \to X$ is \emph{$(-1)$-geometric} if
    for any affine stack $Z$ and any morphism $Z \to X$, the pullback
    $Y \times_{X} Z$ is affine (\ie{} the morphism is affine in the
    sense of Definition~\ref{defn:affmor}).
  \item A morphism of stacks $Y \to X$ is \emph{$(-1)$-smooth} if it
    is $(-1)$-geometric and for any affine stack $Z$ and any morphism
    $Z \to X$, the pullback
    $Y \times_{X} Z \to Z$ is a smooth map of affine stacks.
  \item An \emph{$n$-smooth atlas} on a stack $X$ is an
    effective epimorphism $\coprod_{i \in I} U_{i} \to X$ (where $I$
    is a set) such that each $U_{i}$ is affine and each map $U_{i} \to
    X$ is $(n-1)$-smooth.
  \item A stack $X$ is \emph{$n$-geometric} if the diagonal $X \to X
    \times X$ is $(n-1)$-geometric and there exists an $n$-smooth
    atlas of $X$.
  \item A morphism of stacks $Y \to X$ is \emph{$n$-geometric} if
    for any affine stack $Z$ and any morphism $Z \to X$, the pullback
    $Y \times_{X} Z$ is an $n$-geometric stack.
  \item A morphism of stacks $Y \to X$ is \emph{$n$-smooth} if it is
    $n$-geometric and for any affine stack $Z$ and any morphism $Z \to
    X$, there exists an $n$-smooth atlas $\coprod U_{i} \to Z
    \times_{X} Y$ of the $n$-geometric stack $Z \times_{X} Y$ such
    that each composite $U_{i} \to Z$ is a smooth map of affine
    stacks.
  \item A stack $X$ is \emph{geometric} if it is $n$-geometric for
    some $n$.
  \end{enumerate}
  We write $\dStg \subseteq \dSt$ and $\Fun(\Delta^{1}, \dSt)^{\txt{geom}} \subseteq \Fun(\Delta^{1}, \dSt)$ for the
  full subcategories spanned by the
  geometric stacks and geometric morphisms, respectively. We also say an $S$-stack $X$ is \emph{geometric} if
  the morphism $X \to S$ is geometric, and write $\dStg_{S}$ for the full
  subcategory of geometric $S$-stacks.
\end{definition}

\begin{remark}
  This is the definition from \cite{HAG2}, which is slightly different from the definition used in
  \cite{AntieauGepnerBrauer,GaitsgoryRozenblyum1}: there $(-1)$-geometric (or
  actually $0$-geometric in \cite{GaitsgoryRozenblyum1}) is
  defined as an arbitrary coproduct of affines. There is also a characterization of $n$-geometric stacks as
  geometric realizations of smooth groupoids in $(n-1)$-geometric
  stacks, see \cite{HAG2}*{\S 1.3.4}.  
\end{remark}

\begin{definition}
  A morphism $f \colon X \to Y$ of derived stacks is \emph{locally of
    finite presentation} if it is $n$-geometric for some $n$, and for
  any affine stack $Z$ and any morphism $Z \to Y$ there exists an
  $n$-smooth atlas $\coprod U_{i} \to X \times_{Y} Z$ such that the
  composites $U_{i} \to Z$ (which are maps of affine derived stacks)
  are all finitely presented.
\end{definition}

\begin{definition}\label{defn:Artst}
  Following \cite{PTVV}, we say, for brevity, that a
  derived stack $X$ is an \emph{$n$-Artin stack} if it is
  $n$-geometric and locally of finite presentation over $\Spec(\k)$,
  and that it is an \emph{Artin stack} if it is an $n$-Artin stack for
  some $n$. Similarly, we say a morphism $X \to S$ that is
  ($n$-)geometric and locally of finite presentation is an
  \emph{($n$-)Artin $S$-stack}. We write $\dStArt_{S}$ for the full
  subcategory of $\dSt_{S}$ spanned by the Artin $S$-stacks.
\end{definition}

\begin{proposition}\label{propn:geomfinlim}
  A finite limit of $n$-geometric morphisms is $n$-geometric, and a
  finite limit of Artin $S$-stacks is an Artin $S$-stack.
\end{proposition}
\begin{proof}
  This is \cite{AntieauGepnerBrauer}*{Lemma 4.35}.
\end{proof}

\section{Base Change and the Projection Formula}\label{subsec:basechg}
In this section we look at base change and projection formulas on
quasicoherent sheaves for a certain class of morphisms of derived
stacks. These results are originally due to Ben-Zvi, Francis and
Nadler~\cite{BenZviFrancisNadler}*{Proposition 3.10}, but we include a
self-contained exposition of their proof to make it clear that it
works under the abstract categorical condition we need rather than the
stronger geometric hypotheses they consider; see also
\cite{HalpernLeistnerPreygel}*{Proposition A.1.5},
\cite{GaitsgoryRozenblyum1}*{Chapter 3, Proposition 2.2.2}, and
\cite{SAG}*{Proposition 2.5.4.5} for other accounts, using different
geometric conditions.

\begin{definition}
  A commutative square in $\dSt$
\csquare{X'}{X}{Y'}{Y}{\xi}{f'}{f}{\eta}
induces a commutative square of left adjoints
\csquare{\QCoh(Y)}{\QCoh(Y')}{\QCoh(X)}{\QCoh(X'),}{\eta^{*}}{f^{*}}{f'^{*}}{\xi^{*}}
and hence a Beck-Chevalley transformation
\[ \eta^{*}f_{*} \to f'_{*}f'^{*}\eta^{*}f_{*} \simeq
f'_{*}\xi^{*}f^{*}f_{*} \to f'_{*} \xi^{*}.\]
We say that the commutative square  in $\dSt$ \emph{satisfies base change} if the
Beck-Chevalley transformation is an equivalence, and we say that $f$
 \emph{satisfies base change} if every such \emph{cartesian}
square satisfies base change.
\end{definition}

\begin{definition}
  For any $f \colon X \to Y$ the functor $f^{*} \colon \QCoh(Y) \to
  \QCoh(X)$ is symmetric monoidal, and so induces a natural map
  \[ f_{*} \mathcal{E} \otimes \mathcal{F} \to f_{*}(\mathcal{E}
  \otimes f^{*}\mathcal{F}),\]
  adjoint to
  $f^{*}(f_{*}\mathcal{E} \otimes \mathcal{F}) \simeq
  f^{*}f_{*}\mathcal{E} \otimes f^{*}\mathcal{F} \to \mathcal{E}
  \otimes f^{*}\mathcal{F}$.
  We say that $f$ \emph{satisfies the projection formula} if this is
  an equivalence for all $\mathcal{F} \in \QCoh(Y)$,
  $\mathcal{E} \in \QCoh(X)$.
\end{definition}

We start with the basic case of morphisms of affine stacks:
\begin{lemma}\label{lem:affinebc}\
  \begin{enumerate}[(i)]
  \item Given a morphism $\phi \colon A \to B$ in $\CAlgkc$, the
    morphism \[f := \Spec \phi \colon \Spec B \to \Spec A\] satisfies
    the projection formula.
  \item Given a pushout square
    \csquare{A}{B}{C}{C \otimes_{A} B}{\phi}{\gamma}{\gamma'}{\phi'}
    in $\CAlgkc$, the
  cartesian square
  \csquare{\Spec C \otimes_A B}{\Spec C}{\Spec B}{\Spec A,}{f'}{g'}{g}{f}
  where $g := \Spec \gamma$, $f := \Spec \phi$, $g' := \Spec \gamma'$,
  $f' := \Spec \phi'$, satisfies base change.
  \end{enumerate}
\end{lemma}
\begin{proof}
  The functor $f_{*}$ is just the forgetful functor
  $\Mod_{B} \to \Mod_{A}$, so for $M \in \Mod_{B}$ and
  $N \in \Mod_{A}$ the projection formula morphism
  $f_{*}M \otimes_{A}N \to f_{*}(M \otimes_{B}f^{*}N)$ is just the
  natural equivalence
  $M \otimes_{A} N \to M \otimes_{B}(B \otimes_{A} N)$, which proves
  (i).
  
  For (ii), the Beck--Chevalley transformation
  \[ g^{*}f_{*}M \to f'_{*}g'^{*}M\]
  for $M \in \Mod_{B}$ is an equivalence since the underlying morphism
  of $\k$-modules is the equivalence
\[ C \otimes_{A} f_{*}M \isoto (C \otimes_{A}B) \otimes_{B} M. \qedhere\]
\end{proof}

Next, we consider the case of affine morphisms:
\begin{proposition}\label{propn:affinemorbc} 
  Suppose $\pi \colon \Spec_{S}\mathcal{A} \to S$ is an affine
  morphism. Then for $p \colon \Spec A \to S$ the pullback square
  \csquare{\Spec p^*\mathcal{A}}{\Spec_S \mathcal{A}}{\Spec
    A}{S}{\bar{p}}{\pi_p}{\pi}{p} satisfies base change.
\end{proposition}
\begin{proof} 
  For a morphism \opctriangle{\Spec B}{\Spec A}{S}{\phi}{q}{p} in
  $\dAff_{/S}$, we have a pullback square of derived stacks
  \csquare{\Spec q^{*}\mathcal{A}}{\Spec p^* \mathcal{A}}{\Spec
    B}{\Spec A,}{\bar{\phi}}{\pi_q}{\pi_p}{\phi} \ie{}
  $q^{*}\mathcal{A} \simeq B \otimes_{A} p^{*}\mathcal{A}$. 
  By Lemma~\ref{lem:affinebc} the corresponding Beck--Chevalley transformation
  \[ \phi^{*}\pi_{p,*} \to \pi_{q,*}\bar{\phi}^{*}\]
  is an
  equivalence.
  As we will show in the next section, the maps
  \[ \Mod_{p^{*}\mathcal{A}} \xto{\pi_{p,*}}
\Mod_{p^{*}\mathcal{O}_{S}}\] are natural in $\dAff_{/S}^{\op}$
and so induce a functor
  \[ \Phi \colon \QCoh(X) \simeq \lim_{p} \Mod_{p^{*}\mathcal{A}}\to
\lim_{p} \Mod_{p^{*}\mathcal{O}_{S}}\simeq \QCoh(S),\]
which by construction satisfies $p^{*}\Phi \simeq
\pi_{p,*}\bar{p}^{*}$. Moreover, this functor
$\Phi$ is the right adjoint $\pi_{*}$, since we have natural
equivalences
  \[ 
  \begin{split}
\Map_{\QCoh(S)}(\mathcal{M}, \Phi(\mathcal{N})) & \simeq
\lim_{p}\Map_{\Mod_{p^{*}\mathcal{O}_{X}}}(p^{*}\mathcal{M},
p^{*}\Phi(\mathcal{N})) \\ & \simeq
\lim_{p}\Map_{\Mod_{p^{*}\mathcal{O}_{X}}}(p^{*}\mathcal{M},
\pi_{p,*}\bar{p}^{*}(\mathcal{N}))\\
&  \simeq
\lim_{p}\Map_{\Mod_{p^{*}\mathcal{A}}}(\pi_{p}^{*}p^{*}\mathcal{M},
\bar{p}^{*}(\mathcal{N}))  \\ & \simeq
\lim_{p}\Map_{\Mod_{p^{*}\mathcal{A}}}(\bar{p}^{*}\pi^{*}\mathcal{M},
\bar{p}^{*}(\mathcal{N})) \\ & \simeq \Map_{\QCoh(X)}(\pi^{*}\mathcal{M},
\mathcal{N}).
\end{split}
\]
Thus we indeed have $p^{*}\pi_{*} \simeq \pi_{p,*}\bar{p}^{*}$,
as required.
\end{proof}

\begin{corollary}\label{cor:affinemonadic} 
  Suppose $f \colon X \to S$ is an affine morphism. Then:
  \begin{enumerate}[(i)]
  \item $f_{*} \colon \QCoh(X) \to \QCoh(S)$ preserves colimits and
detects equivalences.
  \item The adjunction $f^{*} \dashv f_{*}$ is monadic.
  \item The functor $f_{*}$ has a right adjoint.
  \end{enumerate}
\end{corollary}
\begin{proof}
  From the equivalence $\QCoh(S) \simeq \lim_{p}
\Mod_{p^{*}\mathcal{O}_{X}}$ we see that the functors $p^{*} \colon \QCoh(S) \to
\Mod_{p^{*}\mathcal{O}_{X}}$ jointly detect equivalences. Since they
also preserve colimits, this implies that to prove that $f_{*}$ preserves
colimits it suffices to show that the functors $p^{*}f_{*}$ preserve
colimits and jointly detect equivalences. By
Proposition~\ref{propn:affinemorbc} if $f$ is equivalent to the
projection $\Spec_{S}\mathcal{A} \to S$ for $\mathcal{A} \in \QCAlgconn(S)$, then we may identify $p^{*}f_{*}$ with $f_{p,*}\bar{p}^{*}$. Here
$f_{p,*}$ is the forgetful functor $\Mod_{p^{*}\mathcal{A}} \to
\Mod_{p^{*}\mathcal{O}_{X}}$, which detects equivalences and preserves
colimits. Since $\bar{p}^{*}$ is a left adjoint, this implies that
$p^{*}f_{*}$ preserves colimits for all $p$, hence $f_{*}$ preserves
colimits. To see that $f_{*}$ detects equivalences, it suffices to
show that a morphism $\phi \colon \mathcal{M} \to \mathcal{N}$ in
$\QCoh(X)$ is an equivalence \IFF{} $p^{*}f_{*}\phi$ is an equivalence
for all $p$. The latter holds \IFF{} $f_{p,*}\bar{p}^{*}\phi$ is an
equivalence, and as $f_{p,*}$ detects equivalences this is true \IFF{}
$\bar{p}^{*}\phi$ is an equivalence. The functors $\bar{p}^{*}$ jointly
detect equivalences, since $\QCoh(X)$ is the limit
$\lim_{p}\Mod_{p^{*}\mathcal{A}}$, \ie{} $\phi$ is an equivalence
\IFF{} $\bar{p}^{*}\phi$ is an equivalence for all $p$. This proves
(i). Now (ii) follows from the monadicity theorem and (iii) from the
adjoint functor theorem, since the \icats{} $\QCoh(X)$ and $\QCoh(S)$
are presentable.
\end{proof}

This leads to a description of quasicoherent sheaves for an affine map
that will be useful later on:
\begin{proposition}\label{propn:relspecqcoh}
  Let $\pi$ denote the projection $\Spec_{S} \mathcal{A} \to S$ for an
  object $\mathcal{A} \in \QCAlgconn(S)$. Then the functor
  \[\pi_{*} \colon \QCoh(\Spec_{S}\mathcal{A}) \to \QCoh(S)\]
  is lax symmetric monoidal. The unit
  $\mathcal{O}_{\Spec_{S}\mathcal{A}}$ is mapped to $\mathcal{A}$ by
  Corollary~\ref{cor:pfOSpecSA} and so this functor canonically
  factors through
  \[ \QCoh(\Spec_{S}\mathcal{A}) \to \Mod_{\mathcal{A}}(\QCoh(S)).\]
  This functor is an equivalence, and under this equivalence the functor
  \[\pi^{*} \colon \QCoh(S) \to \QCoh(\Spec_{S}\mathcal{A})\]
  corresponds to $\mathcal{A} \otimes_{\mathcal{O}_{S}} \blank$.
\end{proposition}
\begin{proof} 
  In the commutative triangle
  \opctriangle{\QCoh(\Spec_{S}\mathcal{A})}{\Mod_{\mathcal{A}}(\QCoh(S))}{\QCoh(S)}{}{\pi_{*}}{}
  the two diagonal morphisms are both right adjoints, with left
  adjoints given respectively by $\pi^{*}$ and the free
  $\mathcal{A}$-module functor. The right-hand adjunction is monadic,
  with the underlying endofunctor of the monad given by
  $\mathcal{A} \otimes_{\mathcal{O}_{X}} \blank$, as is the left-hand
  adjunction by Corollary~\ref{cor:affinemonadic}(ii). Thus by
  \cite{HA}*{Corollary 4.7.3.16} to see that the horizontal morphism
  is an equivalence it suffices to show that the induced natural
  transformation
  $\pi_{*}\pi^{*} \to \mathcal{A} \otimes_{\mathcal{O}_{S}} \blank$ of
  endofunctors on $\QCoh(S)$ is an equivalence. To see this it
  suffices to check this transformation induces an equivalence after
  applying $p^{*}$ for every $p \colon \Spec A \to S$. Using
  Proposition~\ref{propn:affinemorbc} and the symmetric monoidality of
  $p^{*}$ we may identify this with the natural map
  \[ p^{*}\pi_{*}\pi^{*} \simeq \pi_{p,*}\bar{p}^{*}\pi^{*} \simeq
\pi_{p,*} \pi_{p}^{*}p^{*} \to p^{*}\mathcal{A} \otimes_{A}
p^{*}\blank,\] which is indeed an equivalence since $\pi_{p,*}$ is the
forgetful functor $\Mod_{p^{*}\mathcal{A}} \to \Mod_{A}$ with left
adjoint given by $p^{*}\mathcal{A} \otimes_{A} \blank$.
\end{proof}

\begin{definition}\label{defn:pfcc}
  A morphism $f \colon X \to Y$ of derived stacks is
  \emph{cocontinuous} if the functor \[f_{*} \colon
  \QCoh(X) \to \QCoh(Y)\] preserves colimits. We say that $f$ is
  \emph{universally cocontinuous} if for every pullback
  square
  \csquare{X'}{X}{Y'}{Y}{\xi}{f'}{f}{\eta}
  the morphism $f'$ is cocontinuous.
\end{definition}
Our goal is now to show that universally cocontinuous
morphisms of derived stacks satisfy base change and the projection
formula. We first consider the case of morphisms to an affine stack:
\begin{lemma}\label{lem:affproj}
  Suppose $f \colon X \to \Spec A$ is cocontinuous. Then
  $f$ satisfies the projection formula.
\end{lemma}
\begin{proof}
  Given $\mathcal{E} \in \QCoh(X)$ we have a natural transformation
  \[ f_{*}\mathcal{E} \otimes_{A} (\blank) \to f_{*}(\mathcal{E} \otimes_{X}
  f^{*}(\blank))\]
  of functors $\Mod_{A} \to \Mod_{A}$.
  By assumption both functors preserve colimits, so as $\Mod_{A}$ is
  generated under colimits by $A[n]$ for $n \in \ZZ$ it suffices to show that this
  natural transformation gives an equivalence when evaluated at $A[n]$,
  which is clear since the functors also commute with shifts.
\end{proof}

\begin{lemma}\label{lem:affbasebc}
  Suppose $f \colon X \to \Spec A$ is cocontinuous. Then
  for every map \[g := \Spec \gamma \colon \Spec B \to \Spec A,\] the
  pullback square \csquare{X'}{X}{\Spec B}{\Spec A}{g'}{f'}{f}{g}
  satisfies base change.
\end{lemma}
\begin{proof}
  Since $g_{*} \colon \Mod_{B} \to \Mod_{A}$ detects equivalences,
  to show that $g^{*}f_{*}\mathcal{E} \to
  f'_{*}g'^{*}\mathcal{E}$ is an equivalence it suffices to
  show that \[g_{*}g^{*}f_{*}\mathcal{E} \to
  g_{*}f'_{*}g'^{*}\mathcal{E} \simeq
  f_{*}g'_{*}g'^{*}\mathcal{E}\] is an equivalence.
  By Lemma~\ref{lem:affproj} the maps $f$ and $g$ satisfy the
  projection formula, so we have equivalences
  \[ g_{*}g^{*}f_{*}\mathcal{E} \simeq g_{*}B \otimes_{A}
  f_{*}\mathcal{E} \simeq f_{*}(f^{*}g_{*}B \otimes \mathcal{E}).\]
  It is thus enough to show that
  $g'_{*}g'^{*}\mathcal{E}$ is equivalent to
  $f^{*}g_{*}B \otimes \mathcal{E}$. By Proposition~\ref{propn:pbrelspec}
  we can identify $X'$ with $\Spec_{X} f^{*}g_{*}B$, and by
  Proposition~\ref{propn:relspecqcoh} there is then an equivalence $\QCoh(X')
  \simeq \Mod_{f^{*}g_{*}B}(\QCoh(X))$ under which $g'_{*}$
  corresponds to the forgetful functor from $f^{*}g_{*}B$-modules
  to $\QCoh(X)$ and $g'^{*}$
  corresponds to tensoring with $f^{*}g_{*}B$. The natural map 
  $f^{*}g_{*}B \otimes \mathcal{E} \to
  g'_{*}g'^{*}\mathcal{E}$ is therefore an
  equivalence, which completes the proof.
\end{proof}

\begin{proposition}\label{propn:affpbbc}
  Let $f \colon X \to Y$ be a morphism such that for all morphisms $p
  \colon \Spec A \to Y$ the base change $f_{p} \colon X_{p}\to \Spec
  A$ is cocontinuous. Then the cartesian squares
  \csquare{X_p}{X}{\Spec A}{Y}{\bar{p}}{f_p}{f}{p}
  all satisfy base change.
\end{proposition}
\begin{proof}
  We have an equivalence $\QCoh(Y) \simeq \lim_{p \in
    \dAff_{/Y}^{\op}} \Mod_{p^{*}\mathcal{O}_{Y}}$.  For
  $\mathcal{E} \in \QCoh(X)$ consider
  the section $p \mapsto f_{p,*}\bar{p}^{*}\mathcal{E}$ of the
  cocartesian fibration associated to the diagram $p \mapsto
  \Mod_{p^{*}\mathcal{O}_{X}}$. By Lemma~\ref{lem:affbasebc} for every
  map $g := \Spec \phi \colon
  \Spec B \to \Spec A$ the map
  \[ g^{*}f_{p,*}\bar{p}^{*}\mathcal{E} \to
  f_{q,*}\bar{q}^{*}\mathcal{E}\] is an equivalence (where $q =
  p\circ g$), hence this is a cocartesian section. By \cite[Lemma
  5.17]{DAG7} this means it can be extended to a cocartesian section
  of the extended fibration over $\dAff_{/Y}$, which is a relative
  limit diagram. In other words, the limit $\mathcal{F} := \lim_{p}
  p_{*}f_{p,*}\bar{p}^{*}\mathcal{E}$ satisfies $p^{*}\mathcal{F}
  \simeq f_{p,*}\bar{p}^{*}\mathcal{E}$. But on the other hand we have
  $\mathcal{F} \simeq \lim_{p}
  f_{*}\bar{p}_{*}\bar{p}^{*}\mathcal{E}$. Since $f_{*}$ preserves
  limits, this is equivalent to $f_{*} \lim_{p}
  \bar{p}_{*}\bar{p}^{*}\mathcal{E}$. Since $\dSt$ is an
  $\infty$-topos, colimits of derived stacks are universal, and so we
  have an equivalence $X \simeq \colim_{p} X_{p}$, giving an
  equivalence $\QCoh(X) \simeq \lim_{p}\QCoh(X_{p})$. By
  Corollary~\ref{cor:qcohrellim} this means that for $\mathcal{E} \in
  \QCoh(X)$ we have an equivalence $\mathcal{E} \simeq \lim_{p}
  \bar{p}_{*}\bar{p}^{*}\mathcal{E}$, hence $\mathcal{F}$ is
  equivalent to $f_{*}\mathcal{E}$ and we have proved
  $p^{*}f_{*}\mathcal{E} \simeq f_{p,*}\bar{p}^{*}\mathcal{E}$, as
  required.
\end{proof}

\begin{corollary}\label{cor:upfcocts}
  A morphism $f \colon X \to Y$ is universally
  cocontinuous \IFF{} for every morphism $p \colon \Spec A \to
  Y$ the base change $f_{p} \colon X_{p} \to \Spec A$ is
  cocontinuous.
\end{corollary}
\begin{proof}
  The second condition is obviously implied by $f$ being universally
  cocontinuous. Suppose therefore that the the second
  condition holds and consider a pullback square
  \csquare{X'}{X}{Y'}{Y.}{\xi}{f'}{f}{\eta} We want to show that
  $f'_{*}$ preserves colimits. Since the functors
  $p^{*} \colon \QCoh(Y) \to \Mod_{A}$ for $p \colon \Spec A \to Y$
  jointly detect equivalences and preserve colimits, it suffices to
  show that $p^{*}f'_{*}$ preserves colimits for all $p$. Since the
  second condition holds for $f'$ if it holds for $f$, we know from
  Proposition~\ref{propn:affpbbc} that $p^{*}f'_{*}$ is equivalent to
  $f'_{p,*}\bar{p}^{*}$, which preserves colimits since the base
  change $f'_{p}$ is also a base change of $f$.
\end{proof}

\begin{theorem}[Ben-Zvi--Francis--Nadler
  \cite{BenZviFrancisNadler}]\label{thm:bcpf}
  Suppose $f \colon X \to Y$ is universally
  cocontinuous. Then $f$ satisfies base change and the
  projection formula.
\end{theorem}
\begin{proof}
  Suppose we have a cartesian square
  \csquare{X'}{X}{Y'}{Y.}{\xi}{f'}{f}{\eta} Since
  $\QCoh(Y') \simeq \lim_{p \in \dAff_{/Y'}^{\op}}
  \Mod_{p^{*}\mathcal{O}_{Y'}}$,
  the projections $p^{*}$ are jointly conservative by \cite[Lemma
  5.17]{DAG7}. To see that the natural map
  $\eta^{*}f_{*}\mathcal{E} \to f'_{*}\xi^{*}\mathcal{E}$ is an
  equivalence for $\mathcal{E} \in \QCoh(X)$ it therefore suffices to
  show that
  $p^{*}\eta^{*}f_{*}\mathcal{E} \to p^{*}f'_{*}\xi^{*}\mathcal{E}$ is
  an equivalence for every morphism $p \colon \Spec A \to Y'$. Set
  $q := \eta p$, then by Proposition~\ref{propn:affpbbc} we know that the
  cartesian squares
  \[ \nodispcsquare{X'_{p}}{X'}{\Spec A}{Y'}{\bar{p}}{f'_{p}}{f'}{p}
  \qquad \nodispcsquare{X'_{p}}{X}{\Spec A}{Y}{\bar{q}}{f'_{p}}{f}{q}\]
  satisfy base change. Hence we have equivalences
  \[p^{*}\eta^{*}f_{*}\mathcal{E} \simeq q^{*}f_{*}\mathcal{E} \simeq
  f'_{p,*}\bar{q}^{*}\mathcal{E} \simeq
  f'_{p,*}\bar{p}^{*}\xi^{*}\mathcal{E} \simeq
  p^{*}f_{*}\xi^{*}\mathcal{E},\] which implies that $f$ satisfies
  base change.
  
  To see that $f$ satisfies the projection formula, it similarly
  suffices to show that $p^{*}f_{*}(\mathcal{E} \otimes
  f^{*}\mathcal{F}) \to p^{*}(f_{*}\mathcal{E} \otimes
  \mathcal{F}) \simeq p^{*}f_{*}\mathcal{E} \otimes
  p^{*}\mathcal{F}$ is an equivalence for every $p \colon \Spec
  A \to Y$. Using Proposition~\ref{propn:affpbbc} for the pullback along $p$ we
  can identify this with \[f_{p,*}\bar{p}^{*}(\mathcal{E} \otimes
  f^{*}\mathcal{F}) \simeq f_{p,*}(\bar{p}^{*}\mathcal{E} \otimes
  f_{p}^{*}p^{*}\mathcal{F}) \to f_{p,*}\bar{p}^{*}\mathcal{E} \otimes
  p^{*}\mathcal{F} \simeq p^{*}f_{*}\mathcal{E} \otimes
  p^{*}\mathcal{F},\]
  which is an equivalence by Lemma~\ref{lem:affproj}.
\end{proof}

\section{Naturality of Base Change and the Projection Formula}\label{subsec:natbasechg}

In this section we want to construct functors that encode the naturality of pushforward of
quasicoherent sheaves,  as well as its compatibility with tensor
products. We will first show that pushforward of
quasicoherent sheaves determines a functor
\[ \Gamma_{S} \colon \mathcal{Q}_{S} \to \QCoh(S),\] which takes
$(X \xto{f} S, \mathcal{E} \in \QCoh(X))$ to
$\Gamma_{S}\mathcal{E} := f_{*}\mathcal{E} \in \QCoh(S)$, where
$\mathcal{Q}_{S} \to \dSt_{S}^{\op}$ is the cocartesian fibration from
Definition~\ref{def:Qfib}. A morphism
in $\mathcal{Q}_{S}$ from $(X \xto{f} S, \mathcal{E})$ to
$(Y \xto{g} S, \mathcal{F})$ is a morphism $\phi \colon Y \to X$ over
$S$ together with a morphism $\mathcal{E} \to \mathcal{F}$ over
$\phi$, \ie{} a morphism $\phi^{*}\mathcal{E} \to \mathcal{F}$ in
$\QCoh(Y)$ or $\mathcal{E} \to \phi_{*}\mathcal{F}$ in $\QCoh(X)$; the
functor $\Gamma_{S}$ takes this to
$f_{*}\mathcal{E} \to f_{*}\phi_{*}\mathcal{F} \simeq
g_{*}\mathcal{F}$. Next, we will see that if we restrict to the full
subcategory
$\mathcal{Q}_{S}^{\txt{UCC}} := \dSt^{\txt{UCC},\op}_{S}
\times_{\dSt^{\op}} \mathcal{Q}$ lying over the universally
cocontinuous morphisms, then the functors $\Gamma_{S}$ are natural in
$S$ with respect to pullback. More precisely, we have a natural
transformation of functors $\mathcal{Q}^{\txt{UCC}}_{(\blank)} \to
\QCoh(\blank)$, which for
$\sigma \colon S' \to S$ supplies a commutative square
\csquare{\mathcal{Q}_{S'}^{\txt{UCC}}}{\QCoh(S')}{\mathcal{Q}_{S}^{\txt{UCC}}}{\QCoh(S),}{\Gamma_{S'}}{\sigma^{*}}{\sigma^*}{\Gamma_{S}}
where on the left $\sigma^{*}$ takes $(X \to S, \mathcal{E})$ to
$(X' \to S', \xi^{*}\mathcal{E})$, in terms of the pullback square
\csquare{X'}{X}{S'}{S.}{\xi}{}{}{\sigma}
We will also prove analogous statements that take the symmetric
monoidal structure on quasicoherent sheaves into account. The
technical input we need is the following observation:
\begin{construction}\label{constr:cartpb}
  Suppose $\pi \colon \mathcal{E} \to \mathcal{B}$ is a cartesian
  fibration, and consider a commutative diagram
  \[
    \begin{tikzcd}
    \mathcal{C} \arrow{r}{f} \arrow{d}{u} & \mathcal{E}
    \arrow{dd}{\pi} \\
    \mathcal{I} \times \{1\} \arrow{d} \\
    \mathcal{I} \times \Delta^{1} \arrow{r}{p} & \mathcal{B}.
    \end{tikzcd}
  \]
  We can reinterpret this as a diagram
  \[
    \begin{tikzcd}
      \{1\} \arrow{d} \arrow{rr} & & \mathcal{E}^{\mathcal{C}}
      \arrow{d} \\
      \Delta^{1} \arrow[dashed]{urr} \arrow{r} & \mathcal{B}^{\mathcal{I}} \arrow{r} & \mathcal{B}^{\mathcal{C}}.
    \end{tikzcd}
  \]
  Here the right vertical map is again a cartesian fibration by
  \cite{HTT}*{Proposition 3.1.2.1}, so we
  can choose a cartesian lift $\Delta^{1} \to
  \mathcal{E}^{\mathcal{C}}$, which we can unfold to a commutative square
  \[
    \begin{tikzcd}
      \mathcal{C} \times \Delta^{1} \arrow{d}{u \times \Delta^{1}} \arrow{r}{\bar{f}} \arrow{d} &
      \mathcal{E} \arrow{d}{\pi} \\
      \mathcal{I} \times \Delta^{1} \arrow{r}{p} & \mathcal{B}.
    \end{tikzcd}
  \]
  where $\bar{f}|_{\mathcal{C} \times \{1\}} \simeq f$ and
  $\bar{f}(c,0) \to \bar{f}(c,1)$ is the cartesian morphism over
  $p(u(c),0) \to p(u(c),1)$ with target $f(c)$. In particular,
  restricting to the fibre over $\{0\}$ we have a commutative square
  \[
    \begin{tikzcd}
      \mathcal{C} \arrow{d}{u} \arrow{r}{\bar{f}_{0}} & \mathcal{E}
      \arrow{d}{\pi} \\
      \mathcal{I} \arrow{r}{p_{0}} & \mathcal{B},
    \end{tikzcd}
  \]
  where $\bar{f}_{0}(c) \simeq \tilde{p}(u(c))^{*}f(c)$, where
  $\tilde{p}$ is $p$ viewed as a functor $\mathcal{I} \to \mathcal{B}^{\Delta^{1}}$.
\end{construction}

\begin{construction}\label{constr:GammaS}
  For $S \in \dSt$ we have a canonical map
  $\dSt_{S} \times \Delta^{1} \to \dSt$ (given at $0$ by the forgetful
  functor and at $1$ by the constant functor at $S$). We can then
  apply Construction~\ref{constr:cartpb} to the diagram
    \[
    \begin{tikzcd}
    \mathcal{Q}_{S} \arrow{r} \arrow{d}{u} & \mathcal{Q}
    \arrow{dd} \\
    \dSt^{\op}_{S} \times \{1\} \arrow{d} \\
    \dSt^{\op}_{S} \times \Delta^{1} \arrow{r} & \dSt^{\op}
    \end{tikzcd}
  \]
  to get a commutative square
  \csquare{\mathcal{Q}_{S}}{\mathcal{Q}}{\dSt^{\op}_{S}}{\dSt^{\op}.}{}{}{}{}
  Here the bottom horizontal functor is constant at $S$, so we can
  interpret this as a functor $\Gamma_{S} \colon \mathcal{Q}_{S} \to
  \QCoh(S)$.
\end{construction}

\begin{remark}\label{rmk:GammaSradj}
  The functor $\Gamma_{S} \colon \mathcal{Q}_{S} \to \QCoh(S)$ is
  right adjoint to the inclusion of $\QCoh(S)$ as the fibre of
  $\mathcal{Q}_{S}$ over $S$.
\end{remark}

\begin{definition}\label{defn:GammaSO}
  The cocartesian section
  $\mathcal{O} \colon \dSt^{\op} \to \mathcal{Q}$ pulls back to a
  section
  $\mathcal{O}_{/S} \colon \dSt^{\op}_{S} \to
  \mathcal{Q}_{S}$. Composing this with the ``global sections over
  $S$'' functor $\Gamma_{S} \colon \mathcal{Q}_{S} \to \QCoh(S)$ from
  Construction~\ref{constr:GammaS}, we get a functor
  \[\Gamma_{S}\mathcal{O} := \Gamma_{S} \circ \mathcal{O}_{/S} \colon
    \dSt^{\op}_{S} \to \QCoh(S). \] This assigns to $f \colon X \to S$
  the pushforward
  $\Gamma_{S}\mathcal{O}_{X} := f_{*}\mathcal{O}_{X} \in \QCoh(S)$,
  and to a morphism \opctriangle{X}{Y}{S}{\phi}{f}{g} the natural
  morphism
  $g_{*}\mathcal{O}_{Y} \to g_{*}\phi_{*}\phi^{*}\mathcal{O}_{Y}
  \simeq f_{*}\mathcal{O}_{X}$.
\end{definition}

\begin{construction}\label{constr:GammaSotimes}
  We can enhance $\Gamma_{S}$ to take the tensor product of
  quasicoherent sheaves into account: We have a cocartesian fibration
  $\mathcal{Q}^{\otimes} \to \dSt^{\op} \times \xF_{*}$, such
  that the composite $\mathcal{Q}^{\otimes} \to \dSt^{\op}$ is also a
  cartesian fibration, and the functor to $\dSt^{\op} \times
  \xF_{*}$ preserves cartesian morphisms (\ie{} cartesian
  morphisms in $\mathcal{Q}^{\otimes}$ map to isomorphisms in
  $\xF_{*}$). Setting $\mathcal{Q}^{\otimes}_{S} :=
  \dSt_{S}^{\op} \times_{\dSt^{\op}} \mathcal{Q}^{\otimes}$ we can
  apply Construction~\ref{constr:cartpb} to the diagram
    \[
    \begin{tikzcd}
    \mathcal{Q}^{\otimes}_{S} \arrow{r} \arrow{d}{u} & \mathcal{Q}^{\otimes}
    \arrow{dd} \\
    \dSt^{\op}_{S} \times \{1\} \arrow{d} \\
    \dSt^{\op}_{S} \times \Delta^{1} \arrow{r} & \dSt^{\op}
    \end{tikzcd}
  \]
  to get a commutative square
  \csquare{\mathcal{Q}_{S}^{\otimes}}{\mathcal{Q}^{\otimes}}{\dSt^{\op}_{S}}{\dSt^{\op},}{}{}{}{}
  where the bottom map is constant at $S$.
  Since cartesian morphisms map to isomorphisms in $\xF_{*}$,
  this is compatible with the projections to $\xF_{*}$, and we
  get a commutative square
  \[
    \begin{tikzcd}
      \mathcal{Q}_{S}^{\otimes} \arrow{r}{\Gamma_{S}^{\otimes}}
      \arrow{d} & \QCoh(S)^{\otimes} \arrow{d} \\
      \dSt_{S}^{\op} \times \xF_{*} \arrow{r} & \xF_{*}.
    \end{tikzcd}
  \]
  Moreover, since the pushforward functors are lax monoidal, it
  follows that $\Gamma_{S}^{\otimes}$ preserves cocartesian morphisms
  over inert morphisms in $\xF_{*}$ that lie over equivalences
  in $\dSt_{S}^{\op}$.  
\end{construction}

\begin{construction}\label{constr:cobase}
  Continuing the notation of Construction~\ref{constr:cartpb}, as a
  universal case thereof we can apply it to the diagram
  \[
    \begin{tikzcd}
    \txt{ev}_{1}^{*}\mathcal{E} \arrow{r} \arrow{d} & \mathcal{E}
    \arrow{dd}{\pi} \\
    \mathcal{B}^{\Delta^{1}} \times \{1\} \arrow{d} \\
    \mathcal{B}^{\Delta^{1}} \times \Delta^{1} \arrow{r}{\txt{ev}} & \mathcal{B}.
    \end{tikzcd}
  \]
  This gives a functor $\txt{ev}_{1}^{*}\mathcal{E} \times \Delta^{1}
  \to \mathcal{E}$ which takes a pair $(\phi \colon b \to b', x \in
  \mathcal{E}_{b'})$ to the cartesian morphism $\phi^{*}x \to
  x$. Restricting to $\txt{ev}_{1}^{*}\mathcal{E} \times \{0\}$ we get
  a commutative diagram
  \[
    \begin{tikzcd}
      \txt{ev}_{1}^{*}\mathcal{E} \arrow{dr} \arrow{rrrr} & & & & \mathcal{E} \arrow{ddll} \\
      & \mathcal{B}^{\Delta^{1}} \arrow{dr}{\txt{ev}_{0}} \\
      & & \mathcal{B}.
    \end{tikzcd}
  \]
  Suppose $\pi$ is also a cocartesian fibration, and that
  $\mathcal{B}$ has pushouts. Then both the left-hand maps are
  cocartesian fibrations, hence so is the composite
  $\txt{ev}_{1}^{*}\mathcal{E} \to \mathcal{B}$, with the cocartesian
  pushforward over $\phi \colon b \to b'$ at
  $(\psi \colon b \to b'', x \in \mathcal{E}_{b''})$ given in terms of
  the pushout \csquare{b}{b'}{b''}{b'''}{\phi}{\psi}{\psi'}{\phi'} by
  $(\psi' \colon b' \to b''', \phi'_{!}x)$. We can now ask whether the
  functor $\txt{ev}_{1}^{*}\mathcal{E} \to \mathcal{E}$ preserves
  cocartesian morphisms, which it does precisely when, in the
  situation above, the natural morphism
  \[ \phi_{!}\psi^{*}x \to \psi'^{*}\phi'_{!}x\]
  is an equivalence, \ie{} ``cobase change'' holds for the cartesian and
  cocartesian fibration $\pi$.
\end{construction}

\begin{construction}\label{constr:Gamma}
  We can apply Construction~\ref{constr:cobase} to the cartesian and
  cocartesian fibration $\mathcal{Q} \to \dSt^{\op}$, to obtain a 
  commutative diagram
  \[
    \begin{tikzcd}
      \txt{ev}_{1}^{*}\mathcal{Q} \arrow{dr} \arrow{rrrr}{\Gamma} & & &&
      \mathcal{Q} \arrow{ddll} \\
      & (\dSt^{\op})^{\Delta^{1}} \arrow{dr}{\txt{ev}_{0}} \\
      & & \dSt^{\op}.
    \end{tikzcd}
  \]
  The composite
  $\txt{ev}_{1}^{*}\mathcal{Q} \to \dSt^{\op}$ is a cocartesian
  fibration corresponding to a functor $\dSt^{\op} \to \CatI$ that
  assigns to $S \in \dSt$ the \icat{} $\mathcal{Q}_{S}$, and to a
  morphism $\sigma \colon S' \to S$ the functor $\sigma^{*} \colon
  \mathcal{Q}_{S} \to \mathcal{Q}_{S'}$ given by
  \[ (X \xto{f} S, \mathcal{E} \in \QCoh(X)) \mapsto
    (X' \xto{f'} S', \xi^{*}\mathcal{E} \in \QCoh(X')),\]
  where we have a pullback square
  \csquare{X'}{X}{S'}{S.}{\xi}{f'}{f}{\sigma}
  If we restrict to the full subcategory
  $\dSt^{\Delta^{1},\txt{UCC}}$ spanned by the universally
  cocontinuous morphisms we still get a cocartesian
  fibration (since these morphisms are closed under base change), and
  now Theorem~\ref{thm:bcpf} implies that the resulting functor
  \[ \Gamma \colon \dSt^{\Delta^{1},\txt{UCC},\op} \times_{\dSt^{\op}} \mathcal{Q}
    \to \mathcal{Q} \]
  preserves cocartesian morphisms. This amounts to a natural
  transformation of functors $\dSt^{\op} \to \CatI$ from
  $\mathcal{Q}_{(\blank)}^{\txt{UCC}}$ to $\QCoh(\blank)$. In
  particular, for $\sigma \colon S' \to S$ we have a naturality square
  \csquare{\mathcal{Q}_{S}^{\txt{UCC}}}{\QCoh(S)}{\mathcal{Q}_{S'}^{\txt{UCC}}}{\QCoh(S').}{\Gamma_{S}}{\sigma^{*}}{\sigma^{*}}{\Gamma_{S}}
\end{construction}

\begin{construction}\label{constr:GammaO}
  Consider the pullback $\dSt^{\Delta^{1},\UCC,\op}
  \times_{\dSt^{\op}} \mathcal{Q}$, where the pullback is over
  evaluation at $0$. The cocartesian section $\mathcal{O} \colon
  \dSt^{\op} \to \mathcal{Q}$ pulls back to a cocartesian section
  \[\mathcal{O}_{\Delta^{1}} \colon  \dSt^{\Delta^{1},\UCC,\op}
    \to \dSt^{\Delta^{1},\UCC,\op}
    \times_{\dSt^{\op}} \mathcal{Q}.\]
  Composing with the ``global sections'' functor $\Gamma \colon
  \dSt^{\Delta^{1},\UCC,\op} 
  \times_{\dSt^{\op}} \mathcal{Q} \to \mathcal{Q}$ we obtain a
  functor $\Gamma\mathcal{O} := \Gamma \circ
  \mathcal{O}_{\Delta^{1}} \colon \dSt^{\Delta^{1},\UCC,\op} \to
  \mathcal{Q}$ which takes $X \xto{f} S$
  to $f_{*}\mathcal{O}_{X} \in \QCoh(S)$. 

  This functor lives in a commutative triangle
  \opctriangle{\dSt^{\Delta^{1},\txt{UCC},\op}}{\mathcal{Q}}{\dSt^{\op}}{\Gamma\mathcal{O}}{\txt{ev}_{1}}{}
  where both maps to $\dSt^{\op}$ are cocartesian fibrations; since
  $\mathcal{O}$ is a cocartesian section and $\Gamma$ preserves
  cocartesian morphisms over $\dSt^{\op}$, we see that
  $\Gamma \mathcal{O}$ also preserves cocartesian morphisms.
\end{construction}

\begin{construction}\label{constr:Gammaotimes}
  We can also apply Construction~\ref{constr:cobase} to the cartesian and
  cocartesian fibration $\mathcal{Q}^{\otimes} \to \dSt^{\op}$, to obtain a 
  commutative diagram
  \[
    \begin{tikzcd}
      \txt{ev}_{1}^{*}\mathcal{Q}^{\otimes} \arrow{dr} \arrow{rrrr}{\Gamma^{\otimes}} & & &&
      \mathcal{Q}^{\otimes} \arrow{ddll} \\
      & (\dSt^{\op})^{\Delta^{1}} \arrow{dr}{\txt{ev}_{0}} \\
      & & \dSt^{\op},
    \end{tikzcd}
  \]
  where base change again implies that $\Gamma^{\otimes}$ preserves
  cocartesian morphisms when we restrict to universally cocontinuous
  maps of stacks. Taking the functors to $\xF_{*}$ into
  account, we obtain
\[
    \begin{tikzcd}
      (\dSt^{\Delta^{1},\txt{UCC}})^{\op} \times_{\dSt^{\op}}\mathcal{Q}^{\otimes} \arrow{dr} \arrow{rr}{\Gamma^{\otimes}} & & \mathcal{Q}^{\otimes} \arrow{dl} \\
      & \dSt^{\op} \times \xF_{*},
    \end{tikzcd}
  \]
  where $\Gamma^{\otimes}$ preserves cocartesian morphisms that lie
  over inert morphisms in $\xF_{*}$. This we can interpret as a
  ``lax monoidal'' natural transformation of functors $\mathcal{Q}_{(\blank)}^{\otimes}
  \to \QCoh(\blank)^{\otimes}$.
\end{construction}

Now we want to convince ourselves that the projection formula morphism
is natural and compatible with base change. We start by considering
the projection formula for pushforward to a fixed base stack $S$:
\begin{construction}
  Let $S$ be a derived stack. Applying
  Construction~\ref{constr:cartpb} to $\mathcal{Q}_{S} \to
  \dSt_{S}^{\op}$ we get a functor
  \[ \mathcal{Q}_{S} \times \Delta^{1} \to \mathcal{Q}_{S},\]
  lying over the functor $\dSt_{S}^{\op} \times \Delta^{1} \to
  \dSt_{S}^{\op}$ that takes $f \colon X \to S$ to the map
  \opctriangle{X}{S}{S.}{f}{f}{\id_S}
  This sends $(X, \mathcal{E} \in \QCoh(X))$ to the cartesian
  morphism $f_{*}\mathcal{E} \to \mathcal{E}$. Similarly, we have a
  functor
  \[ \dSt_{S}^{\op} \times \QCoh(S) \times \Delta^{1} \to
    \mathcal{Q}_{S},\]
  which takes $(X, \mathcal{F})$ to the cocartesian morphism
  $\mathcal{F} \to f^{*}\mathcal{F}$. Combining these, we get
  \[ \mathcal{Q}_{S} \times \QCoh(S) \times \Delta^{1} \to
    \mathcal{Q}_{S} \times_{\dSt_{S}^{\op}} \mathcal{Q}_{S},\]
  which we can compose with the tensor product
  \[ \mathcal{Q}_{S} \times_{\dSt_{S}^{\op}} \mathcal{Q}_{S} \to \mathcal{Q}_{S}\]
  to get a functor that takes $(X,\mathcal{E},\mathcal{F})$ to the
  natural map $f_{*}\mathcal{E} \otimes_{S} \mathcal{F} \to
  \mathcal{E} \otimes_{X} f^{*}\mathcal{F}$ over $f$. Now taking the
  cartesian pullback in $\mathcal{Q}_{S}$ we get a functor
  \[ \txt{PF}_{S} \colon \mathcal{Q}_{S} \times \QCoh(S) \times \Delta^{1} \to \QCoh(S)\]
  that sends $(X,\mathcal{E},\mathcal{F})$ to the projection formula
  morphism $f_{*}\mathcal{E} \otimes_{S} \mathcal{F} \to f_{*}(\mathcal{E}
  \otimes_{X} f^{*}\mathcal{F})$.
\end{construction}

\begin{construction}
  We can carry out the previous construction more globally:
  First we can define a functor
  \[ \dSt^{\Delta^{1},\op} \times_{\dSt^{\op}} \mathcal{Q} \to
    \mathcal{Q},\]
  where the fibre product is over evaluation at $0$, which takes $(f
  \colon X \to S, \mathcal{E} \in \QCoh(X))$ to the cartesian morphism $f_{*}\mathcal{E}
  \to \mathcal{E}$ over $f \in \dSt^{\op}$. Similarly, we have a
  functor
  \[ \dSt^{\Delta^{1},\op} \times_{\dSt^{\op}} \mathcal{Q} \to
    \mathcal{Q},\] where the fibre product is now over evaluation at
  $1$, which takes $(f \colon X \to S, \mathcal{F} \in \QCoh(S))$ to
  the cocartesian morphism $\mathcal{F} \to f^{*}\mathcal{F}$ over
  $f$.  We can combine these two functors, as well as the tensor
  product functor, to get
  \[ \dSt^{\Delta^{1},\op} \times_{\dSt^{\op} \times \dSt^{\op}}
    \mathcal{Q} \times \mathcal{Q} \times \Delta^{1} \to \mathcal{Q}
    \times_{\dSt^{\op}} \mathcal{Q} \xto{\otimes} \mathcal{Q},\]
  giving a functor that takes $(f \colon X \to S, \mathcal{E},
  \mathcal{F})$ to $f_{*}\mathcal{E} \otimes_{S} \mathcal{F} \to
  \mathcal{E} \otimes_{X}f^{*}\mathcal{F}$. Now a cartesian pullback
  produces from this a functor
  \[ \txt{PF} \colon \dSt^{\Delta^{1},\op} \times_{\dSt^{\op} \times \dSt^{\op}}
    \mathcal{Q} \times \mathcal{Q} \times \Delta^{1} \to \mathcal{Q}\]
  that takes $(f,\mathcal{E},\mathcal{F})$ to the projection formula
  morphism. This lives in a commutative diagram
  \[
    \begin{tikzcd}
      \dSt^{\Delta^{1},\op} \times_{\dSt^{\op} \times \dSt^{\op}}
    \mathcal{Q} \times \mathcal{Q} \times \Delta^{1} \arrow{rrr}
    \arrow{dr} & & & \mathcal{Q} \arrow{ddl} \\
    & \dSt^{\Delta^{1},\op} \arrow{dr}{\txt{ev}_{1}} \\
     & & \dSt^{\op},
    \end{tikzcd}
  \]
  where the two maps to $\dSt^{\op}$ are cocartesian fibrations. We
  can then ask if $\txt{PF}$ preserves cocartesian morphisms. If we
  restrict to universally cocontinuous morphisms we find that it does,
  since we have base change and the pullback functors preserve the
  tensor product. (Moreover, in this case the projection formula
  morphisms are all equivalences.) In particular, for every map
  $\sigma \colon S' \to S$ we get a commutative square
  \[
    \begin{tikzcd}
     \mathcal{Q}_{S} \times \QCoh(S) \times \Delta^{1}
     \arrow{r}{\txt{PF}_{S}} \arrow{d}{\sigma^{*}} & \QCoh(S)
     \arrow{d}{\sigma^{*}} \\
     \mathcal{Q}_{S'} \times \QCoh(S') \times \Delta^{1}
     \arrow{r}{\txt{PF}_{S'}} & \QCoh(S'),
   \end{tikzcd}
 \]
  so that the projection formula is compatible with base change.
\end{construction}

\section{Cotangent Complexes of Derived Stacks}\label{subsec:cotgt}
In this section we recall the definition of the \emph{cotangent
  complex} of a derived stack, and prove some formal properties --- in
particular, we will observe that the cotangent complex determines a
functor of \icats{}.

\begin{definition}
  It follows from the naturality of Remark~\ref{rmk:sqznat} that the square-zero extension functors $\Mod_{A} \to
  \CAlg_{A}$ induce natural square-zero extension functors
  \[ \mathcal{O}_{X}[\blank] \colon \QCoh(X) \simeq \lim_{p \in
    \dAff_{/X}^{\op}} \Mod_{p^{*} \mathcal{O}_{X}} \to \lim_{p \in
    \dAff_{/X}^{\op}} \CAlg_{p^{*} \mathcal{O}_{X}} \simeq
  \QCAlg(X)\] for $X \in \dSt$, which for $f \colon X \to Y$ satisfy
  $f^{*}(\mathcal{O}_{Y}[\mathcal{F}]) \simeq
  \mathcal{O}_{X}[f^{*}\mathcal{F}]$. These restrict to functors
  $\QCohconn(X) \to \QCAlgconn(X)$ and for $\mathcal{F} \in
  \QCohconn(X)$ we write
  \[ X[\mathcal{F}] := \Spec_{X}(\mathcal{O}_{X}[\mathcal{F}]).\]
  The augmentation $\mathcal{O}_{X}[\mathcal{F}] \to \mathcal{O}_{X}$
  induces a morphism $X \to X[\mathcal{F}]$ over $X$.
\end{definition}

\begin{definition}\label{defn:LXS}
  Let $X$ be an $S$-stack. We say a quasicoherent sheaf
  $\mathbb{L}_{X/S} \in \mathcal{Q}_{S}$ is a \emph{(relative) cotangent complex} of $X$ over $S$ if
  there is a natural equivalence 
  \[ \Map_{\mathcal{Q}_{S}}(\mathbb{L}_{X/S}, (Y,\mathcal{F})) \simeq
  \Map_{\dSt_{S}}(Y[\mathcal{F}], X)\] of functors $\QCconn_{S} \to
  \mathcal{S}$.
\end{definition}

\begin{remark}
  Suppose $\mathbb{L} \in \mathcal{Q}_{S}$ is a relative cotangent complex of $X$ over $S$. Let $\pi$ denote
  the projection $\mathcal{Q}_{S} \to \dSt_{S}^{\op}$. Taking
  $\mathcal{F}$ to be
  the zero object of $\QCoh(Y)$, we then have natural equivalences
  \[ \Map_{\dSt_{S}}(Y, X) \simeq \Map_{\dSt_{S}}(Y[0], X) \simeq
  \Map_{\mathcal{Q}_{S}}(\mathbb{L}, (Y,0)) \simeq \Map_{\dSt_{S}}(Y,
  \pi(\mathbb{L})).\] Thus by the Yoneda Lemma $\pi(\mathbb{L}) \simeq
  X$, \ie{} a cotangent complex for $X$ is necessarily
  a quasicoherent sheaf on $X$, as one would expect.
\end{remark}

\begin{remark}\label{rmk:cotgtunique}
  In general, the cotangent complex of an arbitrary derived stack over
  $S$ may well not exist. Moreover, note that we allow the cotangent
  complex of $X$ to be non-connective, while we only consider maps to
  connective $\mathcal{F}$. Thus we are not requiring $\mathbb{L}_{X/S}$
  to satisfy a universal property and so it is not a priori uniquely
  determined if it exists. However, if we assume that
  $\mathbb{L}_{X/S} \in \QCoh^{c[-n]}(X)$ for some $n$, then for
  $\mathcal{F} \in \QCoh^{c[-n]}(X)$ we have
  \[
    \begin{split}
\Map_{\QCoh^{c[-n]}(X)}(\mathbb{L}_{X/S}, \mathcal{F}) & \simeq
  \Omega^{n}\Map_{\QCoh(X)}(\mathbb{L}_{X/S}, \mathcal{F}[n]) \\ & 
  \simeq \Omega^{n}\Map_{\dSt_{X//S}}(X[\mathcal{F}[n]], X),
    \end{split}
  \]
  and so $\mathbb{L}_{X}$ is uniquely determined by the Yoneda
  Lemma. Thus a derived stack has at most one bounded-below cotangent
  complex, and by \cite{AntieauGepnerBrauer}*{Lemma 4.8} if there
  exists one bounded-below cotangent complex then all
  cotangent complexes are equivalent.
\end{remark}

\begin{notation}
  If $X$ is a derived stack over $\Spec A$, we will use the
  abbreviation $\mathbb{L}_{X/A}$ for $\mathbb{L}_{X/\Spec A}$. When
  $A$ is the base field $\k$, we will further abbreviate
  $\mathbb{L}_{X/\k}$ to $\mathbb{L}_{X}$ --- this is the
  \emph{absolute} cotangent complex of $X$ in our context.
\end{notation}

\begin{remark}
  To relate our definition of the cotangent complex to that given in
  \cite{HAG2}, observe that for $\mathbb{L}_{X/S}$ to be a cotangent
  complex for $X$ over $S$ it suffices to have a natural equivalence
  \[ \Map_{\mathcal{Q}_{S}}((X,\mathbb{L}_{X/S}), (Y,\mathcal{M})) \simeq \Map_{\dSt_{S}}(Y[\mathcal{M}], X)\]
  when $Y$ is affine. Indeed, since for a general derived stack $Y$ and quasicoherent sheaf
  $\mathcal{M}$ on $Y$ we have $Y[\mathcal{M}] \simeq \colim_{p}
  \Spec(A \oplus p^{*}\mathcal{M})$ we then have 
  \[ \Map_{\dSt_{S}}(Y[\mathcal{M}], X) \simeq \lim_{p} \Map_{\dSt_{S}}(\Spec(A \oplus
  p^{*}\mathcal{M}), X) \simeq \lim_{p} \Map_{\mathcal{Q}_{S}}(\mathbb{L}_{X/S},
  p^{*}\mathcal{M}).\]
  Taking the fibre at a fixed map $f \colon Y \to X$ we then see
  \[ \Map_{Y//S}(Y[\mathcal{M}], X) \simeq
  \lim_{p}\Map_{\Mod_{A}}(p^{*}f^{*}\mathbb{L}_{X/S}, p^{*}\mathcal{M})
  \simeq \Map_{\QCoh(Y)}(f^{*}\mathbb{L}_{X/S}, \mathcal{M}), \]
  as required. Note that, on the other hand, it is not enough to
  consider only $\Map_{X//S}(X[\mathcal{F}], X)$ for
  $\mathcal{F} \in \QCohconn(X)$ --- this does characterize
  $\mathbb{L}_{X/S}$ as an object of $\QCoh(X)$ if it exists, but in
  general it is a weaker property than that defining the cotangent
  complex.
\end{remark}

\begin{example}
  Suppose $X$ is a derived stack over $T$ and $T$ is a derived stack
  over $S$, such that the cotangent complexes
  $\mathbb{L}_{X/S}$ and $\mathbb{L}_{T/S}$ exist. Then if we define
  $\mathbb{L}_{X/T}$ by the cofibre sequence
  \nolabelcsquare{f^{*}\mathbb{L}_{T/S}}{\mathbb{L}_{X/S}}{0}{\mathbb{L}_{X/T},}
  it follows immediately from the definition that $\mathbb{L}_{X/T}$
  is a relative cotangent complex for $X$ over $T$.
\end{example}

\begin{theorem}\label{thm:geomhascotgt}\ 
  \begin{enumerate}[(i)]
  \item Suppose $f \colon X \to S$ is an $n$-geometric morphism. Then
    $f$ has an (essentially unique) $(-n-1)$-connective cotangent
    complex.
  \item Suppose $X$ is an Artin $S$-stack. Then the
    relative cotangent complex $\mathbb{L}_{X/S}$ is a dualizable
    object of $\QCoh(X)$.
  \end{enumerate}
\end{theorem}
\begin{proof}
  (i) is part of \cite{AntieauGepnerBrauer}*{Proposition 4.45}; the
  existence of the cotangent complex is also part of
  \cite{HAG2}*{Corollary 2.2.3.3}. 
  (ii) follows from the same inductive argument as in
  \cite{AntieauGepnerBrauer}*{Proposition 4.45}, appealing to 
  \cite[Theorem 7.4.3.18]{HA} on affines.
\end{proof}

\begin{lemma}\label{lem:relcotgtpb}
  Suppose $f \colon X \to S$ has a relative cotangent complex
  $\mathbb{L}_{X/S}$, and let
  \csquare{X'}{X}{S'}{S}{\xi}{f'}{f}{\sigma}
  be a pullback square. Then $\xi^{*}\mathbb{L}_{X/S}$ is a relative
  cotangent complex for $X'$ over $S'$.
\end{lemma}
\begin{proof}
  The adjunction
  \[ \sigma_{!} : \dSt_{/S'} \rightleftarrows \dSt_{/S} : \sigma^{*},\]
  where $\sigma_{!}$ is given by composition with $\sigma$ and
  $\sigma^{*}$ by pullback along $\sigma$, lifts to an adjunction
  \[ \sigma^{*} : \mathcal{Q}_{S} \rightleftarrows \mathcal{Q}_{S'} :
    \sigma_{!},\]
  where $\sigma^{*}$ takes $(X \to S, \mathcal{E} \in \QCoh(X))$ to $(X',
  \xi^{*}\mathcal{E})$ and $\sigma_{!}$ takes $(X \to S',
  \mathcal{E})$ to $(X \to S, \mathcal{E})$.

  We therefore have natural equivalences
  \[
    \begin{split}
\Map_{\mathcal{Q}_{S'}}(\xi^{*}\mathbb{L}_{X/S}, (Y,\mathcal{M})) & 
    \simeq \Map_{\mathcal{Q}_{S}}(\mathbb{L}_{X/S}, (Y, \mathcal{M})) \\ &
    \simeq \Map_{\dSt_{S}}(Y[\mathcal{M}], X) \\ & \simeq
    \Map_{\dSt_{S'}}(Y[\mathcal{M}], X'),
    \end{split}
  \]
  which exhibit $\xi^{*}\mathbb{L}_{X/S}$ as a cotangent complex for
  $X'$ over $S'$.
\end{proof}

\begin{example}
  Suppose $X$ is an Artin $S$-stack and $\Sigma$ is an arbitrary
  $S$-stack. Then the projection $X \times_{S} \Sigma \to \Sigma$ has
  a relative cotangent complex $\mathbb{L}_{X \times_{S} \Sigma / \Sigma}
  \simeq \pi_{X}^{*}\mathbb{L}_{X/S}$, where $\pi_{X}$ denotes the
  projection $X \times_{S} \Sigma \to X$.
\end{example}

\begin{lemma}\label{lem:QCohIsusp}
  The projection
  $\Map_{\mathcal{Q}_{S}}((X,\mathcal{F}), (Y,\mathcal{G})) \to
  \Map_{\dSt_{S}}(Y,X)$
  has a canonical section, taking $f \colon Y \to X$ to the zero map
  $\mathcal{F} \to f_{*}\mathcal{G}$. We write
  $\Omega_{\Map_{/S}(Y,X)}\Map_{\mathcal{Q}_{S}}((X, \mathcal{F}), (Y,
  \mathcal{G}))$
  for the pullback
  \nolabelcsquare{\Omega_{\Map_{/S}(Y,X)}\Map_{\mathcal{Q}_S}((X, \mathcal{F}),
    (Y,
    \mathcal{G}))}{\Map_{\dSt_S}(Y,X)}{\Map_{\dSt_S}(Y,X)}{\Map_{\mathcal{Q}_S}((X,
    \mathcal{F}), (Y, \mathcal{G}))}
  of two copies of this section.
  Then there are natural equivalences
  \[
    \begin{tikzcd}
      \Map_{\mathcal{Q}_{S}}((X, \mathcal{F}[1]), (Y, \mathcal{G}))\arrow{d}{\sim} \\
\Omega_{\Map_{/S}(Y,X)}\Map_{\mathcal{Q}_{S}}((X, \mathcal{F}), (Y,
\mathcal{G})) \\
      \Map_{\mathcal{Q}_{S}}((X, \mathcal{F}), (Y, \mathcal{G}[-1])) \arrow{u}[swap]{\sim} 
    \end{tikzcd}
  \]
  
\end{lemma}
\begin{proof}
  We consider the case of $\mathcal{F}[1]$; the case of
  $\mathcal{G}[-1]$ is proved similarly. The definition of
  $\mathcal{F}[1]$ as a pushout gives a commutative square
  \nolabelcsquare{\Map_{\mathcal{Q}_S}((X, \mathcal{F}[1]), (Y,
    \mathcal{G}))}{\Map_{\mathcal{Q}_S}((X, 0), (Y,
    \mathcal{G}))}{\Map_{\mathcal{Q}_S}((X, 0), (Y,
    \mathcal{G}))}{\Map_{\mathcal{Q}_S}((X, \mathcal{F}), (Y, \mathcal{G})),}
  and, identifying
  $\Map_{\mathcal{Q}_{S}}((X, 0), (Y, \mathcal{G})) \simeq \Map_{\dSt_{S}}(Y, X)$,
  this induces a map
  \[\Map_{\mathcal{Q}_{S}}((X, \mathcal{F}[1]), (Y, \mathcal{G})) \to
  \Omega_{\Map_{/S}(Y,X)}\Map_{\mathcal{Q}_{S}}((X, \mathcal{F}), (Y,
  \mathcal{G})).\]
  To see that this is an equivalence it suffices to see it is an
  equivalence on the fibres over each point of $\Map_{\dSt_{S}}(Y,X)$; at
  $f \colon Y \to X$ the map on fibres is the natural equivalence
  \[\Map_{\QCoh(X)}(\mathcal{F}[1], f_{*}\mathcal{G}) \isoto
  \Omega\Map_{\QCoh(X)}(\mathcal{F}, f_{*}\mathcal{G}).\qedhere\]
\end{proof}

\begin{remark}\label{rmk:relcotgtrepr}
  If $\mathbb{L}_{X/S}$ is the relative cotangent complex of $X \to
  S$, then for $\mathcal{G} \in \QCoh(Y)$ the zero
  section $\Map_{\dSt}(Y, X) \to \Map_{\mathcal{Q}}((X,\mathbb{L}_{X/S}),
  (Y, \mathcal{G}))$ corresponds to the commutative square
  \nolabelcsquare{\Map_{\dSt}(Y, X)}{\Map_{\dSt}(Y[\mathcal{G}],
    X)}{\Map_{\dSt}(Y, S)}{\Map_{\dSt}(Y[\mathcal{G}], S),}
  where the vertical maps are given by composition with $X \to S$
  and the horizontal maps by composition with $Y[\mathcal{G}]  \to Y$. Thus by Lemma~\ref{lem:QCohIsusp} the space
  $\Map_{\mathcal{Q}}((X, \mathbb{L}_{X/S}[n]), (Y, \mathcal{G}))$ is
  equivalent to \[\Omega^{n}_{\Map(Y,X)}\left(\Map_{\dSt}(Y,S)
  \times_{\Map_{\dSt}(Y[\mathcal{G}], S)} \Map_{\dSt}(Y[\mathcal{G}], X)\right).\]
\end{remark}

\begin{lemma}
  The  relative cotangent complex determines a functor
  \opctriangle{\Fun(\Delta^{1},\dSt)^{\txt{geom},\op}}{\mathcal{Q}}{\dSt}{}{\txt{ev}_{0}}{}
\end{lemma}
\begin{proof}
  Let
  $\Lambda_{n} \colon \Fun(\Delta^{1},\dSt)^{\txt{geom}} \to
  \mathcal{P}(\mathcal{Q}^{c[-n],\op})$ be the functor sending a
  morphism $f \colon X \to S$ to the presheaf
  \[(Y, \mathcal{G}) \mapsto
  \Omega^{n}_{\Map(Y,X)}\left(\Map_{\dSt}(Y,S)
    \times_{\Map_{\dSt}(Y[\mathcal{G}[n]], S)}
    \Map_{\dSt}(Y[\mathcal{G}[n]], X)\right).\]
  If the morphism $X \to S$ has a $(-n)$-connective relative cotangent complex,
  then by Remark~\ref{rmk:relcotgtrepr} the presheaf
  $\Lambda_{n}(Y, \mathcal{G})$ is naturally equivalent to
\[ \Map_{\mathcal{Q}}((X, \mathbb{L}_{X/S}[n]), (Y,
\mathcal{G}[n])) \simeq \Map_{\mathcal{Q}^{c[-n]}}((X,
\mathbb{L}_{X/S}), (Y, \mathcal{G})).\]

Denote by $\Fun(\Delta^{1}, \dSt)^{n\txt{-geom}}$ the full subcategory
of $\Fun(\Delta^{1}, \dSt)$ spanned by the $n$-geometric stacks. Then
the restriction of the functor  $\Lambda_{n+1}$ to the subcategory
$\Fun(\Delta^{1}, \dSt)^{n\txt{-geom}}$ lands in the full subcategory
of $\mathcal{P}(\mathcal{Q}^{c[-n],\op})$ spanned by the representable
presheaves by Theorem~\ref{thm:geomhascotgt}. By the Yoneda Lemma,
this means it determines a functor
$\mathbb{L}^{n} \colon \Fun(\Delta^{1}, \dSt)^{n\txt{-geom}} \to
\mathcal{Q}^{c[-n-1],\op}$.
Moreover, these functors are clearly compatible as we increase $n$,
and so combine to give a functor
\[\mathbb{L} \colon \Fun(\Delta^{1}, \dSt)^{\txt{geom}} \to \mathcal{Q}^{\op},\] as required.
\end{proof}

\begin{proposition}[\cite{AntieauGepnerBrauer}*{Lemma 4.11}]\label{propn:tgtlim}
  The cotangent complex functor takes finite limits in
  $\Fun(\Delta^{1},\dSt)^{\txt{geom}}$ to relative colimits in
  $\QC$. In other words, if
  $p \colon \mathcal{I}^{\triangleleft} \to
  \Fun(\Delta^{1},\dSt)^{\txt{geom}}$
  is a finite limit diagram, taking the initial object to $X \to S$,
  then the induced diagram
  $(\mathcal{I}^{\op})^{\triangleright} \to \QCoh(X)$ that sends
  $i \in \mathcal{I}$ to $p(\pi_{i})^{*}\mathbb{L}_{p(i)}$, with
  $\pi_{i}$ the unique map $-\infty \to i$ in
  $\mathcal{I}^{\triangleleft}$, is a colimit diagram.
\end{proposition}
\begin{proof}
  This is immediate from the definition, since the functor
  corepresented by the cotangent complex preserves limits in
  the target variable.
\end{proof}

\begin{definition}
  Let $X \to S$ be a morphism of derived stacks that has a relative
  cotangent complex $\mathbb{L}_{X/S}$. We write $\mathbb{T}_{X/S}$ for its
  dual $\mathbb{L}_{X/S}^{\vee}$, and refer to this as the \emph{relative tangent
    complex} of $X$ over $S$. If $X \to S$ is a relative Artin stack,
  then we also have $\mathbb{L}_{X/S} \simeq
  \mathbb{T}_{X/S}^{\vee}$ by Theorem~\ref{thm:geomhascotgt}.
\end{definition}

Our final goal in this section is to show that, under strong
finiteness assumptions, we can identify the tangent complex of a
derived mapping stack.

\begin{definition}\label{defn:Ocompact}
  We say a morphism $f \colon X \to Y$ is \emph{perfect}
  if the functor $f_{*} \colon \QCoh(X) \to \QCoh(Y)$ preserves
  dualizable objects. We say $f$ is \emph{universally
  perfect} if any base change of $f$ is
  perfect. We say $f$ is \emph{$\mathcal{O}$-compact} if
  $f$ is both universally perfect and universally
  cocontinuous.
\end{definition}
\begin{warning}
  This notion of perfect morphism differs from that considered in
  \cite{BenZviFrancisNadler}. As we will only encounter this notion as
  part of the notion of $\mathcal{O}$-compact morphism, this should
  not cause any confusion.
\end{warning}

\begin{lemma}\label{lem:Ocpt}
  A morphism $f \colon X \to Y$ is $\mathcal{O}$-compact \IFF{} for
  every map $p \colon \Spec A \to Y$, the base change $f_{p} \colon
  X_{p} \to \Spec A$ is perfect and cocontinuous.
\end{lemma}
\begin{proof}
  Since the second condition holds for a base change of $f$ if it
  holds for $f$ it suffices to show that it implies that $f_{*}$
  preserves colimits and dualizable objects. For colimits we already
  proved this in Corollary~\ref{cor:upfcocts}. For dualizable objects
  it follows from base change (in the form of
  Proposition~\ref{propn:affpbbc}) together with
  \cite{HA}*{Proposition 4.6.1.11}, which implies that an object of
  $\QCoh(Y)$ is dualizable \IFF{} its image under $p^{*}$ is
  dualizable in $\Mod_{A}$ for each $p \colon \Spec A \to Y$.
\end{proof}

\begin{proposition}\label{propn:UCCcolim}
  For any derived stack $S$, the full subcategories of $\dSt_{S}$
  spanned by the
  \begin{enumerate}[(i)]
  \item cocontinuous,
  \item perfect,
  \item universally cocontinuous,
  \item universally perfect,
  \item $\mathcal{O}$-compact
  \end{enumerate}
  morphisms are all closed under finite colimits.
\end{proposition}
\begin{proof}
  Consider a diagram $\mathcal{I} \to \dSt_{S}$ taking
  $i \in \mathcal{I}$ to $p_{i}\colon X_{i} \to S$ with colimit
  $p \colon X \to S$. To prove (i) we must show that if each functor
  $p_{i,*} \colon \QCoh(X_{i}) \to \QCoh(S)$ preserves colimits, then
  so does the functor $p_{*}$. Since $X$ is a colimit, we have an
  equivalence
  $\QCoh(X) \simeq \lim_{i \in \mathcal{I}^{\op}} \QCoh(X_{i})$. If we
  write $f_{i} \colon X_{i} \to X$ for the unique map in the colimit
  diagram, this means that for $\mathcal{E} \in \QCoh(X)$ we have a
  natural equivalence
  \[ \mathcal{E} \simeq \lim_{i \in\mathcal{I}^{\op}}
    f_{i,*}f_{i}^{*}\mathcal{E},\] by Corollary~\ref{cor:qcohrellim},
  and hence we have an equivalence of functors
  \[p_{*} \simeq \lim_{i \in \mathcal{I}^{\op}} p_{i,*}f_{i}^{*}.\]
  The functors $f_{i}^{*}$ preserve colimits (being left adjoints) and
  the functors $p_{i,*}$ preserve colimits by assumption. Moreover,
  since $\QCoh(S)$ is a stable \icat{}, finite limits commute with
  colimits. Thus, if the \icat{} $\mathcal{I}$ is finite, the functor
  $p_{*}$ also preserves colimits. This proves (i).

  The same argument proves (ii): If we assume the functors $p_{i,*}$
  preserve dualizable objects, and $\mathcal{E} \in \QCoh(X)$ is
  dualizable, then $p_{*}\mathcal{E} \simeq \lim_{i \in
    \mathcal{I}^{\op}} p_{i,*}f_{i}^{*}\mathcal{E}$ is again
  dualizable. To see this, first observe that  the functors $f_{i}^{*}$
  preserve dualizable objects (since they are symmetric monoidal), so
  each object $p_{i,*}f_{i}^{*}\mathcal{E}$ is dualizable. Then we use
  that  a finite limit of dualizable objects is again dualizable in a stable
  \icat{} with a symmetric monoidal structure where the tensor product
  preserves finite colimits in each variable.

  Now (iii) and (iv) follow since for every map $\sigma \colon S' \to S$, the
  pullback functor $\sigma^{*} \colon \dSt_{S} \to \dSt_{S'}$
  preserves colimits (as $\dSt$ is an $\infty$-topos). Finally,
  (v) is just a combination of (iii) and (iv).
\end{proof}

\begin{notation}
  Given $X \to S$ and $S' \to S$ in $\dSt$, we write $X^{S'}_{S}$ for
  the internal Hom in $\dSt_{/S}$.
\end{notation}

\begin{proposition}\label{propn:mapscotgt}
  Suppose $X \to S$ is a relative Artin stack, and $S' \to S$ is an
  $\mathcal{O}$-compact morphism.
  If $\txt{ev}$ denotes the evaluation map
  $S' \times_{S} X^{S'}_{S} \to X$ and $\pi$ the projection
  $S' \times_{S} X^{S'}_{S} \to X^{S'}_{S}$, then
  $(\pi_{*}\txt{ev}^{*}\mathbb{T}_{X/S})^{\vee}$ is a relative
  cotangent complex for $X^{S'}_{S}$ over $S$.  Under the induced
  identification of $\mathbb{L}_{S' \times_{S} X^{S'}_{S}/S'}$ with
  $\pi^{*}(\pi_{*}\txt{ev}^{*}\mathbb{T}_{X/S})^{\vee}$ via
  Lemma~\ref{lem:relcotgtpb}, the natural map
  $\txt{ev}^{*}\mathbb{L}_{X/S} \to \mathbb{L}_{S' \times_{S}
    X^{S'}_{S}/S'}$
  corresponds to the dual of the counit map
  $\pi^{*}\pi_{*}\txt{ev}^{*}\mathbb{T}_{X/S} \to
  \txt{ev}^{*}\mathbb{T}_{X/S}$.
\end{proposition}
\begin{proof}
  For any derived stack $Z$ over $S$ and 
  $\mathcal{G} \in \QCohconn(Z)$, we have a natural equivalence
  \[ \Map_{\dSt_{S}}(Z[\mathcal{G}], X^{S'}_{S}) \simeq
    \Map_{\dSt_{S}}(S' \times_{S} Z[\mathcal{G}], X) \simeq
    \Map_{\dSt_{S}}((S' \times_{S} Z)[\pi_{Z}^{*}\mathcal{G}], X),\]
  where $\pi_{Z}$ is the projection $S' \times_{S}Z \to Z$ and the
  second equivalence follows from
  Proposition~\ref{propn:pbrelspec}. Since $X$ has a relative
  cotangent complex over $S$, we may identify the right-hand side here
  with $\Map_{\mathcal{Q}_{S}}(\mathbb{L}_{X/S}, (S' \times_{S} Z,
  \pi_{Z}^{*}\mathcal{G}))$; its fibre at $\phi \colon S' \times_{S} Z \to X$
  is given by $\Map_{\QCoh(S' \times_{S}Z)}(\phi^{*}\mathbb{L}_{X/S}, \pi_{Z}^{*}\mathcal{G})$.

  Since $\mathbb{L}_{X/S}$ is dualizable, so is
  $\phi^{*}\mathbb{L}_{X/S}$, giving
\[ \Map_{\QCoh(Z
  \times_{S}S')}(\phi^{*}\mathbb{L}_{X/S}, \pi_{Z}^{*}\mathcal{G})
\simeq \Map_{\QCoh(Z
  \times_{S}S')}(\mathcal{O}_{Z \times_{S}S'},
\phi^{*}\mathbb{T}_{X/S} \otimes_{Z \times_{S}S'}
\pi_{Z}^{*}\mathcal{G}).\]
Since $\mathcal{O}_{Z \times_{S}S'} \simeq \pi_{Z}^{*}\mathcal{O}_{Z}$
we have by adjunction
  \begin{multline*}
\Map_{\QCoh(Z
  \times_{S}S')}(\mathcal{O}_{Z \times_{S}S'},
\phi^{*}\mathbb{T}_{X/S} \otimes_{Z \times_{S}S'}
\pi_{Z}^{*}\mathcal{G}) \\ \simeq \Map_{\QCoh(Z)}(\mathcal{O}_{Z}, \pi_{Z,*}(\phi^{*}\mathbb{T}_{X/S} \otimes_{Z \times_{S}S'}
\pi_{Z}^{*}\mathcal{G})).
\end{multline*}
Applying the projection formula to $\pi_{Z}$ (which is universally
cocontinuous since it is a base change of $S' \to S$) we
then get
\[ \Map_{\QCoh(Z)}(\mathcal{O}_{Z}, \pi_{Z,*}(\phi^{*}\mathbb{T}_{X/S} \otimes_{Z \times_{S}S'}
\pi_{Z}^{*}\mathcal{G})) \simeq \Map_{\QCoh(Z)}(\mathcal{O}_{Z},
\pi_{Z,*}\phi^{*}\mathbb{T}_{X/S} \otimes_{Z} \mathcal{G}).\]
Let $f \colon Z \to X^{S'}_{S}$ be the map corresponding to
$\phi$. Then we have a commutative diagram
\[
\begin{tikzcd}
  Z \times_{S} S' \arrow{r}{f \times_{S} S'} \arrow{d}{\pi_{Z}}
  \arrow[bend left=30]{rr}{\phi}
  & X^{S'}_{S} \times_{S}S'  \arrow{d}{\pi} \arrow{r}{\txt{ev}} & X \\
  Z \arrow{r}{f} & X^{S'}_{S},
\end{tikzcd}
\]
where the square is cartesian. Applying base change for this pullback
square, we have a natural equivalence
$\pi_{Z,*}\phi^{*}\mathbb{T}_{X/S} \simeq
f^{*}\pi_{*}\txt{ev}^{*}\mathbb{T}_{X/S}$. 
Since $\pi$ is perfect, being a base change of $S' \to
S$, the quasicoherent sheaf $\pi_{*}\txt{ev}^{*}\mathbb{T}_{X/S}$ is
dualizable, with dual preserved by the symmetric monoidal functor
$f^{*}$. Thus we get
\[
  \begin{split}
\Map_{\QCoh(Z)}(\mathcal{O}_{Z},
\pi_{Z,*}\phi^{*}\mathbb{T}_{X/S} \otimes_{Z} \mathcal{G}) & \simeq
\Map_{\QCoh(Z)}(f^{*}(\pi_{*}\txt{ev}^{*}\mathbb{T}_{X/S})^{\vee},
\mathcal{G}) \\ & \simeq
\Map_{\mathcal{Q}_{S}}((\pi_{*}\txt{ev}^{*}\mathbb{T}_{X/S})^{\vee}, \mathcal{G})_{f}.
  \end{split}
\]
Putting these natural equivalences together, we have
\[ \Map_{\dSt_{S}}(Z[\mathcal{G}], X^{S'}_{S})_{f} \simeq \Map_{\mathcal{Q}_{S}}((\pi_{*}\txt{ev}^{*}\mathbb{T}_{X/S})^{\vee}, \mathcal{G})_{f},\]
naturally in $f$ and $\mathcal{G}$, which identifies
$(\pi_{*}\txt{ev}^{*}\mathbb{T}_{X/S})^{\vee}$ as a relative cotangent complex
of $X^{S'}_{S} \to S$.

Unwinding these equivalences for $Z = X^{S'}_{S}$,
$f = \id_{X^{S'}_{S}}$, and
$\mathcal{G} = (\pi_{*}\txt{ev}^{*}\mathbb{T}_{X/S})^{\vee}$ we see
that the identity of $(\pi_{*}\txt{ev}^{*}\mathbb{T}_{X/S})^{\vee}$ is
identified with the map
$\txt{ev}^{*}\mathbb{L}_{X/S} \to
\pi^{*}(\pi_{*}\txt{ev}^{*}\mathbb{T}_{X/S})^{\vee}$ that is dual to the
counit map $\pi^{*}\pi_{*}\txt{ev}^{*}\mathbb{T}_{X/S} \to
\txt{ev}^{*}\mathbb{T}_{X/S}$. On the other hand, the
canonical map $X^{S'}_{S}[\mathbb{L}_{X^{S'}_{S}/S}] \to X^{S'}_{S}$
corresponding to the identity of $\mathbb{L}_{X^{S'}_{S}/S}$ is
identified under the first equivalence with the composite
\[(S' \times_{S} X^{S'}_{S})[\pi^{*}\mathbb{L}_{X^{S'}_{S}/S}] \simeq
S' \times_{S} (X^{S'}_{S}[\mathbb{L}_{X^{S'}_{S}/S}]) \to S'
\times_{S} X^{S'}_{S} \xto{\txt{ev}} X,\]
which precisely corresponds to the canonical map
$\txt{ev}^{*}\mathbb{L}_{X/S} \to
\mathbb{L}_{S'\times_{S}X^{S'}_{S}/S'}$.
\end{proof}

\begin{remark}
  Note that in the situation of Proposition~\ref{propn:mapscotgt},
  although $X^{S'}_{S} \to S'$ has a relative cotangent complex,
  it need not be a geometric morphism.
\end{remark}

In this paper we are  interested in mapping stacks of the
form $X_{S}^{\Sigma_{B} \times S}$ where $\Sigma_{B}$ is the Betti stack of a
finite cell complex $\Sigma \in \mathcal{S}$, defined as the colimit
$\colim_{X}\Spec \k$ of the constant functor with value $\Spec \k$.
The Betti stack of a finite cell complex is $\mathcal{O}$-compact by Lemma~\ref{lem:BettiOcpt}, and we
 now spell
out what Proposition~\ref{propn:mapscotgt} gives in this case:
\begin{corollary}
  Let $\Sigma \in \mathcal{S}$ be a finite cell complex. Then for any
  Artin $S$-stack $X$, the mapping stack $X_{S}^{\Sigma_{B} \times S}$ is an
  Artin $S$-stack, and its tangent complex is given by 
  $\pi_{*}\txt{ev}^{*}\mathbb{T}_{X/S}$, where
  $\txt{ev} \colon \Sigma_{B} \times X_{S}^{\Sigma_{B} \times S} \to X$ is the
  evaluation map and $\pi \colon \Sigma_{B} \times X_{S}^{\Sigma_{B} \times S} \to
  X_{S}^{\Sigma_{B} \times S}$ is the projection.
\end{corollary}
\begin{proof}
  The stack $X_{S}^{\Sigma_{B} \times S}$ is the limit of the constant
  diagram in $\dSt_{S}$ with
  value $X$ indexed by $\Sigma$, hence it is an Artin stack as
  these are closed under finite limits. The identification of the
  tangent complex follows from Proposition~\ref{propn:mapscotgt}.
\end{proof}

\section{The de Rham Complex}\label{subsec:deRham}
In this section we review the construction of the (relative) de Rham
complex of a morphism of derived stacks, as introduced in
\cite{PTVV,CPTVV}, which will be used to define closed differential
forms in the next section. We also prove some basic properties of
these objects, which we will use to construct pushforward maps for
relative differential forms in \S\ref{subsec:DFpush}.

We first briefly recall the notion of \emph{graded mixed complexes};
see \cite{CPTVV}*{\S 1.1--1.2} for more details.
\begin{definition}
  Let $\Ch_{\k}$ denote the ordinary category of cochain complexes of
  $\k$-vector spaces; for $A$ a commutative dg-algebra over $\k$, we
  write $\sMod_{A}(\Ch_{\k})$ for the ordinary category of differential
  graded $A$-modules. A \emph{graded mixed $A$-module} is
  a 
  family of $A$-modules $E(p) \in \sMod_{A}(\Ch_{\k})$ for $p \in \ZZ$
  together with morphisms $\epsilon \colon E(p) \to E(p+1)[1]$ such
  that $\epsilon^{2} = 0$. We write $\sGMM_{A}(\Ch_{\k})$ for the
  ordinary category of graded mixed $A$-modules; this has a symmetric
  monoidal model structure where the weak equivalences are the
  componentwise quasi-isomorphisms and the symmetric monoidal
  structure is given by
  \[ (E\otimes E')(p) := \bigoplus_{i+j=p} E(i) \otimes E'(j)\]
  (and the symmetry does not involve signs). Let $\GMM_{A}$ denote
  the \icat{} obtained by inverting these weak equivalences; this
  inherits a symmetric monoidal structure, and we write $\GMC_{A}$ for
  the \icat{} $\CAlg(\GMM_{A})$; equivalently, this is the \icat{}
  underlying the model category of commutative algebras in
  $\sGMM_{A}(\Ch_{\k})$. We obtain a functor of \icats{} $\GMC \colon \CAlg_{\k}
  \to \PrL$.
\end{definition}

\begin{remark}
The functor $(\blank)(0) \colon \GMC_{A} \to \CAlg_{A}$ is accessible
and preserves limits, hence it has a left adjoint \[\DR_{A} \colon
\CAlg_{A} \to \GMC_{A}\] (see \cite{CPTVV}*{Proposition 1.3.8}).
\end{remark}

\begin{definition}
  Given a morphism of derived stacks $X \to \Spec A$, we define $\DR_{A}(X) \in
  \GMC_{A}$ to be $\lim_{p \in \dAff_{/X}^{\op}}
  \DR_{A}(p^{*}\mathcal{O}_{X})$. This determines a functor \[\DR_{A}
  \colon \dSt_{/\Spec A}^{\op} \to \GMC_{A}.\]
\end{definition}

\begin{definition}
  If $X$ is a derived stack over $S$, we define \[\DR_{S}(X)
  \colon \dAff_{/S}^{\op} \to \GMC_{\k}\] by sending $p \colon \Spec A
  \to S$ to $\DR_{A}(X_{p})$, where $X_{p}$ is the pullback of $X$
  along $p$. This is a (non-quasicoherent) sheaf of
  graded mixed commutative $\mathcal{O}_{S}$-algebras, giving a
  functor
  \[ \DR_{S} \colon \dSt_{S}^{\op} \to
  \Mod_{\mathcal{O}_{S}}(\Sh_{S}(\GMC_{\k})).\]
  We write $\DR(X/S)$ for the global sections
  \[\Gamma_{S}\DR_{S}(X) \simeq \DR_{S}(X)(S) \simeq \lim_{p \colon \Spec A \to S}
  \DR_{A}(X_{p}).\]
\end{definition}

\begin{lemma}\label{lem:relDRpb}
  For any commutative square \csquare{X'}{X}{S'}{S,}{\xi}{f'}{f}{\sigma}
  we have a natural morphism
  \[ \sigma^{*}\DR_{S}(X) \to \DR_{S'}(X').\]
  If the square is cartesian then this is an equivalence.
\end{lemma}
\begin{proof}
  Immediate from the definition of $\DR_{S}(X)$, since for $p \colon
  \Spec A \to S'$ we have natural morphisms 
  \[p^{*}\sigma^{*}\DR_{S}(X)
  \simeq \DR_{A}(X_{\sigma p}) \to \DR_{A}(X'_{p}) \simeq
  p^{*}\DR_{S'}(X'),\]
  which are equivalences if the square is cartesian.
\end{proof}

\begin{remark}\label{rmk:DRfunctor}
  Given a commutative square \csquare{X'}{X}{S'}{S,}{}{}{}{\sigma} the map
  of sheaves $\sigma^{*}\DR_{S}(X) \to \DR_{S'}(X')$ gives a map
  \[ \sigma^{*} \colon \DR(X/S) \to \DR(X'/S') \]
  as the composite
  \[\DR(X/S) \simeq \DR_{S}(X)(S) \to \DR_{S}(X)(S') \simeq \Gamma_{S'}(\sigma^{*}\DR_{S}(X)) \to \DR(X'/S'),\]
  where the last map is an equivalence if the square is
  cartesian. Given another square
  \csquare{X''}{X'}{S''}{S'}{}{}{}{\sigma'}
  we get a commutative diagram
  \[
    \begin{tikzcd}
      \DR_{S}(X)(S) \arrow{r} & \DR_{S}(X)(S') \arrow{r} \arrow{d} &
      \DR_{S}(X)(S'') \arrow{d}\\
      & \DR_{S'}(X')(S') \arrow{r} & \DR_{S'}(X')(S'')\arrow{d}\\
       & & \DR_{S''}(X'')(S''),
    \end{tikzcd}
  \]
  which shows that $(\sigma\sigma')^{*} \colon \DR(X/S) \to
  \DR(X''/S'')$ agrees with the composite $\DR(X/S) \xto{\sigma^{*}}
  \DR(X'/S') \xto{\sigma'^{*}} \DR(X''/S'')$. Extending this
  observation to arbitrary sequences of maps, we see that $\DR(\blank/\blank)$ determines a
  functor $\dSt^{\Delta^{1},\op} \to \GMC_{\k}$.
\end{remark}

\begin{lemma}\label{lem:DR/Sdesc}
  If $S$ is a colimit $\colim_{i \in \mathcal{I}} S_{i}$, then
  $\DR(X/S)$ satisfies descent in the sense that the natural map
  \[ \DR(X/S) \to \lim_{i \in \mathcal{I}^{\op}} \DR(X_{i}/S_{i})\]
  is an equivalence, where $X_{i}$ denotes the pullback of $X$ to $S_{i}$.
\end{lemma}
\begin{proof}
  Since $\DR_{S}(X)$ is a sheaf, we have
  \[
    \begin{split}
\DR(X/S) & \simeq \DR_{S}(X)(S) \simeq \lim_{i \in
      \mathcal{I}^{\op}} \DR_{S}(X)(S_{i}) \simeq \lim_{i \in
      \mathcal{I}^{\op}} (\sigma_{i}^{*}\DR_{S}(X))(S_{i})
    \\
    & \simeq
    \lim_{i \in \mathcal{I}^{\op}} \DR_{S_{i}}(X_{i})(S_{i}) \simeq
    \lim_{i \in \mathcal{I}^{\op}} \DR(X_{i}/S_{i}).      \qedhere
    \end{split}
\]
\end{proof}

\begin{proposition}\label{propn:symLXAdesc}
  For $f \colon X \to \Spec A$ a geometric morphism, the natural map
  \[ f_{*}\Sym_{\mathcal{O}_{X}}(\mathbb{L}_{X/A}[-1]) \to \lim_{p \in
  \dAff_{/X}^{\op}} (fp)_{*}\Sym_{p^{*}\mathcal{O}_{X}}(
\mathbb{L}_{p^{*}\mathcal{O}_{X}/A}[-1])\]
is an equivalence.
\end{proposition}
\begin{proof}
  The proof of \cite{PTVV}*{Lemma 1.15} shows that for a geometric
  groupoid $X_{\bullet}$ with colimit $X$ in $\dSt_{/\Spec A}$ we have
  \[ \Sym_{\mathcal{O}_{X}}(\mathbb{L}_{X/A}) \simeq \lim_{\simp}
  \pi_{\bullet,*}
  \Sym_{\mathcal{O}_{X_{\bullet}}}(\mathbb{L}_{X_{\bullet}/A}),\]
  where $\pi_{i}$ denotes the canonical map $X_{i} \to X$.  
  The statement holds for $-1$-geometric stacks over
  $\Spec A$, and we can induct on $n$ as in the proof of
  \cite{PTVV}*{Proposition 1.14} using this equivalence to
  show it holds for $n$-geometric stacks.
\end{proof}

\begin{corollary}\label{cor:DRAXisSymLXA}
  For $f \colon X \to \Spec A$ a geometric morphism, the underlying
  commutative $A$-algebra of $\DR_{A}(X)$ is
  $f_{*}\Sym_{\mathcal{O}_{X}}(\mathbb{L}_{X/A}[-1])$.
\end{corollary}
\begin{proof}
  Let $U_{A}$ denote the forgetful functor $\GMC_{A} \to
  \CAlg_{A}$.
  By definition, $\DR_{A}(X)$ is
  $\lim_{p \in \dAff_{/X}^{\op}}
  \DR_{A}(p^{*}\mathcal{O}_{X})$.
  Using \cite{CPTVV}*{Proposition 1.3.12} we have a natural
  equivalence
 \[U_{A}\DR_{A}(p^{*}\mathcal{O}_{X}) \simeq
  \Sym_{p^{*}\mathcal{O}_{X}}(\mathbb{L}_{p^{*}\mathcal{O}_{X}/A}[-1]).\]
  Since $U_{A}$ preserves limits, being a right adjoint, this implies
  \[ U_{A}\DR_{A}(X) \simeq \lim_{p}
  \Sym_{p^{*}\mathcal{O}_{X}}(\mathbb{L}_{p^{*}\mathcal{O}_{X}/A}[-1])
  \simeq f_{*}\Sym_{\mathcal{O}_{X}}(\mathbb{L}_{X/A}[-1]),\]
  where the second equivalence holds by
  Proposition~\ref{propn:symLXAdesc} (since $(fp)_{*}$ is just a
  forgetful functor between \icats{} of modules).
\end{proof}

\begin{remark}
  For $f \colon X \to S$ geometric, consider a pullback square
  \csquare{X_p}{X}{\Spec A}{S.}{\bar{p}}{f_p}{f}{p}
  We have $p^{*}\DR_{S}(X) \simeq \DR_{A}(X_{p})$ by definition, and
  by applying Corollary~\ref{cor:DRAXisSymLXA} and Lemma~\ref{lem:relcotgtpb}
  we get
  \[ U_{A}\,p^{*}\DR_{S}(X) \simeq
  f_{p,*}\Sym_{\mathcal{O}_{X_{p}}}(\mathbb{L}_{X_{p}/A}[-1]) \simeq
  f_{p,*}\bar{p}^{*} \Sym_{\mathcal{O}_{X}}(\mathbb{L}_{X/S}[-1]).\]
  Now \emph{if} $f$ satisfies base change, we can rewrite this as
  \[ U_{A}\,p^{*}\DR_{S}(X) \simeq
  p^{*}f_{*}\Sym_{\mathcal{O}_{X}}(\mathbb{L}_{X/S}[-1]),\]
  and so in this case the underlying sheaf of
  $\mathcal{O}_{X}$-algebras of $\DR_{S}(X)$ would be
  $f_{*}\Sym_{\mathcal{O}_{X}}(\mathbb{L}_{X/S}[-1])$, which in
  particular is quasicoherent. In general, we still have:
\end{remark}

\begin{corollary}\label{cor:geomDR}
  For $f \colon X \to S$ a geometric morphism, the underlying
  graded commutative $\k$-algebra of $\DR(X/S)$ is
  $\Gamma_{X}(\Sym_{\mathcal{O}_{X}}(\mathbb{L}_{X/S}[-1]))$.
\end{corollary}
\begin{proof}
  By definition, $\DR(X/S)$ is
  $\lim_{p \in \dAff_{/S}^{\op}} \DR_{p^{*}\mathcal{O}_{X}}(X_{p})$.
  From Corollary~\ref{cor:DRAXisSymLXA} we see that the underlying
  commutative algebra of this is
  \[\lim_{p \in \dAff_{/S}^{\op}}
  f_{p,*}\Sym_{\mathcal{O}_{X_{p}}}(\mathbb{L}_{\mathcal{O}_{X_{p}}/p^{*}\mathcal{O}_{S}}[-1])
  \simeq 
  \lim_{p \in \dAff_{/S}^{\op}}
  f_{p,*}\bar{p}^{*}\Sym_{\mathcal{O}_{X}}(\mathbb{L}_{X/S}[-1]),\]
  using Lemma~\ref{lem:relcotgtpb}.
  By functoriality of pushforwards, we can identify this limit with
  the global sections 
  $\Gamma_{X}(\lim_{p \in \dAff_{/S}^{\op}}
  \bar{p}_{*}\bar{p}^{*}\Sym_{\mathcal{O}_{X}}(\mathbb{L}_{X/S}[-1]))$. Since
  $\QCoh(X) \simeq \lim_{p \in \dAff_{/S}^{\op}} \QCoh(X_{p})$ we have
  that $\mathcal{M} \simeq \lim_{p  \in \dAff_{/S}^{\op}}\bar{p}_{*}\bar{p}^{*}\mathcal{M}$
  for any $\mathcal{M} \in \QCoh(X)$, hence we get
  \[ \DR(X/S) \simeq
  \Gamma_{X}(\Sym_{\mathcal{O}_{X}}(\mathbb{L}_{X/S}[-1])),\]
  as required.
\end{proof}

\begin{remark}
  The identification of Corollary~\ref{cor:geomDR} is natural, meaning
  that we have an equivalence
  \[U\DR(\blank/\blank) \simeq
    \Gamma(\Sym_{\mathcal{O}}(\mathbb{L}_{\blank/\blank}[-1])\]
  of functors
  $\dSt^{\Delta^{1},\txt{geom},\op} \to \GC_{\k}$.
\end{remark}

\begin{proposition}\label{propn:relDRpbaff}
  Suppose given a pullback square \csquare{X'}{X}{S'}{\Spec
    A.}{\xi}{f'}{f}{\sigma}
  Then there is a natural map
  \[ \sigma_{*}\mathcal{O}_{S'} \otimes_{A} \DR_{A}(X) \to
  \DR(X'/S').\]
  If $f$ is geometric and $\sigma$ is cocontinuous, then
  this map is an equivalence.
\end{proposition}
\begin{proof}
  By definition, we have
  \[
    \begin{split}
\DR(X'/S') &  \simeq \lim_{p \in \dAff_{/S'}^{\op}}
  p^{*}\DR_{S'}(X') \\  & \simeq \lim_{p \in \dAff_{/S'}^{\op}}
  p^{*}\sigma^{*}\DR_{A}(X) \\ & \simeq \lim_{p \in \dAff_{/S'}^{\op}}
  p^{*}\mathcal{O}_{S'} \otimes_{A} \DR_{A}(X).
    \end{split}
  \]
  On the other hand, by descent we get 
$\sigma_{*}\mathcal{O}_{S'} \simeq \lim_{p \in \dAff_{/S'}^{\op}}
p^{*}\mathcal{O}_{S'}$, giving
\[\sigma_{*}\mathcal{O}_{S'} \otimes_{A}
\DR_{A}(X) \simeq \left( \lim_{p \in \dAff_{/S'}^{\op}}
  p^{*}\mathcal{O}_{S'} \right) \otimes_{A} \DR_{A}(X),\]
so there is indeed a natural map
  \[ \sigma_{*}\mathcal{O}_{S'} \otimes_{A} \DR_{A}(X) \to
  \DR(X'/S').\]
  To see that this is an equivalence, it suffices to show the
  underlying map of commutative $\k$-algebras is an equivalence. If $f$
  is geometric, then by Corollary~\ref{cor:DRAXisSymLXA} we may
  identify this underlying map with the natural map
  \[ \sigma_{*}\mathcal{O}_{S'} \otimes_{A}
  f_{*}\Sym_{\mathcal{O}_{X}}(\mathbb{L}_{X/A}[-1]) \to
  \sigma_{*}\sigma^{*}f_{*}\Sym_{\mathcal{O}_{X}}(\mathbb{L}_{X/A}[-1]).\]
  This is an equivalence by the projection formula, which holds for
  $\sigma$ by Lemma~\ref{lem:affproj}.
\end{proof}

\begin{remark}
  In particular, if $X$ is a geometric stack and $\Sigma$ is a stack
  such that $\mathcal{O}_{\Sigma}$ is compact in $\QCoh(\Sigma)$
  (which is equivalent to $\Sigma \to \Spec \k$ being
  cocontinuous) then we have an equivalence
  \[ \Gamma_{\Sigma}(\mathcal{O}_{\Sigma}) \otimes \DR(X) \simeq \DR(X
  \times \Sigma / \Sigma).\]
\end{remark}

\begin{corollary}\label{cor:relDRpb}
  Suppose given a pullback square \csquare{Y}{X}{T}{S.}{}{g}{f}{\phi}
  Then there is a natural map
  \[ \Gamma_{S}(\phi_{*}\mathcal{O}_{T} \otimes_{S} \DR_{S}(X)) \to
  \DR(Y/T).\]
  If $f$ is geometric and $\phi$ is universally
  cocontinuous, then this map is an equivalence.
\end{corollary}
\begin{proof}
  By Lemma~\ref{lem:DR/Sdesc} we have an equivalence
  \[ \DR(Y/T) \simeq \lim_{p \in \dAff_{/S}^{\op}}
  \DR(Y_{p}/T_{p}).\]
  Using Proposition~\ref{propn:relDRpbaff} we then get a map
  \[ \lim_{p \in \dAff_{/S}^{\op}} \phi_{p,*}\mathcal{O}_{T_{p}}
  \otimes_{p^{*}\mathcal{O}_{S}} \DR_{p^{*}\mathcal{O}_{S}}(X_{p}) \to
  \lim_{p \in \dAff_{/S}^{\op}} \DR(Y_{p}/T_{p}) \simeq
  \DR(Y/T),\]
  which is an equivalence if $f$ is geometric and $\phi$ is
  universally cocontinuous.
  Here $\phi_{p,*}\mathcal{O}_{T_{p}} \simeq
  \phi_{p,*}\bar{p}^{*}\mathcal{O}_{T}$ so by the naturality of
  Beck-Chevalley transformations we have a map
  \[ \lim_{p \in \dAff_{/S}^{\op}} p^{*}\phi_{*}\mathcal{O}_{T}
  \otimes_{p^{*}\mathcal{O}_{S}} \DR_{p^{*}\mathcal{O}_{S}}(X_{p}) \to
  \lim_{p \in \dAff_{/S}^{\op}} \phi_{p,*}\mathcal{O}_{T_{p}}
  \otimes_{p^{*}\mathcal{O}_{S}} \DR_{p^{*}\mathcal{O}_{S}}(X_{p}),\]
  which is an equivalence if $\phi$ is universally
  cocontinuous by Theorem~\ref{thm:bcpf}.  The
  left-hand side can be identified as
  \[
    \begin{split}
\lim_{p \in \dAff_{/S}^{\op}} p^{*}\phi_{*}\mathcal{O}_{T}
  \otimes_{p^{*}\mathcal{O}_{S}} \DR_{p^{*}\mathcal{O}_{S}}(X_{p})
  & \simeq \lim_{p \in \dAff_{/S}^{\op}}
  p^{*}(\phi_{*}\mathcal{O}_{T} \otimes_{S} \DR_{S}(X)) \\ & \simeq
  \Gamma_{S}(\phi_{*}\mathcal{O}_{T} \otimes_{S} \DR_{S}(X)),
    \end{split}
  \]
  so putting these two morphisms together we indeed get a natural map
  \[ \Gamma_{S}(\phi_{*}\mathcal{O}_{T} \otimes_{S} \DR_{S}(X)) \to
  \DR(Y/T),\]
  which is an equivalence if $f$ is geometric and $\phi$ is
  universally cocontinuous.
\end{proof}

\begin{remark}\label{rmk:pbderhamunderlying}
  On underlying graded commutative $\k$-algebras the map \[ \Gamma_{S}(\phi_{*}\mathcal{O}_{T} \otimes_{S} \DR_{S}(X)) \to
  \DR(Y/T)\]
  is the composite of projection formula and base change morphisms
  \[
  \begin{split}
  \Gamma_{S}(\phi_{*}\mathcal{O}_{T} \otimes f_{*}\Sym
  \mathbb{L}_{X/S}[-1]) & \to \Gamma_{S}(\phi_{*}(\mathcal{O}_{T}
  \otimes \phi^{*}f_{*}\Sym
  \mathbb{L}_{X/S}[-1])) \\ & \simeq \Gamma_{T}(\phi^{*}f_{*}\Sym
  \mathbb{L}_{X/S}[-1]) \\ & \to \Gamma_{T}(g_{*}\xi^{*}\Sym
  \mathbb{L}_{X/S}[-1]) \\ & \simeq \Gamma_{Y}(\Sym
  \mathbb{L}_{Y/T}[-1])    
  \end{split}
\]
\end{remark}

\begin{remark}\label{rmk:drpbbasechange}
  The construction of Corollary~\ref{cor:relDRpb}
  is compatible with base change in the following
  sense: Suppose we have a commutative cube
  \[
    \begin{tikzcd}
      Y' \arrow{rr} \arrow{dr}{\eta} \arrow{dd} & & X'
      \arrow{dr}{\xi} \arrow{dd} \\
      & Y \arrow[crossing over]{rr} & & X\arrow{dd} \\
      T' \arrow{rr}{\phi'}  \arrow{dr}{\tau}& & S' \arrow{dr}{\sigma} \\
      & T \arrow[leftarrow,crossing over]{uu} \arrow{rr}{\phi} && S
    \end{tikzcd}
  \]
  where all faces are cartesian. Then, unwinding the definitions, we
  have a commutative square
  \[
    \begin{tikzcd}
    \Gamma_{S}(\phi_{*}\mathcal{O}_{T} \otimes_{S} \DR_{S}(X))
    \arrow{r} \arrow{d} & \DR(Y/T) \arrow{d}{\eta^{*}} \\
    \Gamma_{S}(\phi'_{*}\mathcal{O}_{T'} \otimes_{S'} \DR_{S'}(X'))
    \arrow{r} & \DR(Y'/T'),
    \end{tikzcd}
  \]
  where the right vertical map is the natural map from
  Remark~\ref{rmk:DRfunctor}, while the left vertical map is the
  composite
  \[
    \begin{split}
    \Gamma_{S}(\phi_{*}\mathcal{O}_{T} \otimes_{S} \DR_{S}(X)) & \to
    \Gamma_{S}(\sigma_{*}\phi'_{*}\mathcal{O}_{T'} \otimes_{S}
    \DR_{S}(X)) \\  & \to \Gamma_{S}\sigma_{*}(\phi'_{*}\mathcal{O}_{T'}
    \otimes_{S} \sigma^{*}\DR_{S}(X)) \\ & \isoto
    \Gamma_{S'}(\phi'_{*}\mathcal{O}_{T'} \otimes_{S} \DR_{S'}(X')),      
    \end{split}
\]
  where the first map is induced by the map
  $\phi_{*}\mathcal{O}_{T} \to \sigma_{*}\phi'_{*}\mathcal{O}_{T'}$
  adjoint to the base change morphism
  $\sigma^{*}\phi_{*}\mathcal{O}_{T} \to
  \phi'_{*}\tau^{*}\mathcal{O}_{T} \simeq \phi'_{*}\mathcal{O}_{T'}$,
  the second can be thought of as the projection formula morphism for
  $\sigma^{*} \dashv \sigma_{*}$, and the third is the equivalence of
  Lemma~\ref{lem:relDRpb}.   
\end{remark}

\section{Differential Forms}\label{subsec:diffform}
In this section we briefly review the definitions of (closed)
differential forms on derived stacks from \cite{PTVV}.

\begin{definition}\label{defn:pforms}
  Let $\k(p)[-p-s] \in \GMM_{\k}$ denote the graded mixed $\k$-module
  given by $\k[-p-s]$ in weight $p$ with $0$ elsewhere. Given
  $X \to S$, the space of \emph{$s$-shifted closed relative $p$-forms}
  is defined by
   \[ \Apcl_{S}(X, s) := \Map_{\GMM_{\k}}(\k(p)[-p-s], \DR(X/S)), 
   \]
   while the space of    \emph{$s$-shifted relative $p$-forms} is
   defined by 
   \[ \Ap_{S}(X, s) := \Map(\k(p)[-p-s], U\DR(X/S)),\]
   where $U$ is the forgetful functor from graded mixed commutative
   algebras to graded commutative algebras.
\end{definition}

\begin{proposition}\label{propn:formscotgt}
  Suppose $f \colon X \to S$ is a geometric morphism. Then there is
  a natural equivalence 
  \[ \Ap_{S}(X, s) \simeq \Map(\k[-s],
  \Gamma_{X}(\Lambda^{p}\mathbb{L}_{X/S})) \simeq
  \Map_{\QCoh(X)}(\mathcal{O}_{X}[-s], \Lambda^{p}\mathbb{L}_{X/S}).\]
If $f$ is moreover a relative Artin stack, then $\Ap_{S}(X,s)$ is also naturally
  equivalent to $\Map_{\QCoh(X)}(\Lambda^{p}\mathbb{T}_{X/Y}, \mathcal{O}_{X}[s])$.
\end{proposition}
\begin{proof}
  From Corollary~\ref{cor:geomDR} we have an equivalence
  \[ \Ap_{S}(X, s) \simeq \Map(\k(p)[-p-s],
  \Gamma_{X}(\Sym_{\mathcal{O}_{X}}(\mathbb{L}_{X/S}[-1])).\]
  Since the weight-$p$ part of
  $\Gamma_{X}(\Sym_{\mathcal{O}_{X}}(\mathbb{L}_{X/S}[-1]))$ is 
  $(\Gamma_{X}\Lambda^{p}\mathbb{L}_{X/S})[-p])$, this gives the equivalences
  \[
    \begin{split}
\Ap_{S}(X, s) &  \simeq \Map(\k[-s-p],
  (\Gamma_{X}\Lambda^{p}\mathbb{L}_{X/S})[-p]) \\ & \simeq \Map(\k[-s],
  \Gamma_{X}\Lambda^{p}\mathbb{L}_{X/S}) \\ & \simeq
  \Map(\mathcal{O}_{X}[-s], \Lambda^{p}\mathbb{L}_{X/S}).
    \end{split}
  \]
  Since we're working over a field of characteristic zero, we may
  identify $(\Lambda^{p}\mathbb{L}_{X/S})^{\vee}$ with
  $\Lambda^{p}\mathbb{T}_{X/S}$, giving the last equivalence for $X \to
  S$ a relative Artin stack.
\end{proof}

\begin{proposition}[\cite{PTVV}*{Proposition 1.11}]\label{propn:Apsheaf}
  The functors $\Apcl_{S}(\blank,s)$ and $\Ap_{S}(\blank, s)$ are sheaves
  in the derived \'{e}tale topology, and thus give limit-preserving
  functors $\dSt_{/S}^{\op} \to \mathcal{S}$. \qed
\end{proposition}

\begin{lemma}\label{lem:twoformftr}
  There is a functor
  $((\dStg_{S})_{/\mathcal{A}^{p}_{S}(s)})^{\op} \to
  \mathcal{Q}^{\Delta^{1}}$ that sends a geometric morphism $X \to S$ equipped with a
  relative $p$-form $\omega$ to the corresponding map $\mathcal{O}_{X} \to
  \otimes^{p}\mathbb{L}_{X/S}[s]$.
\end{lemma}
\begin{proof}
  The forgetful functor $\dStg_{S/\Ap_{S}(s)} \to \dStg$ is the right
  fibration for the functor $\dStgop_{S} \to \mathcal{S}$
  represented by $\Ap_{S}(s) \in \dSt_{S}$, which is
  $\Ap_{S}(\blank,s)$. By Proposition~\ref{propn:formscotgt} the
  functor $\mathcal{A}^{p}_{S}(\blank,s)$ is the composite
  \[ \dStgop_{S} \xto{\Lambda^{p}\mathbb{L}_{(\blank)/S}[s]} \mathcal{Q} \xto{\Map_{\mathcal{Q}}(\mathcal{O}, \blank)} \mathcal{S},\]
 hence the corresponding right fibration fits in a
  pullback square
  \csquare{\dStg_{S/\Ap_{S}(s)}}{(\mathcal{Q}^{\op})_{/\mathcal{O}}}{\dStg_{S}}{\mathcal{Q}^{\op}.}{}{}{}{\Lambda^{p}\mathbb{L}_{(\blank)/S}[s]}
  Composing with the forgetful map
  $(\mathcal{Q}^{\op})_{/\mathcal{O}} \to \Fun(\Delta^{1},\mathcal{Q}^{\op})$ we get a
  functor
  \[ \dStgop_{S/\Ap_{S}(s)} \to \Fun(\Delta^{1},
    \mathcal{Q}),\] which 
  restricts to $\mathcal{O}$ and $\Lambda^{p}\mathbb{L}_{\blank/S}[s]$
  over $0$ and $1$, respectively. We can now form the composite of
  this with the natural transformation
  $\Lambda^{p}\mathbb{L}_{\blank/S} \to
  \otimes^{p}\mathbb{L}_{\blank/S}$.
\end{proof}

\begin{lemma}\label{lem:pbdiffform}
  Suppose given a pullback square
  \csquare{X'}{X}{S'}{S}{\xi}{f'}{f}{\sigma}
  where $f$ is geometric. Then there is a natural equivalence
  \[ \Ap_{S'}(X', s) \simeq \Map(\mathcal{O}_{X'}[-s],
  \xi^{*}\Lambda^{p}\mathbb{L}_{X/S}).\]
  If $\sigma$ is moreover universally cocontinuous, then we also have
  an equivalence
  \[ \Ap_{S'}(X', s) \simeq \Map(\mathcal{O}_{S}[-s],
  \sigma_{*}\mathcal{O}_{S'} \otimes
  f_{*}\Lambda^{p}\mathbb{L}_{X/S}).\] Finally, if in addition $f$ is
relative Artin stack, then we have an identification
\[ \Ap_{S'}(X', s) \simeq
  \Map(\Lambda^{p}\mathbb{T}_{X/S}, f^{*}\sigma_{*}\mathcal{O}_{S'}[s]).\]
\end{lemma}
\begin{proof}
 By Lemmas~\ref{propn:formscotgt} and \ref{lem:relcotgtpb} we have equivalences
\[ \Ap_{S'}(X',s)\simeq \Map_{\QCoh(X')}(\mathcal{O}_{X'}[-s],
\Lambda^{p}\mathbb{L}_{X'/S'}) \simeq
\Map(\mathcal{O}_{X'}[-s], \xi^{*}\Lambda^{p}\mathbb{L}_{X/S}).\]
By adjunction, this is equivalent to
$\Map_{\QCoh(S)}(\mathcal{O}_{S}[-s],
\sigma_{*}f'_{*}\xi^{*}\Lambda^{p}\mathbb{L}_{X/S})$. 
If $\sigma$ is universally cocontinous, then this is
equivalent to 
\[ \Map_{\QCoh(S)}(\mathcal{O}_{S}[-s],
\sigma_{*}\sigma^{*}f_{*}\Lambda^{p}\mathbb{L}_{X/S}) \simeq
\Map_{\QCoh(S)}(\mathcal{O}_{S}[-s], \sigma_{*}\mathcal{O}_{S'}
\otimes f_{*} \Lambda^{p}\mathbb{L}_{X/S})\]
using base change and the projection formula. Alternatively, if $f$ is
a relative Artin stack, then we have
\[ 
\begin{split}
\Ap_{S'}(X',s) & \simeq
\Map_{\QCoh(X')}(\xi^{*}\Lambda^{p}\mathbb{T}_{X/S},
\mathcal{O}_{X'}[s]) \\ & \simeq
\Map_{\QCoh(X)}(\Lambda^{p}\mathbb{T}_{X/S},
\xi_{*}f'^{*}\mathcal{O}_{S'}[s]) \\ & \simeq
\Map_{\QCoh(X)}(\Lambda^{p}\mathbb{T}_{X/S},
f^{*}\sigma_{*}\mathcal{O}_{S'}[s]). \qedhere
\end{split}
\]
\end{proof}

\begin{remark}\label{rmk:formfunctor}
  The functoriality of the de Rham complex, as noted in
  Remark~\ref{rmk:DRfunctor}, makes $\Apcl_{(\blank)}(\blank,s)$ a
  functor $\dSt^{\Delta^{1},\op} \to \mathcal{S}$. Let $\mathrm{ClDF}^{p,s}
  \to \dSt^{\Delta^{1}}$ be the corresponding right fibration. The
  composite
  \[ \mathrm{ClDF}^{p,s} \to \dSt^{\Delta^{1}} \xto{\ev_{1}} \dSt \]
  is then a cartesian fibration, since $\ev_{1}$ is a cartesian
  fibration (as $\dSt$ has pullbacks). Over $S$ in $\dSt$, the right
  fibration $\mathrm{ClDF}^{p,s}_{S} \to \dSt_{S}$ corresponds to the presheaf
  $\Apcl_{S}(\blank,s)$ and so $\mathrm{ClDF}^{p,s}_{S} \simeq
  \dSt_{S/\Apcl_{S}(s)}$. The functor $\dSt^{\op} \to \CatI$
  corresponding to $\mathrm{ClDF}^{p,s} \to \dSt$ thus gives compatible functors
  \[ \sigma^{*} \colon \dSt_{S/\Apcl_{S}(s)} \to
    \dSt_{S'/\Apcl_{S'}(s)} \]
  for $\sigma \colon S' \to S$, given on underlying stacks by pullback
  along $S$. In particular, we have a canonical map
  $\sigma^{*}\Apcl_{S}(s) \to \Apcl_{S'}(s)$ in $\dSt_{S'}$, and the
  cartesian morphism over $\sigma$ with source $(X, \omega \colon X
  \to \Apcl_{S}(s))$ is given by the diagram
  \[
    \begin{tikzcd}
      X \arrow{d} & X' \arrow{l}\arrow{d} \arrow{dr} \\
      \Apcl_{S}(s) \arrow{d} & \sigma^{*}\Apcl_{S}(s) \arrow{l}
      \arrow{d} \arrow{r} & \Apcl_{S'}(s) \arrow{dl} \\
      S & S' \arrow{l}{\sigma}
    \end{tikzcd}
    \]
    where the two squares are cartesian.
\end{remark}

\backmatter




\bib{AKSZ}{article}{
  author={Alexandrov, M.},
  author={Kontsevich, M.},
  author={Schwarz, A.},
  author={Zaboronsky, O.},
  title={The geometry of the master equation and topological quantum field theory},
  journal={Internat. J. Modern Phys. A},
  volume={12},
  date={1997},
  number={7},
  pages={1405--1429},
  eprint={arXiv:hep-th/9502010},
}

\bib{AmorimBenBassatLag}{article}{
  author={Amorim, Lino},
  author={Ben-Bassat, Oren},
  title={Perversely categorified Lagrangian correspondences},
  journal={Adv. Theor. Math. Phys.},
  volume={21},
  date={2017},
  number={2},
  pages={289--381},
  eprint={arXiv:1601.01536},
}

\bib{AntieauGepnerBrauer}{article}{
  author={Antieau, Benjamin},
  author={Gepner, David},
  title={Brauer groups and \'etale cohomology in derived algebraic geometry},
  journal={Geom. Topol.},
  volume={18},
  date={2014},
  number={2},
  pages={1149--1244},
}

\bib{AtiyahTQFT}{article}{
  author={Atiyah, Michael},
  title={Topological quantum field theories},
  journal={Inst. Hautes \'Etudes Sci. Publ. Math.},
  number={68},
  date={1988},
  pages={175--186},
}

\bib{AyalaFrancisFib}{article}{
  eprint={arXiv:1702.02681},
  author={Ayala, David},
  author={Francis, John},
  title={Fibrations of $\infty $-categories},
  journal={High. Struct.},
  volume={4},
  date={2020},
  number={1},
  pages={168--265},
}

\bib{BaezDolanTQFT}{article}{
  author={Baez, John C.},
  author={Dolan, James},
  title={Higher-dimensional algebra and topological quantum field theory},
  journal={J. Math. Phys.},
  volume={36},
  date={1995},
  number={11},
  pages={6073--6105},
}

\bib{BarwickThesis}{book}{
  author={Barwick, Clark},
  title={$(\infty ,n)$-{C}at as a closed model category},
  note={Thesis (Ph.D.)--University of Pennsylvania},
  date={2005},
}

\bib{BasterraTAQ}{article}{
  author={Basterra, M.},
  title={Andr\'e-Quillen cohomology of commutative $S$-algebras},
  journal={J. Pure Appl. Algebra},
  volume={144},
  date={1999},
  number={2},
  pages={111--143},
}

\bib{BenZviFrancisNadler}{article}{
  author={Ben-Zvi, David},
  author={Francis, John},
  author={Nadler, David},
  title={Integral transforms and Drinfeld centers in derived algebraic geometry},
  journal={J. Amer. Math. Soc.},
  volume={23},
  date={2010},
  number={4},
  pages={909--966},
}

\bib{CalaqueTFT}{article}{
  author={Calaque, Damien},
  title={Lagrangian structures on mapping stacks and semi-classical {T}{F}{T}s},
  conference={ title={Stacks and categories in geometry, topology, and algebra}, },
  book={ series={Contemp. Math.}, volume={643}, publisher={Amer. Math. Soc., Providence, RI}, },
  date={2015},
  pages={1--23},
  eprint={arXiv:1306.3235},
}

\bib{CalCot}{article}{
  author={Calaque, Damien},
  title={Shifted cotangent stacks are shifted symplectic},
  journal={Annales de la Facult\'e des sciences de Toulouse: Math\'ematiques},
  pages={67--90},
  publisher={Universit\'e Paul Sabatier, Toulouse},
  volume={Ser. 6, 28},
  number={1},
  year={2019},
  eprint={arXiv:1612.08101},
}

\bib{CS}{article}{
  author={Calaque, Damien},
  author={Scheimbauer, Claudia},
  title={A note on the $(\infty ,n)$-category of cobordisms},
  journal={Algebr. Geom. Topol.},
  volume={19},
  date={2019},
  number={2},
  pages={533--655},
  eprint={arXiv:1509.08906},
}

\bib{CPTVV}{article}{
  author={Calaque, Damien},
  author={Pantev, Tony},
  author={To\"{e}n, Bertrand},
  author={Vaqui\'{e}, Michel},
  author={Vezzosi, Gabriele},
  title={Shifted Poisson structures and deformation quantization},
  journal={J. Topol.},
  volume={10},
  date={2017},
  number={2},
  pages={483--584},
  eprint={arXiv:1506.03699},
}

\bib{CatFel}{article}{
  author={Cattaneo, Alberto S.},
  author={Felder, Giovanni},
  title={On the AKSZ formulation of the Poisson sigma model},
  journal={Lett. Math. Phys.},
  volume={56},
  date={2001},
  pages={163--179},
}

\bib{CMNBV}{article}{
  author={Cattaneo, Alberto S.},
  author={Mnev, Pavel},
  author={Reshetikhin, Nicolai},
  title={Classical BV theories on manifolds with boundary},
  journal={Comm. Math. Phys.},
  volume={332},
  date={2014},
  number={2},
  pages={535--603},
}

\bib{CisinskiCat}{book}{
  author={Cisinski, Denis-Charles},
  title={Higher categories and homotopical algebra},
  series={Cambridge Studies in Advanced Mathematics},
  volume={180},
  publisher={Cambridge University Press, Cambridge},
  date={2019},
  note={Available from \url {http://www.mathematik.uni-regensburg.de/cisinski/CatLR.pdf}.},
}

\bib{CostelloGwilliam2}{book}{
  author={Costello, Kevin},
  author={Gwilliam, Owen},
  title={Factorization algebras in quantum field theory. Volume 2},
  fjournal={New Mathematical Monographs},
  journal={New Math. Monogr.},
  volume={41},
  year={2021},
  publisher={Cambridge: Cambridge University Press},
}

\bib{DimcaSheaves}{book}{
  author={Dimca, Alexandru},
  title={Sheaves in topology},
  series={Universitext},
  publisher={Springer-Verlag, Berlin},
  date={2004},
}

\bib{GaitsgoryRozenblyum1}{book}{
  author={Gaitsgory, Dennis},
  author={Rozenblyum, Nick},
  title={A study in derived algebraic geometry. Vol. I. Correspondences and duality},
  series={Mathematical Surveys and Monographs},
  volume={221},
  publisher={American Mathematical Society, Providence, RI},
  date={2017},
  note={Available from \url {http://www.math.harvard.edu/~gaitsgde/GL/}.},
}

\bib{GaitsgoryRozenblyum2}{book}{
  author={Gaitsgory, Dennis},
  author={Rozenblyum, Nick},
  title={A study in derived algebraic geometry. Vol. II. Deformations, Lie theory and formal geometry},
  series={Mathematical Surveys and Monographs},
  volume={221},
  publisher={American Mathematical Society, Providence, RI},
  date={2017},
  note={Available from \url {http://www.math.harvard.edu/~gaitsgde/GL/}.},
}

\bib{GRW}{article}{
  author={Galatius, S\o ren},
  author={Randal-Williams, Oscar},
  title={Monoids of moduli spaces of manifolds},
  journal={Geom. Topol.},
  volume={14},
  date={2010},
  number={3},
  pages={1243--1302},
  issn={1465-3060},
  review={\MR {2653727}},
  doi={10.2140/gt.2010.14.1243},
}

\bib{GMTW}{article}{
  author={Galatius, S\o ren},
  author={Tillmann, Ulrike},
  author={Madsen, Ib},
  author={Weiss, Michael},
  title={The homotopy type of the cobordism category},
  journal={Acta Math.},
  volume={202},
  date={2009},
  number={2},
  pages={195--239},
  issn={0001-5962},
  review={\MR {2506750}},
  doi={10.1007/s11511-009-0036-9},
}

\bib{freepres}{article}{
  author={Gepner, David},
  author={Haugseng, Rune},
  author={Nikolaus, Thomas},
  title={Lax colimits and free fibrations in $\infty $-categories},
  journal={Doc. Math.},
  volume={22},
  date={2017},
  pages={1225--1266},
  eprint={arXiv:1501.02161},
}

\bib{spans}{article}{
  author={Haugseng, Rune},
  title={Iterated spans and classical topological field theories},
  journal={Math. Z.},
  volume={289},
  issue={3},
  pages={1427--1488},
  date={2018},
  eprint={arXiv:1409.0837},
}

\bib{paradj}{article}{
  author={Haugseng, Rune},
  author={Hebestreit, Fabian},
  author={Linskens, Sil},
  author={Nuiten, Joost},
  title={Lax monoidal adjunctions, two-variable fibrations and the calculus of mates},
  date={2021},
  eprint={arXiv:2011.08808},
}

\bib{cois}{article}{
  author={Haugseng, Rune},
  author={Melani, Valerio},
  author={Safronov, Pavel},
  title={Shifted coisotropic correspondences},
  journal={J. Inst. Math. Jussieu},
  volume={21},
  date={2022},
  number={3},
  pages={785--849},
  eprint={arXiv:1904.11312},
}

\bib{HalpernLeistnerPreygel}{article}{
  author={Halpern-Leistner, Daniel},
  author={Preygel, Anatoly},
  title={Mapping stacks and categorical notions of properness},
  date={2014},
  eprint={arXiv:1402.3204},
}

\bib{HinichYoneda}{article}{
  author={Hinich, Vladimir},
  title={Yoneda lemma for enriched $\infty $-categories},
  journal={Adv. Math.},
  volume={367},
  date={2020},
  pages={107129},
  eprint={arXiv:1805.07635},
}

\bib{Ikeda}{article}{
  author={Ikeda, Noriaki},
  title={Two-dimensional gravity and nonlinear gauge theory},
  journal={Annals Phys.},
  volume={235},
  date={1994},
  number={2},
  pages={435--464},
}

\bib{IversenSheaves}{book}{
  author={Iversen, Birger},
  title={Cohomology of sheaves},
  series={Universitext},
  publisher={Springer-Verlag, Berlin},
  date={1986},
}

\bib{Kosinski}{book}{
  author={Kosinski, Antoni A.},
  title={Differential manifolds},
  series={Pure and Applied Mathematics},
  volume={138},
  publisher={Academic Press, Inc., Boston, MA},
  date={1993},
}

\bib{Laures}{article}{
  author={Laures, Gerd},
  title={On cobordism of manifolds with corners},
  journal={Trans. Amer. Math. Soc.},
  volume={352},
  date={2000},
  number={12},
  pages={5667--5688},
}

\bib{DAG5}{article}{
  author={Lurie, Jacob},
  title={Derived algebraic geometry {V}: Structured spaces},
  date={2009},
  eprint={http://math.ias.edu/~lurie/papers/DAG-V.pdf},
}

\bib{HTT}{book}{
  author={Lurie, Jacob},
  title={Higher Topos Theory},
  series={Annals of Mathematics Studies},
  publisher={Princeton University Press},
  address={Princeton, NJ},
  date={2009},
  volume={170},
  note={Available at \url {http://math.ias.edu/~lurie/papers/highertopoi.pdf}},
}

\bib{LurieGoodwillie}{article}{
  author={Lurie, Jacob},
  title={($\infty $,2)-Categories and the {G}oodwillie Calculus {I}},
  date={2009},
  eprint={http://math.ias.edu/~lurie/papers/GoodwillieI.pdf},
}

\bib{LurieCob}{article}{
  author={Lurie, Jacob},
  title={On the classification of topological field theories},
  conference={ title={Current developments in mathematics, 2008}, },
  book={ publisher={Int. Press, Somerville, MA}, },
  date={2009},
  pages={129--280},
  eprint={http://math.ias.edu/~lurie/papers/cobordism.pdf},
}

\bib{DAG7}{article}{
  author={Lurie, Jacob},
  title={Derived algebraic geometry {VII}: spectral schemes},
  date={2011},
  eprint={http://math.ias.edu/~lurie/papers/DAG-VII.pdf},
}

\bib{HA}{book}{
  author={Lurie, Jacob},
  title={Higher Algebra},
  date={2014},
  note={Available at \url {http://math.ias.edu/~lurie/papers/higheralgebra.pdf}},
}

\bib{SAG}{book}{
  author={Lurie, Jacob},
  title={Spectral Algebraic Geometry},
  date={2017},
  note={Available at \url {http://math.ias.edu/~lurie/}},
}

\bib{MadsenBoekstedt}{article}{
  author={B\"{o}kstedt, Marcel},
  author={Madsen, Ib},
  title={The cobordism category and Waldhausen's $K$-theory},
  conference={ title={An alpine expedition through algebraic topology}, },
  book={ series={Contemp. Math.}, volume={617}, publisher={Amer. Math. Soc., Providence, RI}, },
  date={2014},
  pages={39--80},
}

\bib{PTVV}{article}{
  author={Pantev, Tony},
  author={To{\"e}n, Bertrand},
  author={Vaqui{\'e}, Michel},
  author={Vezzosi, Gabriele},
  title={Shifted symplectic structures},
  journal={Publ. Math. Inst. Hautes \'Etudes Sci.},
  volume={117},
  date={2013},
  pages={271--328},
}

\bib{PortaVezzosiSqZ}{article}{
  author={Porta, Mauro},
  author={Vezzosi, Gabriele},
  title={Infinitesimal and square-zero extensions of simplicial algebras},
  date={2015},
  eprint={arXiv:1310.3573},
}

\bib{PymSafronov}{article}{
  author={Pym, Brent},
  author={Safronov, Pavel},
  title={Shifted symplectic Lie algebroids},
  journal={Int. Math. Res. Not. IMRN},
  volume={2020},
  date={2020},
  number={21},
  pages={7489--7557},
}

\bib{QuillenK}{article}{
  author={Quillen, Daniel},
  title={Higher algebraic $K$-theory. I},
  conference={ title={Algebraic $K$-theory, I: Higher $K$-theories (Proc. Conf., Battelle Memorial Inst., Seattle, Wash., 1972)}, },
  book={ publisher={Springer}, place={Berlin}, },
  date={1973},
  pages={85--147. Lecture Notes in Math., Vol. 341},
}

\bib{RezkCSS}{article}{
  author={Rezk, Charles},
  title={A model for the homotopy theory of homotopy theory},
  journal={Trans. Amer. Math. Soc.},
  volume={353},
  date={2001},
  number={3},
  pages={973--1007 (electronic)},
}

\bib{RW}{article}{
  author={Rozansky, Lev},
  author={Witten, Edward},
  title={Hyper-K\"ahler geometry and invariants of three-manifolds},
  journal={Sel. math., New ser.},
  volume={3},
  date={1997},
  pages={401--458},
}

\bib{SchStr}{article}{
  author={Schaller, Peter},
  author={Strobl, Thomas},
  title={Poisson structure induced (topological) field theories},
  journal={ Modern Phys. Lett. A},
  volume={9},
  date={1994},
  number={33},
  pages={3219--3136},
}

\bib{SegalCatCohlgy}{article}{
  author={Segal, Graeme},
  title={Categories and cohomology theories},
  journal={Topology},
  volume={13},
  date={1974},
  pages={293--312},
}

\bib{Stevenson}{article}{
  author={Stevenson, Danny},
  title={Model Structures for correspondences and bifibrations},
  date={2018},
  eprint={arXiv:1807.08226},
}

\bib{HAG2}{article}{
  author={To{\"e}n, Bertrand},
  author={Vezzosi, Gabriele},
  title={Homotopical algebraic geometry {I}{I}: geometric stacks and applications},
  journal={Mem. Amer. Math. Soc.},
  volume={193},
  date={2008},
  number={902},
  eprint={arXiv:math/0404373},
}

\bib{ToenEMS}{article}{
  author={To\"en, Bertrand},
  title={Derived Algebraic Geometry},
  journal={EMS Surv. Math. Sci.},
  volume={1},
  date={2014},
  number={2},
  pages={153--240},
}

\bib{WeinsteinSymplGeom}{article}{
  author={Weinstein, Alan},
  title={Symplectic geometry},
  journal={Bull. Amer. Math. Soc. (N.S.)},
  volume={5},
  date={1981},
  number={1},
  pages={1--13},
}

\bib{WeinsteinSymplCat}{article}{
  author={Weinstein, Alan},
  title={The symplectic ``category''},
  conference={ title={Differential geometric methods in mathematical physics (Clausthal, 1980)}, },
  book={ series={Lecture Notes in Math.}, volume={905}, publisher={Springer, Berlin-New York}, },
  date={1982},
  pages={45--51},
}



\begin{bibdiv}
 \begin{biblist}
     \bibselect{refs}
 \end{biblist}
 \end{bibdiv}
\end{document}